\documentclass[10pt,a4paper,reqno]{amsart}

\usepackage[utf8]{inputenc}
\usepackage[T1]{fontenc}
\usepackage[largesc]{newtxtext}
\usepackage{newtxmath}
\usepackage{microtype}
\usepackage{amsmath}
\usepackage{amsthm}
\usepackage{amsfonts}
\usepackage{thmtools}
\usepackage{relsize}
\usepackage[mathscr]{eucal}
\usepackage[linktocpage]{hyperref}
\usepackage[sort,capitalize]{cleveref}
\usepackage{enumitem}			
\usepackage[dvipsnames]{xcolor}
\usepackage[msc-links, abbrev, non-sorted-cites]{amsrefs}
\usepackage{todonotes}
\usepackage{caption}

\newcommand{\filledbox}{%
    \begin{tikzpicture}[baseline={(0,0)}]
        \draw[line width=1pt] (0,0) rectangle (.5em,.5em); 
    \end{tikzpicture}%
}

\makeatletter
\def\nonumberfootnote{\xdef\@thefnmark{}\@footnotetext}			
\makeatother

\definecolor{colorred}{HTML}{B00000}
\definecolor{colorgreen}{HTML}{258300}
\definecolor{colorblue}{HTML}{2e32fa}
\definecolor{coloryellow}{HTML}{cbbb1a}
\hypersetup{colorlinks=true, linkcolor=colorred, citecolor=colorgreen, urlcolor=coloryellow, pdfusetitle=true}

\numberwithin{equation}{section}

\newcommand{\nlb}{{\ensuremath{\textnormal{b}}}}

\newcommand{\nlC}{{\ensuremath{\textnormal{C}}}}

\newcommand{\rmc}{{\ensuremath{\mathrm{c}}}}
\newcommand{\rmd}{{\ensuremath{\mathrm{d}}}}
\newcommand{\rme}{{\ensuremath{\mathrm{e}}}}

\newcommand{\bdH}{{\ensuremath{\boldsymbol{H}}}}

\newcommand{\sfd}{{\ensuremath{\mathsf{d}}}}
\newcommand{\sfe}{{\ensuremath{\mathsf{e}}}}

\newcommand{\sfg}{{\ensuremath{\mathsf{g}}}}
\newcommand{\sfh}{{\ensuremath{\mathsf{h}}}}

\newcommand{\sfp}{{\ensuremath{\mathsf{p}}}}

\newcommand{\sfG}{{\ensuremath{\mathsf{G}}}}
\newcommand{\sfH}{{\ensuremath{\mathsf{H}}}}

\newcommand{\sfJ}{{\ensuremath{\mathsf{J}}}}

\newcommand{\sfP}{{\ensuremath{\mathsf{P}}}}
\newcommand{\sfQ}{{\ensuremath{\mathsf{Q}}}}
\newcommand{\sfR}{{\ensuremath{\mathsf{R}}}}

\newcommand{\scrB}{{\ensuremath{\mathscr{B}}}}
\newcommand{\scrC}{{\ensuremath{\mathscr{C}}}}
\newcommand{\scrD}{{\ensuremath{\mathscr{D}}}}
\newcommand{\scrE}{{\ensuremath{\mathscr{E}}}}

\newcommand{\scrH}{{\ensuremath{\mathscr{H}}}}

\newcommand{\scrM}{{\ensuremath{\mathscr{M}}}}

\newcommand{\scrP}{{\ensuremath{\mathscr{P}}}}

\newcommand{\scrT}{{\ensuremath{\mathscr{T}}}}
\newcommand{\scrU}{{\ensuremath{\mathscr{U}}}}
\newcommand{\scrV}{{\ensuremath{\mathscr{V}}}}

\newcommand{\bdpi}{{\ensuremath{\boldsymbol{\pi}}}}

\newcommand{\N}{\boldsymbol{\mathrm{N}}}						
\newcommand{\R}{\boldsymbol{\mathrm{R}}}						
\renewcommand{\d}{\,\mathrm{d}}				

\let\limsup\undefined
\let\liminf\undefined
\let\div\undefined
\DeclareMathOperator*{\limsup}{limsup}		
\DeclareMathOperator*{\liminf}{liminf}		
\DeclareMathOperator*{\esssup}{esssup}		
\DeclareMathOperator*{\essinf}{essinf}		
\DeclareMathOperator{\supp}{spt}			
\DeclareMathOperator{\div}{div}				

\theoremstyle{definition}
\newtheorem{bump}{Bump}[section]

\theoremstyle{plain}
\newtheorem{theorem}[bump]{Theorem}
\newtheorem{proposition}[bump]{Proposition}
\newtheorem{definition}[bump]{Definition}
\newtheorem{lemma}[bump]{Lemma}
\newtheorem{corollary}[bump]{Corollary}

\theoremstyle{remark}
\newtheorem{remark}[bump]{Remark}
\newtheorem{example}[bump]{Example}

\crefname{theorem}{Theorem}{Theorems}
\crefname{proposition}{Proposition}{Propositions}
\crefname{definition}{Definition}{Definitions}
\crefname{lemma}{Lemma}{Lemmas}
\crefname{corollary}{Corollary}{Corollaries}
\crefname{hypothesis}{Hypothesis}{Hypotheses}
\crefname{remark}{Remark}{Remarks}
\crefname{example}{Example}{Examples}
\crefname{notation}{Notation}{Notations}

\crefformat{section}{{§}#2#1#3}
\crefformat{subsection}{{§}#2#1#3}
\crefformat{subsubsection}{{§}#2#1#3}
\crefformat{appendix}{{§}#2#1#3}

\crefmultiformat{theorem}{Theorems #2#1#3}{ and #2#1#3}{, #2#1#3}{, and #2#1#3}
\crefmultiformat{proposition}{Propositions #2#1#3}{ and #2#1#3}{, #2#1#3}{, and #2#1#3}
\crefmultiformat{definition}{Definitions #2#1#3}{ and #2#1#3}{, #2#1#3}{, and #2#1#3}
\crefmultiformat{lemma}{Lemmas #2#1#3}{ and #2#1#3}{, #2#1#3}{, and #2#1#3}
\crefmultiformat{corollary}{Corollaries #2#1#3}{ and #2#1#3}{, #2#1#3}{, and #2#1#3}
\crefmultiformat{hypothesis}{Hypotheses #2#1#3}{ and #2#1#3}{, #2#1#3}{, and #2#1#3}
\crefmultiformat{remark}{Remarks #2#1#3}{ and #2#1#3}{, #2#1#3}{, and #2#1#3}
\crefmultiformat{example}{Examples #2#1#3}{ and #2#1#3}{, #2#1#3}{, and #2#1#3}
\crefmultiformat{notation}{Notations #2#1#3}{ and #2#1#3}{, #2#1#3}{, and #2#1#3}

\crefmultiformat{section}{{§§}#2#1#3}{ and #2#1#3}{, #2#1#3}{, and #2#1#3}
\crefmultiformat{subsection}{{§§}#2#1#3}{ and #2#1#3}{, #2#1#3}{, and #2#1#3}
\crefmultiformat{subsubsection}{{§§}#2#1#3}{ and #2#1#3}{, #2#1#3}{, and #2#1#3}

\crefrangeformat{equation}{#3\textcolor{black}{(}#1\textcolor{black}{)}#4 to #5\textcolor{black}{(}#2\textcolor{black}{)}#6}

\newcommand{\mms}{\mathsf{M}}				
\newcommand{\met}{\sfd}						

\newcommand{\Meas}{\mathfrak{M}}
\newcommand{\Rmet}{g}						
\newcommand{\meas}{\mathfrak{m}}				

\newcommand{\Leb}{\mathscr{L}}				
\newcommand{\vol}{\mathrm{vol}}				
\newcommand{\Prob}{\mathscr{P}}		        
\newcommand{\OptTGeo}{\mathrm{OptTGeo}}
\newcommand{\ac}{{\mathrm{ac}}}

\newcommand{\TCD}{\mathsf{TCD}}

\newcommand{\TMCP}{\mathsf{TMCP}}
\newcommand{\MCP}{\mathsf{MCP}}
\newcommand{\CD}{\mathsf{CD}}

\newcommand{\bounded}{\nlb}					
\newcommand{\comp}{\rmc}						
\newcommand{\bc}{\mathrm{bc}}
\newcommand{\loc}{\mathrm{loc}}				
\newcommand{\pr}{\mathrm{pr}}				
\newcommand{\HS}{{\mathrm{HS}}}				
\newcommand{\Ric}{\mathrm{Ric}}				
\newcommand{\Cont}{\nlC}					
\newcommand{\LCC}{\mathrm{LCC}}
\newcommand{\Ell}{\mathit{L}}				

\newcommand{\End}{\mathrm{bp}}

\newcommand{\Dom}{\scrD}					
\DeclareMathOperator{\BOX}{\filledbox}

\DeclareMathOperator{\Ent}{Ent}				
\DeclareMathOperator{\Hess}{Hess}			
\DeclareMathOperator{\diam}{diam}			
\newcommand{\eval}{\sfe}					

\newcommand{\push}{\sharp}					

\newcommand{\cl}{\mathrm{cl}}				
\newcommand{\Len}{\mathrm{Len}}
\newcommand{\TGeo}{\mathrm{TGeo}}

\newcommand{\One}{1}
\newcommand{\ray}{\sfh}

\newcommand{\Quot}{\mathsf{Q}}

\usepackage{blindtext}

\allowdisplaybreaks

\let\oldtocsection=\tocsection

\let\oldtocsubsection=\tocsubsection

\let\oldtocsubsubsection=\tocsubsubsection

\renewcommand{\tocsection}[2]{\hspace{0em}\oldtocsection{#1}{#2}}
\renewcommand{\tocsubsection}[2]{\hspace{1em}\oldtocsubsection{#1}{#2}}
\renewcommand{\tocsubsubsection}[2]{\hspace{2em}\oldtocsubsubsection{#1}{#2}}

\newcommand{\nocontentsline}[3]{}
\newcommand{\tocless}[2]{\bgroup\let\addcontentsline=\nocontentsline#1{#2}\egroup}

\newcommand{\mres}{\mathbin{\vrule height 1.6ex depth 0pt width 0.13ex\vrule height 0.13ex depth 0pt width 1.3ex}}

\newcommand{\cost}{\mathfrak{c}}
\newcommand{\SIN}{\sin}
\newcommand{\COS}{\cos}
\newcommand{\Tr}{\scrT}
\renewcommand{\u}{\mathtt{u}}
\renewcommand{\v}{\mathtt{v}}
\newcommand{\w}{\mathsf{w}}
\renewcommand{\q}{\mathfrak{q}}

\DeclareMathOperator{\COMP}{comp}
\newcommand{\Diff}{\boldsymbol{\mathrm{D}}}

\DeclareMathOperator{\LCR}{\mathrm{LCR}}
\DeclareMathOperator{\RCR}{\mathrm{RCR}}
\DeclareMathOperator{\sgn}{sgn}
\newcommand{\Pert}{\mathrm{Pert}}
\newcommand{\FPert}{\mathrm{FPert}}

\newcommand{\TCut}{\mathrm{TC}}
\newcommand{\ETCut}{\mathrm{ETC}}
\newcommand{\hh}{\mathfrak{h}}

\newcommand{\PP}{{p}}
\newcommand{\QQ}{{q}}

\allowdisplaybreaks

\makeatletter
\@namedef{subjclassname@2020}{\textup{2020} Mathematics Subject Classification}
\makeatother

\begin{document}

\title[Exact d'Alembertian for Lorentz distance functions]{Exact d'Alembertian for Lorentz distance functions}
\author{Mathias Braun}
\address{Institute of Mathematics, EPFL, 1015 Lausanne, Switzerland}
\email{\href{mailto:mathias.braun@epfl.ch}{mathias.braun@epfl.ch}}
\thanks{I am grateful to Nicola Gigli, Robert McCann, and Clemens Sämann for their valuable comments.}
\thanks{Financial support by the EPFL through a Bernoulli Instructorship is gratefully acknowledged. A large part of this work was carried out during my Postdoctoral Fellowship at the University of Toronto.}

\begin{abstract} We refine a recent distributional notion of d'Alembertian of a signed Lorentz distance function  to an achronal set in a metric measure spacetime obeying the timelike measure contraction property. We show precise representation formulas and comparison estimates (both upper and lower bounds). Under a condition we call ``infinitesimally strict concavity'' (known for infinitesimally Minkowskian structures and established here  for Finsler spacetimes), we prove the associated  distribution is a signed measure certifying the integration by parts formula. This   treatment of the d'Alembertian using techniques from metric geometry expands upon its recent nonlinear yet elliptic interpretation; even in the smooth case, our formulas seem to pioneer its exact shape across the timelike cut locus. Two central ingredients our  contribution unifies are the localization paradigm of Cavalletti--Mondino and the  Sobolev calculus of Beran--Braun--Calisti--Gigli--McCann--Ohanyan--Rott--S\"amann.

In the second part of our work, we present several applications of these insights. First, we show the  equivalence of the timelike curvature-dimension condition with a Bochner-type inequality. Second, we set up synthetic mean curvature (as well as barriers for CMC sets) exactly. Third, we prove synthetic volume and area estimates of Heintze--Karcher-type, which enable  us to show several synthetic volume singularity theorems.
\end{abstract}

\maketitle

\thispagestyle{empty}

\tableofcontents

\addtocontents{toc}{\protect\setcounter{tocdepth}{2}}

\section{Introduction}\label{Ch:Intro}

\subsection{Objective} Recent advances to questions from   mathematical general relativity from a new family of Lagrangians suggest the nonlinear \emph{$\PP$-d'Alembertian}
\begin{align}\label{Eq:Lapl formally}
\Box_\PP\u := \div\!\big[\vert\rmd\u\vert^{\PP-2}\,\nabla\u\big] 
\end{align}
exhibits elliptic characteristics, unlike the classical linear yet hyperbolic d'Alembertian --- here $\PP\in (-\infty,1)\setminus \{0\}$ and the function $\u$ has timelike gradient. This ellipticity was first realized by Beran--Braun--Calisti--Gigli--McCann--Ohanyan--Sämann--Rott \cite{beran-braun-calisti-gigli-mccann-ohanyan-rott-samann+-} (abbreviated by \emph{Beran et al.} from now on). Motivated by this and a family of comparison theorems for $\smash{\Box_\PP}$ they established, they introduced a distributional notion of $\PP$-d'Alembertian based on the integration by parts formula. Starting from this and reflecting preceding approaches in metric measure geometry by Gigli \cite{gigli2015}, Cavalletti--Mondino \cite{cavalletti-mondino2020-new}, and Ketterer \cite{ketterer2017}, the ambition of this exposition is two-fold.


First, using tools from optimal transport and convex geometry --- notably  Cavalletti--Mondino's localization paradigm  \cite{cavalletti-mondino2020} ---, we prove exact   representation formulas (which include the timelike cut locus) and comparison theorems for $\PP$-d'Alembertians of signed Lorentz distance functions (and more general $1$-steep functions) and their powers. This \emph{constructive} approach to  existence of these operators  complements the \emph{abstract} results derived by Beran et al.~\cite{beran-braun-calisti-gigli-mccann-ohanyan-rott-samann+-}. The distributions are  upgraded to generalized signed Radon measures under an anti-Lipschitz condition of Kunzinger--Steinbauer \cite{kunzinger-steinbauer2022} (inspired by Sormani--Vega \cite{sormani-vega2016}) on the function in question. Moreover, inspired by work of Gigli  \cite{gigli2015} in positive signature, we will identify a condition called \emph{infinitesimal strict concavity} --- which we show for general Finsler spacetimes --- relative to certain causal functions $\u$  implying the considered $\PP$-d'Alembertian of $\u$ is unique. 

Second, our formulas yield a synthetic Bochner-type inequality, set up a nonsmooth  notion of mean curvature, and help derive volume singularity theorems (a concept recently  introduced by García-Heveling \cite{garcia-heveling2023-volume} inspired by Treude--Grant \cite{treude-grant2013}) from separately shown synthetic volume and area inequalities of Heintze--Karcher-type \cite{heintze-karcher1978}. 

Our results  hold in the general abstract framework of metric measure spacetimes with constant or variable timelike Ricci curvature bounds after Cavalletti--Mondino \cite{cavalletti-mondino2020} or Braun--McCann \cite{braun-mccann2023}, respectively.

\subsection{Motivation} Before describing our results and related literature in  \cref{Sub:Contri}, we list some of the   motivations that underpin the recent attention the $\PP$-d'Alembertian has received. In particular, we outline manifested and expected features of the operator \eqref{Eq:Lapl formally}.

\subsubsection{Laplacian in Riemannian signature}\label{Sub:Lapl Riem} The Laplacian  is \emph{the} ubiquitous object in the study of Riemannian manifolds and more general metric measure spaces. The interplay of analysis, probability, and geometry it supports flourishes in the presence of classical or synthetic  lower bounds on the Ricci curvature (in the sense of Sturm \cite{sturm2006-i,sturm2006-ii} and Lott--Villani \cite{lott-villani2009}), often based on the Laplace comparison theorem and Bochner's  inequality. 
Classically, the Laplace comparison theorem (cf.~the book of Cheeger--Ebin \cite{cheeger-ebin1975} for Riemannian comparison theory) gives an upper bound on the Laplacian of a \emph{distance function} in terms of geometric data. This is relevant e.g.~for Cheeger--Gromoll's splitting theorem \cite{cheeger-gromoll1972} and the interrelation of mean curvature \cite{heintze-karcher1978,choe-fraser2018}, De Giorgi's minimal surfaces \cite{giusti1984}, and the isoperimetric problem \cite{cavalletti-mondino2017-isoperimetric,antonelli-pasqualetto-pozzetta-semola2022,mondino-semola2023+} (see also Caffarelli--Sire's  survey \cite{caffarelli-sire2020}).

As well-known, at cut points distance functions cease to be smooth.  Although the cut locus typically has measure zero, a nonsmooth interpretation of their Laplacians and said bounds on them is needed across this degeneracy for global considerations. Traditionally, this is tackled in the barrier sense of Calabi \cite{calabi1958}, distributionally by Cheeger--Gromoll \cite{cheeger-gromoll1972}, and in the viscosity sense of Crandall--Lions \cite{crandall-lions1983} (see also Wu \cite{wu1979}). The intermediate approach relies on the integration by parts formula, while the two extremal viewpoints come from the maximum principle. Its availability follows from   \emph{ellipticity}, a key feature of the Laplacian in Riemannian signature.

Via the more flexible toolkit from metric measure geometry, several related singular Laplacians for distance functions and Laplace comparison theorems have been proposed over time. Two  distributional approaches based upon the integration by parts formula will be most influential to our work.
\begin{itemize}
\item Gigli \cite{gigli2015} employed a metric form of the Legendre--Fenchel duality of cotangent and tangent spaces (Lorentzified by Beran et al. \cite{beran-braun-calisti-gigli-mccann-ohanyan-rott-samann+-}). His Laplacian, meaningful on general metric measure spaces, is a priori multivalued, as the ``integration by parts formula'' may generally fail to be an equality. Yet, he pointed out the vast class of \emph{infinitesimally strictly convex} metric measure spaces containing e.g.~Finsler manifolds which do not exhibit this ambiguity. Gigli's comparison theorem for CD spaces \cite{gigli2015} was  a key  ingredient for his splitting theorem for RCD spaces \cite{gigli2013}. 
\item The approach of \cite{gigli2015} was adapted by  Cavalletti--Mondino \cite{cavalletti-mondino2020-new}. They combined \cite{gigli2015} with  their localization paradigm for CD spaces \cite{cavalletti-mondino2017-isoperimetric} (Lorentzified by them \cite{cavalletti-mondino2020}). By reducing the task of ``removing the derivative from the test function'' to a one-dimensional problem, they derived exact representation formulas for Laplacians of distance functions. Using these, \cite{cavalletti-mondino2020-new} generalized upper and established \emph{lower} comparison inequalities. The precise behavior of the considered Laplacians on the cut locus can be read off from their formulas. This adds a notable value to the mere information a priori provided by the Laplace comparison theorem that the distributional Laplacian in question has a sign on its singular set.
\end{itemize}

\subsubsection{D'Alembertian in Lorentzian signature}\label{Sub:DAL LOR} Lorentzian manifolds constitute the default setting of mathematical general relativity. Unlike \cref{Sub:Lapl Riem}, their induced ``Laplacian'', the  \emph{d'Alembertian}, has been considered far less with elliptic tools. Apart from Eschenburg's earlier work \cite{eschenburg1988}, comparison theory has been systematically studied by Treude \cite{treude2011}, Treude--Grant \cite{treude-grant2013}, Graf \cite{graf2016}, and Lu--Minguzzi--Ohta \cite{lu-minguzzi-ohta2022-range} in relative recency compared to the age of relativity theory. A seemingly exhaustive account on Bochner's technique in  the Lorentz\-ian case constitute the works of Lichnerowicz \cite{lichnerowicz1955-theorie-relativiste} and Stepanov \cite{stepanov1993} on relativistic  electromagnetism and ideal fluids, Romero--Sánchez \cite{romero-sanchez1995,romero-sanchez1998} on con\-formal symmetries, and Galloway \cite{galloway1989-splitting} within his splitting theorem. 

Arguably, the main issue is the indefinite signature of the metric tensor. Consequently, the d'Alembertian is a \emph{hyperbolic} operator. 
Therefore, until the contribution of Beran et al. \cite{beran-braun-calisti-gigli-mccann-ohanyan-rott-samann+-}, no control over the timelike cut locus was available in the literature. A further drawback is maximum principle techniques only apply after restricting the d'Alembertian in question to spacelike hypersurfaces (where it becomes the intrinsic Riemannian Laplacian), as e.g.~discussed by Galloway \cite{galloway1989-splitting}. This process requires  nontrivial \emph{a priori} regularity of the function in question and its level sets; conceptually, a certain degree of smoothness is usually a \emph{consequence} rather than a hypothesis of  elliptic regularity theory. 

On the other hand, \cref{Sub:Lapl Riem} serves as a blueprint for applications in Lorentzian geometry. Ricci curvature naturally enters through the Einstein equations \cite{rendall2002}; lower bounds on it correspond to energy conditions, which become especially relevant when the precise matter distribution in spacetime is unknown. Mean curvature relates to the d'Alembertian of  Lorentz distance functions \cite{treude2011,treude-grant2013}, which are well-known to lose smoothness at their timelike cut loci.  Moreover, it appears as an initial condition in the classical singularity theorems of general relativity \cite{penrose1965,hawking1967,hawking-penrose1970}.  \emph{Constant} mean curvature surfaces are simple solutions to the Einstein constraint equations, cf.~Andersson-Moncrief  \cite{andersson-moncrief2004}.  Their existence is a characteristic of Bartnik's splitting conjecture \cite{bartnik1988} (see also Galloway's survey \cite{galloway2019}). The foliation of such surfaces defines the geometric center of mass of Huisken--Yau \cite{huisken-yau1996}.

\subsubsection{A nonlinear yet elliptic variant of the d'Alembertian} In our recent work \cite{beran-braun-calisti-gigli-mccann-ohanyan-rott-samann+-} with Beran et al., inter alia a weak formulation of the estimate ``$\Box_\PP\phi \leq \mu$'' was proven. Here, $\PP$ is a nonzero number less than one, $\Box_\PP$ is formally from \eqref{Eq:Lapl formally}, $\phi$ denotes a Kantorovich potential for a Lorentzian optimal transport problem \cite{eckstein-miller2017,suhr2018-theory,mccann2020,mondino-suhr2022,cavalletti-mondino2020} (cf.~\cref{Sub:Lor opt tr} below), and $\mu$ forms a signed measure. From this, we derived the existence of a  distributional $p$-d'Alembertian of $\phi$ (and, by using the chain rule, of the Lorentz distance function $\u$ from a point) by \emph{abstract} representation theorems for nonnegative distributions. This holds in a  synthetic setting as pioneered by Kunzinger--Sämann  \cite{kunzinger-samann2018}  and Cavalletti--Mondino \cite{cavalletti-mondino2020-new} detailed in \Cref{Sub:Non dal,Sub:TMCP,Sub:TCD}, which largely exceeds classical Lorentz spacetimes. 

The latter machinery --- using tools from optimal transport --- is known to characterize the strong energy condition of Hawking--Penrose \cite{penrose1965,hawking1967,hawking-penrose1970} and solutions to the Einstein equations by contributions of McCann \cite{mccann2020} and Mondino--Suhr \cite{mondino-suhr2022}, respectively. The work \cite{mondino-suhr2022} already noted the linearization of $\Box_\PP$ to appear  in the second derivative of the Boltzmann entropy along displacement interpolations.

Several adjacent observations motivate a more systematic study of the $\PP$-d'\!Alembertian \eqref{Eq:Lapl formally} as suggested by Mondino--Suhr \cite{mondino-suhr2022} and Beran et al.~\cite{beran-braun-calisti-gigli-mccann-ohanyan-rott-samann+-}.
\begin{itemize}
\item \textbf{Geometry.} The functions $\u$ we mainly (but not exclusively) focus on have unit magnitude $\vert\rmd\u\vert=1$ on the relevant set, cf.~\cref{Cor:Constant slope}. Consequently, if $\u$ is a Lorentz distance function, irrespective of $\PP$ its $\PP$-d'Alembertian becomes the usual d'Alembertian and inherits its geometric relevance from \cref{Sub:DAL LOR}. 
\item \textbf{Legendre--Fenchel duality.} The study of $\Box_\PP$  naturally continues the seminal works of McCann  \cite{mccann2020} and Mondino--Suhr \cite{mondino-suhr2022}. By replacing the Lagrangian $\vert \dot\gamma\vert$ in the customary length functional by a strictly concave dependence $\smash{\QQ^{-1}\,\vert\dot\gamma\vert^\QQ}$  (where $\PP$ and $\QQ$ are mutually conjugate), they gained enhanced properties of optimal mass transports. An expert reader will realize that the quantity $\smash{\vert\rmd\u\vert^{\PP-2}\,\nabla\u}$ from \eqref{Eq:Lapl formally} is the Legendre transform of the differential $\rmd\u$ of $\u$ induced by the Lagrangian $\smash{\QQ^{-1}\,\vert\dot\gamma\vert^\QQ}$. This aligns with the  Laplacian from Hamiltonian geometry\footnote{However, we point out the Lagrangians of \cite{mccann2020,mondino-suhr2022} are no special cases of this Hamiltonian formalism, since they degenerate as $\smash{\dot{\gamma}}$ approaches the light cone.}  \cite{agrachev2008}, where Laplace comparison theorems and  Bochner's identity are due to Ohta \cite{ohta2014}. 
\item \textbf{Analysis.} As mentioned, Mondino--Suhr \cite{mondino-suhr2022} noted the linearization of $\Box_\PP$ shows up in the differentiation of the Boltzmann entropy.  They also proved an analog of Ohta's Bochner identity. This relates $\Box_\PP$ to the influential Otto calculus \cite{otto-villani2000}.
\item \textbf{Flexibility.} The tools from Lorentzian optimal transport we also use are robust enough to avoid regularity issues. For instance, the weak comparison estimates of Beran et al.~\cite{beran-braun-calisti-gigli-mccann-ohanyan-rott-samann+-} hold across timelike cut loci. On a broader scale, this toolkit  has recently helped Cavalletti--Mondino \cite{cavalletti-mondino2020,cavalletti-mondino2024} and Braun--McCann \cite{braun-mccann2023} to establish sharp geometric inequalities, synthetic Hawking-type singularity theorems, and isoperimetric inequalities. We also point out the  synthetic descriptions of the null energy condition by Ketterer \cite{ketterer2024}, McCann \cite{mccann2023-null}, and Cavalletti--Manini--Mondino \cite{cavalletti-manini-mondino2024+}.
\item \textbf{Regularity.} In the specified range $\PP\in (-\infty,1)\setminus\{0\}$, $\Box_\PP$ exhibits elliptic characteristics in two ways. Up to a prefactor, it is the  variational derivative of a convex functional on the cone of causal functions, cf.~Beran et al.~\cite{beran-braun-calisti-gigli-mccann-ohanyan-rott-samann+-}, \eqref{Eq:q-Cheeger}, and \cref{Re:Brezis}. Also, its linearization around a suitable function \emph{is} in fact   elliptic. In Braun--Gigli--McCann--Ohanyan--Sämann \cite{braun-gigli-mccann-ohanyan-samann+} (abbreviated by \emph{Braun et al.} from now on), this allows us to give a simplified proof of the classical Lorentzian splitting theorem \cite{eschenburg1988,galloway1989-splitting,newman1990} and a new extension to low regularity using techniques from elliptic PDE. On a broader scale, for RCD spaces heat flow regularization  \cite{ambrosio-gigli-savare2014-calculus,ambrosio-gigli-savare2014-riemannian,ambrosio-gigli-savare2015-bakry,savare2014} based on ellipticity is crucial  e.g.~in developing a second order calculus as accomplished by Gigli  \cite{gigli2018}.
\end{itemize}

\subsection{Contributions}\label{Sub:Contri} 

\subsubsection{Exact representation formulas for d'Alembertians}\label{Sub:Ex1} To fix the main ideas, we first stick to the case of a globally hyperbolic Finsler spacetime $\mms$ endowed with a smooth measure $\meas$; their basic notions are summarized in \cref{Sub:LorentzFinsler}. Then $\meas$ induces the $N$-Ricci curvature $\smash{\Ric^N_\meas}$, where $N\in (1,\infty)$ is an upper bound for $\dim\mms$. For simplicity, assume $\smash{\Ric^N_\meas}$ is nonnegative in all timelike directions (relating to the \emph{strong energy condition} of Hawking--Penrose  \cite{penrose1965,hawking1967,hawking-penrose1970}). If $\mms$ is Lorentzian with metric tensor $\Rmet$, $\meas := \vol_\Rmet$, and $N:=\dim\mms$, then $\smash{\Ric^N_\meas}$ is the usual Ricci tensor $\Ric_g$, cf.~\cref{Ex:TCD}. Our representation formulas will already be new in this case.

Let $\Sigma$ denote a smooth, compact, achronal spacelike hyper\-surface with induced Lorentz distance function $l_\Sigma$, see \cref{Sec:Signed}. The negative gradient flow  of $l_\Sigma$ foliates $\smash{I^+(\Sigma)}$ minus the future timelike cut locus $\smash{\TCut^+(\Sigma)}$  into normal geodesic rays $\mms_\alpha$  which pass    through $\alpha\in Q$. Here $Q$ is an index set (which, in the current case, can be taken to be $\Sigma$, cf. \cref{Re:Compat smooth}). We  write  $a_\alpha$ and $b_\alpha$ for the unique intersection points of $\cl\,\mms_\alpha$ with $\Sigma$ and  $\TCut^+(\Sigma)$ (if existent), respectively.

As pioneered in the Lorentzian context by Cavalletti--Mondino \cite{cavalletti-mondino2020}, later refined by them  \cite{cavalletti-mondino2024} and Braun--McCann \cite{braun-mccann2023} (summarized in \cref{Sub:Disintegr}), this foliation induces a disintegration of the restriction of $\meas$ to $\smash{I^+(\Sigma)}$, viz.
\begin{align*}
\meas \mres I^+(\Sigma) = \int_Q\meas_\alpha\d\q(\alpha).
\end{align*}
Here $\q$ is a Borel probability measure on $Q$ (which, in the current case, can be taken as a density times the  area measure $\smash{\scrH_\Sigma^{\dim\mms-1}}$, cf.~\cref{Re:Compat smooth}). It is merely used to label the rays; geometric information enters as follows. For $\q$-a.e.~$\alpha\in Q$, $\meas_\alpha$ is a Radon measure on $\mms_\alpha$ with the following properties recalled in \cref{Th:Disintegration}.
\begin{itemize}
\item It is absolutely continuous with respect to the Lebesgue measure $\Leb_\alpha^1$ on $\mms_\alpha$ with Radon--Nikodým density $h_\alpha$ (interpreted as an ``area element'').
\item It inherits the ambient nonnegative curvature, viz.~$\smash{h_\alpha^{1/(N-1)}}$ is concave.
\end{itemize}
Moreover, $h_\alpha$ is locally Lipschitz continuous and positive. Consequently, its logarithmic derivative $\smash{(\log h_\alpha)'}$ along the negative gradient flow  exists $\meas_\alpha$-a.e. 

Based on these preparations, let us now introduce a special case of our main result: an exact representation formula for the d'Alembertian of $l_\Sigma$ together with an integration by parts formula. After some preparations, its proof is given in \cref{Sub:Conclusion}.

\begin{theorem}[Finslerian d'Alembertian for Lorentz distance functions]\label{Th:Main I Finsler} Let $\mms$ be a globally hyperbolic Finsler spacetime with a smooth measure $\meas$ on it. Suppose its $N$-Ricci curvature is nonnegative in all timelike directions, where $N\in (1,\infty)$. Let $\Sigma$ be a smooth, compact, achronal space\-like hypersurface. Moreover, let $\q$ designate a disintegration of $\smash{\meas\mres I^+(\Sigma)}$ as above. Then the assignment 
\begin{align}\label{Eq:Assignmentttt}
\BOX l_\Sigma &:= \int_Q (\log h_\alpha)'\,\meas_\alpha\d\q(\alpha) - \int_Q h_\alpha(b_\alpha)\,\delta_{b_\alpha}\d\q(\alpha)
\end{align}
constitutes an element of the topological dual space $\smash{\Cont_\comp(I^+(\Sigma))'}$. 

More strongly, it forms a \textnormal{(}nonrelabeled\textnormal{)} generalized signed Radon measure on $I^+(\Sigma)$ which satisfies the following integration by parts formula for every $\smash{f\in \Cont_\comp^\infty(I^+(\Sigma))}$:
\begin{align*}
\int_{\mms} \rmd f(\nabla l_\Sigma)\d\meas = -\int_\mms f\d\BOX l_\Sigma.
\end{align*}
\end{theorem}

Here and in the sequel, we set $h_\alpha(b_\alpha) := 0$ if $b_\alpha$ does not exist for $\alpha\in Q$. Moreover, by a ``generalized signed Radon measure'' on $\smash{I^+(\Sigma)}$ we will mean an element of $\smash{\Cont_\comp(I^+(\Sigma))'}$ which can be represented as the difference of two Radon measures, cf.~ \cref{Sub:Fixed}. This property is slightly weaker than being a signed Radon measure, since the respective ``positive and negative parts'' may both be infinite. 

The curvature hypothesis enters the conclusion of \Cref{Th:Main I Finsler} only via the regularity properties of the conditional densities. Unlike the Riemannian situation, as quantified by McCann \cite{mccann2023-null}, in Lorentzian geometry a locally uniform lower bound on the Ricci curvature in all timelike directions is not automatically satisfied,  cf.~\cref{Re:Smooothspa}.

In positive signature, the analog of \cref{Th:Main I Finsler}  (and its described extensions) --- i.e. exact representation formulas for the Laplacian of distance functions and more general $1$-Lipschitz functions for CD \cite{sturm2006-i,sturm2006-ii,lott-villani2009} or MCP \cite{sturm2006-ii,ohta2007-mcp} metric measure spaces --- was established by Cavalletti--Mondino \cite{cavalletti-mondino2020-new}   assuming Rajala--Sturm's essential non\-branching \cite{rajala-sturm2014}. They combined an approach of Gigli \cite{gigli2015} to  distributional Laplacians with their localization paradigm \cite{cavalletti-mondino2017-isoperimetric}. 

\cref{Th:Main I Finsler} describes the usual  d'Alembertian distributionally; no nonlinear version of it seems involved. Yet, qua our approach (cf.~\cref{Sub:Non dal} below) and \cref{Cor:Constant slope}, its claim should be read as ``for all  $\PP\in (-\infty,1)\setminus\{0\}$, \eqref{Eq:Assignmentttt} defines the $\PP$-d'Alembertian $\smash{\BOX_\PP l_\Sigma}$ with the stated properties''. In particular, the inherent integration by parts formula is
\begin{align*}
\int_\mms \rmd f(\nabla l_\Sigma)\,\vert\rmd l_\Sigma\vert^{\PP-2}\d\meas = -\int_\mms f\d\BOX_\PP l_\Sigma. 
\end{align*}

There are several remarkable observations around \cref{Th:Main I Finsler}. 
\begin{itemize}
\item \textbf{Novelty.} \cref{Th:Main I Finsler} seems to be the first instance in Lo\-rentzian geometry which provides a precise representation formula for the d'Alembertian (as a generalized signed Radon measure).
\item \textbf{Lebesgue decomposition.} By the disintegration formula, \eqref{Eq:Assignmentttt} decomposes into an $\meas$-absolutely continuous and an $\meas$-singular part. Both are  explicitly given. The $\meas$-singular part is concentrated on the set of future timelike cut points $b_\alpha$ of $\Sigma$, where $\alpha\in Q$. Notably, it is \emph{nonpositive}, which confirms the qualitative indication of this sign by the weak d'Alembert comparison theorem of Beran et al.~\cite{beran-braun-calisti-gigli-mccann-ohanyan-rott-samann+-}. In \cite{braun-gigli-mccann-ohanyan-samann+}, we provide an alternative qualitative explanation.
\item \textbf{Variational identity.} The integration by parts formula holds with an \emph{equality}. On Finsler spacetimes, the distributional d'Alembertian of $l_\Sigma$ is therefore  uniquely determined by this property. 
\end{itemize}

\subsubsection{Sharp comparison estimates}\label{Sub:Ex2} In the setting of \cref{Th:Main I Finsler}, we write the density of the $\meas$-absolutely continuous part of $\BOX l_\Sigma$ as $\Box l_\Sigma$. That is, for $\q$-a.e.~$\alpha\in Q$, 
\begin{align}\label{Eq:Delta log ha}
\Box l_\Sigma = (\log h_\alpha)'\quad\meas_\alpha\textnormal{-a.e.}
\end{align}

\begin{theorem}[Finslerian d'Alembert comparison I, see also \cref{Th:mmmmm}]\label{Th:Main II Finsler} Let $\mms$  be a globally hyperbolic Finsler spacetime with a smooth measure $\meas$ on it. Assume its induced $N$-Ricci curvature is nonnegative in all timelike directions, where $N\in (1,\infty)$.  Let $\Sigma$ be a smooth, compact, achronal space\-like hypersurface. Moreover, let $\q$ be a disintegration of $\meas \mres I^+(\Sigma)$ as above. Then for $\q$-a.e.~$\alpha\in Q$,
\begin{align*}
-(N-1)\,\frac{1}{l(\cdot,a_\alpha)} \leq \Box l_\Sigma \leq (N-1)\,\frac{1}{l_\Sigma}\quad \meas_\alpha\textnormal{-a.e.}
\end{align*}
\end{theorem}

By the nonpositivity of the contribution of $\smash{\TCut^+(\Sigma)}$ to $\BOX l_\Sigma$, this implies
\begin{align*}
\BOX l_\Sigma \leq (N-1)\,\frac{1}{l_\Sigma}\,\meas\mres  I^+(\Sigma).
\end{align*} 

This follows from \eqref{Eq:Assignmentttt} by noting the concavity properties of the conditional densities transfer to estimates on their logarithmic derivatives, cf.~\cref{Le:logarithmic derivative} and \cref{Re:Logarithmic derivative}. Such density bounds go back to Cavalletti \cite{cavalletti2014-monge}. Inter alia, they were employed by Cavalletti--Mondino \cite{cavalletti-mondino2020-new} to establish the predecessor of \cref{Th:Main II Finsler} (and  \cref{Th:mmmmm} below) for essentially nonbranching MCP metric measure spaces. 

In the Lorentzian case, these density bounds originate in Cavalletti--Mondino \cite{cavalletti-mondino2020,cavalletti-mondino2024} and have been extended for variable curvature bounds by Braun--McCann \cite{braun-mccann2023}.

Even in the smooth case, \cref{Th:Main II Finsler} extends the hypersurface d'Alembert comparison results of Eschenburg \cite{eschenburg1988}, Treude \cite{treude2011}, Treude--Grant \cite{treude-grant2013}, Graf \cite{graf2016}, and Lu--Minguzzi--Ohta \cite{lu-minguzzi-ohta2022-range} which give  upper bounds on the d'Alembertian of $l_\Sigma$ only outside  $\smash{\TCut^+(\Sigma)}$; \cite{graf2016} studied mean curvature comparison geometry for $\smash{\Cont^{1,1}}$-Lorentz spacetimes by approximation. With the main result of Braun--Calisti \cite{braun-calisti2023} (cf.~\cref{Sub:Non dal} below), \cref{Th:Main II Finsler} extends Graf's results to $\smash{\Cont^1}$-regularity plus Cavalletti--Mondino's timelike nonbranching \cite{cavalletti-mondino2020}. Conjecturally, \cite{braun-calisti2023} holds in the Lipschitz case, which is optimal from the viewpoint of a generally well-behaved causality, as pointed out by Chru\'sciel--Grant   \cite{chrusciel-grant2012}. The sharpness of these (hence our) results is known, cf.~Treude--Grant \cite{treude-grant2013} for the identification of the model spaces. When $\Sigma$ consists of a point, the upper bound of \cref{Th:Main II Finsler} has been  established with different yet related  techniques by Beran et al.~\cite{beran-braun-calisti-gigli-mccann-ohanyan-rott-samann+-}. 

We point out that \cref{Th:Main II Finsler} also establishes a \emph{lower} bound. The latter estimates  the d'Alembertian $\Box l_\Sigma$ from below by the distance  to its cut point with respect to $l$ (if existent, cf.~\cref{Re:Unbounded I,Re:Unbounded II} for the correct interpretation else). An intuition thank\-fully pointed out to me by Nicola Gigli why a lower curvature bound entails lower (instead of the customary upper) comparison estimates  based on the maximum principle is found in Cavalletti--Mondino \cite{cavalletti-mondino2020-new}*{Rem.~4.11}.

The last inequality of \cref{Th:Main II Finsler} reflects the well-known fact that the d'Alembertian of $l_\Sigma$ admits per se no uniform control near $\Sigma$. Certain powers of Lorentz distance functions behave better  ``near the origin'', as for similar features  of the squared distance function in Riemannian geometry, cf.~the book of Cheeger--Ebin   \cite{cheeger-ebin1975}. This induces  another class of comparison results we show; it relates somewhat more to the shape of the operator \eqref{Eq:Lapl formally}.

\begin{theorem}[Finslerian d'Alembert comparison II, see also \cref{Th:Main II Finsler}]\label{Th:mmmmm} Let $\mms$ be a globally hyperbolic Finsler spacetime endowed with a smooth measure $\meas$. Suppose its $N$-Ricci curvature is nonnegative in all timelike directions, where $N\in(1,\infty)$. Let $\Sigma$ be a smooth, compact, achronal space\-like hypersurface. Let $\q$ designate a disintegration of $\smash{\meas\mres I^+(\Sigma)}$ as above. Lastly, let $\PP$ and $\QQ$ be mutually conjugate nonzero exponents less than $1$. Then $\q$-a.e.~$\alpha\in Q$ satisfies
\begin{align*}
1- (N-1)\,l_\Sigma\,\frac{1}{l(\cdot,a_\alpha)} \leq \Box_\PP [\QQ^{-1}\,l_\Sigma^\QQ] \leq N\quad\meas_\alpha\textnormal{-a.e.}
\end{align*}
\end{theorem}

These enhanced inequalities  help upgrade the distributional d'Alembertian from \cref{Th:Main I Finsler} to a generalized  signed Radon measure, cf.~\cref{Th:11}. This passage through comparison results to get ``measure-valued'' Laplacians  is customary in metric measure geometry \cite{kuwae-shioya2007,gigli2015,cavalletti-mondino2020-new} and was already used in the Lorentzian setting by Beran et al.  \cite{beran-braun-calisti-gigli-mccann-ohanyan-rott-samann+-}. 

As further described in \cref{Sub:Non dal}, we also  establish  representation formulas for such power functions. In turn, these entail
\begin{align*}
\BOX_\PP [\QQ^{-1}\,l_\Sigma^\QQ] \leq N\,\meas\mres I^+(\Sigma).
\end{align*}
This extends our results with Beran et al.~\cite{beran-braun-calisti-gigli-mccann-ohanyan-rott-samann+-} for $\Sigma$ being a singleton.

\subsubsection{Extensions and conceptions}\label{Sub:Non dal}  \cref{Th:Main I Finsler,Th:Main II Finsler,Th:mmmmm} have wide    generalizations to the recent advances to mathematical general relativity through optimal transport and metric measure geometry. Such an abstract approach seems highly desirable from various perspectives, as outlined in the introductions of Kunzinger--Sämann \cite{kunzinger-samann2018} and Cavalletti--Mondino \cite{cavalletti-mondino2020} and detailed in the overviews of the latter \cite{cavalletti-mondino2022-review} and Steinbauer \cite{steinbauer2023} in honor of Sir Roger Penrose's 2020 Nobel Prize in Physics.

Our setting is that of a \emph{metric measure spacetime} $\scrM$ \cite{kunzinger-samann2018,cavalletti-mondino2020,minguzzi-suhr2022,mccann2023-null,braun-mccann2023,beran-braun-calisti-gigli-mccann-ohanyan-rott-samann+-}\footnote{The original ``metric'' generalization of  Lorentz spacetimes are the so-called \emph{Lorentzian length spaces} of Kunzinger--Sämann \cite{kunzinger-samann2018}. The terminology ``metric measure spacetimes'' was proposed by McCann \cite{mccann2023-null}. The current number of variations of \cite{kunzinger-samann2018} comes from the diversity of causality phenomena in Lorentzian geometry,  e.g.~Minguzzi \cite{minguzzi2019-causality}, which need to be selected on a case-by-case basis. Here, we follow the setup of Braun--McCann \cite{braun-mccann2023} which builds on Minguzzi--Suhr's bounded Lorentzian metric spaces \cite{minguzzi-suhr2022}, but in either framework our methods can be developed with some modifications.}. This is a triple $(\mms,l,\meas)$ consisting of 
\begin{itemize}
\item a possibly nonsmooth space $\mms$ endowed with
\item a signed time separation function $l$ and
\item a reference Radon measure $\meas$
\end{itemize}
satisfying certain compatibility conditions specialized in \cref{Sub:MMS}. We assume the so-called \emph{entropic timelike measure contraction property} $\smash{\TMCP^e(k,N)}$, where $k$ is a lower semicontinuous function  on $\mms$ and $N\in (1,\infty)$. Occasionally (cf.~\cref{Sub:BOCH}), we also consider the stronger \emph{entropic timelike curvature-dimension condition} $\smash{\TCD_\beta^e(k,N)}$, where $\beta$ forms a nonzero transport exponent less than one. Both  synthesize 
\begin{itemize}
\item the ``timelike Ricci curvature'' of $\scrM$ being bounded from below by $k$, and 
\item the ``dimension'' of $\scrM$ being bounded from above by $N$.
\end{itemize}
Inspired by the seminal contributions of McCann \cite{mccann2020} and Mondino--Suhr \cite{mondino-suhr2022}, they were introduced by Cavalletti--Mondino \cite{cavalletti-mondino2020} for constant $k$; see Braun \cite{braun2023-renyi} for an alternative approach and Beran et al.~\cite{beran-braun-calisti-gigli-mccann-ohanyan-rott-samann+-} for negative $\beta$. For variable $k$, they were generalized by Braun--McCann \cite{braun-mccann2023}. See \cref{Sub:TMCP,Sub:TCD} for   details. Both conditions include  $\Cont^2$  \cite{mccann2020,mondino-suhr2022} and $\Cont^1$  \cite{braun-calisti2023} regular Lorentz spacetimes and Finsler spacetimes \cite{braun-ohta2024}. The \emph{null energy condition} of Penrose \cite{penrose1965} is equivalent to a variable lower bound on the Ricci curvature in all timelike directions by  McCann \cite{mccann2023-null}, hence covered as well. 

We also assume $\scrM$ is timelike $\beta$-essentially nonbranching after Braun \cite{braun2023-renyi}, a measure-theoretic extension of nonbranching of timelike geodesics by Cavalletti--Mondino \cite{cavalletti-mondino2020}. Except for \cref{Th:From TCD to Bochner,Th:From Bochner to TCD}, $\beta$ does never  enter  our conclusions. It is conjectured by Cavalletti--Mondino \cite{cavalletti-mondino2022-review}, under appropriate nonbranching assumptions $\smash{\TCD_\beta^e(k,N)}$ is  independent of $\beta$,  which is work in progress by Akdemir \cite{akdemir+}.

In the range $\PP\in (-\infty,1)\setminus\{0\}$, our formulas are based on a new distributional notion of \emph{$\PP$-d'Alembertian} of a given $l$-causal function $\u$ on $\mms$ (i.e.~$\u$ is nondecreasing along the causality  relation $\leq$ induced by $l$) from Beran et al.~\cite{beran-braun-calisti-gigli-mccann-ohanyan-rott-samann+-}, cf.~\cref{Def:DAlem}. The property of $\u$ having a d'Alembertian is then set up by $\u$ belonging to the domain of the $\PP$-d'Alembertian for every such $\PP$. We first establish several calculus rules, e.g.~a chain rule in \cref{Pr:Chain rule}. The approach of \cite{beran-braun-calisti-gigli-mccann-ohanyan-rott-samann+-}  also yields a synthetic meaning of $\u$ being $\PP$-harmonic, $\PP$-superharmonic, or $\PP$-subharmonic, cf.~\cref{Def:Harmon,Re:Superharmonicity} below. As for much of our work, our strategy mirrors  Gigli \cite{gigli2015} and  Cavalletti--Mondino \cite{cavalletti-mondino2020-new} in positive signature. We also refer to Björn--Björn \cite{bjorn-bjorn2011} and Gigli--Mondino \cite{gigli-mondino2013} for studies of nonlinearly harmonic functions on metric measure spaces.

Remarkably, our representation formulas hold and make sense in high generality for ``abstract'' test functions from  Beran et al.~\cite{beran-braun-calisti-gigli-mccann-ohanyan-rott-samann+-}  (which we later call \emph{finite perturbations}). This follows from the good almost continuity properties of such functions from \cite{beran-braun-calisti-gigli-mccann-ohanyan-rott-samann+-} (see also similar results by Minguzzi \cite{minguzzi2019-causality}), cf.~\cref{Le:Almost rightleft,Re:Prop fin p}. They are valid under natural hypotheses on $\meas$, cf.~\cref{Sub:MMS}. 

Our choice for $\u$ are mostly signed Lorentz distance functions $l_\Sigma$ (negative on $I^-(\Sigma)$ and positive on $I^+(\Sigma)$). Here $\Sigma$ forms a fixed Borel subset in $\mms$ which is timelike complete \cite{galloway1986}, cf.~\cref{Sec:Signed}. The extensions of \cref{Th:Main I Finsler,Th:Main II Finsler,Th:mmmmm} are given in \cref{Cor:22}, \cref{Th:11}, and  \cref{Cor:1515}, respectively. We give  analogs for more general $1$-steep functions $\u$ in  \cref{Th:Meas val Alem I,Cor:Sharp comparison}. Although $1$-steepness is a condition on the \emph{first} ``derivative'' of $\u$, \cref{Th:Meas val Alem I} asserts its \emph{twice} differentiability on a suitable subset $\Tr$ of $\mms$, cf.~\cref{Sub:Transport rels}. Lastly, by ``flipping'' the causal orientation on $I^-(\Sigma)$, we derive representation formulas and comparison results for the ``d'Alembertian'' of the absolute value $\vert l_\Sigma\vert$ and its powers, neither of which is $l$-causal on $\smash{I^+(\Sigma)\cup \Sigma\cup I^-(\Sigma)}$, in \cref{Sub:Unsign time}.

\subsection{Applications} We finally establish several  consequences of our results.

\subsubsection{Bochner-type inequality}\label{Sub:BOCH} The first is a Bochner-type inequality for e.g.~signed Lorentz  distance functions. The simplifying feature of such functions is the vanishing of their ``Hessian'' precisely in the direction where the ``metric tensor'' fails to be positive definite. 

In \cref{Th:From TCD to Bochner,Th:From Bochner to TCD}, we characterize the time\-like curvature-dimension condition $\smash{\TCD_\beta^e(k,N)}$ by a family of such Bochner-type inequalities. On smooth spacetimes, this result is straightforward, albeit no source seems to have stated it explicitly: following von Renesse--Sturm \cite{von-renesse-sturm2005}, simply construct signed Lorentz distance functions which --- if plugged into e.g.~Mondino--Suhr's Bochner identity \cite{mondino-suhr2022} --- recover the lower bounded\-ness of the Ricci curvature by $k$ in timelike directions  (and then apply McCann \cite{mccann2020}). Yet, this argument passes through a tensor, which we avoid. Instead, we use our identification \eqref{Eq:Delta log ha} of the $\meas$-absolutely continuous part of the relevant d'Alembertian, a differential inequality for the appearing logarithmic derivative found in Cavalletti--Milman \cite{cavalletti-milman2021} (also stated in the localization of the TCD condition for constant $k$ in Cavalletti--Mondino \cite{cavalletti-mondino2024}), cf. \cref{Le:Diffchar}, and a characterization of $\smash{\TCD_\beta^e(k,N)}$ in progress by Akdemir \cite{akdemir+}.

This Lorentzifies the corresponding prior contribution of Cavalletti--Mondino \cite{cavalletti-mondino2020-new}. It supports the  equivalence of the Lagrangian $\smash{\TCD_\beta^e(k,N)}$ condition and an Eulerian description conjectured  by Cavalletti--Mondino \cite{cavalletti-mondino2022-review}, which is yet to be developed. In positive signature, this relation pioneered by Ambrosio--Gigli--Savaré \cite{ambrosio-gigli-savare2014-riemannian,ambrosio-gigli-savare2015-bakry} based on   Gigli's  infinitesimal Hilbert\-iani\-ty \cite{gigli2015} was the initial ignition for the theory of RCD spaces. It was e.g.~succeeded by  its  adaptations to finite dimensions \cite{erbar-kuwada-sturm2015,  ambrosio-mondino-savare2019} and variable curvature bounds \cite{braun-habermann-sturm2021,sturm2020}. In the Lorentzian case, such an equivalence conceivably relies on infinitesimal Minkowskianity, which was  recently proposed by Beran et al.~\cite{beran-braun-calisti-gigli-mccann-ohanyan-rott-samann+-}. Notably, this is \emph{not} assumed in our equivalence. Thus, it extends to Finsler spacetimes, where it  reflects the Finslerian characterization of Ohta--Sturm \cite{ohta-sturm2014} in positive signature.

\subsubsection{Mean curvature}\label{SUB:MCUR} In the framework of their Hawking-type singularity theorem, a synthetic notion of mean curvature \emph{bounds} for a sufficiently regular subset $\Sigma$ of $\mms$ was given by Cavalletti--Mondino \cite{cavalletti-mondino2020}. In \cref{Th:Mean curv bd equiv}, we characterize (suitable variable extensions of) these by upper and lower bounds for the distributional d'Alembertian of $l_\Sigma$. Consequently, the synthetic Hawking-type singularity theorems of  Cavalletti--Mondino \cite{cavalletti-mondino2020} and Braun--McCann \cite{braun-mccann2023} hold under a trappedness condition on $\Sigma$ in terms of our d'Alembert mean curvature bounds set up in \cref{Def:DAL MEAN C} below. 

For essentially nonbranching MCP spaces, Laplacian bounds for distance functions by  Cavalletti--Mondino \cite{cavalletti-mondino2020-new} were connected to synthetic mean curvature bounds after Ketterer \cite{ketterer2020-heintze-karcher} by Burtscher--Ketterer--McCann--Woolgar \cite{burtscher-ketterer-mccann-woolgar2020} and Ketterer \cite{ketterer2023-rigidity}.

Our result motivates the notion of a \emph{mean curvature barrier} for $\Sigma$. Roughly speaking, this is a function on $\Sigma$ which bounds $\BOX l_\Sigma$ ``approximately'' from above and below near $\Sigma$ in a quantitative way;  cf.~\cref{Def:Barriers}. This induces a natural analog of constant mean curvature surfaces in nonsmooth mathematical general relativity. Our approach  mirrors Antonelli--Pasqualetto--Pozzetta--Semola \cite{antonelli-pasqualetto-pozzetta-semola2022}  (see also Ketterer \cite{ketterer2023-rigidity} and Mondino--Semola \cite{mondino-semola2023+}). They also pointed out some metric  examples of sets which admit an entire range of mean curvature barriers, which justifies the different name compared to e.g.~a  ``constant mean curvature set''. This ambiguity is a genuinely nonsmooth phenomenon; in the smooth case, these two descriptions are equivalent.   See \cref{Re:CMC,Re:Compat smooth}.

In fact, inspired by Ketterer \cite{ketterer2020-heintze-karcher} we define the \emph{exact} synthetic mean curvature of $\Sigma$ in \cref{Def:MCurv}, using a suitable ``restriction'' of the d'Alembertian of its signed Lorentz  distance function.

\subsubsection{Volume and area bounds meet Cauchy developments} Besides the Laplace comparison theorem, the Heintze--Karcher inequality \cite{heintze-karcher1978} is another key  volume comparison result from Riemannian geometry. It estimates the volume of a  tubular neighborhood of a smooth hypersurface by a surface integral involving its mean curvature and other geometric data. Subsequent generalizations were given to the weighted case by Bayle \cite{bayle2004}  and Morgan \cite{morgan2005} and essentially nonbranching MCP spaces by Ketterer \cite{ketterer2020-heintze-karcher}.  As argued by Treude--Grant \cite{treude-grant2013}, this reflects the setting of the Hawking singularity theorem \cite{hawking1967}, which makes such estimates very interesting for Lorentzian geometry.


In \cref{Th:HeintzeKarcher}, based on our  mean curvature outlined in \cref{SUB:MCUR} we adapt \cite{heintze-karcher1978,bayle2004,ketterer2020-heintze-karcher} to our synthetic Lorentzian case, mirroring the approach of  Ketterer \cite{ketterer2020-heintze-karcher}. Let us formulate a special case in the Finslerian framework of \cref{Sub:Ex1,Sub:Ex2}.

\begin{theorem}[Heintze--Karcher-type inequality]\label{Th:HKARIntro} Let $\mms$ be  a globally hyperbolic Finsler spacetime with a smooth measure $\meas$. Assume its induced $N$-Ricci curvature is nonnegative in all timelike directions, where $N\in (1,\infty)$. Let $\Sigma$ be a smooth, compact, achronal spacelike hypersurface. Assume its forward mean curvature is bounded from above by a number $H_0$. Lastly, let $\hh_0$ denote its surface measure and $\Sigma_{[0,t]}$ its  future development with temporal height $t>0$, respectively. Then
\begin{align}\label{Eq:Posp}
\meas[\Sigma_{[0,t]}] \leq \hh_0[\Sigma]\int_0^t \Big[1 + \frac{H_0}{N-1}\,\theta\Big]_+^{N-1}\d\theta.
\end{align}
\end{theorem}

Volume bounds by mean curvature for CMC surfaces in Riemannian  warped products (notably positive signature analogs of the de Sitter--Schwarzschild and Reissner--Nordstrøm solutions) are proven by Brendle \cite{brendle2013}. On Lorentz spacetimes, \cref{Th:HKARIntro} was established  by  Treude--Grant \cite{treude-grant2013} (smooth spacetimes, constant curvature bounds), Graf \cite{graf2016} ($\smash{\Cont^{1,1}}$-spacetimes, constant curvature bounds), as well as Graf--Sormani \cite{graf-sormani2022} (smooth spacetimes, nonnegative timelike Ricci curvature,  integral mean curvature bounds). We point out our results improve the volume bounds from \cite{graf-sormani2022}, where the positive part is only taken on the mean curvature.   Although $H_0$ is constant in  \cref{Th:HKARIntro} --- chosen in favor of a simpler presentation ---, we stress that \cref{Th:HeintzeKarcher} covers variable bounds (both on the Ricci and the mean curvature). Our results already extend these to low regularity Lorentz, \smash{Bakry--Émery} spacetimes (cf.~Galloway--Woolgar \cite{galloway-woolgar2014} and Woolgar--Wylie \cite{woolgar-wylie2018}), and Finsler spacetimes. Moreover, Treude--Grant \cite{treude-grant2013} assumed an upper bound on the length of the rays in question \emph{a priori}, which we avoid.

Our methods also allow for a corresponding synthetic extension of \emph{area estimates} à la \cite{treude2011,treude-grant2013,graf2016,graf-sormani2022}, which are given in \cref{Th:AREA}.


\begin{remark}[Volume estimates vs.~stability] Recall a \emph{Cauchy development} from an initial data set $(\Sigma,r,\kappa)$ (cf.~the survey of Rendall \cite{rendall2002}) is a globally hyperbolic Lorentz spacetime with metric tensor $\Rmet$ which solves the vacuum Einstein equations with Cauchy hypersurface $\Sigma$ such that 
\begin{itemize}
\item $\Rmet$ restricts to $r$ on $\Sigma$, where $r$ is a Riemannian metric, and
\item $\kappa$ is the second fundamental form of $\Sigma$.
\end{itemize}
Their existence is a deep classical result of Choquet-Bruhat \cite{choquet-bruhat1952}. A main relevance of volume estimates from higher codimension constraints is the question of their stability under a reasonable convergence of the initial data sets. We refer to the introduction of Graf--Sormani \cite{graf-sormani2022} (see also the survey of Sormani  \cite{sormani2023}) for an overview. 

A candidate which may apply to sequences of Cauchy developments is the \emph{spacetime intrinsic flat convergence}, inspired by Sormani--Wenger's intrinsic flat convergence for integral current spaces \cite{sormani-wenger2011}. The latter admits a compactness theorem that only requires uniform volume and diameter controls, as proven by Wenger \cite{wenger2011}. The idea is that by estimates as e.g.~provided by  \cref{Th:HKARIntro}, convergence of initial data sets entails  the necessary volume bounds in a possible version of Wenger's result in Lorentzian signature. Uniform diameter bounds could be drawn from the Hawking-type singularity theorems of Cavalletti--Mondino \cite{cavalletti-mondino2020} and Braun--McCann \cite{braun-mccann2023}. In this sense,   \cref{Th:HeintzeKarcher}  completes the hypotheses of this (synthetic) ``pattern compactness theorem''. \hfill{\footnotesize{$\blacksquare$}}
\end{remark}

\subsubsection{Volume singularities} As indicated above, the work of Treude--Grant \cite{treude-grant2013} has in fact been motivated by the contribution of Heintze--Karcher \cite{heintze-karcher1978}, with the goal of proving ``volume'' versions of Hawking's singularity theorem. Based on their partial results and attempting to better understand Penrose's cosmic censorship conjectures, recently García-Heveling \cite{garcia-heveling2023-volume} proposed the concept of \emph{volume singularities} for Lorentz spacetimes.  It has a direct generalization to our setting: $\scrM$ is future volume incomplete if it contains a point whose chronological future has finite $\meas$-measure, cf.~\cref{Def:Vol sing,Re:Finite vol}. 

\cref{Th:HKARIntro} and its add-on of \cref{Th:HeintzeKarcher} imply four volume singularity theorems. The simple yet powerful idea is that the positive part  in e.g.~\eqref{Eq:Posp} --- provided its argument eventually becomes nonpositive, as  is implied by suitable  trappedness conditions --- forces the volume of $\smash{\Sigma_{[0,t]}}$ to ``stop growing'' as $t\to\infty$.
\begin{itemize}
\item \cref{Th:Const vol} generalizes the Hawking-type volume singularity theorem from \cite{garcia-heveling2023-volume} beyond Lorentz spacetimes. It is a volume version of Cavalletti--Mondino's synthetic Hawking singularity theorem  \cite{cavalletti-mondino2020}. It thus yields a synthetic resolution of the ``volume version of Penrose's singularity theorem'' of  \cite{garcia-heveling2023-volume}*{Conj. 4.3}. In fact, it predicts volume singularities in cases which are known to be geodesically \emph{complete}, cf.~\cref{Re:Comments}. On the other hand, to a large extent the proof of \cite{garcia-heveling2023-volume} relies on the Hawking singularity theorem. Ours does not, which underlines the logical independence of geodesic and volume incompleteness \cite{garcia-heveling2023-volume}.
\item Volume singularities under variable curvature bounds are pioneered in \cref{Th:Vol2,Th:Vol1,Th:Vol3}. The first two are (independent) volume analogs of Braun--McCann \cite{braun-mccann2023}. This confirms a conjecture about their existence from \cite{garcia-heveling2023-volume}.  
\end{itemize}

\subsection{Organization} In \cref{Ch:Prerequisites}, we collect some basic background material. In particular, our standing framework is described in \cref{Sub:MMS}. In \cref{Ch:Dalem}, after summarizing fundamental notions of our work \cite{beran-braun-calisti-gigli-mccann-ohanyan-rott-samann+-} with Beran et al., \cref{Sub:Def basis} recalls the distributional d'Alembertian; \cref{Sub:Calc rulez} proves several of its properties,  especially pertaining to the chain rule and harmonicity. Then we pass over to spaces with synthetic timelike curvature and dimension bounds. \cref{Sec:Localization} collects the necessary prerequisites from prior works. In particular, since it plays a major role in our work, the localization paradigm is thoroughly reviewed in \cref{Sub:Disintegr}. In \cref{Ch:Repr form}, we develop some analytic tools in order to conduct our integration by parts approach. This culminates in the connection of Cavalletti--Mondino's localization paradigm with Beran et al.'s  Lorentzian Sobolev calculus in \cref{Sub:Loc hor vert}. In \cref{Sub:Formula1,Sec:Formula2}, we address the representation formulas and adjacent comparison estimates for the d'Alembertian of general $1$-steep functions and, more specifically, signed Lorentz distance functions, respectively. Enhanced properties hold under the stronger hypothesis of topological anti-Lipschitzness (which is set up in \cref{Sub:Top Anti}); it is explicitly justified for Finsler spacetimes in \cref{App:A}. Finally, \cref{Ch:Appli} is devoted to applications of \cref{Ch:Repr form}. Inter alia, we discuss our Bochner inequality in \cref{Sub:Nonsmooth Bochner}, our nonsmooth notion of mean curvature in \cref{Sub:Dalem Mean Curv}, and our Heintze--Karcher-type inequality in \cref{Sub:HKIn}.

\section{Prerequisites}\label{Ch:Prerequisites}

Throughout the paper, if not explicitly stated otherwise, by a function we mean a map with codomain $\R$, thus excluding it from taking the values $\pm\infty$.  Let $\Vert f\Vert_\infty$ denote the pointwise supremum norm of a function $f$. For points $a$ and $b$ in a specified space and a function $f$ on it, we frequently use the abbreviations
\begin{align*}
\big[f\big]_a^b &:= f(b) - f(a),\\
\big[f\big]^b &:= f(b),\\
\big[f\big]_a &:= -f(a).
\end{align*}

By an interval we will always mean a convex subset of $\R$, possibly of infinite or zero diameter and possibly empty.

Sometimes we will need to work with an \emph{everywhere} defined derivative of a continuous function $f$ on a closed interval $I$ (even though generically, $f$ will at least be differentiable $\Leb^1\mres I$-a.e.). To this aim, we define its \emph{right} and \emph{left derivative} $f'^\pm$ on $I$ by
\begin{align*}
f'^\pm(t) := \limsup_{\tau\to 0\pm} \frac{f(t+\tau)-f(t)}{\tau}
\end{align*}
subject to the conventions
\begin{align}\label{Eq:CONVEN}
\begin{split}
f'^+(\sup I) &:= -\infty,\\
f'^-(\inf I) &:= \infty.
\end{split}
\end{align}

\subsection{Measure theoretic notation}\label{Sub:Fixed} 

Let $\mms$ be a locally compact and $\sigma$-compact topological space. A \emph{Radon measure} is a Borel measure $\mu$ on $\mms$ which is inner regular and finite on every compact set (and thus outer regular, cf.~Folland  \cite{folland1999}*{Cor.~7.6}). A \emph{signed Radon measure} is an $\R\cup\{\pm\infty\}$-valued map $\mu$ on the Borel $\sigma$-algebra of $\mms$ which is the difference $\smash{\mu - \nu}$ of two Radon measures which satisfy $\smash{\mu[\mms]<\infty}$ or $\smash{\nu[\mms]<\infty}$. In this case, the Radon measure $\smash{\vert \mu\vert := \mu + \nu}$ is called the \emph{total variation} of $\mu$; if $\vert \mu\vert[\mms] < \infty$, $\mu$ is termed to be \emph{finite}. If  $\mu$ and $\nu$ are infinite, the difference $\mu-\nu$ is made sense of in a distributional manner, see below.

Let $\mu$ be a Radon measure. For a $\mu$-measurable set $A$ in  $\mms$, the \emph{restriction} of $\mu$ to $A$ is the Radon measure $\mu\mres A$ defined by $\mu\mres A[B] := \mu[A\cap B]$.

Let $\smash{\Cont_\comp(\mms)}$ be the class of continuous functions with compact support in $\mms$. A \emph{Radon functional} is a linear functional $T$ on $\smash{\Cont_\comp(\mms)}$ such that for every compact subset $W$ of $\mms$, there is a constant $C$ such that for every $f\in \Cont_\comp(\mms)$ with support in $W$,
\begin{align*}
\vert T(f)\vert \leq C\,\Vert f\Vert_\infty.
\end{align*}
If $T$ is nonnegative (i.e.~the value $T(f)$ is nonnegative if $\smash{f\in \Cont_\comp(\mms)}$ is), it is represented by (integration against) a unique Radon measure $\mu$ on $U$ by the Riesz--Markov--Kakutani representation theorem, cf.~Folland  \cite{folland1999}*{Thm.~7.2}. 

In general,  not every Radon functional is represented by a signed Radon measure by the possible failure of $\sigma$-additivity thereof, as is the case e.g.~for the Hilbert transform on $\R$, cf.~Cavalletti--Mondino \cite{cavalletti-mondino2020-new}*{p.~2103}. On the other hand, the difference $T := \mu-\nu$ of two infinite Radon measures $\mu$ and $\nu$ does make sense as a Radon functional; if $T$ can be written this way, we call it a \emph{generalized signed Radon measure}. Since such quantities will be relevant for our purposes (cf.~\cref{Sub:Loc bdd var!,Sub:Differentiation}), we introduce some adjacent notation. Let $\meas$ be a given Radon measure. We will  call density of the $\meas$-absolutely continuous part of $T$ the unique locally $\meas$-integrable function $f$ with the following property. For every compact subset $C$ of $\mms$, the Lebesgue decomposition of  $\mu\mres C-\nu\mres C$ with respect to $\meas$ reads
\begin{align*}
\mu\mres C - \nu\mres C = f\,\meas\mres C + (\mu\mres C - \nu\mres C)^\perp.
\end{align*} 
Moreover, we say $T$ has nonnegative $\meas$-singular part and write
\begin{align*}
T^\perp \geq 0
\end{align*}
if the signed Radon measure $(\mu\mres C - \nu\mres C)^\perp$ from above is nonnegative (i.e.~a Radon measure) for every $C$ as above. Analogously, $T$ has nonpositive $\meas$-singular part if $-T$ has nonnegative $\meas$-singular part.

Evident versions of the above notions make sense and will be used relative to an open subset $U$ of $\mms$ instead of $\mms$.

\subsection{Monotone and locally BV functions on the real line}\label{Sub:Monotone real} Next, let us collect several well-known  properties of monotone functions and, more generally, functions of locally bounded variation on a given interval $I$. We refer to the book of Folland \cite{folland1999} for details.

\subsubsection{Monotonicity}\label{Sub:MONOT} To discuss monotone functions $f$ on $I$, we focus on $f$ being nondecreasing, the non\-in\-creasing situation is analogous. Then $f$ has  right and left limits $f(x+)$ and $f(x-)$ at every $x\in I$. At cocountably many $x\in I$, the values $f(x)$, $f(x+)$, and $f(x-)$ coincide. We write $\RCR f$ and $\LCR f$ for the functions $f(\cdot\,+)$ and $f(\cdot\,-)$ and call them the \emph{right-} and \emph{left-continuous representative} of $f$.  The functions  $f$, $\RCR f$, and $\LCR f$ are differentiable $\smash{\Leb^1 \mres I}$-a.e.~and their derivatives coincide $\Leb^1\mres I$-a.e. \cite{folland1999}*{Thm.~3.23}.

The differentiation of monotone and right-continuous  functions can also be quantified via the Radon--Nikodým  theorem. Since $\RCR f$ is nondecreasing and right-continuous, there is  a unique complete Radon measure $\Diff \RCR f$ on $I$, the \emph{Lebesgue--Stieltjes measure} of $\RCR f$, with  $\Diff \RCR f(a,b] = \RCR f(b) - \RCR f(a)$ for  every $a,b\in I$ with $a<b$  \cite{folland1999}*{Thm.~1.16}. Moreover, $(\RCR f)'$ (and thus $f'$) forms an $\smash{\Leb^1}$-version of the density of the $\smash{\Leb^1}$-absolutely continuous part  of $\Diff \RCR f$  \cite{folland1999}*{Thm.~3.23}, i.e.
\begin{align*}
\Diff \RCR f = (\RCR f)'\,\Leb^1 + (\Diff \RCR f)^\perp = f'\,\Leb^1 + (\Diff \RCR f)^\perp.
\end{align*}

\subsubsection{Locally bounded variation}\label{Sub:Loc bdd var!} Next, we say a function $f$ is of \emph{bounded variation} on a compact interval \cite{folland1999}*{p.~102} if its usual total variation thereon is finite. We say $f$ has \emph{locally bounded variation} on $I$ if has bounded variation on every compact subinterval of $I$. This property is equivalent to $f$ being the difference of two nondecreasing functions $f_1$ and $f_2$ on $I$ \cite{folland1999}*{Thm.~3.27}. The discussion from \cref{Sub:MONOT} then extends to $f$ as follows. 

First, $f$ has right and left limits everywhere on $I$ as well as right-and left-continuous representatives $\RCR f$ and $\LCR f$  it coincides with $\smash{\Leb^1\mres I}$-a.e., respectively.  

Second, the distributional derivative $\Diff\RCR f$ of $\RCR f$   satisfies
\begin{align*}
\Diff\RCR f = \Diff \RCR f_1 - \Diff \RCR f_2.
\end{align*}
This difference makes sense as a signed Radon measure if $f$ has compact support and as a generalized signed Radon measure after \cref{Sub:Fixed} otherwise. 

Third, the derivatives of $f$, $\RCR f$, and $\LCR f$ exist and coincide $\Leb^1\mres I$-a.e. They are $\Leb^1$-versions of the density of the $\Leb^1$-absolutely continuous part of $\Diff \RCR f$.

The following integration by parts formula will be crucial in what follows.

\begin{theorem}[Integration by parts I {\cite{folland1999}*{Lem.~3.34, Thm.~3.36}}, see also \cref{Le:IBP II}]\label{Th:Partial integration f} If on the compact and nonempty interval $[a,b]$, the function $f$ has bounded variation and  $h$ is  absolutely continuous,
\begin{align*}
\int_{(a,b]} h\d\Diff \RCR f = -\int_{(a,b]} f\,h'\d\Leb^1 + \big[\!\RCR f\,h\big]_a^b.
\end{align*}
\end{theorem}

\subsection{Variable Ricci curvature bounds on the real line} Next, we recall the distortion coefficients induced by variable potentials, following  Ketterer \cite{ketterer2017}. Then we sketch how variable Ricci curvature bounds on real intervals are characterized by convexity properties of their densities, along with their  properties. This constitutes a straightforward generalization of the constant situation, e.g.~Cavalletti--Milman \cite{cavalletti-milman2021}, Cavalletti--Mondino \cite{cavalletti-mondino2020-new}, Ketterer \cite{ketterer2020-heintze-karcher}, and Burtscher--Ketterer--McCann--Woolgar \cite{burtscher-ketterer-mccann-woolgar2020}, to which we refer for details.

\subsubsection{Variable distortion coefficients}\label{Sub:Var dist coeff} For a lower semicontinuous function $\kappa$ on $[0,\infty)$, let $\smash{\SIN_\kappa}$ be the unique upper semicontinuous distributional solution to the ODE $u'' +\kappa\, u = 0$ with initial conditions $u(0) = 0$ and $u'(0) = 1$,  called \emph{generalized sine function}. Whenever needed, we extend $\kappa$ symmetrically and $\SIN_\kappa$ antisymmetrically to all of $\R$.

Throughout the sequel, $\pi_\kappa$ is the positive maximal length $L$ of the interval $[0,L)$ on which $\SIN_\kappa$ is positive. If $\kappa$ is nonpositive, it is infinite; if $\kappa$ is positive, then $\pi_\kappa$ is the first positive root of $\SIN_\kappa$ (e.g.~$\pi_\kappa = \pi/\sqrt{\kappa}$ if $\kappa$ is also constant).

Since $\kappa$ is uniformly bounded from below on every compact interval $[a,b]$ of $(0,\infty)$, say by a real number $K$, $\SIN_\kappa$ is a distributional supersolution to the ODI $u'' + K\,u \leq 0$ on $(a,b)$; since it is locally bounded from above on $(a,b)$, it is continuous and locally semiconcave on $[a,b]$ and locally Lipschitz continuous on $(a,b)$. The right and left derivatives $\smash{\SIN_\kappa'^\pm}$ exist everywhere as true limits on $(a,b)$, respectively. Moreover, $\smash{\SIN_\kappa'^+}$ also exists in $a$ and $\smash{\SIN_\kappa'^-}$ also exists in $b$ (both as true limits); the first is right-continuous on $[a,b)$ and the second is left-continuous on $(a,b]$, respectively. Lastly, on $(a,b)$ we have $\smash{\SIN_\kappa'^+\leq \SIN_\kappa'^-}$ with equality precisely at those cocountably many points where $\smash{\SIN_\kappa}$ is differentiable.

We define the \emph{generalized cosine function} $\COS_\kappa$ on $[0,\infty)$ by
\begin{align}\label{Eq:GenCOS}
\COS_\kappa(\theta) := 1-\int_0^\theta \kappa(r)\SIN_\kappa(r)\d r,
\end{align}
symmetrically extended to all of $\R$. By monotone approximation of the given potential and the above  regularity properties of $\SIN_\kappa$, the function $\smash{\COS_\kappa}$ is equal to $\smash{\SIN_\kappa'^\pm}$ at cocountably many points and in fact everywhere if $\kappa$ is continuous.

\begin{remark}[The constant case I, see also \cref{Re:Const dist coeff,Re:Logarithmic derivative}]\label{Re:Const sink cosk} If $\kappa$ is constant,  the two functions $\smash{\SIN_\kappa}$ and $\smash{\COS_\kappa}$ introduced above  read
\begin{align*}
\SIN_\kappa(\theta) &= \begin{cases} \sqrt{\kappa}^{\,-1} \sin(\sqrt{\kappa}\,\theta) & \textnormal{if }\kappa > 0,\\
\theta & \textnormal{if }\kappa=0,\\
\sqrt{-\kappa}^{\,-1}\sinh(\sqrt{-\kappa}\,\theta) & \textnormal{otherwise},
\end{cases}\\
\COS_\kappa(\theta)  &= \begin{cases} \cos(\sqrt{\kappa}\,\theta) & \textnormal{if }\kappa  >0,\\
1 & \textnormal{if }\kappa=0,\\
\cosh(\sqrt{-\kappa}\,\theta) & \textnormal{otherwise}.
\end{cases}\tag*{{\footnotesize{$\blacksquare$}}}
\end{align*}
\end{remark}

\begin{definition}[Distortion coefficients]\label{Def:Dist coeff} Given any $t\in[0,1]$, the \emph{distortion coefficients} $\smash{\sigma_\kappa^{(t)}}\colon [0,\infty)\to [0,\infty]$ associated with $\kappa$ are defined by
\begin{align*}
\sigma_{\kappa}^{(t)}(\theta) := \begin{cases} \displaystyle\frac{\SIN_\kappa(t\,\theta)}{\SIN_\kappa(\theta)} & \textnormal{\textit{if}\,}\theta < \pi_\kappa,\\
\infty & \textnormal{\textit{otherwise}}.
\end{cases}
\end{align*}
\end{definition}

\begin{remark}[The constant case II, see also \cref{Re:Const sink cosk,Re:Logarithmic derivative}]\label{Re:Const dist coeff} If  $\kappa$ is constant,  the distortion coefficients from \cref{Def:Dist coeff} take the following explicit form:
\begin{align*}
\sigma_{\kappa}^{(t)}(\theta) &= \begin{cases} \displaystyle\frac{\sin(\sqrt{\kappa}\,t\,\theta)}{\sin(\sqrt{\kappa}\,\theta)} & \textnormal{if } 0 < \kappa\,\theta^2 <\pi^2,\\
\infty & \textnormal{if } \kappa\,\theta^2 \geq \pi^2,\\
t &  \textnormal{if }\kappa\,\theta^2 = 0,\\
\displaystyle\frac{\sinh(\sqrt{-\kappa}\,t\,\theta)}{\sinh(\sqrt{-\kappa}\,\theta)} & \textnormal{otherwise}.
\end{cases}\tag*{{\footnotesize{$\blacksquare$}}}
\end{align*}
\end{remark}

For any $\theta\geq 0$,  Sturm--Picone's oscillation theorem as quoted e.g.~in \cite[Thm.~3.2]{ketterer2017}  implies either $\smash{\sigma_\kappa^{(t)}(\theta)}$ is infinite for every $t\in[0, 1]$, or $\smash{\sigma_\kappa^{(t)}(\theta)}$ is finite for every $t\in[0, 1]$. In the latter case, the assignment $\smash{u(t) := \sigma_\kappa^{(t)}(\theta)}$ obeys  $u(0) = 0$, $u(1) = 1$ and it induces a distributional solution to the ODE $\smash{u'' + \kappa(\cdot\,\theta)\,\theta^2\,u=0}$ on $(0,1)$. This yields
\begin{align*}
\sigma_\kappa^{(t)}(\theta) = \sigma_{\kappa\theta^2}^{(t)}(1),
\end{align*}
where $\kappa\,\theta^2$ is understood as the function $\smash{\tilde{\kappa}}$ on $[0,1]$ defined by
\begin{align*}
\tilde{\kappa}(t) := \kappa(t\,\theta)\,\theta^2.
\end{align*}

\begin{remark}[Properties of distortion coefficients]\label{Re:Props dist coeff} The following hold.
\begin{itemize}
\item If $\kappa$ is nonpositive,  $\smash{\sigma_{\kappa}^{(t)}(\theta)}$ is nonincreasing in $\theta\geq 0$ for every $t\in [0,1]$.
\item If a lower semicontinuous function $\kappa'$ is no larger than $\kappa$, then $\smash{\sigma_{\kappa'}^{(t)}(\theta) \leq \sigma_\kappa^{(t)}(\theta)}$ for every $t\in[0,1]$ and every $\theta\geq 0$, cf.~Ketterer \cite{ketterer2017}*{Prop.~3.4}.
\item The distortion coefficients are lower semicontinuous in the following way \cite{ketterer2017}*{Prop.~3.15}. Let $(\kappa_\imath)_{\imath\in\N}$ be a given sequence of lower semicontinuous functions on $\R$ such that for every $\theta\geq 0$,
\begin{align*}
\kappa(\theta) \leq\liminf_{\imath\to \infty}\kappa_\imath(\theta).
\end{align*}
Then for every $t\in[0,1]$ and every $\theta\geq 0$,
\begin{align*}
\sigma_\kappa^{(t)}(\theta) \leq \liminf_{\imath\to\infty}\sigma_{\kappa_\imath}^{(t)}(\theta).\tag*{{\footnotesize{$\blacksquare$}}}
\end{align*}
\end{itemize}
\end{remark}

\subsubsection{MCP densities}\label{Sub:MCP densities}  Let $I$ be a fixed interval with interior $\smash{I^\circ}$. Given  $x_0,x_1\in I$,  define the forward and back\-ward potentials $\smash{\kappa_{x_0,x_1}^\pm}$ on the interval $[0,\vert x_1-x_0\vert]$  by
\begin{align*}
\kappa_{x_0,x_1}^+(t\,\vert x_1-x_0\vert) &:= \kappa((1-t)\min\{x_0,x_1\} + t\max\{x_0,x_1\}),\\
\kappa_{x_0,x_1}^-(t\,\vert x_1-x_0\vert) &:= \kappa(t\min\{x_0,x_1\} + (1-t)\max\{x_0,x_1\}).
\end{align*}
Observe every $t\in[0,1]$ satisfies
\begin{align}\label{Eq:t vs 1-t}
\kappa_{x_0,x_1}^-(t\,\vert x_1-x_0\vert) = \kappa_{x_1,x_0}^+((1-t)\,\vert x_1-x_0\vert).
\end{align}

\begin{definition}[MCP density]\label{Def:MCP density} A Borel function $h$ on $I$ is called an \emph{$\MCP(\kappa,N)$ den\-sity} if it is nonnegative and every $x_0,x_1\in I$ and every $t\in[0,1]$ satisfy
\begin{align*}
h((1-t)\,x_0 + t\,x_1)^{1/(N-1)} \geq \sigma_{\kappa_{x_0,x_1}^-/(N-1)}^{(1-t)}(\vert x_1-x_0\vert)\,h(x_0)^{1/(N-1)}.
\end{align*}
\end{definition}

Here,  recall our simplifying assumption $1<N<\infty$ from \cref{Sub:Fixed}. The class of $\MCP(\kappa,1)$ densities $h$, e.g.~defined by saying $h$ is an $\MCP(\kappa,1+\delta)$ density for every $\delta>0$,  only consists of constant functions, e.g.~Cavalletti--Mondino \cite{cavalletti-mondino2017-isoperimetric}*{Thm.~4.2}.

\begin{remark}[Properties of MCP densities]\label{Re:Properties MCP} Let $h$ be an $\MCP(\kappa,N)$ density on $I$.
\begin{itemize}
\item The metric measure space $\smash{(\cl\, I,\vert\cdot-\cdot\vert,h\,\Leb^1\mres I)}$ obeys $\MCP(\kappa,N)$ in the variable adaptation of Cavalletti--Milman \cite{cavalletti-milman2021}*{Def.~6.8}, cf.~Ketterer \cite{ketterer2017}*{Def.~4.4}, if and only if $h$ has an $\smash{\Leb^1\mres I}$-version which is an $\MCP(\kappa,N)$ density. The proof is the same as in the constant case by Cavalletti--Mondino  \cite{cavalletti-mondino2020-new}*{Lem.~2.13}, see also  Cavalletti--Milman  \cite{cavalletti-milman2021}*{Prop.~9.1}.
\item The function $h$ can be continuously extended to the closure of $I$ \cite{cavalletti-mondino2020-new}*{Rem. 2.14}. 
 It is locally Lipschitz continuous on $I^\circ$; in  particular, it is differentiable $\Leb^1\mres I$-a.e. However, unlike the stronger case of $\CD(\kappa,N)$ densities (cf.~\cref{Re:Prop CD}), the one-sided derivatives $\smash{h^\pm}$  need not exist as true limits, and there is no general relation between both quantities. Furthermore, either $h$  vanishes identically on $I$ or it is positive on $I^\circ$. These are qualitative statements which are either clear or can be deduced from  the constant case by using local uniform lower boundedness of $\kappa$.
\item If $\kappa$ is nonnegative and $I$ is all of $\R$, then $h$ is constant \cite{cavalletti-mondino2020-new}*{Lem.~2.17}.\hfill{\footnotesize{$\blacksquare$}}
\end{itemize}
\end{remark}

Through the requirement of finiteness of $h$, it is built into the definition that the involved distortion coefficients are  finite as soon as $h$ does not vanish identically\footnote{Since $h$ will typically come from a rescaled  probability density in the sequel, its vanishing will never  occur in every relevant case for trivial reasons.}; in this case, the diameter of $I$ necessarily does not exceed $\pi_\kappa$.

Lastly, let us recall the following quantitative fact. Its standard proof basically uses \cref{Re:Properties MCP} and follows as in the constant case, e.g.~Cavalletti--Milman \cite{cavalletti-milman2021}*{Lem.~A.8} and Cavalletti--Mondino \cite{cavalletti-mondino2020-new}*{Lem.~2.17}. For variable $\kappa$, the first claim has already been stated in  Braun--McCann \cite{braun-mccann2023}*{Cor.~6.40} without proof. 

\begin{lemma}[Comparison I, see also \cref{Le:CompII}]\label{Le:logarithmic derivative} Let $h$ be a positive $\MCP(\kappa,N)$ density on a bounded open interval $(a,b)$. Then for every $x,y\in (a,b)$ with $x<y$,
\begin{align*}
\frac{\SIN_{\kappa_{x,b}^-/(N-1)}(b-y)^{N-1}}{\SIN_{\kappa_{x,b}^-/(N-1)}(b-x)^{N-1}} \leq \frac{h(y)}{h(x)} \leq \frac{\SIN_{\kappa_{a,y}^+/(N-1)}(y-a)^{N-1}}{\SIN_{\kappa_{a,y}^+/(N-1)}(x-a)^{N-1}}.
\end{align*}

In particular, at $\Leb^1$-a.e.~$x\in[a, b]$ we have
\begin{align*}
-(N-1)\,\frac{\COS_{\kappa_{x,b}^-/(N-1)}(b-x)}{\SIN_{\kappa_{x,b}^-/(N-1)}(b-x)} \leq (\log h)'(x) \leq  (N-1)\,\frac{\COS_{\kappa_{a,x}^+/(N-1)}(x-a)}{\SIN_{\kappa_{a,x}^+/(N-1)}(x-a)}.
\end{align*}
\end{lemma}

\begin{proof} To prove the first inequality, apply the certifying estimate from \cref{Def:MCP density} to $x_0 := x$ and $x_1 := b$; for every $t\in[0, 1]$, we get
\begin{align*}
h((1-t)\,x + t\,b)\geq \sigma_{\kappa_{x,b}^-/(N-1)}^{(1-t)}(b-x)^{N-1}\,h(x).
\end{align*}
We choose $t$ with $(1-t)\,x + t\,b = y$, or in other words $1-t = (b-y)/(b-x)$. Employing the definition of the involved distortion coefficient, this gives the desired inequality
\begin{align*}
h(y) \geq \frac{\SIN_{\kappa_{x,b}^-/(N-1)}(b-y)^{N-1}}{\SIN_{\kappa_{x,b}^-/(N-1)}(b-x)^{N-1}}\,h(x).
\end{align*}

The corresponding upper bound is argued analogously, this time by considering $x_0 := y$ and $x_1 := a$ while taking into account \eqref{Eq:t vs 1-t}.

The claimed bounds for $(\log h)'$ are now straightforward consequences of the already established inequalities and \cref{Re:Properties MCP}. The first holds at every differentiability point $x\in (a, b)$ of $h$ and therefore of $\log h$ since $h$ is  positive. The second is true at every differentiability point $y \in (a, b)$ of $h$.
\end{proof}

\begin{remark}[The constant case III, see also \cref{Re:Const sink cosk,Re:Const dist coeff}]\label{Re:Logarithmic derivative}  If  $\kappa$ is constant, by \cref{Re:Const sink cosk} the estimates from \cref{Le:logarithmic derivative}  simplify to the following. 
\begin{itemize}
\item In the case $\kappa > 0$,
\begin{align*}
&\frac{\sin(\sqrt{\kappa/(N-1)}\,(b-y))^{N-1}}{\sin(\sqrt{\kappa/(N-1)}\,(b-x))^{N-1}}\\ 
&\qquad\qquad \leq \frac{h(y)}{h(x)} \leq \frac{\sin(\sqrt{\kappa/(N-1)}\,(y-a))^{N-1}}{\sin(\sqrt{\kappa/(N-1)}\,(x-a))^{N-1}}
\end{align*}
and 
\begin{align*}
&-\sqrt{\kappa\,(N-1)} \,\frac{\cos(\sqrt{\kappa/(N-1)}\,(b-x))}{\sin(\sqrt{\kappa/(N-1)}\,(b-x))}\\
&\qquad\qquad  \leq (\log h)'(x)  \leq \sqrt{\kappa\,(N-1)}\,\frac{\cos(\sqrt{\kappa/(N-1)}\,(x-a))}{\sin(\sqrt{\kappa/(N-1)}\,(x-a))}.
\end{align*}
\item In the case $\kappa=0$,
\begin{align*}
\frac{(b-y)^{N-1}}{(b-x)^{N-1}} \leq \frac{h(y)}{h(x)} \leq \frac{(y-a)^{N-1}}{(x-a)^{N-1}}
\end{align*}
and
\begin{align*}
-(N-1)\,\frac{1}{b-x}\leq (\log h)'(x) &\leq (N-1)\,\frac{1}{x-a}.
\end{align*}
\item In the case $\kappa < 0$,
\begin{align*}
&\frac{\sinh(\sqrt{-\kappa/(N-1)}\,(b-y))^{N-1}}{\sinh(\sqrt{-\kappa/(N-1)}\,(b-x))^{N-1}}\\
&\qquad\qquad\leq \frac{h(y)}{h(x)} \leq \frac{\sinh(\sqrt{-\kappa/(N-1)}\,(y-a))^{N-1}}{\sinh(\sqrt{-\kappa/(N-1)}\,(x-a))^{N-1}}
\end{align*}
and
\begin{align*}
&-\sqrt{-\kappa\,(N-1)} \,\frac{\cosh(\sqrt{-\kappa/(N-1)}\,(b-x))}{\sinh(\sqrt{-\kappa/(N-1)}\,(b-x))}\\
&\qquad\qquad \leq (\log h)'(x)  \leq \sqrt{-\kappa\,(N-1)}\,\frac{\cosh(\sqrt{-\kappa/(N-1)}\,(x-a))}{\sinh(\sqrt{-\kappa/(N-1)}\,(x-a))}.\tag*{{\footnotesize{$\blacksquare$}}}
\end{align*}
\end{itemize}
\end{remark}

The following qualitative result we will repeatedly use later  is taken from Cavalletti--Milman \cite{cavalletti-milman2021}*{Lem.~A.8} and Cavalletti--Mondino \cite{cavalletti-mondino2020-new}*{Lems. 2.15, 2.16}.

\begin{lemma}[A priori estimates]\label{Le:Bounds} Let $h$ be a given $\MCP(\kappa,N)$ density on a nonempty, bounded, and open interval $(a,b)$, where $\kappa$ is a negative  \emph{constant}. Then the subsequent a priori estimates hold.
\begin{enumerate}[label=\textnormal{(\roman*)}]
\item \textnormal{\textbf{Zeroth order uniform bound.}} We have
\begin{align*}
\sup_{x\in (a,b)} h(x) \leq \frac{1}{b-a}\,\Big[\!\int_0^1\sigma_{\kappa/(N-1)}^{(r)}(b-a)^{N-1}\d r\Big]^{-1} \int_{[a,b]} h\d\Leb^1.
\end{align*}
\item \textnormal{\textbf{First order integral bound.}} There is a function $\smash{C_\cdot^{\kappa,N}}$ on $(0,\infty)$ which certifies the subsequent properties. It is nondecreasing, uniformly bounded on $(0,R)$ for every $R>0$, we have $\smash{C^{\kappa,N}_r\to \infty}$ as $r\to \infty$, and
\begin{align*}
\int_{[a,b]} \vert h'\vert\d \Leb^1 \leq \frac{1}{b-a}\,C_{b-a}^{\kappa,N} \int_{[a,b]} h\d\Leb^1.
\end{align*}
\end{enumerate}
\end{lemma} 

\subsubsection{CD densities}\label{Sub:CD density} 

\begin{definition}[CD density]\label{Def:CD density} A Borel function $h$ on $I$  is  termed a \emph{$\CD(\kappa,N)$ density} if it is nonnegative and every $x_0,x_1\in I$ and every $t\in[0,1]$ satisfy
\begin{align*}
h((1-t)\,x_0 + t\,x_1)^{1/(N-1)} &\geq \sigma_{\kappa_{x_0,x_1}^-/(N-1)}^{(1-t)}(\vert x_1-x_0\vert)\,h(x_0)^{1/(N-1)}\\
&\qquad\qquad + \sigma_{\kappa_{x_0,x_1}^+/(N-1)}^{(t)}(\vert x_1-x_0\vert)\,h(x_1)^{1/(N-1)}.
\end{align*}
\end{definition}

\begin{remark}[Properties of CD densities]\label{Re:Prop CD} Let $h$ be a $\CD(\kappa,N)$ density on $I$.
\begin{itemize}
\item Evidently, $h$ forms an $\MCP(\kappa,N)$ density. In particular, it satisfies all qualitative properties from \cref{Re:Properties MCP}.
\item More strongly, $h$ is locally semiconcave on $I^\circ$. In turn, the one-sided derivatives $\smash{h'^\pm}$ exist as true limits everywhere. Furthermore, $h'^+$ exists at $\inf I$ and is right-continuous on $I$; $h'^-$ exists at $\sup I$ and is left-continuous on $I$, cf.~Burtscher--Ketterer--McCann--Woolgar \cite{burtscher-ketterer-mccann-woolgar2020}*{Rem.~2.12}. Lastly, $\smash{h'^+\leq h'^-}$ everywhere on $I$, taking into account our convention \eqref{Eq:CONVEN}.
\item The metric measure space $\smash{(\bar{I},\vert\cdot-\cdot\vert, h\,\Leb^1\mres \bar{I})}$ obeys the finite-dimensional variable $\CD(\kappa,N)$  introduced by Ketterer \cite{ketterer2017}*{Def.~4.4} if and only if $h$ has an $\smash{\Leb^1\mres I}$-version which is a $\CD(\kappa,N)$ density.\hfill{\footnotesize{$\blacksquare$}}
\end{itemize}
\end{remark}

A standard characterization is the following, cf.~Cavalletti--Milman \cite{cavalletti-milman2021}*{Lem.~A.3} for constant $\kappa$. Its integrated form later induces our  Bochner inequality, cf.~\cref{Th:From TCD to Bochner}.

\begin{lemma}[Differential characterization]\label{Le:Diffchar} A positive locally semiconcave function $h$ constitutes a $\CD(\kappa,N)$ density on an open interval $I$ if and only if at every twice  differen\-tiability point of $h$ in $I$,
\begin{align*}
(\log h)'' + \frac{1}{N-1}\,\big\vert (\log h)'\big\vert^2 = (N-1)\,\frac{(h^{1/(N-1)})''}{h^{1/(N-1)}} \leq -\kappa.
\end{align*}
\end{lemma}

\begin{remark}[Local-to-global property \cite{cavalletti-milman2021,cavalletti-sturm2012}]\label{Re:Localtoglobal} An important direct  consequence from \cref{Le:Diffchar} is the following. Suppose an open interval $I$ can be covered by open intervals on each of  which $h$ is a $\CD(\kappa,N)$ density. Then $h$ is a $\CD(\kappa,N)$ density on $I$.\hfill{\footnotesize{$\blacksquare$}}
\end{remark}

\subsubsection{Jacobian function}\label{Sub:Ball Jac} Finally, we provide variable adaptations of constructions appearing in  Ketterer \cite{ketterer2020-heintze-karcher} and Burtscher--Ketterer--McCann--Woolgar \cite{burtscher-ketterer-mccann-woolgar2020} for constant $\kappa$ (see the references therein for prior work, especially Milman \cite{milman2015}). 

For a real number $\lambda$, we define the \emph{potential  function} $\smash{\sfP_{\kappa,\lambda}}$ on $\R$ through
\begin{align*}
\sfP_{\kappa,\lambda}(\theta) := \COS_\kappa(\theta) + \lambda\SIN_\kappa(\theta).
\end{align*}
For instance, if $\kappa$ is constant, by \cref{Re:Const sink cosk} we have
\begin{align*}
\sfP_{\kappa,\lambda}(\theta) = \begin{cases} \cos(\sqrt{\kappa}\,\theta) + \lambda\sqrt{\kappa}^{-1}\sin(\sqrt{\kappa}\,\theta) &\textnormal{if }\kappa>0,\\
1 + \lambda\,\theta & \textnormal{if }\kappa=0,\\
\cosh(\sqrt{-\kappa}\,\theta) + \lambda\sqrt{-\kappa}^{-1}\sinh(\sqrt{-\kappa}\,\theta) & \textnormal{otherwise}.
\end{cases}
\end{align*}

\begin{remark}[Transformation under reflection]\label{Re:Trafo refl} In the previous setting, the symmetry and antisymmetry properties of $\cos_\kappa$ and $\sin_\kappa$ on $\R$, respectively, easily entail the formulas $\sfP_{\kappa,\lambda}(-\theta) = \sfP_{\kappa,-\lambda}(\theta)$ and $\sfP_{\kappa,\lambda}'(-\theta) = -\sfP_{\kappa,-\lambda}(\theta)$.\hfill{\footnotesize{$\blacksquare$}}
\end{remark}

We say the pair $(\kappa,\lambda)$ obeys the \emph{ball condition} if $\smash{\sfP_{\kappa,\lambda}}$ has a positive root we denote by $\theta_{\kappa,\lambda}$. We conventionally set $\theta_{\kappa,\lambda} := \infty$ if the ball condition does not hold for $(\kappa,\lambda)$.

\begin{definition}[Jacobian function]\label{Def:Jacobian} For a real number $H$, define the \emph{Jacobian function} $\smash{\sfJ_{\kappa,N,H}}$  by the assignment
\begin{align*}
\sfJ_{\kappa,N,H} := \big[\sfP_{\kappa/(N-1), H/(N-1)} \big]_+^{N-1}.
\end{align*}
\end{definition}

In the setting of \cref{Def:Jacobian} and before, the number $\theta_{\kappa/(N-1),H/(N-1)}$ is precisely the maximal length $L$ of the interval $(0,L)$ on which $\smash{\sfJ_{\kappa,N,H}}$ is positive. It is finite if and only if the pair $\smash{(\kappa/(N-1),H/(N-1))}$ obeys the ball condition. In particular, since the Jacobian function $\smash{\sfJ_{\kappa,N,H}}$ is symmetric, by definition its open positivity interval around the origin is enclosed by $\smash{\pm \theta_{\kappa/(N-1),H/(N-1)}}$.

We state the following results without proof. By employing  \cref{Le:Diffchar}, \cref{Re:Properties MCP},  and the formula \eqref{Eq:GenCOS}, with evident modifications the arguments from the constant setting carry over to the variable case; cf.~Burtscher--Ketterer--McCann--Woolgar \cite{burtscher-ketterer-mccann-woolgar2020}*{Lems.~3.1, 4.1, Rem.~3.2} and Ketterer \cite{ketterer2020-heintze-karcher}*{Cor.~4.3}.

\begin{lemma}[Comparison II, see also \cref{Le:logarithmic derivative}]\label{Le:CompII} Let $h$ constitute a $\CD(\kappa,N)$ density on a compact interval $[a,b]$ such that $h$ is positive on $[a,b)$, where $a\leq 0 < b$. Assume $\smash{H^+ := (\log h)'^+(0)}$ is a real number\footnote{This holds e.g.~if $a$ is negative.}. Then on $(0,b)$ we have 
\begin{align*}
h \leq h(0)\,\sfJ_{\kappa,N, H^+}
\end{align*}
and in particular $b\leq \theta_{\kappa/(N-1),H^+/(N-1)}$.
\end{lemma}

\begin{remark}[Reverse parametrization] If we have instead $a<0\leq b$ in \cref{Le:CompII}, the last conclusion therein translates into $a\geq -\theta_{\kappa/(N-1),H^-/(N-1)}$ with $\smash{H^- := (\log h)'^-(0)}$, provided the latter is a real number.\hfill{\footnotesize{$\blacksquare$}}
\end{remark}

\begin{lemma}[Variable Riccati comparison]\label{Le:Riccati} For a real number $d$, let $v$ be a nonnegative and continuous distributional solution to the ODI $v'' + \kappa\,v\leq 0$ on a given interval  $[0,b)$  satisfying $v(0) = 1$ and $v'(0) \leq d$. Then $b\leq \theta_{\kappa,-d}$ and on $[0,b)$ we have
\begin{align*}
(\log v)'^+ \leq (\log \sfP_{\kappa,-d})'^+.
\end{align*}
\end{lemma}

Recall from \cref{Sub:Var dist coeff} that all right and left  logarithmic  derivatives appearing in the above statements  are always well-defined.

\subsection{Metric measure spacetimes}\label{Sub:MMS} We will work on  (globally hyperbolic, regular length) \emph{metric measure spacetimes} as set up by Braun--McCann \cite{braun-mccann2023}, which are reviewed in this section. Some of the properties we require a priori in \cref{Def:MMS} actually follow from weaker hypotheses. The setup of \cite{braun-mccann2023} was based on Minguzzi--Suhr's  theory of bounded Lorentzian metric spaces \cite{minguzzi-suhr2022}, which itself adapts Kunzinger--Sämann's Lorentzian length spaces  \cite{kunzinger-samann2018}. We refer to McCann \cite{mccann2023-null} and Beran et al.~\cite{beran-braun-calisti-gigli-mccann-ohanyan-rott-samann+-} for related discussions. We stress \cref{Def:MMS} includes globally hyperbolic, regular Lorentz\-ian length spaces in the terminology of Kunzinger--Sämann \cite{kunzinger-samann2018}. We assume the reader to have  a basic knowledge of causality theory, cf.~e.g.~Minguzzi \cite{minguzzi2019-causality}, and will thus be brief with precise definitions and instead fix nomenclature and notation.

A function $\smash{l\colon\mms^2\to [0,\infty)\cup\{-\infty\}}$ on a topological space $\mms$ is called a \emph{signed time separation function}  if 
\begin{itemize}
\item it is upper semicontinuous,
\item its positive part $l_+$ is continuous,
\end{itemize}
and every triple $x,y,z\in\mms$ satisfies
\begin{itemize}
\item $l(x,y) > -\infty$ and $l(y,x)>-\infty$ simultaneously if and only if $x=y$, and
\item the \emph{reverse triangle inequality} 
\begin{align*}
l(x,z) \geq l(x,y) + l(y,z).
\end{align*}
\end{itemize}
We point out the subtle additional assumption $l$ does not attain the value $\infty$.

We write $\leq$ and $\ll$ for the relations of  \emph{causality} and \emph{chronology}  on $\mms$ induced by $l$, respectively. We write $J := \{l\geq 0\}$ and $I := \{l>0\}$. $I$ will have empty intersection with the diagonal of $\smash{\mms^2}$ \cite{braun-mccann2023}*{Cor.~2.14}. For $x,y\in\mms$ and subsets $X$ and $Y$ of $\mms$, we  employ  the usual notations $\smash{J^+(x)}$,  $\smash{J^+(X)}$, $\smash{J^-(y)}$, $\smash{J^-(Y)}$, $J(x,y)$, and $J(X,Y)$ for   causal futures, pasts, diamonds, and --- if $X$ and $Y$ are compact --- emeralds, respectively; analogously with all  occurrences  of $J$ and ``causal'' replaced by $I$ and  ``chronological'', respectively. 

A \emph{causal}, \emph{strictly causal}, or \emph{timelike} curve is a continuous curve\footnote{We use the terms ``curve'' and ``path'' synonymously and do \emph{not} a priori  include continuity therein.} $\gamma\colon[0,1]\to\mms$  along which $\leq$, $\leq$ plus nowhere constancy, or $\ll$ propagate, respectively. If $\gamma$ is not required to be continuous, we add the adjective \emph{rough} to these tags. 

We write $\Len_l$ for the length functional on rough causal curves induced by $l$.

\begin{definition}[Metric measure spacetime]\label{Def:MMS} A triple $\scrM := (\mms,l,\meas)$ is termed a \emph{metric measure spacetime} if it certifies the following properties.
\begin{enumerate}[label=\textnormal{\alph*.}]
\item \textnormal{\textbf{Space.}} $\mms$ is a Polish space.
\item \textnormal{\textbf{Nontrivial chronology} \cite{braun-mccann2023}*{Def.~2.7}\textbf{.}} Every $x\in\mms$ is the midpoint of a timelike curve.
\item \textnormal{\textbf{Length property} \cite{braun-mccann2023}*{Def.~2.7}\textbf{.}} If $x,y\in\mms$ satisfy $x\leq y$, then
\begin{align*}
l(x,y) = \sup\Len_l\gamma,
\end{align*}
where the supremum runs over all strictly causal curves $\gamma$ from $x$ to $y$. 
\item \textnormal{\textbf{Global hyperbolicity} \cite{minguzzi2023}*{Def.~1.1}\textbf{.}} The causality relation $\leq$ is a closed order and $J(C,C)$ is compact for every compact subset $C$ of $\mms$.
\item \textnormal{\textbf{Regularity} \cite{braun-mccann2023}*{Def.~2.40}\textbf{.}} Every strictly causal curve which  maximizes $\Len_l$ and connects chronologically related points is already timelike.
\item \textnormal{\textbf{Measure.}} The measure $\meas$ is a  Radon measure on $\mms$.
\end{enumerate}
\end{definition}

Moreover, to simplify the presentation we assume 
\begin{itemize}
\item $\mms$ is proper\footnote{\cref{Def:MMS} already makes $\mms$ locally compact yet noncompact \cite{braun-mccann2023}*{Cor.~2.17}.}, 
\item the measure $\meas$ has full topological support, symbolically $\supp\meas =\mms$, and 
\item in view of \cref{Sub:Iso pert} and \cref{Ex:timedistfunct} it will be convenient to assume $\meas$ assigns zero outer measure $\smash{\meas^*}$ to every achronal subset of $\mms$.
\end{itemize}
Recall a set is \emph{achronal} if it contains no two timelike related points.

We refer to Braun--McCann \cite{braun-mccann2023}*{Exs.~2.43, 2.44}, \cref{Ex:TCD}, and \cref{Sub:LorentzFinsler}  for examples covering \cref{Def:MMS}.

\begin{remark}[Time-reversal]\label{Re:Causal reversal} A  useful structure in the following is the \emph{time-reversal} of $l$, i.e.~the signed time separation function $\smash{l^\leftarrow\colon\mms^2\to [0,\infty)\cup\{-\infty\}}$ given by
\begin{align*}
l^\leftarrow(x,y) :=l(y,x).
\end{align*}

The metric measure spacetime induced by $\mms$, $\smash{l^\leftarrow}$, and $\meas$ is denoted $\smash{\scrM^\leftarrow}$ and termed the causally reversed structure of $\scrM$.\hfill{\footnotesize{$\blacksquare$}}
\end{remark}

\begin{definition}[Geodesics on $\mms$]\label{Def:GeosM} A curve $\gamma$ will be called an \emph{$l$-geodesic} if for every $s,t\in[0,1]$ with $s<t$, we have
\begin{align*}
l(\gamma_s,\gamma_t) = (t-s)\,l(\gamma_0,\gamma_1)>0.
\end{align*}
\end{definition}

The set of all $l$-geodesics will be denoted by $\TGeo(\mms)$.

A subset $E$ of $\mms$ is \emph{$l$-geodesically convex}  if every $\gamma\in\TGeo(\mms)$ whose endpoints lie in $E$  stays entirely within $E$.

An important point of regularity is that every $x,y\in\mms$ with $x\ll y$ are joined by an $l$-geodesic and in fact every maximizer of $\Len_l$ connecting $x$ and $y$ can be reparametrized to an $l$-geodesic, cf.~Braun--McCann  \cite{braun-mccann2023}*{Thm.~2.36, Prop.~2.37}. Regularity also ensures every element of $\TGeo(\mms)$ is continuous,  as first noted by McCann  \cite{mccann2023-null}*{Lem.~5}. Lastly, for every $r>0$ the set of all $\gamma\in\TGeo(\mms)$ with $l(\gamma_0,\gamma_1)\geq r$ is compact with respect to uniform convergence \cite{braun-mccann2023}*{Cor.~2.47}; in particular, $\TGeo(\mms)$ is $\sigma$-compact, hence Borel.

\begin{remark}[Relaxation] The machinery of Braun--McCann  \cite{braun-mccann2023} (and thus our paper) can also be developed if $\mms$ constitutes merely a causally convex closed subset of a metric measure spacetime. In particular, we may allow $\mms$ to be compact or to violate nontrivial chronology. To retain a clear presentation, we refrain from making this extraneous hypothesis.\hfill{\footnotesize{$\blacksquare$}}
\end{remark}

\subsection{Signed Lorentz distance functions}\label{Sec:Signed} An important class of functions we will consider below are signed Lorentz distance functions, which we shortly review here inspired by Treude--Grant \cite{treude-grant2013} and Cavalletti--Mondino \cite{cavalletti-mondino2020}. Their induced localizations have been used by the latter \cite{cavalletti-mondino2020} and later in Braun--McCann \cite{braun-mccann2023} to establish synthetic Hawking-type singularity theorems. In our context, such functions will become especially relevant in \cref{Sec:Formula2}. The results reviewed now are folklore, hence their elementary proofs are omitted.

Let $\Sigma$ be an achronal subset of $\mms$. Achronality yields  the well-definedness of the \emph{signed Lorentz distance function} $l_\Sigma\colon \mms \to \R \cup\{-\infty,\infty\}$, where
\begin{align*}
l_\Sigma(x) := \sup_{x-\in \Sigma} l_+(x-,x) - \sup_{x+\in \Sigma} l_+(x,x+).
\end{align*}
The function $\smash{\pm l_\Sigma}$ is clearly lower semicontinuous on $\smash{I^\pm(\Sigma)\cup \Sigma}$.

In the following, we consider the set
\begin{align}\label{Eq:ESigmaUSigma}
E_\Sigma := I^+(\Sigma) \cup \Sigma \cup I^-(\Sigma).
\end{align}

When $\Sigma$ is a singleton $\{o\}$, we write $l_\Sigma$ and $E_\Sigma$ as  $l_o$ and $E_o$, respectively.

\begin{lemma}[Geodesic convexity]\label{Le:Geod conv E} The set $\smash{E_\Sigma}$ is $l$-geodesically convex.
\end{lemma}

The following property will be tagged $1$-steepness in \cref{Sub:Iso pert} below.

\begin{corollary}[Steepness]\label{Cor:Steepness} For every $\smash{x,y\in E_\Sigma}$,
\begin{align*}
l_\Sigma(y) \geq l_\Sigma(x) + l(x,y).
\end{align*}
\end{corollary}

A further assumption is required in order for the suprema in the definition of $\smash{l_\Sigma}$ to be attained. To this aim, we recall the following notions introduced by Galloway \cite{galloway1986}.

\begin{definition}[Timelike completeness]\label{Def:timelike complete} We term $\Sigma$
\begin{enumerate}[label=\textnormal{\alph*.}]
\item \emph{future timelike complete}, briefly FTC, if for every $\smash{y\in I^+(\Sigma)}$ the set $\smash{J^-(y) \cap \Sigma}$ has compact closure in $\Sigma$,
\item \emph{past timelike complete}, briefly PTC, if for every $\smash{x\in I^-(\Sigma)}$ the set $\smash{J^+(x)\cap \Sigma}$ has compact closure in $\Sigma$, and
\item \emph{timelike complete}, briefly TC, if it is simultaneously FTC and PTC.
\end{enumerate}
\end{definition}

By straightforward arguments, cf.~e.g.~Cavalletti--Mondino  \cite{cavalletti-mondino2020-new}*{Lem.~1.8}, $\Sigma$ being FTC ensures the supremum in the definition of $\smash{l_\Sigma(y)}$ is attained for every $y\in I^+(\Sigma)$; in other words, there exists $x\in \Sigma$ with  $\smash{l_\Sigma(y) = l(x,y)}$, in which case we call $x$ a \emph{footpoint} of $y$. In particular, $\smash{l_\Sigma}$ is finite on $\smash{I^+(\Sigma)}$. It is also not difficult to show $\smash{l_\Sigma}$ is continuous on $\smash{I^+(\Sigma)}$ under the FTC condition, as established by Treude--Grant \cite{treude2011}*{Cor.~3.2.20}.  Analogous statements hold relative to $\smash{I^-(\Sigma)}$ if $\Sigma$ is PTC. 

\subsection{Topological local anti-Lipschitz condition}\label{Sub:Top Anti} Assuming the following stronger property on a signed Lorentz distance function\footnote{\cref{Def:KS} evidently makes sense by replacing $l_\Sigma$ with a general $l$-causal  function $\u$. Since we only need it in the former case, we stick to the present formulation.} $l_\Sigma$ as above is helpful in turning distributions into generalized signed Radon measures, especially via  \cref{Le:VsTop}. A specific setting where it  applies is described in  \cref{Sub:LorentzFinsler}.

\begin{definition}[Topological local anti-Lipschitz condition \cite{kunzinger-steinbauer2022}*{p.~4328}]\label{Def:KS} Given an open subset $U$ of $E_\Sigma$, we call $\smash{l_\Sigma}$ \emph{topologically anti-Lipschitz} on $U$ if there exists a metric $\met_U$ on $U$ which generates the relative topology $\smash{\mathcal{T}_U}$ of $U$ and every $x,y\in U$ with $x\leq y$ satisfy 
\begin{align*}
l_\Sigma(y) - l_\Sigma(x) \geq \met_U(x,y).
\end{align*}

Accordingly, we call $\smash{l_\Sigma}$ \emph{topologically locally anti-Lipschitz} on $U$ if every point in $U$ has an open neighborhood on which $l_\Sigma$ is topologically anti-Lipschitz.
\end{definition}

In particular, if $l_\Sigma$ is topologically locally anti-Lipschitz on $U$, it is strictly increasing along the causality relation $\leq$.

\begin{remark}[Topological part of \cref{Def:KS}] The previous notion by Kunzinger--Steinbauer \cite{kunzinger-steinbauer2022} was motivated by the anti-Lipschitz condition of Sormani--Vega \cite{sormani-vega2016}*{Def. 4.4} which does not require $\smash{\mathcal{T}_U}$ and the topology generated by $\met_U$ to be compatible. The property we will need for \cref{Le:VsTop} below is continuity of $\met_U$ with respect to $\smash{\mathcal{T}_U}$, which is a priori weaker than $\smash{\met_U}$ generating $\smash{\mathcal{T}_U}$. However, straightforward topological considerations (cf.~e.g.~the proof of Minguzzi--Suhr \cite{minguzzi-suhr2022}*{Prop.~1.5}) show in our locally compact setting the induced weakening of topological local anti-Lipschitzness is in fact equivalent to the second part of \cref{Def:KS}.\hfill{\footnotesize{$\blacksquare$}}
\end{remark}

The feature of the previous anti-Lipschitz condition from Sormani--Vega \cite{sormani-vega2016} is that it characterizes definiteness of their so-called \emph{null distance} \cite{sormani-vega2016}*{Def.~3.2} from the related function, cf.~\cite{sormani-vega2016}*{Prop.~4.5} and \cite{kunzinger-steinbauer2022}*{Prop.~3.12}. We prefer to work with the ``local'' \cref{Def:KS} over the null distance since the finiteness of the latter would require some connectivity hypotheses, cf.~\cite{sormani-vega2016}*{Lem.~3.5}  and \cite{kunzinger-steinbauer2022}*{Def.~3.4}.

\section{Distributional d'Alembertian}\label{Ch:Dalem}

\subsection{Sobolev calculus}\label{Sub:Sobo} 

\subsubsection{Probabilistic notation}\label{Sub:Prob Not} Before reviewing the Lorentzian Sobolev calculus developed in Beran et al.~\cite{beran-braun-calisti-gigli-mccann-ohanyan-rott-samann+-} (inspired by the contribution of Ambrosio--Gigli--Savaré \cite{ambrosio-gigli-savare2014-calculus} and Gigli \cite{gigli2015} in metric measure geometry), we fix some notation.

Let $\scrP(\mms)$ denote the space of Borel probability measures. It is  endowed with the usual narrow topology. Its subsets of compactly supported or $\meas$-absolutely continuous elements are denoted  $\Prob_\comp(\mms)$ or  $\smash{\Prob^\ac(\mms,\meas)}$, respectively; their intersection is denoted  ${\Prob_\comp^\ac(\mms,\meas)}$. 

Given  $t\in[0,1]$, the \emph{evaluation map} $\eval_t\colon \Cont([0,1];\mms)\to\mms$ is defined by $\eval_t(\gamma) := \gamma_t$. By a \emph{plan}, we mean a Borel probability measure $\bdpi$ on $\Cont([0,1];\mms)$. A curve $(\mu_t)_{t\in[0,1]}$ with values in $\Prob(\mms)$ is \emph{represented} by a plan $\bdpi$ if $\mu_t = (\eval_t)_\push\bdpi$ for every $t\in[0,1]$. Lastly, the \emph{reverse plan} of $\bdpi$ is the Borel probability measure $\smash{\bdpi^\leftarrow := \sfR_\push\bdpi}$ on $\Cont([0,1];\mms)$, where $\sfR\colon \Cont([0,1];\mms)\to \Cont([0,1];\mms)$ is defined by $\sfR(\gamma)_t := \gamma_{1-t}$. We also use the same notation for Borel probability measures on path spaces which contain $\Cont([0,1];\mms)$, cf.~\cref{Sub:CSpeed}.

We say $\mu\in\Prob(\mms)$ has \emph{bounded compression} provided $\mu\leq C\,\meas$ for some $C>0$; the least such constant $C$  is written $\COMP\mu$. Accordingly, a plan $\bdpi$ has \emph{bounded compression} provided $\smash{\COMP\bdpi := \sup_{t\in[0,1]} \COMP(\eval_t)_\push\bdpi < \infty}$.

\subsubsection{Causal and steep functions}\label{Sub:Iso pert} We call a function $f$ on $\mms$ 
\begin{itemize}
\item \emph{$l$-causal} if $f(y)\geq f(x)$ whenever $x,y\in\mms$ obey $x\leq y$, and
\item \emph{$L$-steep}, where $L$ is a given nonnegative constant, if every $x,y\in \mms$ satisfy
\begin{align*}
f(y) \geq f(x) + L\,l(x,y).
\end{align*}
\end{itemize}
These properties can also be set up for functions merely defined on an $\meas$-measurable set or which attain the values $\pm\infty$ (subject to appropriate conventions \cite{beran-braun-calisti-gigli-mccann-ohanyan-rott-samann+-}*{§2.1}). To streamline the delivery and since we do not deal with infinite quantities, we concentrate on globally defined and real-valued functions; \cref{Re:Extension,Re:ExtensionII} use locality to extend the theory to functions which are only defined on an appropriate subset of $\mms$. 

Where defined, the \emph{time-reversal} of a function $f$ is 
\begin{align*}
f^{\leftarrow} := -f.
\end{align*}
This notation and terminology are well-motivated from the time-reversal $\smash{l^\leftarrow}$ defined in  \cref{Re:Causal reversal}, for $f$ is $l$-causal if and only if $\smash{f^\leftarrow}$ is $\smash{l^\leftarrow}$-causal. Analogously, the function $f$ is $1$-steep with respect to $l$ if and only if $\smash{f^\leftarrow}$ is $1$-steep with respect to $\smash{l^\leftarrow}$.

For the rest of \cref{Sub:Sobo}, let $f$ be an $l$-causal  function on $\mms$.

Using nontrivial chronology and the fact that $f$ is everywhere real-valued, the following lemma is verified in the same way as continuous functions are established to attain their maxima and minima on compact sets.

\begin{lemma}[Boundedness]\label{Le:Bounded} The function $f$  is locally bounded.
\end{lemma}

Now we shortly discuss right- and left-continuity properties of $f$ in analogy to the real case of \cref{Sub:Monotone real}. The \emph{$l$-right-} and \emph{$l$-left-continuous representative} of $f$, respectively, are the $l$-causal functions $\RCR f$ and $\LCR f$ on $\mms$ defined by
\begin{align*}
\RCR f(x) &:= \inf_{x+\in I^+(x)} f(x+),\\
\LCR f(x) &:= \sup_{x-\in I^-(x)} f(x-).
\end{align*}

\begin{lemma}[Almost right- and left-continuity \cite{beran-braun-calisti-gigli-mccann-ohanyan-rott-samann+-}*{Lem.~3.2}]\label{Le:Almost rightleft} The function $f$ lies between $\LCR f$ and $\RCR f$ everywhere on $\mms$. 

In fact, outside a countable union of achronal sets the functions  $\RCR f$ and $\LCR f$ \textnormal{(}and thus $f$\textnormal{)} coincide and at every such point the function $f$ is in fact continuous.
\end{lemma}

\begin{remark}[Properties of representatives]\label{Re:Prosp} The following hold.
\begin{itemize}
\item The function $\RCR f$ is $l$-right-continuous in the following sense. If $x\in\mms$ is the limit of a sequence in $\smash{I^+(x)}$, then
\begin{align*}
\lim_{n\to\infty}\RCR f(x_n) = \RCR f(x).
\end{align*}
Analogously, $\LCR f$ is $l$-left-continuous in the evident way.
\item As $\meas$ has full support and its outer measure $\meas^*$ charges no  achronal subsets of $\mms$ by \cref{Sub:MMS}, the functions $f$, $\RCR f$, and $\LCR f$ coincide $\meas$-a.e.~by \cref{Le:Almost rightleft}. All three functions are thus $\meas$-measurable \cite{beran-braun-calisti-gigli-mccann-ohanyan-rott-samann+-}*{Cor.~3.3}. 
\item These notions are consistent with \cref{Sub:Monotone real} as follows. If $\gamma$ is a continuous  timelike curve, then $\RCR(f\circ\gamma) = (\RCR f) \circ \gamma$ on $[0,1)$ [sic]. Indeed, the functions $\RCR(f\circ\gamma)$ and --- since  $\gamma$ crosses the set $\{\RCR f \neq f\}$ at most countably many times by \cref{Le:Almost rightleft} --- $(\RCR f)\circ \gamma$ both coincide with $f\circ \gamma$ at cocountably many points. Since $\RCR(f\circ\gamma)$ and $(\RCR f)\circ\gamma$ are right-continuous, the claim follows. The identity $\LCR(f\circ\gamma) = (\LCR f)\circ\gamma$ on $(0,1]$ [sic] is shown analogously.\hfill{\footnotesize{$\blacksquare$}}
\end{itemize}
\end{remark}

\subsubsection{Causal speed}\label{Sub:CSpeed} Let us briefly recall the Lorentzian analog of the metric notion of speed of absolutely continuous curves (discussed e.g.~in Ambrosio--Gigli--Savaré \cite{ambrosio-gigli-savare2008}*{§1.1})  introduced by Beran et al.~\cite{beran-braun-calisti-gigli-mccann-ohanyan-rott-samann+-}. The regularity of $l$-causal curves assumed in their work is left-continuity. The space $\LCC([0,1];\mms)$ of all such paths is in fact Polish \cite{beran-braun-calisti-gigli-mccann-ohanyan-rott-samann+-}*{Prop.~2.29}. As we deal with continuous paths in all relevant cases, we do not go into more details here and refer to \cite{beran-braun-calisti-gigli-mccann-ohanyan-rott-samann+-}*{§2.5} for more information.

Let $\gamma$ be a left-continuous $l$-causal curve. By \cite{beran-braun-calisti-gigli-mccann-ohanyan-rott-samann+-}*{Thm.~2.23} the limit
\begin{align*}
\vert \dot\gamma_t\vert := \lim_{n\to\infty} \frac{l(\gamma_{s_n}, \gamma_{t_n})}{t_n -s_n}
\end{align*}
exists for $\Leb^1$-a.e.~$t\in[0,1]$, where $(s_n)_{n\in\N}$ and $(t_n)_{n\in\N}$ are arbitrary sequences in $[0,1]$ which converge to $t$ such that $s_n<  t_n$ for every $n\in\N$; in particular, $\vert \dot\gamma\vert$ is $\Leb^1\mres[0,1]$-a.e. independent of the choice of these sequences. We call $\vert\dot\gamma\vert$ the \emph{$l$-causal speed} of $\gamma$. This quantity forms the $\smash{\Leb^1\mres[0,1]}$-a.e.~largest function whose average on $[s,t]$ is bounded from above by $l(\gamma_s,\gamma_t)$ for every $s,t\in [0,1]$ with $s<t$ (simply by maximality).

Evidently, if $\gamma\in \TGeo(\mms)$ then every $t\in[0,1]$ satisfies
\begin{align*}
\vert \dot\gamma_t\vert = l(\gamma_0,\gamma_1).
\end{align*}

\subsubsection{Weak subslope and $\PP$-Cheeger energy}

\begin{definition}[Test plan \cite{beran-braun-calisti-gigli-mccann-ohanyan-rott-samann+-}*{Def.~3.7}]\label{Def:Test plan} A Borel probability measure $\bdpi$ on $\LCC([0,1];\mms)$ is called a \emph{test plan} if it has bounded compression.
\end{definition}

For the existence of test plans assuming the curvature hypotheses  from \cref{Sub:TMCP,Sub:TCD} (but without any nonbranching condition), see the contributions of Braun \cite{braun2023-good}*{Thm.~1.2}, Braun--McCann \cite{braun-mccann2023}*{Thm.~3.26}, and Beran et.~al.~\cite{beran-braun-calisti-gigli-mccann-ohanyan-rott-samann+-}*{Thm.~5.7}. A specific test plan with stronger properties is built in \cref{Pr:Plan representing}.

\begin{definition}[Weak subslope \cite{beran-braun-calisti-gigli-mccann-ohanyan-rott-samann+-}*{Def.~3.11}]\label{Def:WSS} An $\meas$-measurable function $G$ is termed a \emph{weak subslope} of $f$ if it is $\meas$-measurable and for every test plan $\bdpi$,
\begin{align*}
\int \big[f(\gamma_1) - f(\gamma_0)\big]\d\bdpi(\gamma) \geq\iint_0^1 G(\gamma_t)\,\vert\dot\gamma_t\vert\d t\d\bdpi(\gamma).
\end{align*}
\end{definition}

\begin{theorem}[Maximal weak subslope \cite{beran-braun-calisti-gigli-mccann-ohanyan-rott-samann+-}*{Thm.~3.14}]\label{Th:MaxWSS} There  exists a unique maximal function $\vert \rmd f\vert$ on $\mms$, called \emph{maximal weak subslope} of $f$, which is a weak subslope of $f$. 
\end{theorem}

Here ``maximal'' means if $G$ is a weak subslope of $f$,  then $\vert \rmd f\vert\geq G$ $\meas$-a.e. Uniqueness  easily follows from this, as do the subsequent results about concavity plus non\-negative $1$-homogeneity \cite{beran-braun-calisti-gigli-mccann-ohanyan-rott-samann+-}*{Prop.~3.21}. If $f$ and $g$ are $l$-causal and $\lambda$ is positive, 
\begin{alignat*}{3}
\vert\rmd(\lambda f + g)\vert &\geq \lambda\,\vert\rmd f\vert + \vert\rmd g\vert &&&\quad &\meas\textnormal{-a.e.,}\\
\vert\rmd(\lambda f)\vert &= \lambda\,\vert\rmd f\vert &&&& \meas\textnormal{-a.e.}
\end{alignat*}

Lastly, by the average property of the $l$-causal speed from  \cref{Sub:CSpeed}, if the function $f$ is $L$-steep for a nonnegative real number $L$, then
\begin{align*}
\vert \rmd f\vert\geq L\quad\meas\textnormal{-a.e.}
\end{align*}

\begin{remark}[Behavior under time-reversal]\label{Re:Behcaus} Since the property of being a test plan is not invariant under time-reversal qua inherent regularity, it is unclear to us if the maximal weak subslope from  \cref{Th:MaxWSS} is the same for both structures. (On the other hand, the notion of \cref{Sub:CSpeed} is $\Leb^1\mres[0,1]$-a.e.~independent on whether right- or left-continuous versions are chosen.) If \cref{Def:WSS} was ``symmetric'' under time-reversal, the $\meas$-a.e.~equality of the two induced weak subslopes would easily follow from maximality.

This does not cause issues for our developments. The relevant test plans constructed in \cref{Pr:Plan representing} are  concentrated on \emph{continuous} paths. Irrespective of the causal orientation, the upper bound from \cref{Th:df vs f'} and  our representation formulas from  \cref{Sub:Formula1,Sec:Formula2} will be the same in both scenarios up to a sign.\hfill{\footnotesize{$\blacksquare$}}
\end{remark}

\begin{lemma}[Pathwise version of \Cref{Def:WSS} \cite{beran-braun-calisti-gigli-mccann-ohanyan-rott-samann+-}*{Lem.~3.13}]\label{Le:Pathwise} Let $\bdpi$ form a test plan. Then $\bdpi$-a.e.~$\gamma$ satisfies the following property for every $s,t\in[0,1]$ with $s<t$:
\begin{align*}
f(\gamma_t) - f(\gamma_s) \geq \int_s^t \vert\rmd f\vert(\gamma_r)\,\vert\dot\gamma_r\vert\d r.
\end{align*}
\end{lemma}

\begin{proposition}[Calculus rules I \cite{beran-braun-calisti-gigli-mccann-ohanyan-rott-samann+-}*{Props.~3.22, 3.23, 3.24}, see also \cref{Pr:Calc rulez II}]\label{Pr:Calc rulez} Let $f$ and $g$ be $l$-causal functions. Then the following statements hold.
\begin{enumerate}[label=\textnormal{\textcolor{black}{(}\roman*\textcolor{black}{)}}]
\item \textnormal{\textbf{Locality I.}} We have
\begin{align*}
\vert\rmd f\vert = \vert \rmd g\vert\quad \meas\mres\{f=g\}\textnormal{-a.e.}
\end{align*}
\item \textnormal{\textbf{Locality II.}} Let $N$ be a Borel subset of $\R$ with $\smash{\Leb^1[N]=0}$. Then 
\begin{align*}
\vert\rmd f\vert = 0\quad\meas\mres f^{-1}(N)\textnormal{-a.e.}
\end{align*}
\item \textnormal{\textbf{Chain rule.}}\label{La:CHAIN} Assume $\varphi$ is a nondecreasing and Lipschitz continuous function on the image of $f$. Then
\begin{align*}
\vert\rmd(\varphi\circ f)\vert = \varphi'\circ f\,\vert\rmd f\vert\quad\meas\textnormal{-a.e.},
\end{align*}
where $\varphi'$ denotes the density of the $\Leb^1$-absolutely continuous part of the Lebesgue--Stieltjes measure $\Diff\RCR\varphi$ and the quantity $\varphi'\circ f$ is defined arbitrarily at all points in the image of $f$ at which $\varphi$ is not differentiable.
\item \textnormal{\textbf{Leibniz rule.}}\label{La:LEIB} If $f$ and $g$ are nonnegative,
\begin{align*}
\vert \rmd(f\,g)\vert \geq f\,\vert\rmd g\vert + g\,\vert \rmd f\vert\quad\meas\textnormal{-a.e.}
\end{align*}
\end{enumerate}
\end{proposition}

The functions considered in \ref{La:CHAIN} and \ref{La:LEIB} above are clearly $l$-causal.

\begin{remark}[Extension I, see also \cref{Re:ExtensionII}]\label{Re:Extension} The locality properties of \cref{Pr:Calc rulez} enable us to make sense of  the maximal weak subslope of $l$-causal functions $f_0$ which are only given on an $\meas$-measurable subset $W$ of $\mms$. Fix $o\in \mms$ and consider the signed Lorentz distance function $l_o$ from \cref{Sec:Signed}. By an evident variant of \cref{Cor:Steepness}, $l_o$ is $1$-steep on $\mms$. Consequently, $f_\varepsilon := f_0 + \varepsilon\, l_o$ is $\varepsilon$-steep on $W$ for every $\varepsilon > 0$. We  extend $f_\varepsilon$ to $\mms$ by Beran et al.'s McShane-type lemma \cite{beran-braun-calisti-gigli-mccann-ohanyan-rott-samann+-}*{Lem.~3.5}, e.g.~through the assignment
\begin{align*}
\tilde{f}_\varepsilon(y) := \sup_{x\in W} \big[f_\varepsilon(x) + \varepsilon\,l(x,y)\big].
\end{align*}
The monotone limit $\smash{\tilde{f}_0 := \lim_{\varepsilon\to 0+}\tilde{f}_\varepsilon}$ is the desired $l$-causal extension of $f_0$\footnote{The extension may attain the values $\pm \infty$. However, since it clearly coincides with $f$ on $W$, it is real-valued on the relevant set $W$.}. 

Finally, on $W$ we set $\smash{\vert \rmd f_0\vert := \vert \rmd \tilde{f}_0\vert}$. By \cref{Pr:Calc rulez}, the quantity $\vert\rmd f_0\vert$ is $\meas\mres W$-a.e.~independent of the chosen extension of $f_0$.\hfill{\footnotesize{$\blacksquare$}}
\end{remark}

Let $\PP$ and $\QQ$ be conjugate nonzero numbers less than $1$. Following Beran et al.~\cite{beran-braun-calisti-gigli-mccann-ohanyan-rott-samann+-}, the convex \emph{$\PP$-Cheeger energy}  on $l$-causal functions defined on an open subset $U$ of $\mms$ is
\begin{align}\label{Eq:q-Cheeger}
\scrE_\PP(f) := \frac{1}{\QQ}\int_U \vert \rmd f\vert^\PP\d\meas.
\end{align}

\subsection{Horizontal vs.~vertical differentiation}\label{Sub:Horiz} Now we introduce a central concept for our subsequent work, namely horizontal and vertical differentiation by Beran et al.~\cite{beran-braun-calisti-gigli-mccann-ohanyan-rott-samann+-} inspired by Gigli's predecessor  \cite{gigli2015} in positive signature. It relates two  approaches to differentiating an $l$-causal function $\u$ ``in the direction'' of a not necessarily $l$-causal function $f$. These two approaches are linked through an integrated inequality based on the reverse Young inequality in \cite{beran-braun-calisti-gigli-mccann-ohanyan-rott-samann+-}*{Prop.~4.1}. We do not state this result and adjacent terminology here; in \cref{Th:df vs f'}, we anyway improve it by localizing it to an a.e.~bound, which will constitute the basis for our  integration by parts approach to the d'Alembertian (\cref{Def:DAlem}).

Throughout the rest of this \cref{Ch:Dalem}, let $\u$ be an $l$-causal function on $\mms$. For simplicity, we assume $\vert \rmd\u\vert$ is positive and finite $\meas$-a.e. This will hold in much stronger form in all relevant cases in \cref{Ch:Repr form,Ch:Appli}, cf. \cref{Cor:Constant slope} and the chain rule from \cref{Pr:Calc rulez}.

\subsubsection{Finite perturbations}\label{Sub:PRTUB} Let $W$ be an $\meas$-measurable subset of $\mms$.

\begin{definition}[Finite perturbation \cite{beran-braun-calisti-gigli-mccann-ohanyan-rott-samann+-}*{Def.~4.3}]\label{Def:Perturbations} A function $f$ on $W$  will be called  a \emph{finite perturba\-tion} of $\u$, symbolically $\smash{f\in\FPert(\u,W)}$, if there is $\tau > 0$ such that $\u + \tau f$ is $l$-causal on $W$.

A finite perturbation $f$ is \emph{symmetric} if $\smash{f^\leftarrow}$ belongs to $\smash{\FPert(\u,W)}$ as well.
\end{definition}

In Beran et al.~\cite{beran-braun-calisti-gigli-mccann-ohanyan-rott-samann+-}, perturbations are in fact allowed to attain the values $\pm\infty$. We do not work with such degenerate functions (hence the addendum ``finite'').

As the property of $l$-causality is closed under addition, every $l$-causal function on $W$ is a (generally nonsymmetric) finite perturbation of $\u$ thereon.

The sets of functions in $\smash{\FPert(\u,W)}$ which are bounded, have relatively compact support in $W$, or obey both properties simultaneously  are denoted $\Pert_\bounded(\u,W)$, $\smash{\FPert_\comp(\u,W)}$, and $\Pert_\bc(\u,W)$, respectively. 

If $W$ is all of $\mms$, we shall abbreviate the sets introduced above by $\smash{\FPert(\u)}$, $\smash{\Pert_\bounded(\u)}$, $\smash{\FPert_\comp(\u)}$, and $\smash{\Pert_\bc(\u)}$, respectively.

\begin{remark}[Admissible range]\label{Re:Adm range!} Thanks to the $l$-causality of $\u$, if $\tau$ is as in \cref{Def:Perturbations}, then $\u+\tau' f$ is $l$-causal for \emph{every} $\tau' \in (0,\tau]$; in other words, the defining property of $f$ is equivalently certified by a  range of positive scalars around zero. 

In particular, the class $\smash{\FPert(\u,W)}$ from \cref{Def:Perturbations} is stable under addition and multiplication with nonnegative real numbers.\hfill{\footnotesize{$\blacksquare$}}
\end{remark}

\begin{remark}[Extension II, see also \cref{Re:Extension}]\label{Re:ExtensionII} By using the decomposition
\begin{align}\label{Eq:Decomp prt}
f = \tau^{-1}\,(\u + \tau f) - \tau^{-1}\,\u
\end{align}
for a suitable $\tau > 0$ and \cref{Re:Extension}, we can extend every finite perturbation $f$ of $\u$ only defined on an $\meas$-measurable subset to all of $\mms$.

In fact, if $f\in \Pert_\bc(\u,U)$ for some \emph{open} subset $U$ of $\mms$, it is not hard to infer that its natural extension by zero to all of $\mms$ is an element of $\Pert_\bc(\u,\mms)$. This is the interpretation of the global extension of functions in $\Pert_\bc(\u,U)$ we use throughout this paper.\hfill{\footnotesize{$\blacksquare$}}
\end{remark}

It thus suffices to make the subsequent basic considerations for $\smash{f\in\FPert(\u)}$.

\begin{remark}[Properties of finite perturbations]\label{Re:Prop fin p} By using \eqref{Eq:Decomp prt}, we extend the properties of $l$-causal functions from \cref{Re:Prosp} to $f$ as follows.
\begin{itemize}
\item The function $f$ has unique $l$-right- and $l$-left continuous representatives $\RCR f$ and $\LCR f$ it coincides with outside countably many achronal sets in $\mms$. In particular, the functions $f$, $\RCR f$, and $\LCR f$ are $\meas$-measurable and coincide $\meas$-a.e.~by our standing  assumption on $\meas$ from  \cref{Sub:MMS}.
\item For every continuous timelike curve $\gamma$, the function $f\circ\gamma$ has bounded variation according to  \cref{Sub:Loc bdd var!}. In particular, we have $\RCR(f\circ \gamma) = (\RCR f)\circ \gamma$ on $[0,1)$ [sic] and $\LCR(f\circ\gamma) = (\LCR f)\circ\gamma$ on $(0,1]$ [sic], respectively.\hfill{\footnotesize{$\blacksquare$}}
\end{itemize}
\end{remark}

The trick of extending results and quantities from $l$-causal  functions to general finite perturbations using \eqref{Eq:Decomp prt} will be recurrent throughout our work, e.g.~\cref{Sub:Analytic preps}.

Since we frequently have to relate different sets of finite  perturbations (pertaining to the chain rule, see especially \cref{Pr:Chain rule,Th:11,Cor:22}), we note the following result separately. 

\begin{lemma}[Relations of classes of finite perturbations \cite{beran-braun-calisti-gigli-mccann-ohanyan-rott-samann+-}*{Prop.~4.8}] \label{Pr:Relations} Let $\varphi$ form a function defined on the image of $\u$. Then the following hold.
\begin{enumerate}[label=\textnormal{\textcolor{black}{(}\roman*\textcolor{black}{)}}] 
\item\label{La:R1} \textnormal{\textbf{Composition I.}} If $\varphi$ is Lipschitz continuous, we have $\smash{\varphi\circ\u\in \FPert(\u)}$.
\item\label{La:R2} \textnormal{\textbf{Composition II.}} If $\varphi$ is strictly increasing with Lipschitz continuous inverse, then every element of $\smash{\FPert(\u)}$ belongs to $\smash{\FPert(\varphi\circ\u)}$. 
\item\label{La:R212} \textnormal{\textbf{Composition III.}} If $\varphi$ is nondecreasing and Lipschitz continuous, then every $\smash{f\in \FPert(\u)}$ satisfies $\smash{\varphi\circ f\in\FPert(\u)}$.
\item\label{La:R3} \textnormal{\textbf{Multiplication.}} If $\varphi$ is nonnegative, bounded, and Lipschitz continuous, for every $\smash{f\in\Pert_\bounded(\u)}$ we have $\smash{f\,\varphi\circ\u \in \Pert_\bounded(\u)}$.
\end{enumerate}
\end{lemma}

As an addendum to item \ref{La:R2} above,  recall that if the function $\varphi$ from \cref{Pr:Relations} is continuously differentiable on an open set containing the image of $\u$, Lipschitz continuity of its inverse is equivalent to the derivative of $\varphi$ being bounded away from zero.

\subsubsection{Vertical difference and differential quotients}\label{Sub:Vert diff} We now recall ``vertical'' notions of  ``differentiation'' of $\u$ in a fixed direction $\smash{f\in\FPert(\u)}$ computed by perturbing $\u$ by $\tau f$ for sufficiently small $\tau > 0$. Vertical \emph{right} differentiation has been set up  by Beran et al.~\cite{beran-braun-calisti-gigli-mccann-ohanyan-rott-samann+-}. In view of \cref{Re:Two sided,Re:Symmmmmmmm} and \cref{Sub:LorentzFinsler}, we recall the analogous concept of vertical \emph{left} differentiation as well. This treatment is inspired by Gigli \cite{gigli2015} in positive signature.

Let $\PP$ be a nonzero number less than $1$. By our assumption on $\vert\rmd \u\vert$ being positive and finite $\meas$-a.e.,  we $\meas$-a.e.~define the \emph{vertical right differential quotient}
\begin{align}\label{Eq:Vert diff}
\rmd^+f(\nabla\u)\,\vert \rmd\u\vert^{\PP-2} &:= \lim_{\tau\to 0+} \frac{\vert\rmd(\u+\tau f)\vert^\PP - \vert\rmd\u\vert^\PP}{\PP\tau}.
\end{align}
By concavity and using the convention that the $\tau$-dependent term is $-\infty$ if $\u+\tau f$ is not an $l$-causal function, the dependence of the term inside the limit on $\tau>0$ is nonincreasing $\meas$-a.e. \cite{beran-braun-calisti-gigli-mccann-ohanyan-rott-samann+-}*{§4.3}. Therefore, \eqref{Eq:Vert diff} is an $\meas$-essential supremum.

By concavity and nonnegative $1$-homogeneity,  if $f$ and $g$ are finite perturbations of $\u$ and $\lambda$ is positive (in particular, $f$ is a finite perturbation of $\lambda\u$ by \cref{Pr:Relations}), 
\begin{alignat*}{3}
\rmd^+(f + g)(\nabla\u)\,\vert\rmd\u\vert^{\PP-2} &\geq \rmd^+f(\nabla\u)\,\vert\rmd\u\vert^{\PP-2} + \rmd^+g(\nabla\u)\,\vert\rmd\u\vert^{\PP-2} &&&\quad &\meas\textnormal{-a.e.,}\\
\rmd^+(\lambda f)(\nabla\u)^{\PP-2}\,\vert\rmd\u\vert^{\PP-2} &= \lambda\,\rmd^+f(\nabla\u)\,\vert\rmd\u\vert^{\PP-2} &&&& \meas\textnormal{-a.e.},\\
\rmd^+f(\nabla(\lambda\u))\,\vert\rmd(\lambda\u)\vert^{\PP-2} &= \lambda^{\PP-1}\,\rmd^+f(\nabla\u)\,\vert\rmd\u\vert^{\PP-2} &&&&\meas\textnormal{-a.e.}
\end{alignat*}

\begin{proposition}[Calculus rules II \cite{beran-braun-calisti-gigli-mccann-ohanyan-rott-samann+-}*{Prop.~4.11}, see also \cref{Pr:Calc rulez}]\label{Pr:Calc rulez II} Let $f$ and $g$ be finite perturbations of both $\u$ and $\v$, where $\v$ obeys the same properties as $\u$. Then the following hold.
\begin{enumerate}[label=\textnormal{\textcolor{black}{(}\roman*\textcolor{black}{)}}]
\item\label{La:Null1} \textnormal{\textbf{Locality I.}} We have
\begin{align*}
\rmd^+ f(\nabla\u)\,\vert \rmd\u\vert^{\PP-2} = \rmd^+g(\nabla\v)\,\vert\rmd\v\vert^{\PP-2}\quad\meas\mres[\{f=g\}\cap \{\u=\v\}]\textnormal{-a.e.}
\end{align*}
\item\label{La:Null2} \textnormal{\textbf{Locality II.}} Let $N$ be a Borel subset of $\R$ with $\smash{\Leb^1[N]=0}$. Then
\begin{align*}
\rmd^+f(\nabla\u)\,\vert\rmd\u\vert^{\PP-2} = 0\quad\meas\mres f^{-1}(N)\textnormal{-a.e.}
\end{align*}
\item\label{La:Null25} \textnormal{\textbf{Chain rule I.}} If $\varphi$ is continuously differentiable and Lipschitz continuous,
\begin{align*}
\rmd^+(\varphi\circ\u)(\nabla\u)\,\vert\rmd\u\vert^{\PP-2} = \varphi'\circ\u\,\vert\rmd\u\vert^\PP\quad\meas\textnormal{-a.e.}
\end{align*}
\item\label{La:Null3} \textnormal{\textbf{Chain rule II.}} Assume  $\varphi$ constitutes a nondecreasing, continuously differentiable, and Lipschitz constinuous function. Then
\begin{align*}
\rmd^+(\varphi\circ f)(\nabla\u)^{\PP-2}\,\vert\rmd\u\vert^{\PP-2} = \varphi'\circ f\,\rmd^+f(\nabla\u)\,\vert\rmd\u\vert^{\PP-2}\quad  \meas\textnormal{-a.e.}
\end{align*}
In addition, if the derivative of $\varphi$ is bounded away from zero,
\begin{align*}
\rmd^+f(\nabla(\varphi\circ \u))\,\vert\rmd(\varphi\circ \u)\vert^{\PP-2} =(\varphi')^{\PP-1}\circ \u \,\rmd^+f(\nabla\u)\,\vert\rmd\u\vert^{\PP-2}\quad\meas\textnormal{-a.e.}
\end{align*}
\item\label{La:Null4} \textnormal{\textbf{Leibniz rule.}} Assume $f$ and $g$ are nonnegative. Then
\begin{align*}
\rmd^+(f\,g)(\nabla\u)\,\vert\rmd\u\vert^{\PP-2} \geq f\,\rmd^+g(\nabla\u)\,\vert\rmd\u\vert^{\PP-2} + g\,\rmd^+f(\nabla\u)\,\vert\rmd\u\vert^{\PP-2}\quad\meas\textnormal{-a.e.}
\end{align*}
In addition, if $f$ is uniformly bounded from below and  $\varphi$ is nonnegative, bounded, continuously differentiable, and Lipschitz continuous, 
\begin{align*}
\rmd^+(f\,\varphi\circ\u)(\nabla\u)\,\vert\rmd\u\vert^{\PP-2}  \geq f\,\varphi'\circ\u\,\vert\rmd\u\vert^\PP + \varphi\circ\u\,\rmd^+f(\nabla\u)\,\vert\rmd\u\vert^{\PP-2}\quad\meas\textnormal{-a.e.}
\end{align*}
\end{enumerate}
\end{proposition}

We turn to vertical left differentiation. Here, we assume  $f$ forms a symmetric finite perturbation of $\u$. Then we $\meas$-a.e.~introduce the \emph{vertical left differential quotient}
\begin{align*}
\rmd^-f(\nabla\u)\,\vert \rmd\u\vert^{\PP-2} &:= \lim_{\sigma\to 0-} \frac{\vert\rmd(\u+\sigma f)\vert^\PP - \vert\rmd\u\vert^\PP}{\PP\sigma}.
\end{align*}
Again by concavity, the dependence of the term inside the limit  on $\sigma<0$ is nonincreasing $\meas$-a.e.  
In particular, the previous limit is an $\meas$-essential infimum.

\begin{remark}[Exponent independence]\label{Re:Indep!} If $\vert \rmd\u\vert = 1$ $\meas$-a.e., then basic features of convex functions show the quantities $\smash{\rmd^+f(\nabla\u)\,\vert\rmd\u\vert^{\PP-2}}$ and $\smash{\rmd^-f(\nabla\u)\,\vert\rmd\u\vert^{\PP-2}}$, whenever defined, are $\meas$-a.e. independent of $\PP$ as the latter ranges over $(-\infty,1)\setminus \{0\}$.\hfill{\footnotesize{$\blacksquare$}}
\end{remark}

In general, we have
\begin{align}\label{Eq:We have ineq}
\rmd^+f(\nabla\u)\,\vert\rmd\u\vert^{\PP-2} \leq \rmd^-f(\nabla\u)\,\vert\rmd\u\vert^{\PP-2}\quad\meas\textnormal{-a.e.}
\end{align}
If $\scrM$ is \emph{infinitesimally Minkowskian} in the sense of Beran et al.~\cite{beran-braun-calisti-gigli-mccann-ohanyan-rott-samann+-}*{Def.~1.4}, this is an equality for a special class of functions $f$ and $\u$ \cite{beran-braun-calisti-gigli-mccann-ohanyan-rott-samann+-}*{Thm.~4.19}. In \cref{Sub:LorentzFinsler}, we will describe a smoother yet ``nonquadratic'' setting in which equality holds, see the proof of \cref{Th:Main I Finsler,Th:Main II Finsler} therein. In general, equality above relates to strict concavity of the given ``Lorentz\-ian norm'' (or  the exclusion of multivalued gradients), see \cref{Re:Inf str conv}. 

The following related \emph{equality} should be compared to its analog by Gigli \cite{gigli2015}*{Prop.~3.15} in positive signature and will become relevant in \cref{Re:Two sided}.

\begin{lemma}[Right vs.~left differentiation]\label{Le:Diff quot} If $f$ is a symmetric finite perturbation of $\u$,
\begin{align*}
\rmd^+(f^\leftarrow)(\nabla\u)\,\vert\rmd\u\vert^{\PP-2} = -\rmd^-f(\nabla\u)\,\vert\rmd\u\vert^{\PP-2}\quad\meas\textnormal{-a.e.}
\end{align*}
\end{lemma}

\begin{proof} This follows by tracing through the definitions. Indeed, since $\u + \tau f$ is $l$-causal for every nonzero $\tau\in \R$ [sic] sufficiently close to zero by \cref{Re:Adm range!},
\begin{align*}
\rmd^+(f^\leftarrow)(\nabla\u)\,\vert\rmd\u\vert^{\PP-2} &= \lim_{\tau\to 0+} \frac{\vert\rmd(\u - \tau f)\vert^\PP - \vert\rmd\u\vert^\PP}{\PP\tau}\\
&= -\lim_{\tau\to 0+} \frac{\vert\rmd(\u - \tau f)\vert^\PP - \vert\rmd\u\vert^\PP}{\PP(-\tau)}\\
&= -\rmd^-f(\nabla\u)\,\vert\rmd\u\vert^{\PP-2}\quad\meas\textnormal{-a.e.}\qedhere 
\end{align*}
\end{proof}

In particular, all properties of the vertical right differential quotient, notably \cref{Pr:Calc rulez II}, transfer to the vertical left differential quotient with suitable  modifications.

\subsection{Definition and basic properties}\label{Sub:Def basis}  We now recall the distributional notion of ($\PP$-)\-d'Alembertian of Beran et al.~\cite{beran-braun-calisti-gigli-mccann-ohanyan-rott-samann+-}, where $\PP$ is any given nonzero number less than one. It reflects the notion of ``generators'' of maximal monotone operators in Hilbert spaces in terms of their subdifferentials and their metric generalizations of Ambrosio--Gigli--Savaré \cite{ambrosio-gigli-savare2008} after which the metric measure Laplacian was set up in by the same authors \cite{ambrosio-gigli-savare2014-calculus}*{Def.~4.13}; compare with \cref{Re:Brezis}   below.

Let $\u$ be an $l$-causal function on an open subset $U$ of $\mms$ with $\vert \rmd \u\vert > 0$ $\meas\mres U$-a.e. A real-valued map $T$ on $\Pert_\bc(\u,U)$ is termed
\begin{itemize}
\item \emph{nonnegatively linear} if $T(\lambda\,f + g) = \lambda\,T(f) + T(g)$ for every $\smash{f,g\in \Pert_\bc(\u,U)}$ and every nonnegative real number $\lambda$,  
\item \emph{linear} if the previous statement holds for all real numbers $\lambda$.
\end{itemize}

\begin{definition}[Weak Radon functional] A map $T$ as above is called a a \emph{weak\footnote{In general, it is unclear how $\smash{\Pert_\bc(\u,U)}$ relates to $\smash{\Cont_\comp(U)}$, hence the addendum ``weak'' in the nomenclature; see \cref{Le:VsTop} for a relation between these sets under a topological local anti-Lipschitz condition, however.} Radon functional} on $U$ if it is nonnegatively linear and for every compact subset $W$ of $U$ there is a constant $C$ such that for every $f\in\Pert_\bc(\u,U)$ with support in $W$,
\begin{align*}
\vert T(f)\vert\leq C\,\Vert f\Vert_\infty.
\end{align*}
\end{definition}

A weak Radon functional $T$ on $U$ is \emph{nonnegative} if $T(f) \geq 0$ for every nonnegative $\smash{f\in\Pert_\bc(\u,U)}$. Given two such weak Radon functionals $S$ and $T$, we say $S\leq T$ in the sense of weak Radon functionals if the Radon functional $T-S$ is nonnegative.

\begin{definition}[D'Alembertian \cite{beran-braun-calisti-gigli-mccann-ohanyan-rott-samann+-}*{Def.~5.30}]\label{Def:DAlem} Given any nonzero $\PP$ less than $1$, we term $\u$ to lie in the domain of the \emph{$\PP$-d'Alembertian}, symbolically $\u\in\Dom(\BOX_\PP\mres U)$, if there exists a weak  Radon functional $T$  on $\smash{\Pert_\bc(\u,U)}$ such that for every $\smash{f\in\Pert_\bc(\u,U)}$, 
\begin{align*}
\int_\mms \rmd^+f(\nabla \u)\,\vert\rmd \u\vert^{\PP-2}\d\meas \leq -T(f).
\end{align*}
\end{definition}

The convex set of functionals $T$ as above is denoted $\BOX_\PP\u\mres U$. We set
\begin{align*}
\Dom(\BOX \mres U) := \bigcap_{\PP\in (-\infty,1)\setminus \{0\}} \Dom(\BOX_\PP\mres U)
\end{align*}
and call every of its elements to belong to the domain of the \emph{d'Alembertian}. A sufficient condition for the latter is  of course the nonemptiness of the set
\begin{align*}
\BOX\u\mres U := \bigcap_{\PP\in (-\infty,1)\setminus \{0\}} \BOX_\PP\u\mres U.
\end{align*}

\begin{remark}[One-sided vs.~two-sided bounds]\label{Re:Two sided} Assume in the context of \cref{Def:DAlem}, $f$ is a symmetric finite perturbation of $\u$ on $U$; recall this means $\smash{f^\leftarrow}$ belongs to $\smash{\Pert_\bc(\u,U)}$ too. Furthermore, assume $T$ is in fact linear\footnote{In our main framework, our functionals $T$ derived in \cref{Sub:Formula1,Sec:Formula2} will be indeed linear.}. Then applying the certifying  inequality from \cref{Def:DAlem} to $f$ and $\smash{f^\leftarrow}$, respectively, and using \cref{Le:Diff quot} yields
\begin{align*}
\int_\mms \rmd^+f(\nabla\u)\,\vert\rmd\u\vert^{\PP-2}\d\meas \leq -T(f) \leq \int_\mms\rmd^-f(\nabla\u)\,\vert\rmd\u\vert^{\PP-2}\d\meas.
\end{align*}
Thus, in ``good'' cases in which $\u$ has sufficiently many symmetric finite  perturbations, our one-sided \cref{Def:DAlem} is thus  able to recover an inequality tightly linked to a non\-smooth integration by parts formula; cf.~e.g.~\cref{Th:Main I Finsler,Th:Main II Finsler} and their proof in \cref{Sub:LorentzFinsler}. An analogous argument has been followed in \cite{beran-braun-calisti-gigli-mccann-ohanyan-rott-samann+-}*{Prop.~5.31} to enhance some of their results under the hypothesis of infinitesimal Minkowskianity.

On metric measure spaces as treated in Gigli \cite{gigli2015}, the analog of the above two-sided bound is  more natural since there is simply no question about symmetry. On the other hand, our above observation is related to its predecessor in this case by Gigli \cite{gigli2015}*{Prop.~4.13}. Therein, he   starts from a one-sided inequality (related to an abstract  Laplace comparison result) and derives a two-sided inequality by changing signs in the relevant test functions. His proof uses the Hahn--Banach theorem, although this abstract way to get a linear functional is redundant under infinitesimal Hilbertianity \cite{gigli2015}*{Rem.~4.14}.  

Cavalletti--Mondino's approach via localization provides such a linear map \emph{explicitly} \cite{cavalletti-mondino2020-new}*{Thms.~4.8, 4.14, Cor. 4.16}.\hfill{\footnotesize{$\blacksquare$}}
\end{remark}

\begin{remark}[About linearity I, see also \cref{Re:About II}]\label{Re:About I} For functions $f,g\in\Pert_\bc(\u,U)$, the nonnegative linearity of $T$ in \cref{Def:DAlem} entails the inequality therein for $f+g$ in place of $f$. By concavity, this bound is stronger than the one which results  from adding up the individual defining  estimates for $f$ and $g$, respectively.\hfill{\footnotesize{$\blacksquare$}}
\end{remark}

\begin{remark}[Connection to \cite{ambrosio-gigli-savare2008,ambrosio-gigli-savare2014-calculus}]\label{Re:Brezis} Recall the definition \eqref{Eq:q-Cheeger} of the $\PP$-Cheeger energy $\smash{\scrE_\PP}$ on $U$ and assume $\smash{\scrE_\PP(u) < \infty}$. For $\smash{f\in\Pert_\bc(\u,U)}$, let $\tau > 0$ small enough such that $\u + \tau f$ is $l$-causal. If $\smash{T\in \BOX_\PP\u \mres U}$, the negativity of $\PP-1$ and definition \eqref{Eq:Vert diff} yield
\begin{align*}
-(\PP-1)\,T(\tau f) &\leq \tau(\PP-1)\int_U \rmd^+f(\nabla\u)\,\vert\rmd\u\vert^{\PP-2}\d\meas \leq \scrE_\PP(\u + \tau f) - \scrE_\PP(\u).
\end{align*}
Therefore, in a generalized sense $-(\PP-1)\,T$ can be interpreted as an element of the sub\-differential of the convex $\PP$-Cheeger energy on $U$. This is in accordance with the Laplacian $\Delta$ on metric measure spaces: for suitable  Sobolev functions $f$, $-\Delta f$ is the unique element with minimal $L^2$-norm in the subdifferential of the Cheeger energy at $f$. This definition of Laplacian for general metric measure spaces was introduced by Ambrosio--Gigli--Savaré  \cite{ambrosio-gigli-savare2014-calculus}*{Def.~4.13}.\hfill{\footnotesize{$\blacksquare$}}
\end{remark}

\subsection{Calculus rules}\label{Sub:Calc rulez}  We list properties of our $\PP$-d'Alembertian, where $\PP$ is as in \cref{Def:DAlem}. They transfer to analogous statements and definitions for the d'Alembertian.

\begin{remark}[Properties of the d'Alembertian]\label{Re:Props alemb} In the framework of \cref{Def:DAlem}, the class $\Dom(\BOX_\PP\mres U)$ is in general not closed under addition, hence no real cone.

However, the $\PP$-d'Alembertian is positively $(p-1)$-homogeneous in the following way. Let $\lambda$ be a positive real number and let $\u\in\Dom(\BOX_\PP\mres U)$. Given  $T\in\BOX_\PP\u\mres U$, note that $\lambda^{p-1}\,T\in \Dom(\BOX_\PP(\lambda\,\u))$ by the positive $(p-1)$-homogeneity of the quantity $\smash{\rmd^+f(\nabla \u)\,\vert\rmd\u\vert^{\PP-2}}$ in the second functional variable. In turn, this yields $\lambda\,\u\in \Dom(\BOX_\PP\mres U)$. Flipping this argument on its head gives 
\begin{align*}
\lambda^{p-1}\,\BOX_\PP\u\mres U = \BOX_\PP(\lambda\,\u)\mres U.
\end{align*}
This property should be reminiscent of the general scaling property of the Laplacian on metric measure spaces, cf.~Ambrosio--Gigli--Savaré \cite{ambrosio-gigli-savare2014-calculus}*{Rem.~4.14} and Gigli \cite{gigli2015}*{p.~37}.

Similarly, the chain rule and locality imply the invariance of $\Dom(\BOX_\PP\mres U)$ under shifts by arbitrary constants.

The $\PP$-d'Alembertian also has the following ``global-to-local property''. If $V$ is an open set contained in $U$, then the restriction of every element of $\BOX_\PP\u\mres U$ to $\Pert_\bc(\u,V)$ belongs to $\BOX_\PP\u\mres V$, symbolically $\smash{\u\big\vert_V \in\Dom(\BOX_\PP\mres V)}$. The left-hand side of the defining inequality localizes to $V$ by locality properties of the integrand.\hfill{\footnotesize{$\blacksquare$}} 
\end{remark}

A further important property is the chain rule. The following  should be compared to  its counterpart for metric measure spaces by Gigli \cite{gigli2015}*{Prop.~4.11}.

\begin{proposition}[Chain rule]\label{Pr:Chain rule} Assume that $\vert\rmd\u\vert^\PP$ is locally $\smash{\meas\mres U}$-integrable. Let $\varphi$ be a strictly increasing twice continuously differentiable function on an open set containing the image $\u(U)$. Then for every $\smash{T\in\BOX_\PP\u\mres U}$, the functional $S$ on $\smash{\Pert_\bc(\varphi\circ\u,U)}$ given by
\begin{align*}
S(f) := T(f\,(\varphi')^{\PP-1}\circ \u) + (\PP-1)\int_\mms f\,\big[(\varphi')^{\PP-2}\,\varphi''\big]\circ \u\, \vert\rmd \u\vert^\PP\d\meas
\end{align*}
belongs to $\smash{\BOX_\PP(\varphi\circ\u)\mres U}$.
\end{proposition}

\begin{proof} Let $\smash{f\in\Pert_\bc(\varphi\circ\u,U)}$ be fixed. By locality, without loss of generality we may and will replace $U$ by a precompact neighborhood $W$ of $\supp f$ whose closure is contained in $U$. Then $\smash{\u(W)}$ is precompact in $\R$ by \cref{Le:Bounded}; consequently, by assumption there is $\varepsilon > 0$ with $\varepsilon < \varphi' < 1/\varepsilon$ on $\cl\,\u(W)$. Moreover, $\varphi''$ is bounded  on the same set.

\cref{Pr:Relations} implies the following correspondences between different classes of finite perturbations we tacitly use through the rest of this proof. First, as $\smash{\varphi^{-1}}$ is strictly increasing and Lipschitz continuous on $\cl\,\u(W)$ we have $\smash{f\in \Pert_\bc(\u,W)}$. Second, since $\smash{(\varphi')^{\PP-1}}$ is Lipschitz continuous on $\cl\,\u(W)$ we obtain $\smash{(\varphi')^{\PP-1}\circ\u\in\FPert(\u,W)}$. Lastly, since $f$ is bounded and $\smash{(\varphi')^{\PP-1}}$ is nonnegative, bounded from above, and Lipschitz continuous on $\cl\,\u(W)$, we obtain $\smash{f\,(\varphi')^{\PP-1}\circ \u\in\Pert_\bc(\u,W)}$. The latter property implies, for some constant $C$ depending only on $W$,
\begin{align*}
\big\vert T(f\,(\varphi')^{\PP-1}\circ\u)\big\vert &\leq C\,\big\Vert f\,(\varphi')^{\PP-1}\circ\u\big\Vert_\infty \leq C\,\varepsilon^{\PP-1}\,\Vert f\Vert_\infty.
\end{align*}
Since $\smash{\vert\rmd\u\vert^\PP}$ is $\meas\mres W$-integrable, the functional $S$ is a weak Radon functional.

Repeatedly using the previous facts, the Leibniz rule as well as the chain rule, $\smash{\meas\mres W}$-integrability of $\smash{\vert\rmd\u\vert^\PP}$, and our hypothesis $\smash{T\in\BOX_\PP\u\mres U}$ we obtain
\begin{align*}
&\int_W \rmd^+f(\nabla(\varphi\circ\u))\,\vert\rmd(\varphi\circ \u)\vert^{\PP-2}\d\meas\\
&\qquad\qquad = \int_W (\varphi')^{\PP-1}\circ\u\d^+f(\nabla\u)\,\vert\rmd\u\vert^{\PP-2}\d\meas\\
&\qquad\qquad \leq \int_W \rmd^+(f\,(\varphi')^{\PP-1}\circ\u)(\nabla\u)\,\vert\rmd\u\vert^{\PP-2}\d\meas\\
&\qquad\qquad\qquad\qquad - (\PP-1)\int_W f\,\big[(\varphi')^{\PP-2}\,\varphi''\big]\circ\u\,\vert\rmd\u\vert^\PP\d\meas\\
&\qquad\qquad \leq -T(f\,(\varphi')^{\PP-1}\circ\u) - (\PP-1)\int_W f\,\big[(\varphi')^{\PP-2}\,\varphi''\big]\circ\u\,\vert\rmd\u\vert^\PP\d\meas.
\end{align*}
Since $W$ was arbitrary, this is the desired claim.
\end{proof} 

\begin{definition}[Harmonicity]\label{Def:Harmon} We call $\u$ \emph{$\PP$-harmonic} on $U$ if $\smash{\BOX_\PP\u\mres U}$ contains $0$.
\end{definition}

The following sufficient condition for $\PP$-harmonicity should also  be compared to Gigli \cite{gigli2015}*{Prop.~4.16} and Gigli--Mondino \cite{gigli-mondino2013}*{Thm.~4.2}. It follows by replacing $f$ by $\tau f$ for sufficiently small $\tau > 0$ and sending $\tau\to 0+$.

\begin{proposition}[Harmonicity and energy minimizers]\label{Pr:Harmonic} Assume that the function $\vert \rmd\u\vert^\PP$ is $\meas\mres U$-integrable. If every $f\in\Pert_\bc(\u,U)$ satisfies
\begin{align*}
\int_U \vert\rmd \u\vert^\PP\d\meas \leq \int_U \vert\rmd(\u+f)\vert^\PP\d\meas,
\end{align*}
then $\u$ is $\PP$-harmonic on $U$.
\end{proposition}

\begin{remark}[Super- and subharmonicity]\label{Re:Superharmonicity} \cref{Def:DAlem} also allows us to define the following natural concepts, see Björn--Björn \cite{bjorn-bjorn2011}, Gigli \cite{gigli2015}, and Gigli--Mondino \cite{gigli-mondino2013} for similar approaches in positive signature. We say $\u$ is \emph{$\PP$-subharmonic} or \emph{$\PP$-superharmonic} on $U$, respectively, if it is in $\smash{\Dom(\BOX_\PP\mres U)}$ and $\smash{\BOX_\PP\u\mres U}$ contains a nonnegative or nonpositive element, respectively.\hfill{\footnotesize{$\blacksquare$}}
\end{remark}

Lastly, the following formula reflects Gigli  \cite{gigli2015}*{Prop.~4.18} from positive signature.

\begin{proposition}[Change of the reference measure]\label{Pr:Change} Given $\smash{V\in\Pert_\bounded(\u,U)}$, we assume $\rmd V(\nabla\u)\,\vert\rmd\u\vert^{\PP-2}$ is locally $\meas\mres U$-integrable. Let $\smash{\BOX_p^V\u\mres U}$ designate the set defined after  \cref{Def:DAlem} with respect to the new metric measure spacetime $\scrM_V := (\mms,l,\meas_V)$, where $\smash{\meas_V := \rme^V\,\meas}$. Then for every $\smash{T\in \BOX_{\PP}\u\mres U}$, the functional $R$ on $\smash{\Pert_\bc(\u,U)}$ defined by
\begin{align*}
R(f) := T(\rme^Vf) - \int_M f\,\rmd V(\nabla\u)\,\vert\rmd\u\vert^{\PP-2}\d\meas_V
\end{align*}
belongs to $\smash{\BOX_p^V\u\mres U}$.
\end{proposition}

\begin{proof} It is easy to see that $R$ is a weak Radon functional. Let us therefore turn to the defining inequality in \cref{Def:DAlem}. Given any $\smash{f\in\Pert_\bc(\u,U)}$, it is easily seen that $\smash{\rme^V f\in \Pert_\bc(\u,U)}$. Hence, by \cref{Pr:Calc rulez II},
\begin{align*}
\int_\mms \rmd^+f(\nabla\u)\,\vert\rmd\u\vert^{\PP-2} \d\meas_V &\leq \int_\mms \rmd(\rme^Vf)(\nabla\u)\,\vert\rmd\u\vert^{\PP-2}\d\meas - \int_\mms f\,\rmd V(\nabla\u)\,\vert\rmd\u\vert^{\PP-2}\d\meas_V\\
&\leq T(\rme^Vf) - \int_\mms f\,\rmd V(\nabla\u)\,\vert\rmd\u\vert^{\PP-2}\d\meas.\qedhere
\end{align*}
\end{proof}

\section{Curvature, dimension, and localization}\label{Sec:Localization}

\subsection{Optimal transport on metric measure spacetimes}\label{Sub:Lor opt tr}  For background on the results subsequently reviewed, we refer to Eckstein--Miller \cite{eckstein-miller2017}, McCann \cite{mccann2020}, Mondino--Suhr \cite{mondino-suhr2022}, Cavalletti--Mondino \cite{cavalletti-mondino2020}, Braun--McCann \cite{braun-mccann2023}, and the references therein. For a discussion about negative transport exponents, see Beran et al.~\cite{beran-braun-calisti-gigli-mccann-ohanyan-rott-samann+-}.

\subsubsection{Causal and chronological couplings} Recall $\Prob(\mms)$ designates the space of Borel probability measures on $\mms$ endowed with the narrow topology. Its subset consisting of all  compactly supported and $\meas$-absolutely continuous elements is denoted $\smash{\Prob_\comp^\ac(\mms,\meas)}$.

Let $\Pi(\mu,\nu)$ be the set of all couplings of $\mu,\nu\in\Prob(\mms)$. We call $\pi\in\Pi(\mu,\nu)$ 
\begin{itemize}
\item \emph{causal}, symbolically $\pi\in\Pi_\leq(\mu,\nu)$, if $\pi[J]=1$, and
\item \emph{chronological}, symbolically $\pi\in\Pi_\ll(\mu,\nu)$, if $\pi[I]=1$.
\end{itemize}
In particular, through the notion of causal couplings, one can define a natural  causality relation $\preceq$ on $\Prob(\mms)$, as initiated by Eckstein--Miller \cite{eckstein-miller2017} and further studied by Suhr \cite{suhr2018-theory} and Braun--McCann  \cite{braun-mccann2023}. A central structural result from  \cite{braun-mccann2023}*{Thm.~B.5} is that through $\preceq$, global hyperbolicity of $\mms$ according to  \cref{Def:MMS} ``lifts'' to its counterpart on $\Prob(\mms)$.

\subsubsection{Lorentz--Wasserstein distance}\label{Sub:LorWas}  The \emph{$\beta$-Lorentz--Wasserstein distance} of $\mu,\nu\in\Prob(\mms)$, where $\beta\in (-\infty,1)\setminus\{0\}$ is fixed, taking values in $[0,\infty)\cup\{-\infty\}$  is defined by
\begin{align*}
\ell_\beta(\mu,\nu) := \sup_{\pi\in \Pi(\mu,\nu)} \Vert l\Vert_{\Ell^\beta(\mms^2,\pi)},
\end{align*}
where we adopt the conventions 
\begin{align*}
(-\infty)^\beta := (-\infty)^{1/\beta} := -\infty.
\end{align*}
The function $\smash{\ell_\beta}$ does inherit the reverse triangle inequality from $l$, cf.~Eckstein--Miller  \cite{eckstein-miller2017}*{Thm.~13}.

We will call a coupling $\pi$ of $\mu$ and $\nu$ \emph{$\ell_\beta$-optimal} if it is causal and it attains the above supremum. If $\mu$ and $\nu$ have compact support, by standard calculus of variations tools they admit an $\smash{\ell_\beta}$-optimal coupling as soon as $\Pi_\leq(\mu,\nu)$ is nonempty, cf.~Villani  \cite{villani2009}*{Thm.~4.1}. 

Following Cavalletti--Mondino \cite{cavalletti-mondino2020}*{Def.~2.18},   a pair $\smash{(\mu,\nu)\in\Prob_\comp(\mms)^2}$ is \emph{timelike $\beta$-dualizable} if $\mu$ and $\nu$ admit a chronological $\smash{\ell_\beta}$-optimal coupling (which is then called \emph{timelike $\beta$-dualizing}). In general, this is stronger than saying $\smash{\ell_\beta(\mu,\nu)>0}$.

\begin{remark}[Uniqueness of optimal couplings]\label{Re:Uniq1} In the relevant framework for our paper (i.e.~timelike $\beta$-essentially nonbranching $\TMCP^e(k,N)$ spaces, cf.~the beginning of \cref{Ch:Repr form}), \emph{chronological} $\smash{\ell_\beta}$-optimal couplings will in fact be unique if they exist, provided one of their marginals is $\meas$-absolutely continuous (and this coupling is  induced by a map from that marginal). See the results by McCann \cite{mccann2020}*{Thm.~5.8}, Braun--Ohta \cite{braun-ohta2024}*{Thm.~4.17}, Cavalletti--Mondino \cite{cavalletti-mondino2020}*{Thm.~3.20}, Braun \cite{braun2023-renyi}*{Thm.~4.16}, and Braun--McCann  \cite{braun-mccann2023}*{Thm. 3.23} which increase in their generality.\hfill{\footnotesize{$\blacksquare$}}
\end{remark}

\subsubsection{Geodesics} We come to the measure analog of \cref{Def:GeosM}. It was introduced by McCann \cite{mccann2020}*{Def.~1.1}.

\begin{definition}[Geodesics on $\Prob(\mms)$]\label{Def:GeodPM} A curve $(\mu_t)_{t\in[0,1]}$ in $\Prob(\mms)$ is called an \emph{$\smash{\ell_\beta}$-geodesic} if it is narrowly continuous and every $s,t\in[0,1]$ with $s<t$  satisfy
\begin{align*}
0 < \ell_\beta(\mu_s,\mu_t) = (t-s)\,\ell_\beta(\mu_0,\mu_1) < \infty.
\end{align*}
\end{definition}

\begin{remark}[About narrow continuity] As a consequence of regularity of $\mms$, the additional requirement of narrow continuity in \cref{Def:GeodPM} is superfluous if the endpoints $\mu_0$ and $\mu_1$ of $(\mu_t)_{t\in[0,1]}$ are compactly supported and all $\smash{\ell_\beta}$-optimal couplings of $\mu_0$ and $\mu_1$ are chronological, cf.~the result of Braun--McCann \cite{braun-mccann2023}*{Cor.~2.65}. The last property holds if $\smash{\supp\mu_0\times\supp\mu_1}$ is a subset of $I$.\hfill{\footnotesize{$\blacksquare$}}
\end{remark}

For two $\mu_0,\mu_1\in\Prob(\mms)$, let $\smash{\OptTGeo_\beta(\mu_0,\mu_1)}$ denote the set of all Borel probability measures $\bdpi$ on $\Cont([0,1];\mms)$ such that
\begin{itemize}
\item $\bdpi$ is concentrated on the Borel set $\TGeo(\mms)$, and 
\item $\smash{(\eval_0,\eval_1)_\push\bdpi}$ constitutes an $\smash{\ell_\beta}$-optimal coupling of $\mu_0$ and $\mu_1$.
\end{itemize}
We call such a $\bdpi$ a \emph{timelike $\smash{\ell_\beta}$-optimal dynamical plan}.

A curve $(\mu_t)_{t\in[0,1]}$ in $\Prob(\mms)$ is termed a \emph{displacement $\smash{\ell_\beta}$-geodesic} if it is represented by an element $\bdpi\in\OptTGeo_\beta(\mu_0,\mu_1)$  \cite{braun-mccann2023}*{Def. 2.62}. Every displacement $\smash{\ell_\beta}$-geodesic is an $\smash{\ell_\beta}$-geodesic, cf.~Cavalletti--Mondino \cite{cavalletti-mondino2020}*{Rem.~2.32} and Braun--McCann \cite{braun-mccann2023}*{Rem.~2.63}.

\begin{remark}[Uniqueness of geodesics]\label{Re:Uniq2}  Analogously to \cref{Re:Uniq1}, in the setting from the beginning of \cref{Ch:Repr form} below, we have uniqueness of \emph{displacement $\ell_\beta$-geodesics} (a tightening of \cref{Def:GeodPM} set up by Braun--McCann \cite{braun-mccann2023}*{Def. 2.62})  between any two timelike $\beta$-dualizable elements of $\Prob(\mms)$ at least one of which is $\meas$-absolutely continuous. 

If in addition, every $\smash{\ell_\beta}$-optimal coupling of $\mu_0$ and $\mu_1$ is chronological, we  also have uniqueness of $\smash{\ell_\beta}$-geodesics (and these are thus displacement $\smash{\ell_\beta}$-geodesics). See the results of McCann  \cite{mccann2020}*{Cor.~5.9}, Braun--Ohta \cite{braun-ohta2024}*{Cor.~4.18}, Cavalletti--Mondino \cite{cavalletti-mondino2020}*{Thm.~3.21}, Braun \cite{braun2023-renyi}*{Thm.~4.17}, and  Braun--McCann \cite{braun-mccann2023}*{Thm.~3.24, Cor. 3.25} which increase in their generality.\hfill{\footnotesize{$\blacksquare$}}
\end{remark}

\subsubsection{Timelike nonbranching phenomena} Recall that a subset $G$ of $\TGeo(\mms)$ is \emph{timelike nonbranching}  if for every $\gamma,\gamma'\in G$ and every $s,t\in[0,1]$ with $s<t$, 
\begin{align*}
\gamma\big\vert_{[s,t]} = \gamma'\big\vert_{[s,t]}\quad\Longrightarrow\quad \gamma = \gamma'.
\end{align*}
This definition originates in Cavalletti--Mondino \cite{cavalletti-mondino2020}*{Def.~1.10}.

The following notion was inspired by Rajala--Sturm \cite{rajala-sturm2014} in positive signature and proposed in the Lorentzian context by Braun \cite{braun2023-renyi}*{Def.~2.21}.

\begin{definition}[Timelike essential nonbranching]\label{Def:TENB} We call  $\scrM$ \emph{time\-like $\beta$-essentially nonbranching} if every $\smash{\bdpi\in\OptTGeo_\beta(\Prob_\comp^\ac(\mms,\meas)^2)}$ is concentrated on a timelike nonbranching subset of $\TGeo(\mms)$.
\end{definition}

\subsubsection{\textnormal{(}Exponentiated\textnormal{)} Boltzmann entropy} The Boltzmann entropy $\Ent_\meas\colon \Prob(\mms) \to \R\cup\{-\infty,\infty\}$ with respect to $\meas$ is defined by
\begin{align*}
\Ent_\meas(\mu) := \begin{cases}\displaystyle \lim_{\varepsilon \to 0+}\int_{\{\rho>\varepsilon\}} \rho\log\rho\d\meas & \textnormal{if } \displaystyle\rho:=\frac{\rmd\mu}{\rmd\meas} \textnormal{ exists},\\
\infty & \textnormal{otherwise}.
\end{cases}
\end{align*}
As usual, $\Dom(\Ent_\meas)$ denotes the set of all $\mu\in\Prob(\mms)$ with real-valued entropy. 

Thanks to Jensen's inequality and the Radon property of $\meas$ assumed in \cref{Sub:MMS}, if $\mu$ has compact support, its entropy does not attain the value $-\infty$. 

Moreover, in terms of the forthcoming dimensional parameter $N \in [1,\infty)$, we define another entropy functional $\scrU_N \colon \Prob(\mms) \to \R \cup\{\infty\}$ by
\begin{align}\label{Eq:UN def}
\scrU_N(\mu) := \rme^{-\Ent_\meas(\mu)/N}.
\end{align}

\subsection{Timelike measure contraction property}\label{Sub:TMCP} In this section, we recall the definition of the timelike measure contraction property. It has been introduced in Cavalletti--Mondino \cite{cavalletti-mondino2020} (see also Braun \cite{braun2023-renyi}) following Ohta \cite{ohta2007-mcp} and Sturm \cite{sturm2006-ii} and extended to the  variable framework in Braun--McCann \cite{braun-mccann2023}.

Throughout this chapter, let $k$ be a lower semicontinuous function on $\mms$ and $N\in [1,\infty)$ be a dimensional parameter.

Recall the $l$-causal speed $\vert\dot\gamma\vert = l(\gamma_0,\gamma_1)$ of a given $\gamma\in\TGeo(\mms)$.   
Let us define the functions $\smash{k_\gamma^\pm}$ on $[0,\vert\dot\gamma\vert]$ by the relations
\begin{align}\label{Eq:kgammapm}
\begin{split}
k_\gamma^+(t\,\vert\dot\gamma\vert) &:= k(\gamma_t),\\
k_\gamma^-(t\,\vert\dot\gamma\vert) &:= k(\gamma_{1-t}).
\end{split}
\end{align}

Next, we define the $L^2$-cost of a timelike $\smash{\ell_\beta}$-optimal dynamical plan $\bdpi$ by
\begin{align*}
\cost_\bdpi := \big\Vert l\circ(\eval_0,\eval_1)\big\Vert_{L^2(\TGeo(\mms),\bdpi)}.
\end{align*}
Jensen's inequality ensures $\smash{\cost_\bdpi}$ is no smaller than $\smash{\ell_\beta((\eval_0)_\push\bdpi,(\eval_1)_\push\bdpi)}$.

Let $\bdpi$ be as above and assume its endpoint marginals are compactly supported. Define the ``superposition'' functions $\smash{k_\bdpi^\pm}$ on $[0,\cost_\bdpi]$ of the quantities in \eqref{Eq:kgammapm} by
\begin{align*}
k_\bdpi^\pm(t\,\cost_\bdpi)\,\cost_\bdpi^2 := \int k_\gamma^\pm(t\,\vert\dot\gamma\vert)\,\vert\dot\gamma\vert^2\d\bdpi(\gamma).
\end{align*}
Since $\bdpi$-a.e.~$\gamma$ does never leave the compact set $J(\supp(\eval_0)_\push\bdpi,\supp(\eval_1)_\push\bdpi))$ and since $k$ is bounded from below on this set, the previous definitions make sense. In turn, Fatou's lemma ensures the functions $\smash{k_\bdpi^\pm}$ are lower semicontinuous. Hence, the discussion from \cref{Sub:Var dist coeff} applies and yields induced distortion coefficients $\smash{\sigma_{k_\bdpi^\pm}^{(t)}}$.

Recall \eqref{Eq:UN def} for the definition of the exponentiated Boltzmann entropy $\scrU_N$.

\begin{definition}[Timelike measure contraction property]\label{Def:TMCP} We say $\scrM$ obeys the \emph{entropic timelike measure contraction property} $\smash{\TMCP^e(k,N)}$ if the following holds. For every $\smash{\mu_0 \in \Prob_\comp(\mms)\cap\Dom(\Ent_\meas)}$ and every $x_1\in\mms$  such that either
\begin{enumerate}[label=\textnormal{\alph*.}]
\item $x_1$ is in the $l$-chronological future of all points in $\supp\mu_0$, there exist
\begin{itemize}
\item an $\smash{\ell_{1/2}}$-geodesic $(\mu_t)_{t\in[0,1]}$ from $\mu_0$ to $\smash{\mu_1 := \delta_{x_1}}$ and
\item a plan $\smash{\bdpi\in\OptTGeo_{1/2}(\mu_0,\mu_1)}$, or
\end{itemize}
\item $x_1$ is in the $l$-chronological past of all points in $\supp\mu_0$, there exist
\begin{itemize}
\item an $\smash{\ell_{1/2}}$-geodesic $(\mu_t)_{t\in[1,0]}$ from $\mu_1$ to $\mu_0$ and
\item a plan $\smash{\bdpi\in\OptTGeo_{1/2}(\mu_1,\mu_0)}$,
\end{itemize}
\end{enumerate}
such that for every $t\in[0,1]$,
\begin{align*}
\scrU_N(\mu_t) \geq \sigma_{k_\bdpi^-/N}^{(1-t)}(\cost_\bdpi)\,\scrU_N(\mu_0).
\end{align*}
\end{definition}

As observed by Cavalletti--Mondino \cite{cavalletti-mondino2022-review}*{Rem.~2.4} (see also their survey article   \cite{cavalletti-mondino2020}*{Rem.~3.8}, Braun \cite{braun2023-renyi}*{Rem.~4.3}, and Braun--McCann \cite{braun-mccann2023}*{Rem. A.2}), the exponent $1/2$ in \cref{Def:TMCP} could equivalently replaced by any $\beta$ as in \cref{Sub:Fixed}.

We refer to Braun--McCann \cite{braun-mccann2023}*{§A.2} for further basic properties of this notion.

\subsection{Timelike curvature-dimension condition}\label{Sub:TCD} A strictly stronger property (cf.~Braun--McCann \cite{braun-mccann2023}*{Prop. A.3} and \cref{Ex:TCD} below) is the timelike curvature-dimension condition subsequently reviewed. Inspired by McCann \cite{mccann2020} and Mondino--Suhr \cite{mondino-suhr2022}, it has been introduced in Cavalletti--Mondino \cite{cavalletti-mondino2020} (see also Braun \cite{braun2023-renyi}) and extended to the variable framework in Braun--McCann \cite{braun-mccann2023} following Ketterer \cite{ketterer2017}. The proposals of \cite{cavalletti-mondino2020,braun2023-renyi}  for constant $k$ are partly known to be equivalent in our setting described in \cref{Ch:Repr form}, cf.~Braun \cite{braun2023-renyi}*{Thm.~3.35} and their full equivalence --- along with the independence of the subsequent condition on the given transport exponent $\beta$ --- is explored by Akdemir \cite{akdemir+}; see \cref{Re:Potential independence} below.

\begin{definition}[Timelike curvature-dimension condition] We say $\scrM$ obeys the \emph{entropic timelike curvature-dimension condition} $\smash{\TCD_\beta^e(k,N)}$ if the following is satisfied. For every time\-like $\beta$-dualizable pair $(\mu_0,\mu_1)$ of measures $\smash{\mu_0,\mu_1\in\Prob_\comp(\mms)\cap\Dom(\Ent_\meas)}$, there exist
\begin{itemize}
\item an $\smash{\ell_\beta}$-geodesic $(\mu_t)_{t\in[0,1]}$ from $\mu_0$ to $\mu_1$ and
\item a plan $\smash{\bdpi\in\OptTGeo_\beta(\mu_0,\mu_1)}$
\end{itemize}
such that for every $t\in[0,1]$,
\begin{align*}
\scrU_N(\mu_t) \geq \sigma_{k_\bdpi^-/N}^{(1-t)}(\cost_\bdpi)\,\scrU_N(\mu_0) + \sigma_{k_\bdpi^-/N}^{(t)}(\cost_\bdpi)\,\scrU_N(\mu_1).
\end{align*}
\end{definition}

We refer to Braun--McCann \cite{braun-mccann2023}*{§§3.3, 7.2} for basic properties of this notion.

\begin{example}[Smooth spacetimes]\label{Ex:TCD} It is well-known, cf.~e.g.~Kunzinger--Sämann \cite{kunzinger-samann2018}*{§5.1}, that any  smooth, globally hyperbolic spacetime $(\mms,\Rmet)$ canonically induces a metric measure spacetime $\scrM$ according to \cref{Def:MMS} with $\smash{\meas := \vol_\Rmet}$. Then $\scrM$ obeys the $\TCD_\beta^e(k,N)$ condition if and only if $\smash{\Ric_\Rmet \geq k}$ holds in all timelike directions and $N\geq \dim\mms$. For constant $k$, this was independently discovered in McCann \cite{mccann2020}*{Cors.~6.6, 7.5, Thm.~8.5} and Mondino--Suhr \cite{mondino-suhr2022}*{Thm.~4.3}. The necessity of the dimensional constraint is shown in Cavalletti--Mondino \cite{cavalletti-mondino2020}*{Cor.~A.2} and McCann \cite{mccann2023-null}*{Thm.~25}. With evident adaptations, the proofs carry over to  the variable framework, see Braun--McCann \cite{braun-mccann2023}*{Thm.~4.1}.

An analog of the previous results also holds in the weighted case, i.e.~when the measure $\smash{\vol_\Rmet}$ is replaced by $\smash{v\,\vol_\Rmet}$, where $v$ is positive and twice continuously differentiable. If $N>\dim\mms$ we also have to replace $\smash{\Ric_\Rmet}$ by the \emph{Bakry--Émery--Ricci tensor}
\begin{align*}
\Ric_\Rmet - (N-\dim\mms)\,\frac{\Hess_\Rmet v^{1/(N-\dim\mms)}}{v^{1/(N-\dim\mms)}}.
\end{align*}

For the timelike measure contraction property from \cref{Def:TMCP}, this correspondence is more restrictive. Indeed, $\scrM$ (with reference measure $\smash{\meas:=\vol_\Rmet}$) is a $\smash{\TMCP^e(k,\dim\mms)}$ space if and only if $\Ric_\Rmet\geq k$ in all timelike directions (see Cavalletti--Mondino \cite{cavalletti-mondino2020}*{Thm.~A.1} for a proof for constant $k$). On the other hand, $\TMCP^e(k,N)$ for general $N$ does not necessarily imply $\Ric_\Rmet\geq k$ in all timelike directions \cite{cavalletti-mondino2020}*{Rem.~A.3}.\hfill{\footnotesize{$\blacksquare$}}
\end{example}

\subsection{MCP and CD disintegrations}\label{Sub:Disintegr} A central ingredient for us will be the localization paradigm, basics of which we review now. It has been pioneered in the Lorentzian context by Cavalletti--Mondino \cite{cavalletti-mondino2020}  and subsequently extended by them \cite{cavalletti-mondino2024} and Braun--McCann \cite{braun-mccann2023}; we refer to these works for details. We mainly follow the notation and the nomenclature of \cite{braun-mccann2023}, but we directly intertwine the results requiring nonbranching and curvature hypotheses with those which do not to shorten the presentation.

\subsubsection{Framework}\label{Sub:FRAME} Our subsequent  discussion (especially the proof of \cref{Th:From Bochner to TCD} and, connected with this, \cref{Ex:timedistfunct}) requires a slight modification of the hypotheses from \cite{cavalletti-mondino2020,cavalletti-mondino2024,braun-mccann2023}. We shortly outline them here. 

Here and in the sequel, we fix a function $\u$, $1$-steep with respect to $l$, defined  on a Borel subset $E$ on $\mms$. In Braun--McCann \cite{braun-mccann2023}*{§6.1}, $E$ is assumed $l$-geodesically convex. We relax this assumption as follows. Using a suggestive terminology in view of \cref{Sub:Transport rels}, we say $E$ is \emph{$\smash{\smash{\sim}_{\u}}$-convex} if for every $x,y\in E$ such that $\u(y) - \u(x) = l(x,y) > 0$, no $\gamma\in\TGeo(\mms)$ connecting the two points $x$ and $y$ leaves $E$; compare with  \cref{Def:Convexity} below. As for $l$-geodesic convexity, this notion is the same relative to $\u$ and $l$  or $\smash{\u^\leftarrow}$ and $\smash{l^\leftarrow}$, respectively.  Moreover,  $E$ is $\smash{\sim_{\u}}$-convex if it is $l$-geodesically convex \cite{braun-mccann2023}*{Lem.~6.4}. The  discussion of \cite{braun-mccann2023}*{§6} goes through under this weaker hypothesis (see especially \cite{braun-mccann2023}*{Lem.~6.15} therein), which we impose from now on.

By our assumptions from \cref{Sub:MMS}, the conditions on  $\meas$ from \cite{braun-mccann2023}*{§6.1} are satisfied, as pointed out before the disintegration \cref{Th:Disintegration} below.

\begin{example}[Lorentz distance functions]\label{Ex:timedistfunct} Akin to \cref{Sec:Signed}, let $\Sigma$ be an achronal Borel subset of $\mms$ and $\smash{l_\Sigma}$ the associated signed Lorentz distance function. Then $\smash{E_\Sigma}$ and $\smash{l_\Sigma}$ obey the above properties in place of $E$ and $\u$, respectively, by \cref{Le:Geod conv E} and \cref{Cor:Steepness}.\hfill{\footnotesize{$\blacksquare$}}
\end{example}

In \cref{Ex:timedistfunct}, pathologies may still occur through  the lack of ``many distinct points in relation to each other'' in \cref{Sub:Transport rels}. In \cite{cavalletti-mondino2020}, Cavalletti--Mondino rule this out by assuming $\Sigma$ is FTC after \cref{Def:timelike complete} (recall this entails the existence of footpoints of every point in $\smash{I^+(\Sigma)}$) and  $\smash{\meas[\Sigma]=0}$ \cite{cavalletti-mondino2020}*{Lem.~4.4, Thm.~4.17}. The latter property holds by default in our setting by \cref{Sub:MMS}. Since in \cref{Ex:timedistfunct}, both $I^-(\Sigma)$ and $I^+(\Sigma)$ will require the existence of footpoints, for simplicity we assume $\Sigma$ to be TC.

\begin{definition}[Signed TC Lorentz distance function]\label{Def:footpts} We will call $l_\Sigma$ a \emph{signed TC Lorentz  distance function} if its defining achronal set $\Sigma$ is TC.
\end{definition}

\subsubsection{Transport relations}\label{Sub:Transport rels} There are two conceivable relations to study, namely 
\begin{itemize}
\item the one induced by $\u$ relative to $l$, or
\item the one induced by $\u^\leftarrow$ relative to $l^\leftarrow$. 
\end{itemize}
As detailed especially in \cref{Re:INV CAUS RE,Re:INV CAUS RE II}, both essentially lead to the same sets and quantities. The difference will occur once we decide in which causal orientation rays are parametrized (and thus where initial and final points lie);  this is outsourced to \cref{Sub:Ray map}.

Define two relations $\smash{\preceq_{\u}}$ and  $\smash{\succeq_{\u}}$  on $E$ by
\begin{align*}
x \preceq_{\u} y\quad &:\Longleftrightarrow\quad x=y\quad\textnormal{or}\quad \u(y) - \u(x) = l(x,y) > 0,\\
x \succeq_{\u} y \quad &:\Longleftrightarrow\quad y\preceq_{\u} x.
\end{align*}
These constitute partial orders on $E$ \cite{braun-mccann2023}*{Lem.~6.3}.

To ensure symmetry, we define an additional relation $\smash{\sim_\u}$ on $E$ by
\begin{align*}
x\sim_\u y \quad :\Longleftrightarrow\quad x\preceq_\u y\quad\textnormal{or}\quad x\succeq_\u y.
\end{align*}
The \emph{transport relation with bad points} $\smash{E_{\sim_\u}^{2,\End}}$ is defined as the set of all pairs $\smash{(x,y)\in E^2}$ with $\smash{x\sim_\u y}$. The \emph{transport set with bad points} $\smash{\Tr^\End}$ is defined as the set of all $x\in E$ for which there is a point $y\in E\setminus \{x\}$ with $\smash{x\sim_\u y}$ \cite{braun-mccann2023}*{Def.~6.6}. Then $\smash{E_{\sim_\u}^{2,\End}}$ is Borel, while $\smash{\Tr^{\End}}$ is Suslin. Note that if $l_\Sigma$ forms a signed TC Lorentz distance function  after \cref{Def:footpts}, then $\smash{\Tr^\End = E_\Sigma}$ by nontrivial chronology.

To turn $\smash{\sim_\u}$ into an equivalence relation, we have to exclude \emph{bad points}. By these we mean  elements of the union of the following sets \cite{braun-mccann2023}*{Defs.~6.7, 6.9}.
\begin{itemize}
\item The \emph{timelike branching set} $\scrB$ of all $x \in \Tr^\End$ for which there are points $y_1,y_2\in \Tr^\End$ with $\smash{y_1\not\sim_\u y_2}$ such that either $\smash{x\preceq_\u y_1}$ and $\smash{x\preceq_\u y_2}$ or $\smash{x \succeq_\u y_1}$ and $\smash{x\succeq_\u y_2}$.
\item The \emph{endpoint set}\footnote{Strictly speaking, the endpoint set does not have to be excluded for $\smash{\sim_\u}$ to be an equivalence relation. However, its exclusion ensures the absolute continuity of the conditional measures from the disintegration \cref{Th:Disintegration} stipulated in \cref{Def:MCP disintegration,Def:CD disintegration}.} $e$, which is the complement (relative to $\smash{\Tr^\End}$) of the set of all $\smash{x\in \Tr^\End}$ such that there exist $\smash{x\pm\in \Tr^\End\setminus \{x\}}$ with $\smash{x-\preceq_\u x \preceq_\u x+}$.
\end{itemize}
Both $\scrB$ and $e$ are Suslin. Then the \emph{transport set without bad points} is 
\begin{align*}
\Tr := \Tr^\End\setminus (\scrB\cup e);
\end{align*}
the \emph{transport relation without bad points} $\smash{E_{\sim_\u}^2}$ is simply the intersection of $\smash{E_{\sim_\u}^{2,\End}}$ with $\smash{\Tr^2}$ \cite{braun-mccann2023}*{Def.~6.13}. Then $\sim_\u$ constitutes an equivalence relation on $\Tr$, cf.~Cavalletti--Mondino \cite{cavalletti-mondino2020}*{Prop.~4.5} and Braun--McCann \cite{braun-mccann2023}*{Thm. 6.14}. The equivalence class of $\smash{x\in\Tr}$ with respect to $\smash{\sim_\u}$ is denoted by $\smash{\tilde{x}^\u}$. Few is lost from $\smash{\Tr^\End}$: on a timelike $\beta$-essentially nonbranching $\smash{\TMCP^e(k,N)}$ space $\scrM$, the set $\scrB\cup e$ of bad points is $\meas$-negligible as shown in Cavalletti--Mondino \cite{cavalletti-mondino2020}*{Cor.~4.15} and Braun--McCann \cite{braun-mccann2023}*{Cor.~6.29}. This is a consequence of the uniqueness results described in \cref{Re:Uniq1,Re:Uniq2};  the curvature hypothesis is sufficient, but not necessary.

\begin{remark}[Uniqueness of footpoints]\label{Re:Uniqfttpts} If $\smash{l_\Sigma}$ is a signed TC Lorentz distance function,  the footpoint of every $y\in \Tr$ is in fact unique. If $y\in \Sigma$, this is clear by achronality. Otherwise, again by achronality the existence of two distinct footpoints would contradict $y\notin \scrB$.\hfill{\footnotesize{$\blacksquare$}}
\end{remark}

Finally, recall a map $\Quot\colon \Tr\to\Tr$ is a \emph{quotient map} for the equivalence relation $\smash{\sim_\u}$ if its graph is contained in $\smash{E_{\sim_\u}^2}$ and if $x,y\in\Tr$ satisfy $\smash{x\sim_\u y}$, then $\Quot(x) = \Quot(y)$. We shall henceforth fix such a map $\Quot$ which is  additionally $\meas$-measurable, which is possible by Cavalletti--Mondino \cite{cavalletti-mondino2020}*{Prop.~4.9} and Braun--McCann \cite{braun-mccann2023}*{Prop.~6.18}. Existence of $\Quot$ can  be deduced from the axiom of choice, but it can also be constructed  rather explicitly, as pointed out by Cavalletti--Mondino \cite{cavalletti-mondino2020-new}*{p.~2111}.

Lastly, following \cite{braun-mccann2023}*{Def.~6.19} the set 
\begin{align*}
Q := \Quot(\Tr).
\end{align*}
is called \emph{quotient set}. Any set of the form 
\begin{align*}
\mms_\alpha := \Quot^{-1}(\alpha),
\end{align*}
where $\alpha\in Q$, will be termed a \emph{ray}. Note that $\smash{\mms_\alpha = \tilde{\alpha}^\u}$ for every such $\alpha$;  thus, $Q$ can be interpreted as an ``index set'' labeling the  rays.

\begin{remark}[Invariance under time-reversal I, see also \cref{Re:INV CAUS RE II}]\label{Re:INV CAUS RE} It is clear that $\smash{\sim_\u}$, the transport relations with and without bad points, the transport sets with and without bad points, the timelike branching set $\scrB$, the endpoint set $e$, and the induced rays are all independent of the causal orientation of $\u$. The quotient set is universally fixed.\hfill{\footnotesize{$\blacksquare$}}
\end{remark}

\subsubsection{Ray map and endpoints}\label{Sub:Ray map} Every equivalence classes with respect to $\smash{\sim_\u}$ is formed by precisely one unbroken timelike curve of positive (possibly infinite) $l$-length each of whose sub\-seg\-ments lie in $\TGeo(\mms)$ after a suitable reparametrization \cite{braun-mccann2023}*{Lem.~6.15}. In fact, every equivalence class can be identified order isometrically with an interval by  \cref{Pr:Ray}, cf.~Cavalletti--Mondino \cite{cavalletti-mondino2020}*{Lem.~4.6} and Braun--McCann \cite{braun-mccann2023}*{Cor.~6.16}. 

Let $\Dom(\sfg)$ be the set of all pairs $(t,x)\in\R\times\Tr$ such that there is $\smash{y\in \tilde{x}^\u}$ with $\smash{l_x(y)=t}$. This set is clearly independent of the causal orientation of $\u$.

\begin{proposition}[Ray map]\label{Pr:Ray} There exists a Borel map $\sfg \colon \Dom(\sfg)\to \Tr$ such that for every $x\in \Tr$, the map $\smash{\sfg_\cdot(x)}$ is a bijection between an open interval and $\smash{\tilde{x}^\u}$. 

In fact, it is even an order isometry, in the sense that for every $x\in \Tr$ and every $s,t \in\R$ with $s <t$ such that $\smash{\sfg_s(x)}$ and $\smash{\sfg_t(x)}$ are defined,
\begin{align*}
\u\circ \sfg_t(x) - \u \circ\sfg_s(x) = l(\sfg_s(x),\sfg_t(x)) = t-s.
\end{align*}
\end{proposition}

\begin{remark}[Interpretation]\label{Re:Interpr} Given $x\in\Tr$, the map $\smash{\sfg_\cdot(x)}$ should be interpreted as the \emph{negative} gradient flow of $\u$ through $x$. Indeed, if $\u$ is a Lorentz distance function on a Lorentz spacetime, radial $l$-geodesics follow the negative gradient flow of $\u$, as discussed e.g.~in the smooth  compatibility of localization by Cavalletti--Mondino \cite{cavalletti-mondino2020}*{Rem.~5.4}. It appears naturally in the Bochner identity for suitable distance    functions, cf.~\cref{Sub:Nonsmooth Bochner}. \hfill{\footnotesize{$\blacksquare$}}
\end{remark}

By a slight abuse of notation, $\smash{\sfg^{-1}}$ denotes the ``inverse'' of $\smash{\sfg}$ in the subsequent sense  \cite{braun-mccann2023}*{Cor.~6.20}. The real number $\smash{\sfg^{-1}(x)}$ is the signed Lorentz distance function of $x\in\Tr$ to the base point $\Quot(x)$ of the ray $x$ lies in; explicitly, $\smash{\sfg^{-1}(x) := l_{\Quot(x)}(x)}$ in the notation of \cref{Sec:Signed}.

Our choice of $\sfg$ also fixes the ``location'' of the endpoints, for which we introduce some notation. If either point exists for $\alpha\in Q$, we define
\begin{itemize}
\item the \emph{initial point} $a_\alpha$ of the ray $\mms_\alpha$, symbolically $a_\alpha\in a$, as the unique element of $\smash{\cl\,\mms_\alpha}$ in the chronological past  of all points in $\smash{\mms_\alpha}$ with respect to $l$, and
\item the \emph{final point} $b_\alpha$ of the ray $\mms_\alpha$, symbolically $b_\alpha\in b$, as the unique element of $\smash{\cl\,\mms_\alpha}$  in the chronological future  of all points in $\smash{\mms_\alpha}$ with respect to $l$.
\end{itemize}
The points $a_\alpha,b_\alpha\in e$ depend Suslin measurably on $\alpha\in Q$. Note $a_\alpha$ exists if and only if the domain of definition of $\sfg_\cdot(\alpha)$ bounded from below; analogously for $b_\alpha$. 

\begin{remark}[Endpoints and timelike cut loci]\label{Re:Loci} In the case of a signed TC Lorentz distance function according to  \cref{Def:footpts}, the set $a$ coincides with the past  timelike cut locus $\smash{\TCut^-(\Sigma)}$. The latter is the complement of all points in $y\in I^-(\Sigma)$ relative to $I^-(\Sigma)$ such that there exists $\gamma\in \TGeo(\mms)$ with $\gamma_t\in \Sigma$ and $\gamma_1 = y$ for some $t\in (0,1)$.

Analogously, $b$ equals the future timelike cut locus $\smash{\TCut^+(\Sigma)}$. \hfill{\footnotesize{$\blacksquare$}}
\end{remark}

\begin{remark}[Identification]\label{Re:Identification} Occasionally, we identify a ray $\mms_\alpha$ and certain quantities associated to it with their image under $\smash{\ray_\cdot(\alpha)}$, respectively, where $\alpha\in Q$, through $\smash{\sfg}$ or its inverse. We use a barred notation whenever an object lives on the real line. For instance, $\smash{\bar{\mms}_\alpha}$ denotes the domain of $\smash{\sfg_\cdot(\alpha)}$. Given a Borel measure $\sigma$ on $\mms_\alpha$, we set
\begin{align*}
\bar{\sigma}_\alpha := \sfg^{-1}_\push\sigma.
\end{align*}
This identification will mostly be used in proofs, not statements.

These identifications can also be turned on their heads. For $v,w\in \Tr$ such that $\smash{v\preceq_\u w}$, $[v,w]$ will denote the intersection of $J(v,w)$ with $\smash{\tilde{v}^\u}$; ``intervals'' of the form $[v,w)$, $(v,w]$, and $(v,w)$ are set up analogously. The Lebesgue measure on $\mms_\alpha$ (corresponding to the one-dimensional Hausdorff measure thereon by \cref{Pr:Ray}) is defined by
\begin{align*}
\Leb_\alpha^1 := \sfg_\cdot(\alpha)_\push\Leb^1.\tag*{{\footnotesize{$\blacksquare$}}}
\end{align*}
\end{remark}

\begin{definition}[Real ray representative] For $\alpha\in Q$, the \emph{real ray representative} on the interval $\smash{\bar{\mms}_\alpha}$ of an $\meas$-measurable function $f$ on $\Tr$ is 
\begin{align*}
\bar{f}_\alpha := f \circ \sfg_\cdot(\alpha).
\end{align*}
\end{definition}

In the framework of a signed TC Lorentz distance function $l_\Sigma$ from \cref{Def:footpts} (also recall \cref{Ex:timedistfunct}), whenever convenient one can reparametrize all real ray representatives in such a way that ``they pass through the zero level set $\Sigma$ of $l_\Sigma$  at time zero''; cf. \cref{Sub:Dalem Mean Curv} below, Cavalletti--Milman \cite{cavalletti-milman2021}*{Prop.~10.4}, and  Ketterer \cite{ketterer2023-rigidity}*{Rem.~3.6} for details.

\subsubsection{Disintegration theorem} Now we report the disintegration \cref{Th:Disintegration} obtained under non\-branching and curvature hypotheses by Cavalletti--Mondino  \cite{cavalletti-mondino2020,cavalletti-mondino2024} and Braun--McCann \cite{braun-mccann2023}. To avoid pathological statements, we assume the condition $\smash{\meas[\Tr] > 0}$.

For details about the disintegration theorem for probability measures, we refer to the monograph of Fremlin  \cite{fremlin2006}.

\begin{definition}[Strong disintegration]\label{Def:Strong disint} A Borel probability measure $\q\in \Prob(Q)$ is a \emph{strong disintegration} relative to $\u$ if it is mutually absolutely continuous with respect to $\smash{\Quot_\push[\meas\mres\Tr]}$ and there exists a map $\smash{\meas_\cdot \colon Q\to \Meas(\mms)}$ certifying the disintegration formula
\begin{align*}
\meas\mres \Tr^\End = \meas\mres \Tr = \int_Q\meas_\alpha\d\q(\alpha)
\end{align*}
as well as the following properties.
\begin{enumerate}[label=\textnormal{\alph*.}]
\item \textnormal{\textbf{Measurability.}} For every $\meas$-measurable subset $B$ of $E$, the evaluation $\meas_\alpha[B]$ depends $\q$-measurably on $\alpha\in Q$.
\item \textnormal{\textbf{Strong consistency.}} For $\q$-a.e.~$\alpha\in Q$, the measure $\meas_\alpha$ is concentrated on $\mms_\alpha$.
\item \textnormal{\textbf{Disintegration.}} For every $\meas$-measurable subset $B$ of $E$ and every $\q$-measurable subset $A$ of $Q$, we have
\begin{align*}
\meas[B\cap \Quot^{-1}(A)] = \int_A \meas_\alpha[B]\d\q(\alpha).
\end{align*}
\item \textnormal{\textbf{Local finiteness.}} For every compact subset $C$ of $E$,
\begin{align*}
\q\textnormal{-}\!\esssup_{\alpha\in Q} \meas_\alpha[C]< \infty.
\end{align*}
\end{enumerate}
\end{definition}

The above measures $\meas_\alpha$, where $\alpha\in Q$, are called \emph{conditional measures}.

The identity $\smash{\meas\mres\Tr^\End = \meas\mres \Tr}$ is   stated as a separate property. As already observed in \cref{Sub:Transport rels}, in the relevant setting of the disintegration \cref{Th:Disintegration} it will hold (and in this case, the preceding hypothesis $\meas[\Tr]>0$ is equivalent to $\smash{\meas[\Tr^\End]>0}$). 

From Fremlin \cite{fremlin2006}*{Prop.~452F}, it follows in particular that if $f$ is an $\smash{\R\cup\{\pm\infty\}}$-valued Borel function on $\Tr$ whose mean $\smash{\int_\Tr f\d\meas}$ is well-defined in $\R\cup\{\pm\infty\}$, we have
\begin{align*}
\int_{\Tr} f\d\meas = \int_Q\int_{\mms_\alpha} f\d\meas_\alpha\d\q(\alpha).
\end{align*}

As indicated above, suitable nonbranching and curvature hypotheses on $\scrM$ pass over to each conditional measure, as recalled now.

\begin{definition}[MCP disintegration]\label{Def:MCP disintegration} Let $\q$ form a strong disintegration. We call $\q$ an \emph{$\MCP(k,N)$ disintegration} relative to $\u$ if for $\q$-a.e.~$\alpha\in Q$, the conditional measure $\meas_\alpha$ is $\smash{\Leb_\alpha^1}$-absolutely continuous with continuous Radon--Nikodým density $h_\alpha$ whose real ray representative $\smash{\bar{h}_\alpha}$ is an $\smash{\MCP(\bar{k}_\alpha,N)}$ density on $\smash{\bar{\mms}_\alpha}$.
\end{definition}

\begin{definition}[CD disintegration]\label{Def:CD disintegration} Let $\q$ be a strong disintegration. We call $\q$ a \emph{$\CD(k,N)$ disintegration} relative to $\u$ if for $\q$-a.e.~$\alpha\in Q$, the conditional measure $\meas_\alpha$ is $\smash{\Leb_\alpha^1}$-absolutely continuous with continuous Radon--Nikodým density $h_\alpha$ whose real ray representative $\smash{\bar{h}_\alpha}$ is a $\smash{\CD(\bar{k}_\alpha,N)}$ density on $\smash{\bar{\mms}_\alpha}$.
\end{definition}

\begin{remark}[Comparison with \cite{cavalletti-milman2021}] Modulo minor modifications, \cref{Def:MCP disintegration,Def:CD disintegration} constitute Lorentzian analogs of the $\smash{\MCP^1(k,N)}$ and $\smash{\CD^1(k,N)}$ properties for metric measure spaces, introduced in Cavalletti--Milman \cite{cavalletti-milman2021}*{Def.~8.6} for constant $k$ (evidently  interpreted in the variable case). Their Lorentzian versions are studied by Akdemir \cite{akdemir+}.\hfill{\footnotesize{$\blacksquare$}}
\end{remark}

The following theorem is due to Cavalletti--Mondino  \cite{cavalletti-mondino2020}*{Thm.~4.17} and \cite{cavalletti-mondino2024}*{Thms.~3.2, 5.2} as well as Braun--McCann  \cite{braun-mccann2023}*{Thms. 6.37, A.5} in variying generality. Its conclusion will be central in our study. We remark that since its proof uses the disintegration theorem for probability measures, a normalization is required therein  to turn certain restrictions of $\meas$ into probability measures. This is where the properness hypothesis on $\mms$ from \cref{Sub:MMS} tacitly enters, cf.~Cavalletti--Mondino \cite{cavalletti-mondino2020-new}*{Lem.~3.3}.

\begin{theorem}[Disintegration]\label{Th:Disintegration} Assume $\scrM$ forms a timelike $\beta$-essentially nonbranching metric measure spacetime. If it satisfies $\smash{\TMCP^e(k,N)}$ or $\smash{\TCD_\beta^e(k,N)}$, respectively, then it admits an $\MCP(k,N)$ or $\CD(k,N)$ disintegration $\q$ relative to $\u$, respectively.

In either case, the disintegrations are $\q$-essentially unique, i.e.~if $\smash{\meas'_\cdot\colon Q\to\Meas(\mms)}$ is a map satisfying the properties stated in \cref{Def:Strong disint}, 
\begin{align*}
\meas_\cdot = \meas_\cdot'\quad\q\textnormal{-a.e.}
\end{align*}
\end{theorem}

\begin{remark}[Smooth case]\label{Re:Smooothspa} If $\scrM$ comes from a globally hyperbolic Finsler space\-time (cf.~\cref{Sub:LorentzFinsler} below), by the rich existence of transport maps stated in \cref{Re:Uniq1,Re:Uniq2} there are already strong disintegrations with absolutely continuous conditional measures  without any a priori curvature hypothesis, cf.~Braun--McCann \cite{braun-mccann2023}*{Thm.~6.35}. However, a synthetic curvature bound is not automatically satisfied by the rays, cf.~McCann  \cite{mccann2023-null}*{Thm.~31}.\hfill{\footnotesize{$\blacksquare$}}
\end{remark}

\begin{remark}[Invariance under time-reversal II, see also \cref{Re:INV CAUS RE}]\label{Re:INV CAUS RE II} The quantities and statements from the disintegration \cref{Th:Disintegration} are independent of the causal orientation of $\u$. Although the real ray representatives from \cref{Def:MCP disintegration,Def:CD disintegration} are directed relative to $\smash{l}$,  \cref{Def:MCP density,Def:CD density} are invariant under time reversal. \hfill{\footnotesize{$\blacksquare$}}
\end{remark}

\section{Representation formulas for the d'Alembertian}\label{Ch:Repr form}

Throughout the rest of this paper, we assume the following framework.

The metric measure spacetime $\scrM$ according to \cref{Def:MMS} is  timelike $\beta$-essentially nonbranching after \cref{Def:TENB}, where $\beta\in (-\infty,1)\setminus\{0\}$. This hypothesis is abbreviated by calling $\scrM$  ``timelike essentially nonbranching''. At the level of nonbranching, the precise value of $\beta$ mostly does not matter with the partial  exceptions of \cref{Th:From TCD to Bochner,Th:From Bochner to TCD}; see, however, \cref{Re:Potential independence} below.

Assume $\scrM$ obeys $\smash{\TMCP^e(k,N)}$, where $k$ is a lower semicontinuous function on $\mms$ and $N\in (1,\infty)$. The exclusion of the value $1$ for $N$ simplifies the presentation, yet all results below hold accordingly for it by taking limits.

We fix a $1$-steep function $u$ with respect to $l$ defined on a $\smash{\sim_\u}$-convex Borel subset $E$ of $\mms$.  Consider an $\MCP(k,N)$ disintegration $\q$ relative to  $\smash{\u}$ according to the disintegration  \cref{Th:Disintegration}, coming with associated transport sets $\smash{\Tr^\End}$ and $\Tr$ with and without bad points, respectively. To avoid pathologies, we assume $\smash{\meas[\Tr^\End]>0}$, which yields $\smash{\meas[\Tr]>0}$.

\subsection{Analytic preparations}\label{Sub:Analytic preps}

\subsubsection{Differentiation along the reverse ray map}\label{Sub:Differentiation} Our subsequent treatise requires a   notion of differentiation of finite perturbations $f$ of $\u$ on $\Tr$  along $\ray$; recall \cref{Def:Perturbations}. In this part, we make this precise using the discussions from \cref{Sub:Fixed,Sub:Monotone real}. 

Recall our convention of measures being nonnegative by default.

Fix $\alpha\in Q$ and assume first that $f$ is $l$-causal.  By the $l$-orientation of $\ray$, the real ray representative $\smash{\bar{f}_\alpha}$ of $f$ is \emph{nondecreasing}. Hence, the distributional derivative $\smash{\Diff\RCR \bar{f}_\alpha}$ of $\smash{\RCR \bar{f}_\alpha}$ is a   Radon measure on $\smash{\bar{\mms}_\alpha}$. In turn, the assignment
\begin{align*}
\Diff_\alpha\RCR f  := \ray_\cdot(\alpha)_\push\Diff\RCR \bar{f}_\alpha
\end{align*}
is a  Radon measure on $\smash{\bar{\mms}_\alpha}$. The tag $\RCR f$ appearing here comes from the observation  the real ray representative of the $l$-left-continuous representative of $f$ (relative to a subset of $\mms$) coincides with the right-continuous representative of $\smash{\bar{f}_\alpha}$ (relative to a subset of $\R$); this association  shows up in the proof of \cref{Le:IBP II}.  

If $f$ is a general finite perturbation, by the decomposition $\smash{f = \tau^{-1}\,(\u + \tau f) - \tau^{-1}\,\u}$ for some $\tau > 0$, the real ray representative $\smash{\bar{f}_\alpha}$ is the difference of two monotone  functions, thus of locally bounded variation on $\smash{\bar{\mms}_\alpha}$. The $\tau$-independent assignment
\begin{align}\label{Eq:Borel measure}
\Diff_\alpha \RCR f &:= \tau^{-1}\,\Diff_\alpha\RCR (\u+\tau f) - \tau^{-1}\,\Diff_\alpha\RCR\u
\end{align}
is a generalized signed Radon measure on $\mms_\alpha$. 

\begin{lemma}[Distributional derivative]\label{Le:Disderd} The assignment
\begin{align*}
\Diff \RCR f := \int_Q \Diff_\alpha\RCR f\d\q(\alpha)
\end{align*}
constitutes a generalized signed Radon measure on $\mms$.
\end{lemma}

\begin{proof} The quantity $\smash{\Diff_\alpha\RCR f}$ depends $\q$-measurably on $\alpha\in Q$ as the property uniquely defining the terms on the right-hand side of \eqref{Eq:Borel measure} does. 

Finally, we claim that
\begin{align*}
\sigma := \int_Q \Diff_\alpha\RCR(\u+\tau f)\d\q(\alpha)
\end{align*}
defines a Radon measure on $\mms$; by \eqref{Eq:Borel measure} and the discussion after this proof, this easily implies $\Diff\RCR f$ is the difference of two Radon measures on $\mms$, interpreted in the sense of Radon functionals. By the disintegration \cref{Th:Disintegration}, $\sigma$ clearly forms a Borel measure on $\mms$ which vanishes outside of $\Tr$. To conclude, it suffices to show $\sigma$ is finite on compact sets, since every such measure is automatically Radon in our setting of \cref{Sub:MMS}, cf.~Folland  \cite{folland1999}*{Thm.~7.8}. Let $C$ be a fixed compact subset of $\mms$. By \cref{Le:Bounded},  $\u+\tau f$ is bounded on $C$, hence has finite oscillation thereon. On the other hand, by the definition of the Lebesgue--Stieltjes measure of $\smash{\bar{\u}_\alpha+\tau\bar{f}_\alpha}$, it is easily seen $\smash{\Diff_\alpha(\u + \tau f)[C]}$ is bounded from above by the oscillation of $\u+\tau f$ on $C$ for $\q$-a.e.~$\alpha\in Q$. Since this upper bound is independent of $\alpha\in Q$, the  finiteness of $\Diff \RCR(\u+\tau f)[C]$ follows.
\end{proof}

In fact, the above decomposition of $f$ contains more information we crucially exploit in \cref{Le:IBP II} et seq.  Indeed, $\smash{\bar{\u}_\alpha}$ is \emph{affine} with slope $1$ by \cref{Pr:Ray}, where $\alpha\in Q$.  Hence, the $\smash{\Leb_\alpha^1}$-singular part of the generalized signed Radon measure
\begin{align*}
\Diff_\alpha \RCR f  = \tau^{-1}\,\Diff_\alpha\RCR(\u + \tau f)- \tau^{-1}\,\Leb_\alpha^1
\end{align*}
is \emph{nonnegative} (irrespective of whether $f$ is $l$-causal  or not), i.e.~
\begin{align}\label{Eq:NONNEG!!!!!!!!!!!}
(\Diff_\alpha\RCR f)^\perp \geq 0.
\end{align}

Lastly, the subsequent assignment is well-defined for $\smash{\Leb_\alpha^1}$-a.e.~$x\in\mms_\alpha$:
\begin{align}\label{Eq:f' def}
f'(x) &:= \lim_{t\to 0} \frac{f\circ\sfg_t(x) - f(x)}{t}.
\end{align}
The $\meas$-measurable function $f'$ on $\Tr$ is called  \emph{derivative} of $f$ along  $\ray$. It is an $\smash{\Leb_\alpha^1}$-version of the density of the $\smash{\Leb_\alpha^1}$-absolutely continuous part of $\Diff_\alpha \RCR f$, i.e.
\begin{align}\label{Eq:Leb decomp dalpha}
\Diff_\alpha\RCR f = f'\,\Leb_\alpha^1 + (\Diff_\alpha\RCR f)^\perp
\end{align}
in the  sense made precise in \cref{Sub:Fixed}.

\begin{remark}[Compatibility]\label{Re:Compat} It is easy to check that \eqref{Eq:f' def} is compatible with the classical derivative $\smash{\bar{f}_\alpha'}$ of the real ray representative $\smash{\bar{f}_\alpha}$ for $\q$-a.e.~$\alpha\in Q$ after  \cref{Sub:Monotone real}, in that
\begin{align*}
f' \circ \sfg_\cdot(\alpha) = \bar{f}_\alpha'\quad \Leb^1\mres\bar{\mms}_\alpha\textnormal{-a.e.}\tag*{{\footnotesize{$\blacksquare$}}}
\end{align*} 
\end{remark}

\begin{corollary}[Integration by parts II, see also \cref{Th:Partial integration f}]\label{Le:IBP II} For $\q$-a.e.~$\alpha\in Q$, we have the following. If $v,w\in \cl\, \mms_\alpha$ are distinct points with $\smash{v\preceq_{\u} w}$, 
\begin{align*}
\int_{(v,w]} h_\alpha\,f'\d\Leb_\alpha^1 &\leq \int_{(v,w]} h_\alpha\d\Diff_\alpha\RCR f\\
&= -\int_{(v,w]} f\, h_\alpha'\d\Leb_\alpha^1 + \big[\!\RCR f\,h_\alpha\big]_v^w.
\end{align*}
\end{corollary}

\begin{proof} Recall \cref{Re:Identification} for the tacit identifications used throughout the proof. We first prove that the real ray representative  $\smash{\bar{h}_\alpha}$ of the conditional density $h_\alpha$ is locally absolutely continuous on $\smash{\cl\,\bar{\mms}_\alpha}$ for $\q$-a.e.~$\alpha\in Q$. In the interior, this follows from local Lipschitz continuity stated in \cref{Re:Properties MCP}. We are  left to study the boundary behavior. To simplify the presentation, we assume \emph{both} endpoints $a_\alpha$ and $b_\alpha$ exist and show the claim for $\smash{v := a_\alpha}$ and $\smash{w:= b_\alpha}$. We shall identify $\smash{\mms_\alpha\cup\{w\}}$ with the  interval $\smash{(\bar{v},\bar{w}]}$; the arguments in the remaining cases are analogous. Owing to the disintegration \cref{Th:Disintegration}, we may and will  assume $\alpha$ is such that $\meas_\alpha$ gives finite mass to all compact subsets of $\mms$. Lastly, by compactness of causal  diamonds there is a negative constant $K$ with $k \geq K$ on $\smash{J(w,v)}$.

By continuity of $\smash{\bar{h}_\alpha}$ on all of $\smash{[\bar{v},\bar{w}]}$, see \cref{Re:Properties MCP} again, it  suffices to prove $\smash{\bar{h}'}$ is $\smash{\Leb^1\mres [\bar{v},\bar{w}]}$-integrable. However, since $\smash{\bar{h}_\alpha}$ is an $\MCP(K,N)$ density in the interior, this is a direct consequence of \cref{Le:Bounds}: there exists a constant $\smash{C_{l(w,v)}^{K,N}}$ such that
\begin{align*}
\int_{[\bar{v},\bar{w}]} \vert\bar{h}_\alpha'\vert\d\Leb^1 \leq \frac{1}{l(w,v)}\,C_{l(w,v)}^{K,N}\,\meas_\alpha[J(w,v)]. 
\end{align*}

Now the statement of the corollary readily follows from  \eqref{Eq:NONNEG!!!!!!!!!!!}, \eqref{Eq:Leb decomp dalpha},  and \cref{Th:Partial integration f} (applied to the nondecreasing function $\smash{\bar{f}_\alpha}$), using  $\smash{\RCR \bar{f}_\alpha}$ coincides with the real ray representative of $\RCR f$ on $\smash{\bar{\mms}_\alpha}$ by the $l$-orientation of  $\sfg$.
\end{proof}

\subsubsection{Convex and compactly extendible sets} Lastly,  we define two properties of subsets of $\Tr$ which will be useful below. These are inspired by Cavalletti--Mondino  \cite{cavalletti-mondino2020-new}*{Def.~4.3}.

\begin{definition}[Convexity]\label{Def:Convexity} A Borel subset $A$ of $\smash{\Tr}$ is termed \emph{$\smash{\sim_{\u}}$-convex} if for every $x\in \Tr$, the image of $\smash{A\cap \tilde{x}^{\u}}$ under the map $\smash{\sfg^{-1}}$ is an interval.
\end{definition}

This definition is easily seen to be compatible with the notion of $\smash{\sim_\u}$-convexity of the ambient set $E$ we have already used in  \cref{Sub:FRAME}.

Next, given $\varepsilon > 0$ the \emph{$2\varepsilon$-enlargement} of a subset $A$ of $\Tr$ relative to $\q$ is the set of all points of the form $\sfg_t(x)$ with  $t\in[-2\varepsilon,2\varepsilon]$  (whenever defined) and $x\in A$. 

\begin{definition}[Compact extendibility]\label{Def:Admissible} Given $\varepsilon>0$, a Borel subset $A$ of $\smash{\Tr}$ is \emph{compactly $2\varepsilon$-extendible} if for every $x\in A$, the map $\sfg_\cdot(x)$ is defined on the entire interval $[-2\varepsilon,2\varepsilon]$ and the induced $2\varepsilon$-enlargement of $A$ is precompact.

It is \emph{compactly extendible} if it is compactly $2\varepsilon$-extendible for some $\varepsilon>0$.
\end{definition}

The rough purposes of the previous notion is the following. 

The ``$\varepsilon$ of room'' provided therein will entail better bounds for the conditional densities of rays restricted to a compactly $2\varepsilon$-extendible subset $A$ of $\Tr$. It will ensure  the arguments of the generalized cotangent functions in \cref{Re:Logarithmic derivative} are bounded away from zero. In turn, the logarithmic derivative of every conditional density is uniformly bounded on every ray segment contained in $A$; see \cref{Pr:Plan representing}.

\begin{remark}[Generation]\label{Re:Gener} Since $\q$-a.e.~ray is order isometric to an \emph{open} interval via the reverse ray map $\ray$,  the intersection of every chronological diamond with $\Tr$  is the countable union of $\smash{\sim_{\u}}$-convex, compactly extendible Borel sets (where the implicit parameter $\varepsilon$ will of course vary over these sets in general). In particular, subsets of $\Tr$ obeying   \cref{Def:Convexity,Def:Admissible} generate the relative Borel $\sigma$-algebra of $\Tr$.\hfill{\footnotesize{$\blacksquare$}}
\end{remark}

\subsection{Test plans from localization}\label{Sub:Repr of} This part contains some preparatory material especially for \cref{Cor:Constant slope} and \cref{Th:df vs f'}. In \cref{Pr:Plan representing} below, we construct test plans related to $\u$ by using localization. This result is a Lorentzian version of Cavalletti--Mondino  \cite{cavalletti-mondino2020-new}*{Prop.~4.4}. In the context of Gigli's splitting theorem for RCD spaces, $\u$ would naturally be chosen as the Busemann function induced by a line. Bounded compression of test plans induced by its negative gradient flow lines then follows from the much stronger property of measure preservation, cf.~Gigli \cite{gigli2013}*{Thm.~3.21}. We also refer to the discussion in Cavalletti--Mondino  \cite{cavalletti-mondino2020-new}*{§7}. In the Lorentzian case, a related construction comes from Braun \cite{braun2023-good}*{Thm.~1.2}.

Let $\varepsilon\in (0,1)$ be fixed. For a $\smash{\sim_{\u}}$-convex, compactly $2\varepsilon$-extendible  (thus precompact) Borel subset $A$ of $\smash{\Tr}$ with $\meas[A]>0$, we define the Borel map $\sfG\colon A \to \Cont([0,1];\mms)$ by ``stopping'' all rays starting at $x\in A$ after time $\varepsilon$, i.e.~$\sfG(x)_t := \sfg_{t\wedge \varepsilon}(x)$. We consider the Borel probability measure $\smash{\bdpi}$ on $\Cont([0,1];\mms)$ defined by
\begin{align}\label{Eq:bdpi A def}
\bdpi := \meas[A]^{-1}\,\sfG_\push[\meas\mres A].
\end{align}
For $t\in[0,\varepsilon]$,  the element
\begin{align}\label{Eq:The measure}
(\eval_t)_\push\bdpi^\leftarrow = \meas[A]^{-1}\,(\eval_t\circ\sfG)_\push[\meas\mres A] = \meas[A]^{-1}\,(\sfg_t)_\push[\meas\mres A]
\end{align}
of $\Prob(\mms)$ resembles the role of the push-forward of $\smash{\meas[A]^{-1}\,\meas\mres A}$ by the negative gradient flow of $\u$ at time $t$.

\begin{proposition}[Induced test plans]\label{Pr:Plan representing} The Borel probability measure $\bdpi$ is a test plan.
\end{proposition}

\begin{proof} Evidently, $\bdpi$ is concentrated on continuous  $l$-causal curves.

Next, we establish  $\bdpi$ has bounded compression. By construction and  \eqref{Eq:The measure}, it suffices to show the existence of a constant $C$ with the property $(\sfg_t)_\push[\meas\mres A] \leq C\,\meas$ for every $t\in[0, \varepsilon]$. The disintegration \cref{Th:Disintegration} implies
\begin{align}\label{Eq:LEB1}
(\sfg_t)_\push[\meas\mres A] = \int_Q (\sfg_t)_\push[\meas_\alpha\mres A]\d\q(\alpha).
\end{align}
Combining the $\smash{\Leb_\alpha^1}$-absolute continuity of $\meas_\alpha$ with positive density $h_\alpha$,   \cref{Re:Properties MCP},  the transformation formula, and the translation invariance of $\smash{\Leb^1}$ we obtain
\begin{align}\label{Eq:LEB!}
(\sfg_t)_\push[\meas_\alpha\mres A] = \frac{h_\alpha\circ\sfg_{-t}}{h_\alpha}\,\meas_\alpha\mres \sfg_t(A)
\end{align}
for $\q$-a.e.~$\alpha\in\Quot(A)$. As the $2\varepsilon$-enlargement of $A$ relative to $\q$ is precompact, there is a negative real constant $K$ such that for $\q$-a.e.~$\alpha\in \Quot(A)$, the inequality $\smash{\bar{k}_\alpha\geq K}$ holds on its closure. Since  $A$ is $\smash{\sim_{\u}}$-convex, the image of $\smash{\mms_\alpha \cap \bigcup_{t\in[0,\varepsilon]}\ray_t(A)}$ under $\smash{\ray^{-1}}$ is an interval of finite diameter, say $[\bar{v}_\alpha,\bar{w}_\alpha]$. Again by precompactness, there exists a real constant $L$ such that $\smash{\vert\bar{w}_\alpha - \bar{v}_\alpha\vert \leq L}$ for $\q$-a.e. $\alpha\in \Quot(A)$. Since $t\in[0,\varepsilon]$ and since $A$ is compactly $2\varepsilon$-extendible, the real ray representative $\smash{\bar{h}_\alpha}$ of $h_\alpha$  is a $\CD(K,N)$ density on the interval  $\smash{(\bar{v}_\alpha - \varepsilon,\bar{w}_\alpha + \varepsilon)}$. We define
\begin{align*}
C:= \frac{\sinh(\sqrt{-K/(N-1)}\,(L+\varepsilon))^{N-1}}{\sinh(\sqrt{-K/(N-1)}\,(\varepsilon))^{N-1}}.
\end{align*}
Applying \cref{Le:logarithmic derivative} and \cref{Re:Const sink cosk}, every $\smash{x\in [\bar{v}_\alpha, \bar{w}_\alpha]}$  satisfies
\begin{align*}
\frac{\bar{h}_\alpha(x-t)}{\bar{h}_\alpha(x)} &\leq \frac{\sinh(\sqrt{-K/(N-1)}\,(\bar{w}_\alpha + \varepsilon - x+t)^{N-1}}{\sinh(\sqrt{-K/(N-1)}\,(\bar{w}_\alpha + \varepsilon - x)^{N-1}}\leq C.
\end{align*}
With \eqref{Eq:LEB1} and \eqref{Eq:LEB!}, this yields the claimed estimate
\begin{align*}
(\sfg_t)_\push[\meas\mres A] \leq C \int_Q \meas_\alpha\mres\sfg_t(A) \d\q(\alpha) \leq C\int_Q\meas_\alpha\d\q(\alpha)\leq  C\,\meas.\tag*{\qedhere}
\end{align*}
\end{proof}

\begin{remark}[Representation of past gradients] Using the reverse Young inequality, \cref{Cor:Constant slope} below, and the affinity of $\u$ along each ray with slope one, one can  verify that $\bdpi$ in fact represents the past $\PP$-gradient of $\u$ according to Beran et al.~\cite{beran-braun-calisti-gigli-mccann-ohanyan-rott-samann+-}*{Def.~4.2}. Here $\PP$ is any nonzero number less than one. We refer to Cavalletti--Mondino \cite{cavalletti-mondino2020-new}*{Prop.~4.4} for a similar argument in positive signature.\hfill{\footnotesize{$\blacksquare$}}
\end{remark}

Next, we use these test plans to show $\smash{\u}$ has constant slope $1$ on $\Tr$. If $\u$ is the unsigned Lorentz distance function from a point, this is standard in the smooth case, cf.~e.g.~Treude--Grant  \cite{treude2011}*{Prop.~3.2.33},  while the singular counterpart of this has been shown recently by Beran et al.~\cite{beran-braun-calisti-gigli-mccann-ohanyan-rott-samann+-}*{Cor.~5.23}.

\begin{corollary}[Constant slope]\label{Cor:Constant slope} We have
\begin{align*}
\vert\rmd \u\vert = 1\quad\meas\mres\Tr\textnormal{-a.e.}
\end{align*}
\end{corollary}

\begin{proof} Since $\u$ is $1$-steep with respect to $l$, we directly obtain ``$\geq$''.

To show the converse inequality, let $A$ be a $\smash{\sim_{\u}}$-convex, compactly $2\varepsilon$-extendible subset of $\Tr$ for some $\varepsilon >0$. Consider the induced test plan $\bdpi$ from \eqref{Eq:bdpi A def}. Then for every rational $t\in (1-\varepsilon,1)$, by \cref{Le:Pathwise} and \cref{Pr:Ray} $\bdpi$-a.e.~$\gamma$ satisfies
\begin{align*}
t = l(\gamma_{1-t},\gamma_1)= \u(\gamma_1) - \u(\gamma_{1-t})  \geq \int_{1-t}^1 \vert \rmd\u\vert(\gamma_r) \d r \geq t.
\end{align*}
Hence, equality holds throughout. Integrating the resulting inequality and arguing as in the proof of \cite{beran-braun-calisti-gigli-mccann-ohanyan-rott-samann+-}*{Prop.~4.1} by using \cref{Pr:Plan representing} yields
\begin{align*}
1 = \lim_{t\to 0+} \frac{1}{t}\iint_{1-t}^1 \vert \rmd\u\vert(\gamma_r)\d r\d\bdpi(\gamma) = \int \vert \rmd\u\vert(\gamma_1)\d\bdpi(\gamma) = \meas[A]^{-1}\int_A \vert \rmd\u\vert\d\meas.
\end{align*}
The arbitrariness of $A$ terminates the proof.
\end{proof}

Since we only consider functions defined on $\scrT$, we consistently abbreviate
\begin{align}\label{Eq:dnabla}
\rmd^+f(\nabla\u) := \lim_{\PP\to 1-}\One_{\Tr}\,\rmd^+f(\nabla\u)\,\vert\rmd\u\vert^{\PP-2}.
\end{align}
In all relevant cases, the limit exists: indeed, the $\PP$-dependent quantity it is $\meas$-a.e.~constant by \cref{Cor:Constant slope} and \cref{Re:Indep!}.

\subsection{Localization of horizontal vs.~vertical differentiation}\label{Sub:Loc hor vert} Now we finally connect the horizontal and the vertical approach to differentiation from \cref{Sub:Horiz} by using  disintegration tools. This improves Beran et al.~\cite{beran-braun-calisti-gigli-mccann-ohanyan-rott-samann+-}*{Thm.~4.15} from an integrated to an a.e.~estimate. This is crucial to perform integration by parts later, cf.~e.g.~the proof of \cref{Th:Meas val Alem I}.

\begin{theorem}[Horizontal vs.~vertical differentiation]\label{Th:df vs f'} Let $\scrM$ be a time\-like essentially nonbranching $\smash{\TMCP^e(k,N)}$ metric measure spacetime. Let $\q$ be an $\MCP(k,N)$-disintegration  relative to $\smash{\u}$. Finally, let $f$ be a finite perturbation of $\u$ on $\smash{\Tr}$. Then
\begin{align*}
\rmd^+f(\nabla \u) \leq f'\quad\meas\mres\Tr\textnormal{-a.e.}
\end{align*}
\end{theorem}

\begin{proof} Let $\tau > 0$ be such that $\u+\tau f$ is $l$-causal. As in the proof of \cref{Cor:Constant slope}, let $A$ be a $\smash{\sim_\u}$-convex, compactly $2\varepsilon$-extendible subset of $\Tr$ for some $\varepsilon>0$. For $s,t\in[0,\varepsilon]$ with $s\leq t$, we consider the Suslin set  $A_{[s,t]}$ consisting of all points of the form $\ray_r(x)$, where $x\in A$ and $r\in[s,t]$. Furthermore, let $\PP,\QQ\in (-\infty,1)\setminus \{0\}$ be mutually conjugate with $\PP$ positive. Applying \cref{Le:Pathwise} to the test plan $\bdpi$ induced by $A$ from  \eqref{Eq:bdpi A def} (noting $\vert \dot\gamma\vert = 1$ $\Leb^1\mres[0,\varepsilon]$-a.e.~for $\bdpi$-a.e.~$\gamma$), employing the disintegration \cref{Th:Disintegration}  and the definition of the Lebesgue--Stieltjes measure, and using locality $\q$-a.e.~$\alpha\in \Quot(A)$ obeys
\begin{align*}
\Diff_\alpha\LCR(\u + \tau f)[A_{[s,t]}] &\geq \int_{A_{[s,t]}}\vert \rmd(\u + \tau f)\vert\d\Leb_\alpha^1\\
&\geq \frac{1}{\PP}\int_{A_{[s,t]}} \vert\rmd(\u+\tau f)\vert^\PP\d\Leb_\alpha^1 + \frac{1}{\QQ}\Leb_\alpha^1[A_{[s,t]}].
\end{align*}
As $\PP$ is positive, $\QQ$ is negative, and closed intervals generate the Borel $\sigma$-algebra of $[0,\varepsilon]$, this reduces to the following inequality between Radon measures:
\begin{align*}
\Diff_\alpha\LCR(\u+\tau f) \mres A_{[0,\varepsilon]} - \QQ^{-1}\,\Leb_\alpha^1 \mres A_{[0,\varepsilon]} \geq \PP^{-1}\,\vert\rmd (\u+\tau f)\vert^\PP \,\Leb_\alpha^1 \mres A_{[0,\varepsilon]}.
\end{align*}

Recalling \eqref{Eq:NONNEG!!!!!!!!!!!}, this improves to
\begin{align}\label{Prev1}
\u' + \tau f' - \QQ^{-1} \geq \PP^{-1}\,\vert\rmd(\u+\tau f)\vert^\PP\quad\Leb_\alpha^1\mres A_{[0,\varepsilon]}\textnormal{-a.e.}
\end{align}
On the other hand, \cref{Pr:Ray} and \cref{Cor:Constant slope}  imply
\begin{align}\label{Prev2}
\u' - \QQ^{-1} = 1 - \QQ^{-1} = \PP^{-1}\,\vert\rmd\u\vert^\PP\quad\Leb_\alpha^1\mres A_{[0,\varepsilon]}\textnormal{-a.e.}
\end{align}
Subtracting \eqref{Prev2} from \eqref{Prev1} and using positivity of $h_\alpha$ on $\mms_\alpha$ yields
\begin{align*}
\frac{\vert\rmd(\u + \tau f)\vert^\PP - \vert\rmd\u\vert^\PP}{\PP\tau} \leq f'\quad\meas_\alpha\mres A_{[0,\varepsilon]}\textnormal{-a.e.}
\end{align*}

By the disintegration \cref{Th:Disintegration} again, the previous inequality holds $\smash{\meas\mres A_{[0,\varepsilon]}}$-a.e. By the arbitrariness of $A$ (and hence $\varepsilon$, cf.~\cref{Re:Gener}), we verify the above inequality to hold $\meas\mres\Tr$-a.e. Letting $\tau \to 0+$ while using \eqref{Eq:Vert diff} and \eqref{Eq:dnabla}, the claim is shown.
\end{proof}

The following should be compared to \cref{Re:Two sided}.

\begin{remark}[Symmetry]\label{Re:Symmmmmmmm} Assume $f$ is a symmetric finite perturbation of $\u$. Then the previous \cref{Th:df vs f'} and \cref{Le:Diff quot} imply
\begin{align*}
\rmd^+f(\nabla\u) \leq f'\leq \rmd^-f(\nabla\u)\quad\meas\mres\Tr\textnormal{-a.e.}
\end{align*}
where the right-hand side is defined in the evident way following \eqref{Eq:dnabla}. 

With a similar argument using \eqref{Eq:NONNEG!!!!!!!!!!!}, it also follows  if $\scrM$ is $q$-infinitesimally strictly concave (cf.~\cref{Re:Inf str conv}) for some $q\in (-\infty,1)\setminus \{0\}$, then $\Diff_\alpha\RCR f$ has no singular part with respect to $\meas_\alpha$ for $\q$-a.e.~$\alpha\in Q$. \hfill{\footnotesize{$\blacksquare$}}
\end{remark}

\subsection{A formula for the d'Alembertian of a general steep function}\label{Sub:Formula1} In \cref{Th:Meas val Alem I} below, we provide a first formula for the d'Alembertian of the  globally fixed $1$-steep function $\u$. \cref{Re:Properties MCP} and \cref{Le:logarithmic derivative} collect basic properties of the conditional densities from  the fixed disintegration $\q$ we tacitly use below. In particular, by the arguments of \cref{Le:IBP II} and  \cref{Le:Radon funct} below, every integral over $\q$-a.e.~ray will be well-defined. 

For general $\u$, in analogy to Cavalletti--Mondino  \cite{cavalletti-mondino2020-new}*{Prop.~4.7, Thm.~4.8} in positive signature, \cref{Th:Meas val Alem I} does require an integrability assumption on the inverse length of its induced rays. Roughly speaking, sufficiently many of them should not be  too short. 

Recall the initial and final points $a_\alpha$ and $b_\alpha$ of the ray $\mms_\alpha$ from \cref{Sub:Ray map}, where $\alpha\in Q$.

\begin{lemma}[Induced weak Radon functional I, see also \cref{Le:Radon funct II}]\label{Le:Radon funct} Assume
\begin{align*}
\int_Q \frac{1}{l(a_\alpha,b_\alpha)}\d\q(\alpha) < \infty.
\end{align*}
Then the  map $T$ on $\Pert_\bc(\u,\Tr)$ given  by
\begin{align*}
T(f) := \int_Q\int_{\mms_\alpha} f\,(\log h_\alpha)'\d\meas_\alpha\d\q(\alpha) - \int_Q \big[\!\RCR f\,h_\alpha\big]_{a_\alpha}^{b_\alpha}\d\q(\alpha)
\end{align*}
defines a weak Radon functional.
\end{lemma}

\begin{proof} Let $W$ be a precompact open set containing $\supp f$, where $f\in\Pert_\bc(\u,\Tr)$. We may and will assume that $W$ is $l$-geodesically convex up to replacing it with its chronological emerald. Let $K$ be a negative  constant with $k \geq K$ on $\cl(\mms_\alpha\cap W)$ for every $\alpha\in\Quot(W)$. Let $L$ be the maximum of $\smash{l_+}$ on the compact set $\smash{(\cl\,W)^2}$. Lastly, given $\alpha\in \Quot(W)$ let $\smash{v_\alpha}$ and $\smash{w_\alpha}$ be the minimal and maximal elements of the set $\smash{\cl(W\cap \mms_\alpha)}$ with respect to $\smash{\preceq_{\u}}$, respectively. Given $\varepsilon > 0$,  let $\Quot(W)_\varepsilon$ denote the set of all $\alpha\in \Quot(W)$ such that
\begin{align*}
l(v_\alpha,w_\alpha) \geq \varepsilon \vee l(a_\alpha,b_\alpha) \wedge L.
\end{align*}

We first bound the first integrand.  \cref{Le:Bounds} and our choices of $K$ and $L$  imply there exists a constant $\smash{C_{L}^{K,N}}$ such that for $\q$-a.e.~$\alpha\in\Quot(W)_\varepsilon$,
\begin{align*}
\int_{\mms_\alpha} \vert h_\alpha'\vert\d\Leb_\alpha^1\mres W \leq \frac{1}{l(v_\alpha, w_\alpha)}\,C_L^{K,N}\,\q\textnormal{-}\!\esssup_{\alpha\in Q}\meas_\alpha[W].
\end{align*}
The right-hand side is finite by the disintegration  \cref{Th:Disintegration}. Since $\supp f$ lies in $W$,
\begin{align*}
\Big\vert\!\int_{\mms_\alpha} f\,(\log h_\alpha)'\d\meas_\alpha\Big\vert &\leq \int_{\mms_\alpha} \vert f\vert\,\vert h_\alpha'\vert\d\Leb_\alpha^1\mres W\\
&\leq \frac{\Vert f\Vert_\infty}{l(v_\alpha, w_\alpha)}\,C_L^{K,N}\,\q\textnormal{-}\!\esssup_{\alpha\in Q}\meas_\alpha[W].
\end{align*}

We turn to the second integrand in the definition of $T(f)$. For simplicity, we assume $a_\alpha$ and $b_\alpha$ exist for the given $\alpha\in\Quot(W)_\varepsilon$, the other cases are dealt with analogously. Then by \cref{Le:IBP II} and since $f$ has compact support in the interval $(v_\alpha,w_\alpha)$,
\begin{align}\label{Eq:bavw}
\big[\!\RCR f\,h_\alpha\big]_{a_\alpha}^{b_\alpha} = \big[\!\RCR f\,h_\alpha\big]_{v_\alpha}^{w_\alpha}.
\end{align}
By \cref{Le:Bounds} again, we obtain
\begin{align*}
\sup_{x\in \mms_\alpha\cap W} h_\alpha(x) \leq \frac{1}{l(v_\alpha,w_\alpha)}\,\Big[\!\int_0^1 \sigma_{K/(N-1)}^{(r)}(L)^{N-1}\d r\Big]^{-1}\,\q\textnormal{-}\!\esssup_{\alpha\in Q} \meas_\alpha[W],
\end{align*}
which directly implies
\begin{align*}
\big\vert\big[f\,h_\alpha\big]_{v_\alpha}^{w_\alpha}\big\vert \leq \frac{2\Vert f\Vert_\infty}{l(v_\alpha, w_\alpha)}\,\Big[\!\int_0^1 \sigma_{K/(N-1)}^{(r)}(L)^{N-1}\d r\Big]^{-1}\,\q\textnormal{-}\!\esssup_{\alpha\in Q}\meas_\alpha[W].
\end{align*}

Combining these two observations leads to
\begin{align*}
&\int_{\Quot(W)_\varepsilon} \Big\vert\!\int_{\mms_\alpha} f\,(\log h_\alpha)'\d\meas_\alpha\Big\vert \d\q(\alpha) + \int_{\Quot(W)_\varepsilon} \big\vert\big[\!\RCR f\,h_\alpha\big]_{a_\alpha}^{b_\alpha}\big\vert\d\q(\alpha)
\\
&\qquad\qquad \leq \Big[C_L^{K,N} + 2\Big[\!\int_0^1 \sigma_{K/(N-1)}^{(r)}(L)^{N-1}\d r\Big]^{-1}\Big]\int_Q \frac{1}{l(a_\alpha,b_\alpha)}\d\q(\alpha)\,\Vert f\Vert_\infty.
\end{align*}
Letting $\varepsilon \to 0+$, the statement follows by  Levi's monotone convergence theorem (using $\q$-a.e.~ray has positive length) and our integrability assumption.
\end{proof}

\begin{remark}[About linearity II, see also \cref{Re:About I}]\label{Re:About II} The map  from \cref{Le:Radon funct} is clearly linear with respect to linear combinations with negative numbers (recall also the footnote in   \cref{Re:Two sided}), unlike the set of finite perturbations of $\u$ on $\Tr$.\hfill{\footnotesize{$\blacksquare$}}
\end{remark}

\begin{theorem}[Distributional d'Alembertian]\label{Th:Meas val Alem I} Let $\scrM$ form a timelike essentially nonbranching $\smash{\TMCP^e(k,N)}$ metric measure spacetime.  Furthermore, let $\q$ be an $\MCP(k,N)$ disintegration relative to $\smash{\u}$ and $E$ satisfying
\begin{align*}
\int_Q \frac{1}{l(a_\alpha,b_\alpha)}\d\q(\alpha) < \infty.
\end{align*}
Let $U$ be an open subset of $\mms$ with $\meas[U\setminus \Tr] = 0$. Then $\u\in\Dom(\BOX, U)$.

More precisely, the functional $T_U$ on $\Pert_\bc(\u,U)$  given by
\begin{align*}
T_U(f) := \int_Q\int_{\mms_\alpha} f\,(\log h_\alpha)'\,\d\meas_\alpha \d\q(\alpha)  - \int_Q \big[\!\RCR f\,h_\alpha\big]_{a_\alpha}^{b_\alpha}\d\q(\alpha)
\end{align*}
defines a weak Radon functional belonging to $\BOX \u\mres U$.
\end{theorem}

\begin{proof} \cref{Le:Radon funct} above shows $T_U$ defines a weak Radon functional.

Let $f\in\Pert_\bc(\u,U)$. For $\q$-a.e.~$\alpha\in Q$, combining \cref{Th:df vs f'} and \cref{Le:IBP II} while recalling the definition \eqref{Eq:Borel measure} of $\smash{\Diff_\alpha\RCR f}$,
\begin{align*}
\int_{\mms_\alpha} \rmd^+f(\nabla\u)\d\meas_\alpha &\leq \int_{\mms_\alpha} f'\,h_\alpha\d\Leb_\alpha^1\\
&\leq \int_{(a_\alpha,b_\alpha]} h_\alpha\d\Diff_\alpha\LCR f\\
&= -\int_{\mms_\alpha} f\,(\log h_\alpha)'\d\meas_\alpha + \big[\!\RCR f\,h_\alpha\big]_{a_\alpha}^{b_\alpha}.
\end{align*}
By \cref{Cor:Constant slope} and the comment after \eqref{Eq:dnabla}, this estimation is notably independent of the choice of the implicit exponent $\PP\in (-\infty,1)\setminus\{0\}$.  Integration with respect to $\q$ yields the two claims simultaneously.
\end{proof}

\begin{remark}[Unbounded rays I, see also \cref{Re:Unbounded II}]\label{Re:Unbounded I} Since $f$ has compact support, if for $\alpha\in Q$ the initial point $a_\alpha$ does not exist, the last integrand in the definition of $T_U$ is naturally interpreted as $\smash{[\RCR f\,h_\alpha]^{b_\alpha}}$; analogously as $\smash{[\RCR f\,h_\alpha]_{a_\alpha}}$ if the final point $b_\alpha$ does not exist. In particular, if both endpoints do not exist, this term simply vanishes.\hfill{\footnotesize{$\blacksquare$}}
\end{remark}

The functional from \cref{Th:Meas val Alem I} can be decomposed as
\begin{align*}
T_U =T_U^\ll + T_U^\perp,
\end{align*}
where the two maps $\smash{T_U^\ll}$ and $\smash{T_U^\perp}$ on $\Pert_\bc(\u,U)$ are defined by
\begin{align*}
T_U^\ll(f) &:=  \int_Q\int_{\mms_\alpha} f\,(\log h_\alpha)'\d\meas_\alpha\d\q(\alpha),\\
T_U^\perp(f) &:= -\int_Q \big[\!\RCR f\,h_\alpha\big]_{a_\alpha}^{b_\alpha}\d\q(\alpha).
\end{align*}
These are  weak Radon functionals by the proof of \cref{Le:Radon funct}. The functional $\smash{T_U^\perp}$ is called the  \emph{singular part} of $T_U$. Note $\smash{T_U^\perp}$ vanishes  identically  if $\q$-a.e.~ray is order isometric to $\R$ by \cref{Re:Unbounded I}. Analogously, we call $\smash{T_U^\ll}$ the \emph{absolutely continuous part} of $T_U$. Somewhat suggestively,  on $\mms_\alpha$ with $\alpha\in Q$ we also define
\begin{align}\label{Eq:Tll}
\Box\u := (\log h_\alpha)'.
\end{align}
\cref{Le:logarithmic derivative} and the equivalence of $\smash{\Leb_\alpha^1}$ and $\meas_\alpha$ on $\mms_\alpha$  (since $h_\alpha$ is positive thereon) for $\q$-a.e.~$\alpha\in Q$ then transfers to comparison bounds for the density $\smash{\Box\u}$ in \cref{Cor:Sharp comparison}. Since the curvature variable in \cref{Le:logarithmic derivative} depends on the reference interval, the corresponding estimates become more complicated. To serve readers only interested in the constant case, we henceforth intertwine our comparison results with their simpler vanishing $k$ (SEC) versions. For constant yet nonvanishing  $k$, recall \cref{Re:Logarithmic derivative}.

\begin{corollary}[Sharp d'Alembert comparison I, see also \cref{Cor:11,Cor:1515}]\label{Cor:Sharp comparison} In the framework of \cref{Th:Meas val Alem I}, the following hold for $\q$-a.e.~$\alpha\in Q$. 
\begin{enumerate}[label=\textnormal{(\roman*)}]
\item \textnormal{\textbf{Variable version.}} Let $v,w\in \cl\,\mms_\alpha$ be distinct points such that $\smash{v\preceq_{\u} w}$  \textnormal{[sic]}. Then $\smash{\meas_\alpha}$-a.e.~$x\in [v,w]$ satisfies
\begin{align*}
&- (N-1) \,\frac{\COS_{(k_\alpha)_{x,w}^-/(N-1)}\circ\,l(x,w)}{\SIN_{(k_\alpha)_{x,w}^-/(N-1)}\circ\,l(x,w)}\\
&\qquad\qquad  \leq \Box\u(x) \leq (N-1)\,\frac{\COS_{(k_\alpha)^+_{v,x}/(N-1)}\circ\, l(v,x)}{\SIN_{(k_\alpha)^+_{v,x}/(N-1)}\circ\, l(v,x)}.
\end{align*}
\item \textnormal{\textbf{SEC version.}} If $k$ vanishes identically,
\begin{align*}
-(N-1)\,\frac{1}{l(\cdot,b_\alpha)} \leq \Box\u \leq (N-1)\,\frac{1}{l(a_\alpha,\cdot)}\quad\meas_\alpha\textnormal{-a.e.}
\end{align*}
with appropriate interpretation if either endpoints do not exist.
\end{enumerate}
\end{corollary}

This first statement notably includes the endpoints once they exist.

\begin{remark}[Unbounded rays II, see also \cref{Re:Unbounded I}]\label{Re:Unbounded II} If in \cref{Cor:Sharp comparison}, $k$ is uniformly bounded from below on $\mms_\alpha$ by a negative real constant $K$ and $\mms_\alpha$ has infinite $\smash{l}$-length, by \cref{Re:Logarithmic derivative} and taking limits the respective critical bounds extend as follows.
\begin{itemize}
\item If $a_\alpha$ does not exist, then
\begin{align*} 
\Box\u \leq \sqrt{-K(N-1)}\quad\meas_\alpha\textnormal{-a.e.}
\end{align*}
\item If $b_\alpha$ does not exist, then
\begin{align*}
 -\sqrt{-K(N-1)} \leq \Box\u\quad\meas_\alpha\textnormal{-a.e.}
\end{align*}
\end{itemize}

An analogous interpretation is adopted for vanishing $K$.

Lastly, if $K$ is positive, both endpoints always exist by the Bonnet--Myers theorem, cf. e.g. Cavalletti--Mondino \cite{cavalletti-mondino2020}*{Prop.~3.6, Rem.~3.9}.\hfill{\footnotesize{$\blacksquare$}}
\end{remark}

\subsection{A formula for the d'Alembertian of a Lorentz  distance function}\label{Sec:Formula2} Now we specialize to $\u$ being the signed Lorentz distance function $l_\Sigma$ from an achronal TC  Borel subset $\Sigma$ of $\mms$, cf.~\cref{Ex:timedistfunct} and \cref{Def:footpts}. Accordingly, $E$ is replaced by the set $E_\Sigma$ from \eqref{Eq:ESigmaUSigma} in the beginning of \cref{Ch:Repr form}. The simpler case of $\Sigma$ being a singleton is treated separately in \cref{Sub:Point}.

We only formulate the subsequent results in terms of the open subset $\smash{I^+(\Sigma)\cup I^-(\Sigma)}$ of $\smash{E_\Sigma}$. By \cref{Re:Props alemb}, analogous statements hold for every open subset thereof. It is easy to see that $\meas[(I^+(\Sigma)\cup I^-(\Sigma)) \setminus \Tr]=0$, as often required in the preceding discussion.

Unlike \cref{Th:Meas val Alem I,Cor:Sharp comparison}, no integrability assumption on the inverse length of the rays is needed in this case. The rough reason for this is that  replacing $l_\Sigma$ by a suitable power of it plays in favor of ``uniform comparison bounds'' akin to  \cref{Le:logarithmic derivative}, as inferred in \cref{Th:11}. Applying the chain rule from \cref{Pr:Chain rule} propagates this to the originally sought d'Alembertian of $l_\Sigma$, cf.~\cref{Cor:22}.

\subsubsection{General reference sets $\Sigma$}\label{Gen Sigm}

\begin{remark}[Sign of $\smash{l_\Sigma}$ at endpoints]\label{Re:Neg!} Since the ray map $\sfg$ is oriented with respect to $l$, for every $\alpha\in Q$ such that the initial point $a_\alpha$ exists we have $\smash{l_\Sigma(a_\alpha) \leq 0}$. 

Analogously, for every $\alpha\in Q$ we have $\smash{l_\Sigma(b_\alpha) \geq 0}$ if the final point $b_\alpha$ exists.\hfill{\footnotesize{$\blacksquare$}}
\end{remark}

Let $\PP$ and $\QQ$ be conjugate numbers less than $1$. Similarly to  \cref{Th:df vs f'}, it is convenient to suppose $\QQ>0$ (hence $\PP<0$). In the chosen range of $\QQ$, the function $\varphi_\QQ$ defined by
\begin{align}\label{Eq:varphip}
\varphi_\QQ(r) := \QQ^{-1}\sgn r\,\vert r\vert^\QQ
\end{align}
is strictly increasing on all of $\R$ and smooth outside zero. Moreover, away from zero it is Lipschitz continuous with Lipschitz continuous inverse. This allows for a ``global'' treatment of the $\QQ$-th power of $l_\Sigma$ simultaneously defined on the chronological future and past of $\Sigma$. Conceptually, this is a choice rather than a technical necessity for what follows; one could equivalently drop the above sign  restrictions on $\PP$ and $\QQ$ and treat the $\PP$-d'Alembertians appearing below separately on $\smash{I^+(\Sigma)}$ and $\smash{I^-(\Sigma)}$, cf.~\cref{Re:Neg!,Re:Ext pos q}. 

Let us consider the $l$-causal function $\smash{\u_\QQ}$ on $\smash{E_\Sigma}$ defined by
\begin{align*}
\u_\QQ := \varphi_\QQ\circ l_\Sigma.
\end{align*}

The following version of \cref{Pr:Relations} is proven in \cref{App:A}.

\begin{lemma}[Equality of classes of finite perturbations]\label{Le:Invar p}  The sets $\Pert_\bc(l_\Sigma,I^+(\Sigma)\cup I^-(\Sigma))$ and $\Pert_\bc(\u_\QQ,I^+(\Sigma)\cup I^-(\Sigma))$ coincide.
\end{lemma}

\begin{lemma}[Induced weak Radon functional II, see also \cref{Le:Radon funct}]\label{Le:Radon funct II} The map $T^\QQ$ on $\smash{\Pert_\bc(\u_\QQ,\Tr)}$ defined by
\begin{align*}
T^\QQ(f) &:= \int_Q\int_{\mms_\alpha} f\sgn l_\Sigma\,\big[1+l_\Sigma\,(\log h_\alpha)' \big]\d\meas_\alpha\d\q(\alpha) \\
&\qquad\qquad - \int_Q \big[\!\RCR f\,\vert l_\Sigma\vert\,h_\alpha\big]_{a_\alpha}^{b_\alpha}\d\q(\alpha)
\end{align*}
defines a weak Radon functional.
\end{lemma}

\begin{proof} Without restriction, we assume $\supp f$ is contained in $\smash{I^+(\Sigma)}$. Let $W$ be as in the proof of \cref{Le:Radon funct}, but now we replace it with the closure of the causal emerald of the closure of $\smash{J^-(\cl\,W)}$ and $\Sigma$ as well as $\cl\,W$. Since $\Sigma$ is TC, this set is still compact. Moreover, now the point $v_\alpha$ corresponds to the unique intersection of $\cl\,\mms_\alpha$ with $\Sigma$, where $\alpha\in Q$. In particular, we have $\LCR f(v_\alpha) = 0$ and $l_\Sigma(v_\alpha) = 0$. 

As in the mentioned proof of  \cref{Le:Radon funct},
\begin{align*}
\big\vert\big[\!\LCR f\,h_\alpha\big]^{w_\alpha}\big\vert \leq \frac{\Vert f\Vert_\infty}{l(v_\alpha,w_\alpha)}\,\Big[\!\int_0^1 \sigma_{K/(N-1)}^{(r)}(L)^{N-1}\d r\Big]^{-1}\,\q\textnormal{-}\!\esssup_{\alpha\in Q}\meas_\alpha[W].
\end{align*}
On the other hand, by \cref{Re:Neg!} we obtain
\begin{align*}
l_\Sigma(w_\alpha) = l(v_\alpha, w_\alpha),
\end{align*}
and consequently
\begin{align*}
\big\vert\big[\!\LCR f\,\vert l_\Sigma\vert\,h_\alpha\big]^{w_\alpha}\big\vert \leq \Vert f\Vert_\infty \,\Big[\!\int_0^1 \sigma_{K/(N-1)}^{(r)}(L)^{N-1}\d r\Big]^{-1}\,\q\textnormal{-}\!\esssup_{\alpha\in Q}\meas_\alpha[W].
\end{align*}
The previous observations and using \eqref{Eq:bavw} lead to
\begin{align*}
\big\vert\big[\!\LCR f\,\vert l_\Sigma\vert\,h_\alpha\big]^{b_\alpha}_{a_\alpha}\big\vert \leq 2\Vert f\Vert_\infty \,\Big[\!\int_0^1 \sigma_{K/(N-1)}^{(r)}(L)^{N-1}\d r\Big]^{-1}\,\q\textnormal{-}\!\esssup_{\alpha\in Q}\meas_\alpha[W].
\end{align*}
Since $\q$ is a Borel probability measure, this establishes the desired upper bound for the last summand  in the definition of $\smash{T^\QQ}$.

In an analogous manner, for $\alpha\in Q$ we use the prefactor $l_\Sigma$ in front of the logarithmic derivative of $h_\alpha$ in the first summand to estimate
\begin{align*}
&\Big\vert\!\int_{\mms_\alpha} f\sgn l_\Sigma\,\big[1+ l_\Sigma\,(\log h_\alpha)'\big]\d\meas_\alpha\Big\vert  \leq \Vert f\Vert_\infty\,\big[1+ C_L^{K,N}\big]\,\q\textnormal{-}\!\esssup_{\alpha\in Q}\meas_\alpha[W].
\end{align*}
Again we use the normalization of $\q$ to terminate the proof.
\end{proof}

Next, in \cref{Le:nu measure} we construct a Radon measure which will ``dominate'' the weak Radon functional from \cref{Th:11} below. This allows us to turn the latter into a weak Radon functional with a sign, to which the Riesz--Markov--Kakutani representation theorem applies. Since our curvature bound is variable, the construction is rather technical; for a simplification to constant curvature bounds, the reader may consult \cref{Re:Simpli const}.

Fix a locally finite open cover $\scrV$ of $\smash{I^+(\Sigma)\cup I^-(\Sigma)}$ the closures of whose elements are all  compactly contained in $\smash{I^+(\Sigma)\cup I^-(\Sigma)}$. Such a cover exists by paracompactness of $\mms$. Without loss of generality, we assume $V$ does not intersect with both $\smash{I^+(\Sigma)}$ and $\smash{I^-(\Sigma)}$ for every $V\in\scrV$. Since $\Sigma$ is TC, given such a $V$ the set $\smash{I^-(V) \cap \Sigma}$ is precompact. Thus, $k$ is bounded from below by a negative constant $\smash{K_V}$ on the  emerald $\smash{J(\cl(I^-(V)\cap \Sigma), \cl\,V)}$ by global hyperbolicity. Using the disintegration formula we now define
\begin{align*}
\mu_V := \sqrt{-K_V(N-1)}\int_Q \vert l_\Sigma\vert\, \frac{\cosh(\sqrt{-K_V/(N-1)}\,\vert l_\Sigma\vert)}{\sinh(\sqrt{-K_V/(N-1)}\,\vert l_\Sigma\vert)} \, \meas_\alpha\mres V \d\q(\alpha).
\end{align*}

\begin{lemma}[Global measure]\label{Le:nu measure} The assignment
\begin{align*}
\mu := \sum_{V\in \scrV} \mu_V
\end{align*}
defines an $\meas$-absolutely continuous Radon measure on $\mms$.
\end{lemma}

\begin{proof} Clearly, $\mu$ is a Borel measure. Since $\scrV$ is locally finite, for every compact subset $C$ of $\mms$, only finitely many $V\in \scrV$ obey $\mu_V[C] > 0$. As in the proof of \cref{Le:Disderd}, it thus suffices to show $\mu_V$ is finite. This follows from local boundedness of $r\coth r$ in $r\in[0,\infty)$ and the boundedness of $\vert l_\Sigma\vert$ on $\cl\,V$ implied by \cref{Cor:Steepness} and \cref{Le:Bounded}. 

The $\meas$-absolute continuity of $\mu_V$ follows from the disintegration \cref{Th:Disintegration}.
\end{proof}

\begin{remark}[Simplification]\label{Re:Simpli const} If $k$  is bounded from below by a constant $K$ of arbitrary sign throughout $\smash{E_\Sigma}$, one can follow the lines of  Cavalletti--Mondino \cite{cavalletti-mondino2020-new}*{Lem.~4.13} and replace $\mu$ by the more explicit
\begin{align*}
\mu' &:= (N-1)\int_Q \vert l_\Sigma\vert \,\frac{\COS_{K/(N-1)}\circ \,l(\cdot,b_\alpha)}{\SIN_{K/(N-1)}\circ\,l(\cdot,b_\alpha)}\,\meas_\alpha\mres I^+(\Sigma)\d\q(\alpha)\\
&\qquad\qquad + (N-1)\int_Q \vert l_\Sigma\vert \,\frac{\COS_{K/(N-1)}\circ \,l(a_\alpha,\cdot)}{\SIN_{K/(N-1)}\circ\,l(a_\alpha,\cdot)}\,\meas_\alpha\mres I^-(\Sigma)\d\q(\alpha)
\end{align*}
in all subsequent considerations. The right-hand side has to be appropriately interpreted if either endpoint does not exist, cf.~\cref{Re:Unbounded II}.\hfill{\footnotesize{$\blacksquare$}}
\end{remark}

The following condition will be  useful to link suitable finite  perturbations of $\smash{l_\Sigma}$ to the topology of $\mms$ and thence to turn distributions into signed Radon measures. Its proof is deferred to \cref{App:A}. A concrete setting where the hypothesis holds is described in \cref{Sub:LorentzFinsler}.

Here and in the sequel, we consider the classes
\begin{align}\label{Eq:scrC}
\scrC^\pm_\Sigma := \{f \in \FPert_\comp(l_\Sigma,I^\pm(\Sigma)) : f\textnormal{ symmetric}\} \cap \Cont(I^\pm(\Sigma))
\end{align}
as well as their union
\begin{align*}
\scrC_\Sigma := \scrC^+_\Sigma\cup \scrC^-_\Sigma.
\end{align*}

\begin{lemma}[Finite perturbations vs.~topology]\label{Le:VsTop}  Assume $\smash{l_\Sigma}$ is topologically locally anti-Lipschitz on $\smash{I^\pm(\Sigma)}$ after \cref{Def:KS}. Then $\smash{\scrC_\Sigma^\pm}$ is uniformly dense in $\smash{\Cont_\comp(I^\pm(\Sigma))}$.
\end{lemma}

\begin{theorem}[Measure-valued d'Alembertian I, see also \cref{Cor:22}]\label{Th:11} Let $\scrM$ be  a timelike essentially nonbranching $\smash{\TMCP^e(k,N)}$ metric measure spacetime. Moreover, let $\q$ be an $\MCP(k,N)$ disintegration relative to $\smash{l_\Sigma}$ and $E_\Sigma$ and let the exponent $\PP$ be negative. Then $\smash{\u_\QQ \in \Dom(\BOX_\PP,I^+(\Sigma)\cup I^-(\Sigma))}$.

More precisely, the map $\smash{T^\QQ}$ on $\smash{\Pert_\bc(\u_\QQ,I^+(\Sigma)\cup I^-(\Sigma))}$ defined through
\begin{align*}
T^\QQ(f) &:= \int_Q\int_{\mms_\alpha} f\sgn l_\Sigma\,\big[1+l_\Sigma\,(\log h_\alpha)' \big]\d\meas_\alpha\d\q(\alpha) \\
&\qquad\qquad - \int_Q \big[\!\RCR f\,\vert l_\Sigma\vert\,h_\alpha\big]_{a_\alpha}^{b_\alpha}\d\q(\alpha)
\end{align*}
constitutes an element of $\smash{\BOX_\PP \u_\QQ \mres [I^+(\Sigma)\cup I^-(\Sigma)]}$.

Lastly, assume $\smash{l_\Sigma}$ is topologically locally anti-Lipschitz  on $\smash{I^+(\Sigma)\cup I^-(\Sigma)}$.  Then the following statements hold.
\begin{enumerate}[label=\textnormal{(\roman*)}]
\item \textnormal{\textbf{Enhanced distribution.}} The map $\smash{T^\QQ}$  extends continuously to $\smash{\Cont_\comp(I^+(\Sigma)\cup I^-(\Sigma))}$ and this extension is the difference
\begin{align*}
T^\QQ=\sgn l_\Sigma\,\meas +\sgn l_\Sigma\,  \mu  -\, ^+\nu +\, ^-\nu,
\end{align*}
i.e.~a generalized signed Radon measure on $\smash{I^+(\Sigma)\cup I^-(\Sigma)}$; here $\mu$ comes from \cref{Le:nu measure} and $\smash{^\pm\nu}$ are Radon measures on $\smash{I^\pm(\Sigma)}$.
\item \textnormal{\textbf{Two-sided integration by parts.}} Every $\smash{f\in\scrC_\Sigma}$ satisfies
\begin{align*}
\int_\mms \rmd^+f(\nabla\u_\QQ)\,\vert\rmd\u_\QQ\vert^{\PP-2}\d\meas \leq -T^\QQ(f)\leq \int_\mms\rmd^-f(\nabla\u_\QQ)\,\vert\rmd\u_\QQ\vert^{\PP-2}\d\meas.
\end{align*}
\end{enumerate}
\end{theorem}

\begin{proof} By \cref{Le:Radon funct II}, we already know $\smash{T^\QQ}$ constitutes a weak Radon functional.

Next, we establish $\smash{T^\QQ\in \BOX_\PP\u_\QQ\mres [I^+(\Sigma)\cup I^-(\Sigma)]}$. Thanks to \cref{Le:Invar p}, any given $\smash{f\in\Pert_\bc(\u_\QQ,I^+(\Sigma)\cup I^-(\Sigma))}$ belongs to $\smash{\Pert_\bc(l_\Sigma,I^+(\Sigma)\cup I^-(\Sigma))}$ and vice versa. Then for $\q$-a.e.~$\alpha\in Q$, the chain rule and the definition \eqref{Eq:dnabla}  yield
\begin{align*}
\rmd^+f(\nabla\u_\QQ)\,\vert\rmd \u_\QQ\vert^{\PP-2} &= \vert l_\Sigma\vert^{(\QQ-1)(\PP-1)}\,\rmd^+ f(\nabla l_\Sigma)\\ 
&= \vert l_\Sigma\vert \,\rmd^+f(\nabla l_\Sigma)\quad\meas_\alpha \mres [I^+(\Sigma)\cup I^-(\Sigma)]\textnormal{-a.e.}
\end{align*}
Here we have tacitly used that $\vert l_\Sigma\vert$ is bounded away from zero on the support of $f$ plus locality. In turn, \cref{Th:df vs f'} and the disintegration \cref{Th:Disintegration} imply
\begin{align*}
\int_{\mms_\alpha} \rmd^+f(\nabla\u_\QQ)\,\vert \rmd \u_\QQ\vert^{\PP-2}\d\meas_\alpha  \leq \int_{\mms_\alpha}f'\,\vert l_\Sigma\vert \d\meas_\alpha.
\end{align*}
Since $\ray$ directs the rays induced by $\q$ relative to $l$ and since $\smash{l_\Sigma}$ is affine along them, we get $\smash{\vert l_\Sigma\vert' = \sgn l_\Sigma}$ $\smash{\Leb_\alpha^1 \mres \supp f}$-a.e. An evident variant of \cref{Le:IBP II} entails
\begin{align*}
\int_{\mms_\alpha} f'\,\vert l_\Sigma\vert\d\meas_\alpha &\leq \int_{\mms_\alpha} \vert l_\Sigma\vert\,h_\alpha\d\Diff_\alpha\RCR f\\
&= -\int_{\mms_\alpha} f\,\big[\vert l_\Sigma\vert\,h_\alpha\big]'\d\Leb_\alpha^1 + \big[\!\RCR f\,\vert l_\Sigma\vert\,h_\alpha\big]_{a_\alpha}^{b_\alpha}\\
&= -\int_{\mms_\alpha} f\sgn l_\Sigma\d\meas_\alpha - \int_{\mms_\alpha} f\,h_\alpha'\,\vert l_\Sigma\vert\d\Leb_\alpha^1 + \big[\!\RCR f\,\vert l_\Sigma\vert\,h_\alpha\big]_{a_\alpha}^{b_\alpha}\\
&= -\int_{\mms_\alpha} f\sgn l_\Sigma\,\big[1 + l_\Sigma\,(\log h_\alpha)'\big]\d\meas_\alpha + \big[\!\RCR f\,\vert l_\Sigma\vert\,h_\alpha\big]_{a_\alpha}^{b_\alpha}.
\end{align*}
Integrating these estimates with respect to $\q$ gives the claim.

Finally, we verify $\smash{T^\QQ}$ is a generalized signed Radon measure on $\smash{I^+(\Sigma)\cup I^-(\Sigma)}$ assuming $\smash{l_\Sigma}$ obeys the described topological local anti-Lipschitz property. Since $\smash{I^+(\Sigma)}$ and $\smash{I^-(\Sigma)}$ are disjoint, it suffices to find representations of $\smash{T^\QQ}$ on each of these separately and then paste together the signed Radon measures.  We first argue for $\smash{I^+(\Sigma)}$. Let $\rho$ designate  the $\meas$-density of the measure $\mu$ from \cref{Le:nu measure}.   Let $\smash{x\in \mms_\alpha \cap I^+(\Sigma)}$, where $\alpha\in Q$, and let $V\in\scrV$ with $x\in V$. Then the \emph{lower} bound on the logarithmic derivative in  \cref{Le:logarithmic derivative} (choosing the footpoint of $x$ in $\Sigma$ as ``final point'') and \cref{Re:Props dist coeff} imply at  $x$ that
\begin{align}\label{Eq:BLBL}
l_\Sigma\,(\log h_\alpha)' \leq \sqrt{-K_V(N-1)}\,\vert l_\Sigma\vert\,\frac{\cosh(\sqrt{-K_V/(N-1)}\,\vert l_\Sigma\vert)}{\sinh(\sqrt{-K_V/(N-1)}\,\vert l_\Sigma\vert)} \leq \rho.
\end{align}
Thus, for every nonnegative $f\in\Pert_\bc(l_\Sigma,I^+(\Sigma))$,
\begin{align*}
T^\QQ(f) &=\int_Q\int_{\mms_\alpha} f\,\big[1+l_\Sigma\,(\log h_\alpha)'\big]\d\meas_\alpha\d\q(\alpha) - \int_Q \big[\!\RCR f\,\vert l_\Sigma\vert\, h_\alpha\big]^{b_\alpha}\d\q(\alpha)\\
&\leq \int_Q\int_{\mms_\alpha} f\,\big[1+l_\Sigma\,(\log h_\alpha)'\big]\d\meas_\alpha\d\q(\alpha)\\
&\leq \int_\mms f\d\meas \mres I^+(\Sigma) + \int_\mms f\d\mu\mres I^+(\Sigma).
\end{align*}
by the disintegration \cref{Th:Disintegration}. Consequently, the functional 
\begin{align*}
S :=  \meas\mres I^+(\Sigma) + \mu \mres I^+(\Sigma)-T^q
\end{align*}
is weak Radon and nonnegative. By the weak Radon property and  \cref{Le:VsTop}, $S$ extends continuously to a nonnegative (and bounded) functional on $\smash{\Cont_\comp(I^+(\Sigma))}$. By the Riesz--Markov--Kakutani representation theorem, $S$ is thus represented by a unique Radon measure $\smash{^+\nu}$ on $\smash{I^+(\Sigma)}$. In turn, for every appropriate $\smash{f\in \scrC_\Sigma^+}$ we have
\begin{align*}
S(f) = \int_{I^+(\Sigma)} f\d{^+\nu}.
\end{align*}

To cover the portion $\smash{I^-(\Sigma)}$, minor adaptations to the previous arguments are needed to remove the singular part due to the different sign of the signed Lorentz distance function. Again let $\smash{x\in\mms_\alpha\cap I^-(\Sigma)}$, where $\alpha\in Q$, and let $V\in \scrV$ with $x\in V$. The \emph{lower} bound on the logarithmic derivative in \cref{Le:logarithmic derivative} (choosing the footpoint of $x$ in $\Sigma$ as ``final point'') and \cref{Re:Props dist coeff} yield the following estimates at $x$ (noting $-l_\Sigma\geq 0$ there):
\begin{align*}
-l_\Sigma\,(\log h_\alpha)' \geq -\sqrt{-K_V(N-1)}\,\vert l_\Sigma\vert\,\frac{\cosh(\sqrt{-K_V/(N-1)}\,\vert l_\Sigma\vert)}{\sinh(\sqrt{-K_V/(N-1)}\,\vert l_\Sigma\vert)} \geq -\rho.
\end{align*}
As above, every nonnegative $\smash{f\in\scrC_\Sigma^-}$ thus satisfies
\begin{align*}
T^\QQ(f) \geq -\int_\mms f\d\meas\mres I^-(\Sigma)  - \int_\mms f\d\mu\mres I^-(\Sigma).
\end{align*}
The same argument as above yields the claimed Radon measure $\smash{^-\nu}$ on $I^-(\Sigma)$.

The last two-sided bounds follow from \cref{Re:Two sided}, recalling that $\smash{\scrC_\Sigma}$ entirely consists of symmetric finite  perturbations of $\smash{l_\Sigma}$.
\end{proof}

Analogously to \eqref{Eq:Tll}, on $\mms_\alpha$ with $\alpha\in Q$ we define
\begin{align*}
\Box_\PP\u_\QQ := \sgn l_\Sigma\,\big[1+l_\Sigma\,(\log h_\alpha)'\big].
\end{align*}

By sharpening the bound \eqref{Eq:BLBL} using \cref{Le:logarithmic derivative}, the following corollary is immediate from the previous proof; see also \cref{Re:Measvers} for a version including the singular parts. 

\begin{corollary}[Sharp d'Alembert comparison II, see also  \cref{Cor:Sharp comparison,Cor:1515}]\label{Cor:11} In the framework of \cref{Th:11}, the following hold for $\q$-a.e. $\alpha\in Q$. 
\begin{enumerate}[label=\textnormal{(\roman*)}]
\item \textnormal{\textbf{Variable version.}} Let $v,w\in \cl\,\mms_\alpha$ be distinct points such that $\smash{v\preceq_{l_\Sigma} w}$. Then $\smash{\meas_\alpha}$-a.e.~$x\in[v,w]$ satisfies
\begin{align*}
&\sgn l_\Sigma(x) - (N-1)\,\vert l_\Sigma\vert(x)\,\frac{\COS_{(k_\alpha)^-_{x,w}/(N-1)}\circ\, l(x,w)}{\SIN_{(k_\alpha)^-_{x,w}/(N-1)}\circ\, l(x,w)}\\
&\qquad\qquad \leq \Box_\PP\u_\QQ(x) \leq \sgn l_\Sigma(x) + (N-1)\,\vert l_\Sigma\vert(x)\,\frac{\COS_{(k_\alpha)^+_{v,x}/(N-1)}\circ\, l(v,x)}{\SIN_{(k_\alpha)^+_{v,x}/(N-1)}\circ\, l(v,x)}.
\end{align*}
\item \textnormal{\textbf{SEC version.}} If $k$ vanishes identically,
\begin{align*}
&\sgn l_\Sigma - (N-1)\,\vert l_\Sigma\vert\,\frac{1}{l(\cdot,b_\alpha)}\\
&\qquad\qquad \leq \Box_\PP\u_\QQ \leq \sgn l_\Sigma + (N-1)\,\vert l_\Sigma\vert\,\frac{1}{l(a_\alpha,\cdot)}\quad\meas_\alpha\textnormal{-a.e.}
\end{align*}
with appropriate interpretation if either endpoints do not exist.
\end{enumerate}
\end{corollary}

Again, this result does include the endpoints once they exist. If either of them does not exist yet $k$ is constant, these estimates can be extended to unbounded rays as described in \cref{Re:Unbounded II}. The same applies to \cref{Cor:1515}  below.

We are now in a position to improve \cref{Th:Meas val Alem I} in our current setting.

\begin{corollary}[Measure-valued d'Alembertian II, see also \cref{Th:11}]\label{Cor:22} Let $\scrM$ be a timelike essentially nonbranching $\smash{\TMCP^e(k,N)}$ metric measure spacetime. Let $\q$ be an $\MCP(k,N)$ disintegration relative to $l_\Sigma$ and $E_\Sigma$. Then $\smash{l_\Sigma\in \Dom(\BOX, I^+(\Sigma)\cup I^-(\Sigma))}$.

More precisely, the map $T$ on $\smash{\Pert_\bc(l_\Sigma,I^+(\Sigma)\cup I^-(\Sigma))}$ defined by
\begin{align*}
T(f) := \int_Q\int_{\mms_\alpha} f\,(\log h_\alpha)'\d\meas_\alpha\d\q(\alpha) - \int_Q \big[\!\RCR f\,h_\alpha\big]_{a_\alpha}^{b_\alpha}\d\q(\alpha)
\end{align*}
defines a weak Radon functional belonging to $\BOX l_\Sigma\mres [I^+(\Sigma)\cup I^-(\Sigma)]$.

Lastly, assume $\smash{l_\Sigma}$ is topologically locally anti-Lipschitz on $\smash{I^+(\Sigma)\cup I^-(\Sigma)}$. Then the following statements hold.
\begin{enumerate}[label=\textnormal{(\roman*)}]
\item \textnormal{\textbf{Enhanced distribution.}} The map  $T$ extends continuously to $\smash{\Cont_\comp(I^+(\Sigma)\cup I^-(\Sigma))}$ and this extension is the difference
\begin{align*}
T = l_\Sigma^{-1}\,\mu - \vert l_\Sigma\vert^{-1}\,(^+\nu -{} ^-\nu),
\end{align*}
i.e.~a generalized signed Radon measure on $\smash{I^+(\Sigma)\cup I^-(\Sigma)}$; here $\mu$ and $\smash{^\pm\nu}$ come from \cref{Le:nu measure} and \cref{Th:11}, respectively.
\item \textnormal{\textbf{Two-sided integration by parts.}} Every $\smash{f\in\scrC_\Sigma}$ satisfies
\begin{align*}
\int_\mms \rmd^+f(\nabla l_\Sigma)\d\meas \leq -T(f)\leq \int_\mms \rmd^-f(\nabla l_\Sigma)\d\meas.
\end{align*}
\end{enumerate}
\end{corollary}

\begin{proof} Throughout the subsequent proof, let us fix exponents $\PP$ and $\QQ$ as in \cref{Th:11}, although the final statement is of course independent of their choice (and notably holds for all exponents $\PP$ of arbitrary sign) by the definition  \eqref{Eq:dnabla}.

We intend to use the chain rule in  \cref{Pr:Chain rule}. To this aim,  consider the strictly increasing inverse $\smash{\psi_\QQ}$ of the function $\varphi_\QQ$ from \eqref{Eq:varphip} defined by 
\begin{align*}
\psi_\QQ(r) := \QQ^{1/\QQ}\sgn r\,\vert r\vert^{1/\QQ}
\end{align*}
In particular, $\psi_\QQ \circ u_\QQ = l_\Sigma$ on $E_\Sigma$. Moreover, $\psi_\QQ$ is smooth on $\R\setminus\{0\}$, where it obeys
\begin{align*}
\psi_\QQ'(r) &= \QQ^{-1/\PP}\,\vert r\vert^{-1/\PP},\\
\psi_\QQ''(r) &= -\QQ^{-1/\PP}\,\PP^{-1}\sgn r\,\vert r\vert^{-1/\PP-1}.
\end{align*}
By \cref{Pr:Chain rule} and \cref{Th:11}, we already get $\smash{l_\Sigma \in \Dom(\BOX_\PP\mres [I^+(\Sigma)\cup I^-(\Sigma)])}$.

To show the claimed formula, let $\smash{T^\QQ}$ be the weak Radon functional from \cref{Th:11}. By the chain rule and \cref{Cor:Constant slope}, we compute
\begin{align}\label{Eq:dupdupdupdududupdupdup}
\vert \rmd \u_\QQ\vert^\PP = \vert l_\Sigma\vert^{(\QQ-1)\PP}\,\vert\rmd l_\Sigma\vert^\PP = \vert l_\Sigma\vert^\QQ\quad\meas\mres [I^+(\Sigma)\cup I^-(\Sigma)]\textnormal{-a.e.}
\end{align}
In particular, $\smash{\vert \rmd \u_\QQ\vert^\PP}$ is locally $\smash{\meas\mres [I^+(\Sigma)\cup I^-(\Sigma)]}$-integrable. By \cref{Pr:Chain rule} and \cref{Le:Invar p}, the following functional is well-defined on $\smash{\Pert_\bc(l_\Sigma,I^+(\Sigma)\cup I^-(\Sigma))}$ and is an element of $\smash{\BOX_\PP \mres [I^+(\Sigma)\cup I^-(\Sigma)]}$:
\begin{align*}
S(f) &:= \int_Q\int_{\mms_\alpha} f\,(\psi_\QQ')^{\PP-1}\circ u_\QQ\sgn l_\Sigma\,\big[1 + l_\Sigma\,(\log h_\alpha)'\big]\d\meas_\alpha\d\q(\alpha)\\
&\qquad\qquad + \int_Q \big[\!\RCR f\,(\psi_\QQ')^{\PP-1}\circ u_\QQ\,\vert l_\Sigma\vert\,h_\alpha\big]_{a_\alpha}^{b_\alpha}\d\q(\alpha)\\
&\qquad\qquad + (\PP-1)\int_\mms f\,\big[(\psi_\QQ')^{\PP-2}\,\psi_\QQ''\big] \circ\u_\QQ \,\vert\rmd\u_\QQ\vert^\PP \d\meas.
\end{align*}
To see the well-definedness of all integrals, simply note that $\vert l_\Sigma\vert$ is bounded and bounded  away from zero on the support of $f$ by \cref{Le:Bounded} and  lower semicontinuity.

By the disintegration \cref{Th:Disintegration}, since
\begin{align*}
(\psi_\QQ')^{\PP-1}\circ \u_\QQ =  \vert l_\Sigma\vert^{-1} = \sgn l_\Sigma\,(l_\Sigma)^{-1}
\end{align*}
on the support of $f$,  the above expression simplifies to
\begin{align*}
S(f) &= \int_\mms f\,l_\Sigma^{-1}\d\meas  + \int_Q\int_{\mms_\alpha} f\,(\log h_\alpha)'\d\meas_\alpha\d\q(\alpha)  - \int_Q \big[\!\RCR f\,h_\alpha\big]_{a_\alpha}^{b_\alpha}\d\q(\alpha)\\
&\qquad\qquad + (\PP-1)\int_\mms f\,\big[(\psi_\QQ')^{\PP-2}\,\psi_\QQ''\big] \circ\u_\QQ \,\vert\rmd\u_\QQ\vert^\PP \d\meas \mres[I^+(\Sigma)\cup I^-(\Sigma)].
\end{align*}

Finally, we claim the first integral above cancels out with the last. On the support of $f$, 
\begin{align*}
(\psi_\QQ')^{\PP-2}\circ \u_\QQ &= \vert l_\Sigma\vert^{\QQ-2},\\
\psi_\QQ'' \circ\u_\QQ &= -\PP^{-1}\,\QQ\sgn l_\Sigma\,\vert l_\Sigma\vert^{1-2\QQ},
\end{align*}
which implies
\begin{align*}
(\PP-1)\,\big[(\psi_\QQ')^{\PP-2}\,\psi_\QQ''\big]\circ\u_\QQ &= -\sgn l_\Sigma\,\vert l_\Sigma\vert^{-1-\QQ} = -(l_\Sigma)^{-1}\,\vert l_\Sigma\vert^{-\QQ}
\end{align*}
on the same set. Combining this with \eqref{Eq:dupdupdupdududupdupdup}, we arrive at
\begin{align*}
S(f) &= \int_Q\int_{\mms_\alpha} f\,(\log h_\alpha)'
\d\meas_\alpha\d\q(\alpha) - \int_Q\big[\!\RCR f\,h_\alpha\big]_{a_\alpha}^{b_\alpha}\d\q(\alpha)
\end{align*}
and consequently $\smash{S = T}$, which terminates the proof of the second claim.

The last statement is straightforward by \cref{Th:11}. Indeed, via \cref{Pr:Relations} and locality we get $\smash{f\,\vert l_\Sigma\vert^{-1}\in \Pert_\bc(l_\Sigma,I^+(\Sigma)\cup I^-(\Sigma))}$. The trivial identity
\begin{align*}
T(f) = T^\QQ(f\,\vert l_\Sigma\vert^{-1}) - \int_\mms f\,\vert l_\Sigma\vert^{-1}\sgn l_\Sigma\d\meas
\end{align*}
implied by the disintegration \cref{Th:Disintegration} gives the claim.
\end{proof}

The following extends \cref{Cor:Sharp comparison} beyond any integrability assumption.

\begin{corollary}[Sharp d'Alembert comparison III, see also  \cref{Cor:Sharp comparison,Cor:11}]\label{Cor:1515} In the framework from the first part of \cref{Cor:22}, the following hold for $\q$-a.e.~$\alpha\in Q$. 
\begin{enumerate}[label=\textnormal{(\roman*)}]
\item \textnormal{\textbf{Variable version.}} Let $v,w\in \cl\,\mms_\alpha$ be distinct points such that $\smash{w\preceq_{l_\Sigma} v}$. Then $\smash{\meas_\alpha}$-a.e.~$x\in [v,w]$ satisfies
\begin{align*}
&- (N-1) \,\frac{\COS_{(k_\alpha)_{x,w}^-/(N-1)}\circ\,l(x,w)}{\SIN_{(k_\alpha)_{x,w}^-/(N-1)}\circ\,l(x,w)}\\
&\qquad\qquad \leq \Box l_\Sigma(x) \leq (N-1)\,\frac{\COS_{(k_\alpha)^+_{v,x}/(N-1)}\circ\, l(v,x)}{\SIN_{(k_\alpha)^-_{v,x}/(N-1)}\circ\, l(v,x)}.
\end{align*}
\item \textnormal{\textbf{SEC version.}} If $k$ vanishes identically,
\begin{align*}
- (N-1) \,\frac{1}{l(\cdot,b_\alpha)}  \leq \Box l_\Sigma \leq (N-1)\,\frac{1}{l(a_\alpha,\cdot)}\quad\meas_\alpha\textnormal{-a.e.}
\end{align*}
with appropriate interpretation if either endpoints do not exist.
\end{enumerate}
\end{corollary}

\begin{remark}[Inclusion of singular parts]\label{Re:Measvers} The upper bound from (say, the variable version of) \cref{Cor:1515} actually extends across the future timelike cut locus of $\Sigma$. Indeed, let $w\colon Q \to \Tr\setminus I^+(\Sigma)$ be a given $\q$-measurable map\footnote{One can e.g.~take $w_\alpha$ as the unique intersection of $\cl\,\mms_\alpha$ with $\Sigma$ for every $\alpha\in Q$.}. Then we have
\begin{align*}
T 
&\leq (N-1)\int_Q \frac{\COS_{(k_\alpha)^+_{v_\alpha,\cdot}/(N-1)}\circ\, l(v_\alpha,\cdot)}{\SIN_{(k_\alpha)^+_{v_\alpha,\cdot}/(N-1)}\circ\, l(v_\alpha,\cdot)}\,\meas_\alpha\d\q(\alpha),\\
T^\QQ 
&\leq \meas + (N-1)\int_Q l_\Sigma\, \frac{\COS_{(k_\alpha)^+_{v_\alpha,\cdot}/(N-1)}\circ\, l(v_\alpha,\cdot)}{\SIN_{(k_\alpha)^+_{v_\alpha,\cdot}/(N-1)}\circ\, l(v_\alpha,\cdot)}\,\meas_\alpha\d\q(\alpha)
\end{align*}
in the sense of weak Radon functionals. Here the maps $T$ and $\smash{T^\QQ}$ are from \cref{Cor:22} and \cref{Th:11}, respectively.

Under topological anti-Lipschitzness of $l_\Sigma$ on $\smash{I^+(\Sigma)}$, either bound holds when replacing finite perturbations by characteristic functions of any compact subset of $I^+(\Sigma)$.

Analogous statements hold for the lower bounds relative to $\smash{I^-(\Sigma)}$.\hfill{\footnotesize{$\blacksquare$}}
\end{remark}

\begin{remark}[Sharpness] The upper and the lower bound of \cref{Cor:1515} are sharp. This follows from Treude--Grant's mean curvature comparison theorem \cite{treude-grant2013}*{Thm.~4.5}, see also Treude \cite{treude2011}*{Thm.~3.3.5}. By the chain rule,  \cref{Cor:11} is thus sharp.\hfill{\footnotesize{$\blacksquare$}}
\end{remark}

\begin{remark}[Extension to positive $\PP$]\label{Re:Ext pos q} If we replace $E_\Sigma$ by $I^+(\Sigma)\cup \Sigma$ (and accordingly, the endpoints $b_\alpha$ from the induced disintegration all lie in $\Sigma$), the hypothesis on the sign of $\PP$ from \cref{Th:11} can actually be dropped by \cref{Cor:22} and the chain rule from \cref{Pr:Chain rule}. The representation formula for $\smash{T^\QQ}$  remains the same.\hfill{\footnotesize{$\blacksquare$}}
\end{remark}

\subsubsection{Reference sets $\Sigma$ with cardinality one}\label{Sub:Point} Now we simplify the statements from the previous discussion in the case where $\Sigma$ merely consists of one given point $o\in\mms$. Similarly to \cref{Re:Ext pos q}, we also replace $E_o$ by $\smash{I^+(o)\cup \{o\}}$. In turn, under this restriction it suffices to assume $\Sigma$ is FTC according to \cref{Def:timelike complete}.

Analogously to above, for any nonzero number $\QQ$ less than $1$  we define
\begin{align*}
\v_\QQ := \QQ^{-1}\,l_o^\QQ.
\end{align*}
The relevant counterpart of \eqref{Eq:scrC} is
\begin{align*}
\scrC^+_o := \{f\in \FPert_\comp(l_o, I^+(o)) : f\textnormal{ symmetric}\}\cap \Cont(I^+(o)).
\end{align*}

\begin{corollary}[Measure-valued d'Alembertian for a singleton I, see also \cref{Cor:44}]\label{Cor:33} Let $\scrM$ be a timelike $\TMCP^e(k,N)$ metric measure spacetime. Furthermore, let $\q$ be an $\MCP(k,N)$ disintegration relative to $l_o$ and $\smash{I^+(o)\cup \{o\}}$. Then for every nonzero number $\PP$ less than $1$, we have $\smash{\v_\QQ\in \Dom(\BOX_\PP,I^+(o))}$.

More precisely, the map $\smash{T^\QQ}$ on $\smash{\Pert_\bc(\v_\QQ, I^+(o))}$ defined by
\begin{align*}
T^\QQ(f) := \int_Q\int_{\mms_\alpha} f\,\big[1+l_o\,(\log h_\alpha)'\big]\d\meas_\alpha\d\q(\alpha) - \int_Q\big[\!\RCR f\,l_o\,h_\alpha\big]^{b_\alpha} \d\q(\alpha)
\end{align*}
is an element of $\smash{\BOX_\PP\v_\QQ\mres I^+(o)}$.

Lastly, assume $\smash{l_o}$ is topologically locally anti-Lipschitz on $I^+(o)$. Then the following statements hold.
\begin{enumerate}[label=\textnormal{(\roman*)}]
\item \textnormal{\textbf{Enhanced distribution.}} The map $T^\QQ$ extends continuously to $\smash{\Cont_\comp(I^+(o))}$ and this extension is the difference
\begin{align*}
T^\QQ = \meas + \mu - \nu,
\end{align*}
i.e.~a generalized signed Radon measure on $I^+(o)$; here $\mu$ comes from \cref{Le:nu measure} and $\nu$ is a Radon measure on $\smash{I^+(o)}$.
\item \textnormal{\textbf{Two-sided integration by parts.}} Every $\smash{f\in \scrC_o^+}$ satisfies
\begin{align*}
\int_\mms \rmd^+f(\nabla \v_\QQ)\d\meas \leq -T(f)\leq \int_\mms \rmd^-f(\nabla \v_\QQ)\d\meas.
\end{align*}
\end{enumerate}
\end{corollary}

\begin{corollary}[Measure-valued d'Alembertian for a singleton II, see also \cref{Cor:33}]\label{Cor:44} Let $\scrM$ be a timelike $\TMCP^e(k,N)$ metric measure spacetime. Let $\q$ form  an $\MCP(k,N)$ disintegration relative to $l_o$ and $\smash{I^+(o)\cup \{o\}}$. Then $\smash{l_o\in \Dom(\BOX,I^+(o))}$.

More precisely, the map $T$ defined on $\smash{\Pert_\bc(l_o,I^+(o))}$ defined by
\begin{align*}
T(f) := \int_Q\int_{\mms_\alpha} f\,(\log h_\alpha)'\d\meas_\alpha\d\q(\alpha) - \int_Q \big[\!\RCR f\,h_\alpha\big]^{b_\alpha}\d\q(\alpha)
\end{align*}
consitutes an element of $\BOX l_o\mres I^+(o)$.

Lastly, suppose $\smash{l_o}$ is topologically locally anti-Lipschitz on $I^+(o)$. Then the following statements hold.
\begin{enumerate}[label=\textnormal{(\roman*)}]
\item \textnormal{\textbf{Enhanced distribution.}} The map $T$ extends continuously to $\smash{\Cont_\comp(I^+(o))}$, and this extension is the difference
\begin{align*}
T^\QQ = l_o^{-1}\,(\mu - \nu),
\end{align*}
i.e.~a generalized signed Radon measure on $I^+(o)$; here $\mu$ comes from \cref{Le:nu measure} and $\nu$ is a Radon measure on $\smash{I^+(o)}$.
\item \textnormal{\textbf{Two-sided integration by parts.}} Every $\smash{f\in \scrC_o^+}$ satisfies
\begin{align*}
\int_\mms \rmd^+f(\nabla l_o)\,\vert\rmd l_o\vert^{\PP-2}\d\meas \leq -T^\QQ(f)\leq \int_\mms \rmd^-f(\nabla l_o)\,\vert\rmd l_o\vert^{\PP-2}\d\meas.
\end{align*}
\end{enumerate}
\end{corollary}

Note that the singular parts of the functionals in \cref{Cor:33,Cor:44}, which are ``concentrated'' on the future cut locus $\smash{\TCut^+(o)}$,  are \emph{nonpositive}.

To further simplify the statements about the accompanying comparison theorems, we directly assume $k$ vanishes identically. 

\begin{corollary}[Sharp d'Alembert comparison for a singleton I, see also \cref{Cor:M}]\label{Cor:N} In the framework from the first part of \cref{Cor:33}, we assume in addition $k$ vanishes identically. Then for $\q$-a.e.~$\alpha\in Q$,
\begin{align*}
&1 - (N-1)\,l_o\,\frac{1}{l(\cdot,b_\alpha)}  \leq \Box_\PP\v_\QQ \leq N\quad\meas_\alpha\textnormal{-a.e.}
\end{align*}
with appropriate interpretation if the final point does not exist.
\end{corollary}

\begin{corollary}[Sharp d'Alembert comparison for a singleton II, see also \cref{Cor:N}]\label{Cor:M} In the framework from the first part of \cref{Cor:33}, we assume in addition $k$ vanishes identically. Then for $\q$-a.e.~$\alpha\in Q$,
\begin{align*}
-(N-1)\,\frac{1}{l(\cdot,b_\alpha)} \leq \Box l_o \leq (N-1)\,\frac{1}{l_o}\quad\meas_\alpha\textnormal{-a.e.}
\end{align*}
with appropriate interpretation if the final point does not exist.
\end{corollary}

\begin{remark}[Eschenburg's d'Alembert comparison theorem]\label{Re:Eschen} The upper estimate from  \cref{Cor:M} extends a causally reversed version of Eschenburg's d'Alembert comparison theorem established in his smooth proof of the Lorentzian splitting theorem \cite{eschenburg1988}*{§5}.\hfill{\footnotesize{$\blacksquare$}}
\end{remark}

\begin{remark}[On the lower bound in \cref{Cor:M}] 
Now let us replace $K$ by $N-1$ for simplicity, although the following discussion has an adaptation to general $K$. For $\varepsilon > 0$, let $D_\varepsilon$ denote the set of all points of the form $\sfg_t(a_\alpha)$, where $\alpha\in Q$ and $t\geq \varepsilon$ (whenever defined). In other words, $D_\varepsilon$ is the set of all points in $\smash{I^+(o)}$ with distance at least $\varepsilon$ to the cut point of  the  ray they lie on. Since $\cot$ is strictly decreasing, \cref{Cor:M} implies
\begin{align*}
\Box l_o \geq -(N-1)\cot \varepsilon\quad\meas\mres D_\varepsilon\textnormal{-a.e.}
\end{align*}
The lower bound notably depends on the dimensional parameter $N$ (and implicitly on $K$) as well as the distance parameter $\varepsilon$, but not on the given structure $\scrM$.\hfill{\footnotesize{$\blacksquare$}}
\end{remark}

\subsubsection{Unsigned Lorentz distance functions}\label{Sub:Unsign time} For this final digression, let us switch back to the setting of \cref{Gen Sigm} where we took into account the chronological future \emph{and} past of $\Sigma$. Let $\PP$ and $\QQ$ be mutually conjugate nonzero numbers less than $1$. We then set
\begin{align*}
\w_\QQ := \QQ^{-1}\,\vert l_\Sigma\vert^\QQ.
\end{align*}
In the consideration of the power $\u_\QQ$ of $l_\Sigma$ in \cref{Th:11} et seq., a sign appeared from the retention of $l$-causality on $I^-(\Sigma)$. In comparison, the unsigned functions $\smash{\w_\QQ}$ and $\vert l_\Sigma\vert$ are $l$-causal on $\smash{I^+(\Sigma)\cup\Sigma}$ and $\smash{l^\leftarrow}$-causal  on $\smash{I^-(\Sigma)\cup \Sigma}$. The goal of \cref{Th:UnsigndalemI} and \cref{Th:UnsigndalemII} is to make sense of the ($\PP$-)d'Alembertian of these two functions.

For a meaningful integration by parts formula, we  need to  disambiguate the potentially different Sobolev calculi produced by $\smash{\scrM}$ and $\smash{\scrM^{\leftarrow}}$ (recall \cref{Re:Behcaus}). To this aim, the condition from \cref{Def:Rev} is natural. It should be compared to reversibility of Finsler spacetimes, cf.~\cref{Sub:LorentzFinsler}. It is trivially satisfied on a genuine Lorentz spacetime.

In the sequel, to relax our notation we employ the following convention. Whenever a left arrow $\leftarrow$ appears in a quantity or formula, every adjacent object (e.g.~maximal weak subslopes) is automatically understood with respect to the structure  $\smash{\scrM^\leftarrow}$.

\begin{definition}[Reversibility]\label{Def:Rev} We call $\scrM$ \emph{reversible} if for every $l$-causal function $f$,
\begin{align*}
\vert\rmd f\vert = \vert\rmd f^\leftarrow\vert\quad\meas\textnormal{-a.e.}
\end{align*}
\end{definition}

\begin{remark}[Variant of \cref{Le:Diff quot}]\label{Re:Vat!} As for \cref{Le:Invar p}, the two disjoint unions
\begin{align*}
\Gamma &:= \Pert_\bc(l_\Sigma, I^+(\Sigma))\cup \Pert_\bc(l_\Sigma^\leftarrow,I^-(\Sigma))\\
\Gamma_\QQ &:= \Pert_\bc(\u_\QQ,I^+(\Sigma)) \cup \Pert_\bc(\u_\QQ^\leftarrow,I^-(\Sigma))
\end{align*}
coincide. Moreover, if $\scrM$ is reversible then for every $f\in \Gamma$,
\begin{alignat*}{3}
\rmd^+f(\nabla l_\Sigma^\leftarrow)\,\vert\rmd l_\Sigma^\leftarrow\vert^{p-2} &= \rmd^+(f^\leftarrow)(\nabla l_\Sigma)\,\vert\rmd l_\Sigma\vert^{\PP-2} &&\meas \mres I^-(\Sigma)\textnormal{-a.e.}\\
\rmd^+f(\nabla\u_\QQ^\leftarrow)\,\vert\rmd\u_\QQ^\leftarrow\vert^{\PP-2} &= \rmd^+(f^\leftarrow)(\nabla\u_\QQ)\,\vert\rmd \u_\QQ\vert^{\PP-2} &\quad &\meas\mres I^-(\Sigma)\textnormal{-a.e.}\tag*{{\footnotesize{$\blacksquare$}}}
\end{alignat*}
\end{remark}

Based on \cref{Re:Vat!}, \eqref{Eq:dnabla}, and locality, for $f\in \Gamma$ it makes sense to define
\begin{align*}
\rmd^+f(\nabla\vert l_\Sigma\vert) &:= \begin{cases} \rmd^+f(\nabla l_\Sigma) & \textnormal{if } f\in \Pert_\bc(l_\Sigma,I^+(\Sigma)),\\
\rmd^+f(\nabla l_\Sigma^\leftarrow) & \textnormal{if } f\in \Pert_\bc(l_\Sigma^\leftarrow,I^-(\Sigma)),
\end{cases}\\
\rmd^+f(\nabla \w_\QQ)\,\vert\rmd\w_\QQ\vert^{\PP-2} &:= \begin{cases} \rmd^+f(\nabla \u_\QQ)\,\vert\rmd \u_\QQ \vert^{\PP-2} & \textnormal{if } f\in \Pert_\bc(\u_\QQ,I^+(\Sigma)),\\
\rmd^+f(\nabla \u_\QQ^\leftarrow)\,\vert\rmd \u_\QQ^\leftarrow\vert^{\PP-2} & \textnormal{if } f\in \Pert_\bc(\u_\QQ^\leftarrow,I^-(\Sigma)).
\end{cases}
\end{align*}
We define $\smash{\rmd^-f(\nabla\vert l_\Sigma\vert)}$ and $\smash{\rmd^-f(\nabla \w_\QQ)\,\vert\rmd\w_\QQ\vert^{\PP-2}}$ analogously provided $\smash{f\in\Gamma\cap \scrC_\Sigma}$.

\begin{theorem}[Unsigned d'Alembertian I, see also \cref{Th:UnsigndalemII}]\label{Th:UnsigndalemI} Let $\scrM$ constitute a reversible timelike essentially nonbranching $\smash{\TMCP^e(k,N)}$ metric measure spacetime. Let $\q$ be an $\MCP(k,N)$ disintegration relative to $l_\Sigma$ and $E_\Sigma$ and let $\PP$ be a nonzero number less than $1$. Then the functional $\smash{R^\QQ}$ on $\smash{\Gamma_\QQ}$ defined by 
\begin{align*}
R^\QQ(f) &:= \int_Q\int_{\mms_\alpha} f\,\big[1+l_\Sigma\,(\log h_\alpha)'\big]\d\meas_\alpha\d\q(\alpha) - \int_Q\big[\!\RCR f\,l_\Sigma\,h_\alpha\big]^{b_\alpha}_{a_\alpha}\d\q(\alpha)
\end{align*}
belongs to $\smash{\BOX_\PP\w_\QQ\mres [I^+(\Sigma)\cup I^-(\Sigma)]}$ in sense that it forms a weak Radon functional on $\smash{\Gamma_\QQ}$ and every $\smash{f\in \Gamma_\QQ}$ satisfies the inequality
\begin{align*}
\int_\mms \rmd^+f(\nabla\w_\QQ)\,\vert\rmd\w_\QQ\vert^{\PP-2}\d\meas \leq -R^\QQ(f).
\end{align*}

Moreover, if $\smash{l_\Sigma}$ is topologically locally anti-Lipschitz on $\smash{I^+(\Sigma)\cup I^-(\Sigma)}$, the following statements hold.
\begin{enumerate}[label=\textnormal{(\roman*)}]
\item \textnormal{\textbf{Enhanced distribution.}} The map $R^\QQ$ extends continuously to $\smash{\Cont_\comp(I^+(\Sigma)\cup I^-(\Sigma))}$ and this extension is the difference
\begin{align*}
R^\QQ = \meas + \mu - \nu,
\end{align*} 
i.e.~a generalized signed Radon measure on $\smash{I^+(\Sigma)\cup I^-(\Sigma)}$; here $\mu$ comes from \cref{Le:nu measure} and $\nu$ is a Radon measure on $\smash{I^+(\Sigma)\cup I^-(\Sigma)}$.
\item \textnormal{\textbf{Two-sided integration by parts.}} Every $\smash{f\in\Gamma_\QQ\cap \scrC_\Sigma}$ satisfies
\begin{align*}
\int_\mms\rmd^+f(\nabla \w_\QQ)\,\vert\rmd\w_\QQ\vert^{\PP-2}\d\meas \leq -R^\QQ(f) \leq \int_\mms\rmd^-f(\nabla \w_\QQ)\,\vert\rmd\w_\QQ\vert^{\PP-2}\d\meas.
\end{align*}
\end{enumerate}
\end{theorem}

\begin{proof} When considered on $\smash{I^+(\Sigma)}$, the statement readily follows from \cref{Th:11} and \cref{Re:Ext pos q}. On $\smash{I^-(\Sigma)}$, the claim follows by considering the mentioned results by replacing the variable $f$ by $\smash{f^\leftarrow}$  (which makes the corresponding signed Lorentz distance function switch its sign compared to $l_\Sigma$). We leave out the details.
\end{proof}

\begin{corollary}[Sharp unsigned d'Alembert comparison I, see also \cref{Cor:UNS2}]\label{Cor:UNS1} In the framework from the first part of  \cref{Th:UnsigndalemI}, the following hold for $\q$-a.e. $\alpha\in Q$. 
\begin{enumerate}[label=\textnormal{(\roman*)}]
\item \textnormal{\textbf{Variable version.}} Let $v,w\in \cl\,\mms_\alpha$ be distinct points such that $\smash{v\preceq_{l_\Sigma} w}$. Then $\smash{\meas_\alpha}$-a.e.~$x\in[v,w]$ satisfies
\begin{align*}
&1 - (N-1)\, l_\Sigma(x)\,\frac{\COS_{(k_\alpha)^-_{x,w}/(N-1)}\circ\, l(x,w)}{\SIN_{(k_\alpha)^-_{x,w}/(N-1)}\circ\, l(x,w)}\\
&\qquad\qquad \leq \Box_\PP\w_\QQ(x) \leq 1 + (N-1)\,l_\Sigma(x)\,\frac{\COS_{(k_\alpha)^+_{v,x}/(N-1)}\circ\, l(v,x)}{\SIN_{(k_\alpha)^+_{v,x}/(N-1)}\circ\, l(v,x)}.
\end{align*}
\item \textnormal{\textbf{SEC version.}} If $k$ vanishes identically,
\begin{align*}
&1 - (N-1)\, l_\Sigma\,\frac{1}{l(\cdot,b_\alpha)}  \leq \Box_\PP\w_\QQ \leq 1 + (N-1)\,l_\Sigma\,\frac{1}{l(a_\alpha,\cdot)}\quad\meas_\alpha\textnormal{-a.e.}
\end{align*}
with appropriate interpretation if either endpoints do not exist.
\end{enumerate}
\end{corollary}

\begin{corollary}[Unsigned d'Alembertian II, see also \cref{Th:UnsigndalemI}]\label{Th:UnsigndalemII} Let $\scrM$ be a reversible timelike essentially nonbranching $\smash{\TMCP^e(k,N)}$ metric measure spacetime. Let $\q$ be  an $\MCP(k,N)$ disintegration induced by $l_\Sigma$ and $E_\Sigma$. Then the functional $R$ on $\Gamma$ defined by
\begin{align*}
R(f) &:= -\int_Q\int_{\mms_\alpha} f\,\sgn l_\Sigma\,(\log h_\alpha)'\d\meas_\alpha\d\q(\alpha) + \int_Q \big[\!\RCR f\,h_\alpha\big]^{b_\alpha}_{a_\alpha} \d\q(\alpha)
\end{align*}
lies in $\smash{\BOX_\PP\vert l_\Sigma\vert \mres [I^+(\Sigma)\cup I^-(\Sigma)]}$ in the sense that it is a weak Radon functional on $\Gamma$ and every $f\in \Gamma$ obeys the inequality
\begin{align*}
\int_\mms \rmd^+f(\nabla\vert l_\Sigma\vert)\d\meas \leq -R(f).
\end{align*}

In addition, if $l_\Sigma$ is topologically locally anti-Lipschitz on $\smash{I^+(\Sigma)\cup I^-(\Sigma)}$, the following statements hold.
\begin{enumerate}[label=\textnormal{(\roman*)}]
\item \textnormal{\textbf{Enhanced distribution.}} The map $R$ extends continuously to $\smash{\Cont_\comp(I^+(\Sigma)\cup I^-(\Sigma))}$ and this extension is the difference
\begin{align*}
R = \vert l_\Sigma\vert^{-1}\,(\mu - \nu),
\end{align*} 
i.e.~a generalized signed Radon measure on $\smash{I^+(\Sigma)\cup I^-(\Sigma)}$; here $\mu$ and $\nu$ are from \cref{Le:nu measure} and \cref{Th:UnsigndalemI}, respectively.
\item \textnormal{\textbf{Two-sided integration by parts.}} Every $\smash{f\in \Gamma\cap\scrC}$ satisfies
\begin{align*}
\int_\mms \rmd^+f(\nabla\vert l_\Sigma\vert)\d\meas \leq -R(f) \leq \int_\mms \rmd^-f(\nabla\vert l_\Sigma\vert)\d\meas.
\end{align*}
\end{enumerate}
\end{corollary}

\begin{corollary}[Sharp unsigned d'Alembert comparison II, see also \cref{Cor:UNS1}]\label{Cor:UNS2} In the framework from the first part of  \cref{Th:UnsigndalemII}, the following hold for $\q$-a.e. $\alpha\in Q$. 
\begin{enumerate}[label=\textnormal{(\roman*)}]
\item \textnormal{\textbf{Variable version.}} Let $v,w\in \cl\,\mms_\alpha$ be distinct points such that $\smash{v\preceq_{l_\Sigma} w}$. Then $\smash{\meas_\alpha}$-a.e.~$x\in[v,w]$ satisfies
\begin{align*}
&- (N-1)\, \sgn l_\Sigma(x)\,\frac{\COS_{(k_\alpha)^-_{x,w}/(N-1)}\circ\, l(x,w)}{\SIN_{(k_\alpha)^-_{x,w}/(N-1)}\circ\, l(x,w)}\\
&\qquad\qquad \leq \Box\vert l_\Sigma\vert(x) \leq (N-1)\,\sgn l_\Sigma(x)\,\frac{\COS_{(k_\alpha)^+_{v,x}/(N-1)}\circ\, l(v,x)}{\SIN_{(k_\alpha)^+_{v,x}/(N-1)}\circ\, l(v,x)}.
\end{align*}
\item \textnormal{\textbf{SEC version.}} If $k$ vanishes identically,
\begin{align*}
&- (N-1)\, \sgn l_\Sigma\,\frac{1}{l(\cdot,b_\alpha)}  \leq \Box\vert l_\Sigma\vert \leq (N-1)\,\sgn l_\Sigma\,\frac{1}{l(a_\alpha,\cdot)}\quad\meas_\alpha\textnormal{-a.e.}
\end{align*}
with appropriate interpretation if either endpoints do not exist.
\end{enumerate}
\end{corollary}

\section{Applications}\label{Ch:Appli}

We retain all hypotheses and the notation from the beginning of \cref{Ch:Repr form}. Moreover, unless explicitly stated otherwise $\Sigma$ is an achronal TC subset of $\mms$, in which case the $\MCP(k,N)$ disintegration $\q$ is understood relative to the set $E_\Sigma$ from \eqref{Eq:ESigmaUSigma}. We remark that whenever only the chronological future or past of $\Sigma$ is relevant (i.e.~$\q$ is understood with respect to $I^+(\Sigma)\cup\Sigma$ or $I^-(\Sigma)\cup\Sigma$), the TC hypothesis can be weakened to $\Sigma$ being FTC or PTC, respectively. Conversely, most results which take ``both sides'' of $\Sigma$ into account possess evident ``one-sided'' analogs.

By a slight abuse of notation, we denote the respective weak Radon functionals from \cref{Th:Meas val Alem I} and  \cref{Cor:22} by $\BOX \u$ and $\BOX l_\Sigma$. We write the restriction of the latter to $\smash{\Pert_\bc(l_\Sigma,I^\pm(\Sigma))}$ as  $\smash{\BOX l_\Sigma\mres I^\pm(\Sigma)}$. Moreover, we recall \eqref{Eq:Tll} for the definition
\begin{align*}
\Box\u := (\log h_\alpha)'
\end{align*}
on $\mms_\alpha$ with $\alpha\in Q$; analogously, $\Box l_\Sigma$ is defined.

\subsection{Bochner inequality}\label{Sub:Nonsmooth Bochner} In \cref{Th:From TCD to Bochner,Th:From Bochner to TCD}  we characterize the TCD condition for $\scrM$ in terms of a new  family of Bochner inequalities for Lorentz distance functions. The key technical ingredient is the characterization of  \cref{Le:Diffchar}. 

Inspired by Cavalletti--Mondino's synthetic approach in positive signature \cite{cavalletti-mondino2020}*{§6} and Galloway's use of Bochner's technique in his proof of the Lorentzian splitting theorem \cite{galloway1989-splitting}*{p.~383}, the idea is the   following. Let $\Sigma$ be a smooth, compact, achronal spacelike hyper\-surface in a given smooth, globally hyperbolic spacetime $\mms$ with nonnegative Ricci curvature in all timelike directions, cf.~\cref{Ex:TCD} above. As recapitulated from  Treude \cite{treude2011} in \cref{Th:Properties smooth dist fct} below, the signed Lorentz distance function $l_\Sigma$ is smooth on $E_\Sigma\setminus(\TCut^+(\Sigma)\cup\{\Sigma\}\cup\TCut^-(\Sigma))$. On its intersection $A_r$ with the level set $\{l_\Sigma = r\}$, where $r\in\R$ is such that $A_r$ is smooth, the usual Bochner identity reads
\begin{align*}
\Box\frac{\vert \rmd l_\Sigma\vert^2}{2} - \rmd \Box l_\Sigma(\nabla l_\Sigma)   = \Ric(\nabla l_\Sigma,\nabla l_\Sigma) + \big\vert\! \Hess l_\Sigma\big\vert_\HS^2.
\end{align*}
The first term vanishes since $\rmd l_o$ has unit magnitude. The third term is nonnegative by our curvature assumption and since $\nabla l_\Sigma$ is timelike. The second term is the derivative $(\Box l_\Sigma)'$ of $\Box l_\Sigma$ along the negative gradient flow of $l_\Sigma$, just as in \eqref{Eq:f' def}. Lastly, the Hessian of $l_\Sigma$ vanishes in the direction $\nabla l_\Sigma$ while it is equal to the submanifold Hessian $\smash{\Hess_{A_r} l_\Sigma}$ in all directions tangential to $\Sigma$, cf.~e.g.~Beem--Ehrlich--Easley \cite{beem-ehrlich-easley1996}*{Sublem.~14.33}. The usual Cauchy--Schwarz inequality estimates the Hilbert--Schmidt norm of $\smash{\Hess_{A_r} l_\Sigma}$ from below by $(\dim\mms-1)^{-1}(\Box l_\Sigma)^2$, recalling $\dim\mms-1=\dim A_r$ \cite{beem-ehrlich-easley1996}*{p.~584}. The  inequality 
\begin{align}\label{Eq:RICC in}
(\Box l_\Sigma)'\big\vert_{A_r} \geq \frac{1}{\dim\mms-1}\,(\Box l_\Sigma)^2\big\vert_{A_r}
\end{align}
is sometimes called ``Riccati inequality'' \cite{beem-ehrlich-easley1996}*{§B.3}.

As $(\Box l_\Sigma)'$ involves taking three derivatives of $l_\Sigma$, we integrate \eqref{Eq:RICC in} along the negative gradient flow trajectories of $l_\Sigma$. Thence we get a more robust candidate for our nonsmooth Bochner inequality which makes perfect sense in our synthetic setting.

\begin{theorem}[From TCD to Bochner inequality]\label{Th:From TCD to Bochner} Let $\scrM$ be a timelike essentially nonbranching $\smash{\TCD_\beta^e(k,N)}$ metric measure spacetime. Let $\q$ denote a $\CD(k,N)$ dis\-in\-te\-gration relative to $\smash{\u}$. Then the following hold.
\begin{enumerate}[label=\textnormal{\textcolor{black}{(}\roman*\textcolor{black}{)}}]
\item\label{La:UBoch} \textnormal{\textbf{1-steep functions.}} Assume the integrability condition
\begin{align}\label{Eq:INTEGR!cond}
\int_Q \frac{1}{l(a_\alpha,b_\alpha)}\d\q(\alpha) < \infty.
\end{align}
Then for $\q$-a.e.~$\alpha\in Q$, the quantities $\Box\u(x)$ and $\Box\u(y)$ are defined at co\-count\-ably many $\smash{(x,y)\in\mms_\alpha^2}$ and satisfy
\begin{align*}
\big[\Box\u\big]_x^y \geq \int_0^{l_x(y)} k\circ\sfg_r(x)\d r + \frac{1}{N-1}\int_0^{l_x(y)} (\Box \u)^2\circ\sfg_r(x)\d r.
\end{align*}
\item \textnormal{\textbf{Signed Lorentz distance functions.}} Let $l_\Sigma$ form a signed TC Lorentz distance function. For $\q$-a.e.~$\alpha\in Q$, the quantities $\Box l_\Sigma(x)$ and $\Box l_\Sigma(y)$ are well-defined at cocountably many $(x,y)\in \mms_\alpha^2$; if in addition $\smash{\sgn l_\Sigma(x) = \sgn l_\Sigma(y)}$,
\begin{align*}
\big[\Box l_\Sigma\big]_x^y \geq \int_0^{l_x(y)} k\circ\sfg_r(x)\d r + \frac{1}{N-1}\int_0^{l_x(y)} (\Box l_\Sigma)^2\circ\sfg_r(x)\d r.
\end{align*}
\end{enumerate}
\end{theorem}

\begin{proof} We only prove the first claim, the second is shown along the same lines by using \cref{Cor:22} instead of \cref{Th:Meas val Alem I}. 

By locally monotone approximation of the real ray representative of $k$ from below on every fixed ray with continuous functions, without restriction we assume $k$ is continuous.  Given $\alpha\in Q$, by \cref{Re:Properties MCP} there exists a subset $D$ of $\mms_\alpha$  with countable complement such that $\log h_\alpha$ is differentiable on $D$ according to \cref{Sub:Differentiation}. In particular, thanks to  \eqref{Eq:Tll} the quantity $\Box\u$ is defined everywhere on the set $D$. Thanks to \eqref{Eq:Tll} again, our statement is easily seen to hold if we prove every $x,y\in \mms_\alpha\setminus D$ to satisfy
\begin{align*}
&\big[(\log h_\alpha)'\big]_x^y \geq \int_0^{l_x(y)} k\circ \sfg_r(x)\d r + \frac{1}{N-1}\int_0^{l_x(y)} \big\vert (\log h_\alpha)'\big\vert^2 \circ \sfg_r(x) \d r.
\end{align*}

To streamline the presentation, we assume $h_\alpha$ is twice continuously differentiable on $\mms_\alpha$; if this does not hold, the argument to follow only has to be modified by a  logarithmic convolution, as developed in Cavalletti-Milman  \cite{cavalletti-milman2021}*{Prop.~A.10} and described in the proof of Cavalletti--Mondino \cite{cavalletti-mondino2020-new}*{Thm.~6.1}. Employing the reinforced assumption on the disintegration $\q$ together with \cref{Le:Diffchar} entails 
\begin{align*}
-(\log\bar{h}_\alpha)'' \geq \bar{k}_\alpha + \frac{1}{N-1}\,\big\vert (\log\bar{h}_\alpha)'\big\vert^2
\end{align*}
everywhere on $\mms_\alpha$. Given $(s,t)\in \ray^{-1}(D)^2$ with $s< t$, this yields
\begin{align*}
\big[(\log \bar{h}_\alpha)'\big]_s^t \geq \int_{[s,t]} \bar{k}_\alpha\d \Leb^1  + \frac{1}{N-1}\int_{[s,t]} \big\vert (\log \bar{h}_\alpha)'\big\vert^2 \d \Leb^1.
\end{align*}
The desired estimate follows by choosing $s := \sfg^{-1}(x)$ and $t := \sfg^{-1}(y)$ while recalling the compatibility of differentiation on $\mms_\alpha$ and on $\smash{\bar{\mms}_\alpha}$ from  \cref{Re:Compat}.
\end{proof}

\begin{remark}[About the integrability condition] The above hypothesis \eqref{Eq:INTEGR!cond} only ensures the logarithmic derivative of $\q$-a.e.~conditional density is related to the d'Alembertian of the given function $\u$ set up in  \cref{Def:DAlem} by \cref{Th:Meas val Alem I}. It did not play a role in the previous proof. The statement of our Bochner inequality \ref{La:UBoch} carries over unchanged whenever the  ``d'Alembertian'' of $\u$ is defined by the formula \eqref{Eq:Tll}.\hfill{\footnotesize{$\blacksquare$}}
\end{remark}

\begin{theorem}[From Bochner inequality to TCD]\label{Th:From Bochner to TCD} Let $k'$ and $N'$ satisfy the same properties as $k$ and $N$ from the beginning of \cref{Ch:Repr form}, respectively. Let $\scrM$ form a timelike $\beta$-essential\-ly nonbranch\-ing $\smash{\TMCP^e(k',N')}$ metric measure spacetime. Assume that for every signed TC Lorentz distance function $l_\Sigma$, the induced $\MCP(k',N')$ disintegration $\q$ relative to $\smash{l_\Sigma}$ satisfies the following property. For $\q$-a.e.~$\alpha\in Q$, for every $x,y\in\mms_\alpha$ such that $\smash{\sgn l_\Sigma(x)} = \smash{\sgn l_\Sigma(y)}$ and the quantities $\Box l_\Sigma(x)$ and $\Box l_\Sigma(y)$ are defined,
\begin{align*}
\big[\Box l_\Sigma\big]_x^y \geq \int_0^{l_x(y)} k\circ\sfg_r(x)\d r + \frac{1}{N-1}\int_0^{l_x(y)} (\Box l_\Sigma)^2\circ \sfg_r(x)\d r.
\end{align*}
Then $\scrM$ certifies  the $\smash{\TCD_\beta^e(k,N)}$ condition.
\end{theorem}

\begin{proof} Let $\u$ and $E$ be as in the beginning of  the previous  \cref{Ch:Repr form}.  The strategy is to construct a disintegration of a signed Lorentz distance function $\smash{l_{\Sigma_c}}$, where $c$ is a constant, to the level set $\smash{\{\u=c\}}$ which is compatible with a given disintegration relative to $\smash{\u}$. By the essential uniqueness asserted in \cref{Th:Disintegration}, the hypothesized Bochner inequality will transfer from the constructed to the original disintegration.

A priori, we consider two disintegrations in this proof. Therefore, we exceptionally tag the associated quantities with the respective function. For instance, $\smash{\Tr_{\u}}$ and $\smash{\Tr_{l_{\Sigma_c}}}$ are the transport sets without bad points induced by $\smash{\u}$ and $\smash{l_{\Sigma_c}}$, respectively.

Let $c$ be a continuity point of $\u$ (recall \cref{Le:Almost rightleft}). We consider the Suslin set
\begin{align*}
E_c := \pr_2\big[E_{\sim_{\u}}^2 \cap (\{\u=c\} \times \mms)\big]
\end{align*}
of all points in $\smash{\Tr_{\u}}$ which are related to some point $x\in \Tr_{\u}$ with respect to $\smash{\sim_{\u}}$ with $\smash{\u(x) = c}$. Without restriction, we assume $\smash{\meas[E_c]>0}$.

Given a compact subset $C$ of $\mms$, the capped set $\smash{\Sigma_c := E_c\cap \{u = c\} \cap C}$ is achronal by $1$-steepness of $\u$ and TC. Again with some abuse of notation, we write $\smash{l_{\Sigma_c}}$ for the restriction of the signed Lorentz distance function from $\Sigma_c$ to $E_c$. By \cref{Cor:Steepness}, $\smash{l_{\Sigma_c}}$ is $1$-steep with respect to $l$. Next, define the relation  $\smash{\sim_{l_{\Sigma_c}}}$ as in \cref{Sub:Disintegr} relative to the set $E_c$. We claim 
\begin{itemize}
\item $\smash{l_{\Sigma_c}}$ coincides with $\smash{\u - c}$ on $E_c$,
\item $\smash{l_{\Sigma_c}}$ is a signed TC Lorentz distance function on $\smash{E_c}$, 
\item the relation $\smash{\sim_{\u}}$ implies $\smash{\sim_{l_{\Sigma_c}}}$, and
\item the induced transport set $\smash{\Tr_{l_{\Sigma_c}}}$ without bad points coincides with $E_c$.
\end{itemize}

We prove several points of this list together. Let $y \in E_c$. Owing to the definition of $E_c$, let $x\in \smash{\tilde{y}^{\u}}$ such that $\u(x) = c$. Without restriction, we assume $\smash{\u(y) \geq c}$, the other case is treated analogously. The assumption on $x$ and $y$ yields $\smash{x \leq y}$. For every $\smash{z\in \Sigma_c}$, 
\begin{align*}
l_{\Sigma_c}(y) \geq l(x,y) = \u(y) - \u(x) = \u(y) - \u(z) \geq l(z,y), 
\end{align*}
which establishes
\begin{align*}
l_{\Sigma_c}(y) = l(x,y) = \u(y) - c
\end{align*}
by the arbitrariness of $z$. The first three claims are thus jointly proven. Lastly, we show no point of $E_c$ is a bad point with respect to $\smash{\sim_{l_{\Sigma_c}}}$, which will entail the remaining statement. Indeed, if $y \in E_c$ then $y$ is no endpoint with respect to $\smash{\sim_{\u}}$ by construction of $E_c$, hence no endpoint with respect to $\smash{\sim_{l_{\Sigma_c}}}$ by the third claim. On the other hand, by using the first claim it is easy to verify that if $y\in E_c$ is a branching point with respect to $\smash{\sim_{l_{\Sigma_c}}}$ then it is a branching point with respect to $\smash{\sim_{\u}}$ --- this is incompatible with the definition of $E_c$. Since $\smash{\Tr_{l_{\Sigma_c}}}$ is a subset of $E_c$ by construction, the last claim follows.

Let $\q$ be an $\MCP(k',N')$ disintegration relative to $\smash{\u}$. By the above discussion, we can disintegrate $\smash{\meas\mres E_c}$ in two ways:
\begin{align*}
\meas\mres E_c = \int_{\Quot(E_c)} \meas_{\u,\alpha} \d\q(\alpha) = \int_{\Quot(E_c)} \meas_{l_{\Sigma_c},\alpha}\d\q(\alpha).
\end{align*}
Note that  $\q$ appears in both disintegrations; $\q$-essential uniqueness from the disintegration \cref{Th:Disintegration} implies $\q$-a.e.~$\alpha\in \Quot(E_c)$ satisfies
\begin{align*}
h_{\u,\alpha} &= h_{l_{\Sigma_c},\alpha}.
\end{align*}
The two given ray maps $\smash{\sfg_{\u,\cdot}}$ and $\smash{\sfg_{l_{\Sigma_c},\cdot}}$ do also coincide. In turn, by our hypothesis --- while taking into account the inherent sign constraint and by \cref{Cor:22} ---,  for every $s,t\in \sfg^{-1}(\mms_{\u,\alpha} \setminus \{\u = c\})$ with $\smash{\sgn (\u(x) - c) = \sgn (\u(y) - c)}$ and $s<t$,
\begin{align*}
\big[(\log \bar{h}_{\u,\alpha})\big]_s^t \geq \int_{[s,t]} \bar{k}_\alpha\d\Leb^1 + \frac{1}{N-1}\int_{[s,t]} \big\vert(\log \bar{h}_{\u,\alpha})'\big\vert^2\d\Leb^1.
\end{align*}
If $\smash{\bar{h}_{\u,\alpha}}$ is twice continuously differentiable, by differentiating the above inequality at $s$ and invoking \cref{Le:Diffchar} implies $\smash{\bar{h}_{\u,\alpha}}$ is a $\smash{\CD(\bar{k}_\alpha,N)}$ density on its support. If this regularity fails, we can restrict ourselves to intervals where $\smash{\bar{k}_\alpha}$ is ``almost constant'' (cf.~the proof of Braun--McCann \cite{braun-mccann2023}*{Thm.~6.37}), then invoke a logarithmic convolution as in the proof of Cavalletti--Mondino \cite{cavalletti-mondino2020-new}*{Thm.~6.2}, and globalize these computations by adding up the local estimates; we omit the details. 

Since both $C$ and $c$ were arbitrary, from the one-dimensional local-to-global property in \cref{Re:Localtoglobal} we deduce $\q$-a.e.~conditional measure is $\smash{\Leb^1}$-absolutely continuous with a $\CD(k_\cdot,N)$ density on its global  support. Finally, by the arbitrariness of $\u$ and $E$ we  invoke the main result from Akdemir \cite{akdemir+} and obtain $\scrM$ obeys the $\smash{\TCD_\beta^e(k,N)}$ condition.
\end{proof}

\begin{remark}[Independence of the transport exponent]\label{Re:Potential independence} If the space  $\scrM$ is  genuinely time\-like nonbranching in \cref{Th:From Bochner to TCD}, then the $\smash{\TCD_\beta^e(k,N)}$ condition would follow for every $\beta$ as in \cref{Sub:Fixed}. This reflects the independence of the timelike curvature-dimension condition on the transport exponent in such a situation, as shown in Akdemir  \cite{akdemir+}.\hfill{\footnotesize{$\blacksquare$}}
\end{remark}

\subsection{Estimating the timelike Minkowski content of cut loci}\label{Sub:Loci} In \cref{Th:Minkowski} we  relate the past timelike Minkowski content (introduced by Cavalletti--Mondino \cite{cavalletti-mondino2024}) of the future and past  timelike cut locus of $\Sigma$ to the singular part of the d'Alembertian of $\smash{l_\Sigma}$ on $\smash{I^+(\Sigma)}$ and $\smash{I^-(\Sigma)}$, respectively. For simplicity, we assume $l_\Sigma$ is topologically locally anti-Lipschitz on $I^+(\Sigma)\cup I^-(\Sigma)$. This connects our d'Alembertian to their recent discussion of Lorentzian isoperimetric inequalities. Our result should be compared to their relation \cite{cavalletti-mondino2020-new}*{Prop.~5.1} of the cut locus of a point to the singular part of the Laplacian of the induced squared distance function in positive signature.

We recall from \cref{Re:Loci} that $\smash{x\in \TCut^\pm(\Sigma)}$ if and only if $x = a_\alpha$ for some $\alpha\in Q$; analogously for $\smash{\TCut^-(\Sigma)}$.

In Cavalletti--Mondino \cite{cavalletti-mondino2024}*{Def.~4.1}, the \emph{future} and \emph{past timelike Minkowski content} of a Borel subset $A$ of $\smash{\TCut^\pm(\Sigma)}$, respectively, have been defined by
\begin{align*}
\meas^\pm[A] := \inf \limsup_{\varepsilon \to 0+} \frac{\meas[\{0< \pm l_A < \varepsilon\} \cap U]}{\varepsilon},
\end{align*}
where the infimum is taken over all open sets $U$ in $\mms$ which contain $A$.

In the sequel, by $\smash{\vert\BOX l_\Sigma^\perp\vert}$ we will mean the total variation of the $\meas$-singular part of the generalized signed Radon measure $\smash{\BOX l_\Sigma}$ after the notation set up in the beginning of the present \cref{Ch:Appli}. Recall by \cref{Cor:22}, $\smash{\BOX l_\Sigma^\perp}$ is nonpositive on $\smash{I^+(\Sigma)}$ and nonnegative on $\smash{I^-(\Sigma)}$, respectively.

\begin{theorem}[Upper bounds for timelike Minkowski contents]\label{Th:Minkowski} Let $\scrM$ form a timelike essentially nonbranching $\smash{\TMCP^e(k,N)}$ metric measure spacetime. Let $\q$ be an $\MCP(k,N)$ disintegration relative to $\smash{l_\Sigma}$ and $E_\Sigma$. Let $A$ be a precompact Borel subset of $\smash{\TCut^\pm(\Sigma)}$ such that $\vert l_\Sigma\vert$ is bounded away from zero on $A$. Then
\begin{align*}
\meas^\mp[A] \leq \big\vert\BOX l_\Sigma^\perp\big\vert [A].
\end{align*}
\end{theorem}

\begin{proof} We only sketch the argument and we confine ourselves to  the claimed upper bound for $\smash{\meas^-[A]}$ provided $A$ lies in $\smash{\TCut^+(\Sigma)}$; the mirrored claim follows analogously. 

By the proof of Cavalletti--Mondino \cite{cavalletti-mondino2024}*{Prop.~4.6}, we know 
\begin{align*}
\meas^-[A] \leq \int_{\Quot(A)} \limsup_{\varepsilon\to 0+} \frac{1}{\varepsilon}\int_{(b_\alpha, b_\alpha - \varepsilon)} h_\alpha\d\Leb_\alpha^1\d\q(\alpha).
\end{align*}
Indeed, the argument therein was given for achronal Borel subsets of $\smash{I^+(\Sigma)}$ with empty past $\Sigma$-boundary  \cite{cavalletti-mondino2024}*{Def.~4.3}. Nevertheless, using \cref{Re:Loci} we straightforwardly extend it to our situation. Here, we implicitly use our assumption on the uniform positivity of $\smash{\vert l_\Sigma\vert}$ on $A$. Since $\smash{h_\alpha}$ extends continuously to $a_\alpha$ for $\q$-a.e.~$\alpha\in Q$ by \cref{Re:Properties MCP}, 
\begin{align*}
\meas^-[A] \leq \int_{\Quot(A)} \big[h_\alpha\big]^{b_\alpha}\d\q(\alpha).
\end{align*}
As the set $A$ is precompact, the right-hand side represents the total variation of the signed measure $\smash{\BOX l_\Sigma^\perp}$ evaluated at $A$ by \cref{Cor:44}. 
\end{proof}

\subsection{Mean curvature}\label{Sub:Dalem Mean Curv} We set up a fundamental quantity for the geometric analysis of surfaces in smooth Lorentzian geometry in the nonsmooth setting, namely mean curvature. Bounds on it appear prominently as initial conditions in the singularity theorems of general relativity (see Steinbauer \cite{steinbauer2023} for an overview), synthetic versions of which have been pioneered by Cavalletti--Mondino \cite{cavalletti-mondino2020} and recently extended by Braun--McCann \cite{braun-mccann2023}. (We also point out Burtscher--Ketterer--McCann--Woolgar's ``Riemannian singularity theorem'' \cite{burtscher-ketterer-mccann-woolgar2020}.) We adapt Ketterer's approach \cite{ketterer2020-heintze-karcher} to this from metric measure geometry.

In \cref{Sub:CMC}, the mean curvature of $\Sigma$ is naturally set up in terms of (quantities associated to) the distributional d'Alembertian of $l_\Sigma$; see e.g.~Treude--Grant \cite{treude-grant2013}. As a motivation, following Burtscher--Ketterer--McCann--Woolgar \cite{burtscher-ketterer-mccann-woolgar2020} and Ketterer \cite{ketterer2023-rigidity} we characterize Cavalletti--Mondino's approach \cite{cavalletti-mondino2020} to synthetic mean curvature \emph{bounds} in terms of our d'Alembertian of $l_\Sigma$ in \cref{Th:Mean curv bd equiv}; Mondino--Semola \cite{mondino-semola2023+}*{§A} have clarified the relation of Laplacian and mean curvature bounds on Riemannian manifolds. This naturally induces the Lorentz\-ian analog of ``mean curvature barriers'' after Antonelli--Pasqualetto--Pozzetta--Semola \cite{antonelli-pasqualetto-pozzetta-semola2022} and ``constant mean curvature sets'' after Ketterer \cite{ketterer2023-rigidity}. This is the nonsmooth analog of surfaces of constant mean curvature; see  \cref{Re:CMC}.

\subsubsection{Synthetic mean curvature bounds}\label{Sub:Synth mc} For $t\in \R$, let $\sfp_t\colon Q\to \cl\,\Tr$ be the map sending $\alpha\in Q$ to the unique point in the intersection $\mms_\alpha\cap \{l_\Sigma = t\}$ provided the latter is nonempty. In particular, $\sfp_0$ plays the role of a ``footpoint projection'' to $\Sigma$. The measurability  of these maps is discussed in detail in Cavalletti--Mondino \cite{cavalletti-mondino2020}*{§5.1}.

It will be convenient to parametrize the rays as follows. Given $\alpha\in Q$, on a suitable closed  interval $I_\alpha$, we define the functions $\smash{\hat{h}_\alpha}$ and $\smash{\hat{k}_\alpha}$ by
\begin{align}\label{Eq:tildeh}
\begin{split}
\hat{h}_\alpha(t) &:= h_\alpha\circ\sfg_t\circ\sfp_0(\alpha),\\
\hat{k}_\alpha(t) &:= k\circ\sfg_t\circ\sfp_0(\alpha).
\end{split}
\end{align}

\begin{remark}[Convention beyond endpoints]\label{Re:Convent endpts} To simplify our notation in the sequel, we extend $\smash{\hat{h}_\alpha}$ by zero and $\smash{\hat{k}_\alpha}$ constantly beyond $\smash{\cl\,\bar{\mms}_\alpha}$, where $\alpha\in Q$. Clearly, these  extensions do not alter the conclusion of \cref{Le:CompII}, which are used in \cref{Th:HeintzeKarcher,Th:AREA}.\hfill{\footnotesize{$\blacksquare$}}
\end{remark}

Given $t\geq 0$ we consider the Borel measure
\begin{align*}
\hh_t := (\sfp_t)_\push[\hat{h}_\cdot(t)\,\q],
\end{align*}
recalling $\smash{\hat{h}_\alpha}$ is continuous for $\q$-a.e.~$\alpha\in Q$ by \cref{Re:Properties MCP}. The Borel measure $\hh_t$ should be thought of as an ``area measure'' of $\{l_\Sigma=t\}$. Indeed, for every $\meas$-measurable subset $A$ of $\smash{\Tr^\End}$ with $\meas[A]<\infty$, Cavalletti--Mondino \cite{cavalletti-mondino2020}*{Prop.~5.1} grants the coarea-type formula
\begin{align}\label{Eq:COar}
\meas[A] = \int_0^\infty\hh_t[A]\d t = \int_0^\infty \hh_t[A \cap \{l_\Sigma=t\}]\d t.
\end{align} 
In general, $\hh_t$ is not a Radon measure, which partly  motivates the following definition.

\begin{definition}[Area-regularity] We call $\Sigma$ \emph{area-regular} if its area measure $\hh_0$ from above is a Radon measure and it satisfies $(\sfp_0)_\push\q\ll\hh_0$. 
\end{definition}

In the sequel, we always assume $\Sigma$ to be area-regular.

As explained in Cavalletti--Mondino \cite{cavalletti-mondino2020}*{§5.1}, this assumption will be natural from the interpretation of the mean curvature as the first variation of the area of $\Sigma$. Moreover, the condition $(\sfp_0)_\push\q \ll \hh_0$ is equivalent to asking $\smash{\hat{h}_\alpha(0) > 0}$ to hold  for $\q$-a.e.~$\alpha\in Q$, which can be interpreted as a ``codimension one'' hypothesis in the sense of Minkowski content \cite{cavalletti-mondino2020}*{Rem.~5.3}. By \cref{Re:Properties MCP}, it holds if $\q$-a.e.~ray (which does not contain endpoints by construction) intersects $\Sigma$, or equivalenty, can be extended to both ``sides'' past $\Sigma$.

The \emph{lift} of a function $f$ defined on $\Sigma$ to $Q$ is given by
\begin{align*}
f_0 := f\circ\sfp_0.
\end{align*}
Following Burtscher--Ketterer--McCann--Woolgar \cite{burtscher-ketterer-mccann-woolgar2020}*{§B}, we denote by  $\smash{\Ell_{\pm\loc}^1(\Sigma,\hh_0)}$ the class of all $\q$-measurable functions $H$ on $\Sigma$ such that their positive or negative parts $\smash{H_\pm}$ are locally $\hh_0$-integrable.

The following is an evident variable generalization of Cavalletti--Mondino \cite{cavalletti-mondino2020}*{Def.~5.2}, inspired by similar considerations by Ketterer \cite{ketterer2020-heintze-karcher,ketterer2023-rigidity} and Burtscher--Ketterer--McCann--Woolgar  \cite{burtscher-ketterer-mccann-woolgar2020} in positive signature.

\begin{definition}[Synthetic mean curvature bounds \cite{cavalletti-mondino2020}*{Def.~5.2}]\label{Def:Synth mean} Let $\smash{H\in \Ell_{+\loc}^1(\Sigma,\hh_0)}$. We say $\Sigma$ has \emph{forward mean curvature bounded from above} by $H$ if for every nonnegative, bounded, and compactly supported Borel function $\phi$ on $\Sigma$, its induced normal variations
\begin{align*}
\Sigma_{t,\phi} := \{l_\Sigma \leq t\,\phi_0\circ\sfQ\} \cap \Tr,
\end{align*}
where $t>0$, satisfy the inequality
\begin{align*}
\liminf_{t\to 0+} \frac{2}{t^2}\Big[\meas[\Sigma_{t,\phi}] - t\int_\Sigma \phi\d\hh_0\Big] \leq \int_\Sigma H\,\phi^2\d\hh_0.
\end{align*}

Accordingly, let $\smash{H\in \Ell_{-\loc}^1(\Sigma,\hh_0)}$. Then $\Sigma$  is said to possess \emph{forward mean curvature  bounded from below} by $H$ if for every $\phi$ as above,
\begin{align*}
\limsup_{t\to 0+}\frac{2}{t^2}\Big[\meas[\Sigma_{t,\phi}] - t\int_\Sigma \phi\d \hh_0\Big] \geq \int_\Sigma H \,\phi^2\d\hh_0.
\end{align*}
\end{definition}

The following should be compared with Ketterer  \cite{ketterer2023-rigidity}*{Prop.~3.6} in positive signature.

\begin{remark}[Scaling] For $\varepsilon > 0$ we consider the $\TMCP^e(\varepsilon^{-2}\,k,N)$ metric measure space\-time given by $\scrM_\varepsilon := (\mms,\varepsilon\, l,\meas)$. Moreover, given $\smash{H\in\Ell_{+\loc}^1(\Sigma,\hh_0)}$, $\Sigma$ has forward  mean curvature bounded from above by $H$ relative to $\scrM$ if and only if its forward mean curvature is  bounded from above by $\varepsilon^{-1}\,H$ with respect to $\scrM_\varepsilon$.

An analogous claim holds for synthetic lower mean curvature bounds.\hfill{\footnotesize{$\blacksquare$}}
\end{remark}

\subsubsection{D'Alembert mean curvature bounds} Recall the potential function from \cref{Sub:Ball Jac}.

\begin{theorem}[Synthetic vs.~d'Alembert mean curvature]\label{Th:Mean curv bd equiv} Let $\scrM$ be a timelike essentially nonbranching $\TMCP^e(k,N)$ metric measure spacetime. Moreover, let $\q$ be an $\MCP(k,N)$ disintegration relative to $l_\Sigma$ and $\smash{I^+(\Sigma)\cup \Sigma}$. Finally, let $\smash{H\in \Ell_{+\loc}^1(\Sigma,\hh_0)}$. 
\begin{enumerate}[label=\textnormal{(\roman*)}]
\item \textnormal{\textbf{Variable version.}} The forward mean curvature of $\Sigma$ is bounded from above by $H$ after \cref{Def:Synth mean} if and only if  $\q$-a.e.~$\alpha\in Q$ satisfies
\begin{align}\label{Eq:BD!!}
\Box l_\Sigma \leq -(N-1)\,\frac{\sfP_{\hat{k}_\alpha/(N-1), -H_0(\alpha)/(N-1)}' \circ l_\Sigma}{\sfP_{\hat{k}_\alpha/(N-1), -H_0(\alpha)/(N-1)} \circ l_\Sigma}\quad\meas_\alpha\textnormal{-a.e.};
\end{align}
in either case, the inequality
\begin{align}\label{Eq:BD!!!}
\BOX l_\Sigma\mres I^+(\Sigma) \leq -(N-1)\int_Q\frac{\sfP_{\hat{k}_\alpha/(N-1), -H_0(\alpha)/(N-1)}' \circ l_\Sigma}{\sfP_{\hat{k}_\alpha/(N-1), -H_0(\alpha)/(N-1)} \circ l_\Sigma}\,\meas_\alpha\d\q(\alpha)
\end{align}
holds in the sense of weak Radon functionals.
\item \textnormal{\textbf{SEC version.}} If $k$ vanishes identically, upper boundedness of the forward mean curvature of $\Sigma$ by $H$ is equivalent to $\q$-a.e.~$\alpha\in Q$ satisfying
\begin{align*}
\Box l_\Sigma \leq (N-1)\,\frac{H_0(\alpha)}{N-1- H_0(\alpha)\,l_\Sigma}\quad\meas_\alpha\textnormal{-a.e.};
\end{align*}
in either case, the inequality
\begin{align}\label{Eq:Ra5}
\BOX l_\Sigma\mres I^+(\Sigma) \leq (N-1)\int_Q \frac{H_0(\alpha)}{N-1- H_0(\alpha)\,l_\Sigma}\,\meas_\alpha\d\q(\alpha)
\end{align}
holds in the sense of weak Radon functionals.
\end{enumerate}

Assuming instead $\smash{H\in \Ell_{-\loc}^1(\Sigma,\hh_0)}$,  the forward mean curvature of $\Sigma$ is bounded from below by $H$ if and only if $\q$-a.e.~$\alpha\in Q$ satisfies
\begin{align*}
\Box l_\Sigma \geq -(N-1)\,\frac{\sfP_{\hat{k}_\alpha/(N-1), -H_0(\alpha)/(N-1)}' \circ l_\Sigma}{\sfP_{\hat{k}_\alpha/(N-1), -H_0(\alpha)/(N-1)} \circ l_\Sigma}\quad\meas_\alpha\textnormal{-a.e.}
\end{align*}
\end{theorem}

\begin{remark}[Equivalence of comparison estimates] The converse implication from \eqref{Eq:BD!!!} to \eqref{Eq:BD!!} holds if $l_\Sigma$ is topologically anti-Lipschitz on $I^+(\Sigma)$. This follows directly from \cref{Cor:22} and  \cref{Le:VsTop}. The latter enables us to choose appropriate  test functions to propagate an integrated to an a.e.~inequality.\hfill{\footnotesize{$\blacksquare$}}
\end{remark}

\begin{remark}[Nonnegative curvature] Assume $k$ vanishes identically in the first part of \cref{Th:Mean curv bd equiv}. If also $H$ vanishes identically, recalling \cref{Re:Superharmonicity} the estimate \eqref{Eq:Ra5} says the signed Lorentz distance function $l_\Sigma$ is superharmonic on $\smash{I^+(\Sigma)}$. This should be compared to the fundamental one-to-one correspondence of minimality of smooth hypersurfaces in positive signature (in the sense of vanishing mean curvature) and subharmonicity of their induced distance function (e.g.~in the viscosity or barrier  sense, cf.~Wu \cite{wu1979} and Choe--Fraser \cite{choe-fraser2018}). Also, inspired by Gromov's proof of the Lévy--Gromov  isoperimetric inequality \cite{gromov2007}*{§C} the mentioned subharmonicity is a necessary condition for \emph{singular} hypersurfaces to be minimal. We refer to  Mondino--Semola \cite{mondino-semola2023+}*{§A} for a  literature overview.\hfill{\footnotesize{$\blacksquare$}}
\end{remark}

This characterization suggests the following analog of Ketterer  \cite{ketterer2023-rigidity}*{Def.~3.10}.

\begin{definition}[D'Alembert mean curvature bounds]\label{Def:DAL MEAN C} Let $\smash{H\in \Ell_{+\loc}^1(\Sigma,\hh_0)}$. We say $\Sigma$ has \emph{forward d'Alembert mean curvature bounded from above} by $H$ if $\q$-a.e.~$\alpha\in Q$ satisfies
\begin{align*}
\Box l_\Sigma \leq -(N-1)\,\frac{\sfP_{\hat{k}_\alpha/(N-1), -H_0(\alpha)/(N-1)}' \circ l_\Sigma}{\sfP_{\hat{k}_\alpha/(N-1), -H_0(\alpha)/(N-1)} \circ l_\Sigma}\quad\meas_\alpha\textnormal{-a.e.}
\end{align*}

Accordingly, let $\smash{H\in\Ell_{-\loc}^1(\Sigma,\hh_0)}$. Then we term $\Sigma$ to have \emph{forward d'Alembert mean curvature bounded from below} by $H$ if $\q$-a.e.~$\alpha\in Q$ satisfies
\begin{align*}
\Box l_\Sigma \geq -(N-1)\,\frac{\sfP_{\hat{k}_\alpha/(N-1), -H_0(\alpha)/(N-1)}' \circ l_\Sigma}{\sfP_{\hat{k}_\alpha/(N-1), -H_0(\alpha)/(N-1)} \circ l_\Sigma}\quad\meas_\alpha\textnormal{-a.e.}
\end{align*}
\end{definition}

\begin{remark}[Synthetic Hawking-type singularity theorems]\label{Re:Syntg} \cref{Th:Mean curv bd equiv} entails the synthetic Hawking-type singularity theorems of Cavalletti--Mondino \cite{cavalletti-mondino2020}*{Thm.~5.6} and Braun--McCann \cite{braun-mccann2023}*{Thm.~7.14} hold under the assumption of appropriate forward d'Alembert mean curvature upper bounds. In the smooth situation, this has already been observed by Treude--Grant \cite{treude-grant2013}*{Thm.~5.1}.\hfill{\footnotesize{$\blacksquare$}}
\end{remark}

Now we proceed to the proof of \cref{Th:Mean curv bd equiv}. The strategy follows Burtscher--Ketterer--McCann--Woolgar  \cite{burtscher-ketterer-mccann-woolgar2020}*{Lem.~B.3, Cor. B.4} and Ketterer \cite{ketterer2023-rigidity}*{Lem.~3.7, Thm.~3.8} in positive signature. An expert reader will notice the similarity of the following proofs to the ones for the synthetic Hawking-type singularity theorems of Cavalletti--Mondino \cite{cavalletti-mondino2020}*{Thm.~5.6} and Braun--McCann \cite{braun-mccann2023}*{Thm.~7.14} (a correspondence which, in light of \cref{Re:Syntg}, is of course no surprise).

\begin{lemma}[Differential inequalities]\label{Le:Diff in} Suppose that $\smash{H\in L_{+\loc}^1(\Sigma,\hh_0)}$. Then $\Sigma$ possesses forward mean curvature bounded from above by $H$ if and only if for $\q$-a.e.~$\alpha\in Q$, the right derivative $\smash{\hat{h}_\alpha'^+(0)}$ from the beginning of \cref{Ch:Prerequisites} exists, is a real number, and satisfies
\begin{align*}
\hat{h}_\alpha'^+(0) \leq H_0(\alpha)\,\hat{h}_\alpha(0).
\end{align*}

Analogously, assume $\smash{H\in \Ell_{-\loc}^1(\Sigma,\hh_0)}$. Then $\Sigma$ has forward mean curvature bounded from below by $H$ if and only if for $\q$-a.e.~$\alpha\in Q$, the right derivative $\smash{\hat{h}_\alpha'^+(0)}$ from the beginning of \cref{Ch:Prerequisites} exists, is a real number, and satisfies
\begin{align*}
\hat{h}_\alpha'^+(0) \geq H_0(\alpha)\,\hat{h}_\alpha(0).
\end{align*}
\end{lemma}

\begin{proof} We only prove the first statement, the second is analogous.

To show the ``only if'' implication, we start with some preparations. As in the proof of \cref{Le:Radon funct}, let $\varepsilon >0$ and let $Q_\varepsilon$ denote the $\q$-measurable set of all $\alpha\in Q$ such that $l_\Sigma(a_\alpha) \geq 2\varepsilon$. Let $\Sigma_\varepsilon$ be a precompact subset of $\Quot^{-1}(Q_\varepsilon)\cap \{l_\Sigma = \varepsilon\}$. As $\Sigma$ is FTC, the set $J(\cl\,\sfp_0(\Sigma_\varepsilon), \cl\,\Sigma_\varepsilon)$ is compact. Let $K$ denote a negative lower bound on $k$ on the latter set. Let $\phi$ be a nonnegative Borel function on $\Sigma$ to be chosen later whose support is contained in $\sfp_0(\Sigma_\varepsilon)$. By scaling, it will  suffice to assume  $\Vert \phi\Vert_\infty \leq \varepsilon$. Given  $t\in (0,\varepsilon)$, \eqref{Eq:COar} implies
\begin{align*}
\meas[\Sigma_{t,\phi}] - t\int_\Sigma\phi\d\hh_0 &= \int_{Q_\varepsilon}\int_0^t \big[\hat{h}_\alpha(r\,\phi_0(\alpha)) - \hat{h}_\alpha(0)\big]\d r\,\phi_0(\alpha)\d\q(\alpha).
\end{align*}

By \cref{Re:Props dist coeff}, \cref{Le:logarithmic derivative},  \cref{Re:Logarithmic derivative}, and our hypothesis on $\phi$,  $\q$-a.e.~$\alpha\in \Quot(\Sigma_\varepsilon)$ satisfies the following inequalities for every $r\in (0,t)$:
\begin{align}\label{Eq:ESST}
\begin{split}
\hat{h}_\alpha(r\,\phi_0(\alpha)) &\geq \frac{\sinh(\sqrt{-K/(N-1)}\,(2\varepsilon - r\,\phi_0(\alpha))^{N-1}}{\sinh(\sqrt{-K/(N-1)}\,(2\varepsilon))^{N-1}}\,\hat{h}_\alpha(0)\\
&\geq \big[1 - c\,r\,\phi_0(\alpha)\big]\,\hat{h}_\alpha(0)
\end{split}
\end{align}
where
\begin{align*}
c:= 2\varepsilon\,\sqrt{-K(N-1)}\,\frac{\cosh(\sqrt{-K/(N-1)}\,(2\varepsilon))^{N-1}}{\sinh(\sqrt{-K/(N-1)}\,(2\varepsilon))^{N-1}},
\end{align*}
and consequently 
\begin{align*}
2\int_0^t \big[\hat{h}_\alpha(r\,\phi_0(\alpha)) - \hat{h}_\alpha(0)\big]\d r\,\phi_0(\alpha) \geq -c\,t^2\,\phi^2_0(\alpha)\,\hat{h}_\alpha(0).
\end{align*}
With the $\hh_0$-integrability of $\phi^2$, this  justifies the use of Fatou's lemma to entail
\begin{align}\label{Eq:Fatu}
&\liminf_{t\to 0+}\frac{2}{t^2}\Big[\meas[\Sigma_{t,\phi}] - t\int_\Sigma\phi\d\hh_0\Big]\geq \int_{Q_\varepsilon}\liminf_{t\to 0+}G_t(\alpha)\,\phi_0(\alpha)\d\q(\alpha),
\end{align}
where
\begin{align*}
G_t(\alpha) := \frac{2}{t^2}\int_0^t \big[\hat{h}_\alpha(r\,\phi_0(\alpha)) - \hat{h}_\alpha(0)\big]\d r
\end{align*}

Recall from \cref{Re:Properties MCP} the right derivative $\smash{\hat{h}_\alpha'^+(0)}$ exists in the extended real numbers for $\q$-a.e.~$\alpha\in Q_\varepsilon$. The bound  \eqref{Eq:ESST} above shows it is not $-\infty$. We now invoke the forward mean curvature bound on $\Sigma$ for the first time to show $\smash{\hat{h}_\alpha'^+(0)}$ is a real number. Assume to the contrary that there exists a subset $R_\varepsilon$ of $Q_\varepsilon$ such that $\q[R_\varepsilon] > 0$ with $\smash{\hat{h}_\alpha'^+(0)=\infty}$ for every $\alpha\in R_\varepsilon$. Given any positive $m$ there is $t\in (0,\varepsilon)$ such that for every $s\in (0,t)$,
\begin{align*}
\hat{h}_\alpha(s) - \hat{h}_\alpha(0) \geq m s.
\end{align*}
In the case $\phi_0(\alpha) > 0$, this implies
\begin{align*}
\liminf_{t\to 0+} G_t(\alpha) \geq m\,\phi_0(\alpha)
\end{align*}
and the arbitrariness of $m$ gives
\begin{align*}
\liminf_{t\to 0+} G_t(\alpha) = \infty.
\end{align*}
Choosing $\phi$ as the indicator function scaled by $\varepsilon$ of a precompact subset $C$ of $\sfp_0(R_\varepsilon)$ in \eqref{Eq:Fatu} and using the forward mean curvature bound on $\Sigma$ gives
\begin{align*}
\varepsilon^2\int_C H_+\d\hh_0 \geq \int_{Q_\varepsilon}\liminf_{t\to 0+} G_t(\alpha)\,\phi_0(\alpha)\d \q(\alpha) = \infty.
\end{align*}
This contradicts the assumed local $\hh_0$-integrability of $H_+$.

In turn, arguing as in the previous paragraph, for $\q$-a.e.~$\alpha\in Q_\varepsilon$ we get
\begin{align*}
\liminf_{t\to 0+} G_t(\alpha) = \hat{h}'^+_\alpha(0)\,\phi_0(\alpha).
\end{align*}
Employing \eqref{Eq:Fatu} and our forward mean curvature bound  again,
\begin{align*}
\int_{Q_\varepsilon} \big[H_0(\alpha)\,\hat{h}_\alpha(0)\big]\,\phi^2_0(\alpha)\d\q(\alpha) \geq \int_{Q_\varepsilon} \hat{h}'^+_\alpha(0)\,\phi^2_0(\alpha)\d\q(\alpha).
\end{align*}
By the arbitrariness of $\phi$ and $\varepsilon$, the claim is proven.

The ``if'' implication follows by tracing back the previous  argumentation.
\end{proof}

\begin{proof}[Proof of \cref{Th:Mean curv bd equiv}] First,  assume the forward mean curvature of $\Sigma$ is bounded from above by a function $H$ as given. For $\q$-a.e. $\alpha\in Q$, by the disintegration \cref{Th:Disintegration} the reparametrization $\smash{\hat{h}_\alpha}$ is a $\CD(\hat{k}_\alpha,N)$ density on the closed interval $I_\alpha$ containing zero. By \cref{Le:Diff in,Le:Riccati}, on the interior of this interval we have
\begin{align*}
(\log \hat{h}_\alpha)'^+ &\leq \big[(\log \sfP_{\hat{k}_\alpha/(N-1), H_0(\alpha)/(N-1)})^{N-1}\big]'.
\end{align*}
Reversing the parametrization after \eqref{Eq:tildeh} and recalling \cref{Re:Trafo refl,Re:Compat} gives
\begin{align}\label{Eq:Ineq log}
-(\log h_\alpha)' \leq -(N-1)\,\frac{\sfP'_{\hat{k}_\alpha/(N-1), -H_0(\alpha)/(N-1)}\circ l_\Sigma}{\sfP_{\hat{k}_\alpha/(N-1), -H_0(\alpha)/(N-1)}\circ l_\Sigma}\quad\meas_\alpha\textnormal{-a.e.}
\end{align}

In turn, in the sense of weak Radon functionals this easily implies
\begin{align*}
\BOX l_\Sigma \mres I^+(\Sigma) &= -\int_Q (\log h_\alpha)'\,\meas_\alpha\d\q(\alpha) + \int_Q \big[\delta_\cdot\,h_\alpha\big]_{a_\alpha}\d\q(\alpha)\\
&\leq -\int_Q (\log h_\alpha)'\,\meas_\alpha\d\q(\alpha)\\
&\leq -(N-1)\int_Q \frac{\sfP'_{\hat{k}_\alpha/(N-1), -H_0(\alpha)/(N-1)}\circ l_\Sigma}{\sfP_{\hat{k}_\alpha/(N-1), -H_0(\alpha)/(N-1)}\circ l_\Sigma}\,\meas_\alpha\d\q(\alpha).
\end{align*}

Conversely, by assumption the above inequality \eqref{Eq:Ineq log} holds. In particular, there is a sequence $(t_n)_{n\in\N}$ of positive real numbers converging to zero such that for every $n\in\N$,
\begin{align*}
(\log \hat{h}_\alpha)'(t_n) \leq (N-1)\,\frac{\sfP'_{\hat{k}_\alpha/(N-1), H_0(\alpha)/(N-1)}(t_n)}{\sfP_{\hat{k}_\alpha/(N-1), H_0(\alpha)/(N-1)}(t_n)}.
\end{align*}
As $n\to\infty$, the left-hand side converges to $\smash{(\log \hat{h}_\alpha)'(0)}$ and the right-hand side converges to $H_0(\alpha)$. Using the chain rule and the positivity of $\smash{\hat{h}_\alpha(0)}$ yields
\begin{align*}
\hat{h}_\alpha'^+(0) \leq H_0(\alpha)\,\hat{h}_\alpha(0),
\end{align*}
and the claim follows from \cref{Le:Diff in}.

The equivalence for lower boundedness is shown in an analogous way.
\end{proof}

\subsubsection{Exact mean curvature and its barriers}\label{Sub:CMC} To define the actual mean curvature of $\Sigma$, we start with some technical discussions, referring to the literature \cite{cavalletti-milman2021,ketterer2020-heintze-karcher,burtscher-ketterer-mccann-woolgar2020,ketterer2023-rigidity} in positive signature for details about  measurability. Since $\Tr$ does not contain endpoints, $\Sigma\cap \Tr$ corresponds to the set of all points in $\Sigma$ that lie on the \emph{interior} of a ray. Accordingly, the $\meas$-measurable set $S := \Quot^{-1}(\Quot(\Sigma\cap\Tr))$ is the set of all points in $\Tr$ which lie on such a ray passing through $\Sigma$. Since $\Sigma$ has footpoints, the $\meas$-measurable sets
\begin{align*}
D^\pm_\Sigma &:= (I^\pm(\Sigma) \cap \Tr)\setminus S
\end{align*}
collect all points in $\smash{I^\pm(\Sigma)}$ which lie on a ray which starts or ends in $\Sigma$. We also set
\begin{align*}
D_\Sigma := D^+_\Sigma\cup D^-_\Sigma.
\end{align*}

This motivates the following Lorentzian analog of Ketterer's notion of ``regularity'' of a set sufficient in order to define its mean curvature \cite{ketterer2020-heintze-karcher}*{Def.~5.7}. 

\begin{definition}[Finite curvature]\label{Def:Forward C} We say $\Sigma$ has
\begin{enumerate}[label=\textnormal{\alph*.}]
\item \emph{finite forward curvature} provided $\smash{\meas[D_\Sigma^+]=0}$,
\item \emph{finite backward curvature} provided $\smash{\meas[D_\Sigma^-]=0}$, and
\item \emph{finite curvature} if it has finite forward and backward curvature simultaneously, or equivalently $\smash{\meas[D_\Sigma]=0}$.
\end{enumerate}
\end{definition}

\begin{remark}[Interpretation in the smooth case] As already noted e.g.~in Ketterer \cite{ketterer2020-heintze-karcher}*{Rem.~5.9} and Cavalletti--Mondino \cite{cavalletti-mondino2020}*{Rem.~5.5}, in the smooth case $\Sigma$ having finite curvature corresponds to a local $\smash{\Ell^\infty}$-bound on its second fundamental form.\hfill{\footnotesize{$\blacksquare$}}
\end{remark}

For the rest of the current \cref{Sub:HKIn}, we assume $\Sigma$ has finite curvature. Again, if only the chronological future or past is relevant, this can be weakened to finite forward or backward curvature, respectively.

This entails enhanced properties on the conditional densities ``on $\Sigma$'' comparable to the regularity assumptions from the previous \cref{Sub:Dalem Mean Curv}. Since for $\q$-a.e.~$\alpha\in \Quot(S)$, $0$ is an \emph{interior} point of the open interval $\smash{\bar{\mms}_\alpha}$, the value $\smash{\hat{h}_\alpha(0)}$ is positive by \cref{Re:Properties MCP}. In particular, the property $(\sfp_0)_\push\q \ll \hh_0$ assumed by area-regularity of $\Sigma$ holds by default. In turn, \cref{Re:Properties MCP} again implies $\smash{\hat{h}_\alpha}$ (and therefore $\smash{\log \hat{h}_\alpha}$) is right and left differentiable at zero and both derivatives are real numbers.

This motivates the subsequent adaptation of Ketterer \cite{ketterer2017}*{Def.~5.7, Rem.~5.10}. Let us point out the analogy of the distributional mean curvature set up below to the d'Alembert representation formula from \cref{Cor:22}.

\begin{definition}[Mean curvature]\label{Def:MCurv} If $\Sigma$ has finite forward curvature, then its \emph{forward mean curvature} is the function $\smash{H^+_\Sigma\colon \Sigma\to \R\cup\{-\infty\}}$ $\hh_0$-a.e.~defined by
\begin{align*}
H_\Sigma^+(z) &:= \begin{cases} (\log\hat{h}_{\Quot(z)})'^+(0) & \textnormal{\textit{if} } z\in S,\\
-\infty & \textnormal{\textit{if} }  \smash{\cl\,\mms_{\Quot(z)} \cap I^+(\Sigma)} \textnormal{ \textit{is empty}},\\
0 & \textnormal{\textit{otherwise}}.
\end{cases}
\end{align*}

If $\Sigma$ has finite backward curvature, then its \emph{backward mean curvature} is the function $\smash{H_\Sigma^-\colon \Sigma\to \R\cup\{\infty\}}$ $\hh_0$-a.e.~defined by
\begin{align*}
H_\Sigma^-(z) := \begin{cases} -(\log\hat{h}_{\Quot(z)})'^-(0) & \textnormal{\textit{if} } z\in S,\\
\infty & \textnormal{\textit{if} } \smash{\cl\,\mms_{\Quot(z)} \cap I^-(\Sigma)} \textnormal{ \textit{is empty}},\\
0 & \textnormal{\textit{otherwise}}.
\end{cases}
\end{align*}

Finally, if $\Sigma$ has finite curvature,  the \emph{extended  mean curvature} of $\Sigma$ is the generalized  signed Radon measure $\bdH_\Sigma$ on $\Sigma$ defined by
\begin{align*}
\bdH_\Sigma = H_\Sigma\,\hh_0 - \int_{\Quot(D_\Sigma)} \big[(\delta_\cdot\mres \Sigma)\,h_\alpha\big]_{a_\alpha}^{b_\alpha} \d\q(\alpha)
\end{align*}
where the \emph{mean curvature} of $\Sigma$ is the function $\smash{H_\Sigma\colon\Sigma\to \R}$ $\hh_0$-a.e.~given by
\begin{align*}
H_\Sigma := H^+_\Sigma \vee (-H_\Sigma^-).
\end{align*}
\end{definition}

In other words, the $\hh_0$-singular part in the definition of $\bdH_\Sigma$ weighs all endpoints which lie on $\Sigma$ with the corresponding conditional density (nonpositive mass if the ray does not extend to $\smash{I^-(\Sigma)}$ and nonnegative mass if it does not extend to $\smash{I^+(\Sigma)}$).

\begin{remark}[Compatibility with the smooth setting]\label{Re:Compat smooth} Suppose $\Sigma$ is a smooth, compact, achronal spacelike hypersurface in a smooth, globally hyperbolic space\-time, cf.~\cref{Ex:TCD}. By Cavalletti--Mondino  \cite{cavalletti-mondino2020}*{Rem.~5.4}, $\hh_0$ is the canonical area measure $\smash{\scrH^{\dim\mms-1}_\Sigma}$ of $\Sigma$. Denoting by $\sfH_\Sigma$ the mean curvature of $\Sigma$,  every smooth and nonnegative function $\phi$ on $\Sigma$ satisfies
\begin{align*}
\lim_{t\to 0+} \frac{2}{t^2}\Big[\meas[\Sigma_{t,\phi}] - t\int_\Sigma \phi\d\hh_0\Big] = \int_\Sigma \sfH_\Sigma\,\phi^2\d\hh_0.
\end{align*}
On the other hand, following the lines of the proof of \cref{Le:Diff in}, 
\begin{align*}
\lim_{t\to 0+} \frac{2}{t^2}\Big[\meas[\Sigma_{t,\phi}] - t\int_\Sigma \phi\d\hh_0\Big] &= \int_Q \hat{h}_\alpha'^+(0)\,\phi_0^2(\alpha)\d\q(\alpha)\\
&= \int_\Sigma (\log \hat{h}_{\Quot(z)})'^+(0)\,\phi^2\d\hh_0(z).
\end{align*}
The arbitrariness of $\phi$ and a similar argument relative to $I^-(\Sigma)$ imply
\begin{align*}
H_\Sigma = \sfH_\Sigma\quad\hh_0\textnormal{-a.e.}
\end{align*}
Here we have implicitly used that every ray (corresponding to the negative gradient flow trajectory of $l_\Sigma$) can be extended both to $I^+(\Sigma)$ and $I^-(\Sigma)$, i.e.~$\Sigma$ has finite curvature. In particular, its extended mean curvature has no $\hh_0$-singular part, whence
\begin{align*}
\bdH_\Sigma = \sfH_\Sigma\,\hh_0.\tag*{{\footnotesize{$\blacksquare$}}}
\end{align*}
\end{remark}

\begin{remark}[Real-valued mean curvature from mean curvature bounds] If $\Sigma$ has forward mean curvature bounded from above by a function $\smash{H\in \Ell_{+\loc}^1(\Sigma,\hh_0)}$, then \cref{Le:Diff in} implies $\smash{H^+\leq H}$ $\hh_0 \mres S$-a.e.

An analogous claim holds for forward lower mean curvature bounds.\hfill{\footnotesize{$\blacksquare$}}
\end{remark}

We now introduce the terminology of barriers for the mean curvature of $\Sigma$. For isoperimetric sets in positive signature, this originates in the preprint of Antonelli--Pasqualetto--Pozzetta--Semola \cite{antonelli-pasqualetto-pozzetta-semola2022}*{Def.~3.6}. A similar notion for metric measure spaces satisfying the measure contraction property is found in Ketterer \cite{ketterer2023-rigidity}*{§A}. It is interpreted as the synthetic analog of \emph{constant mean curvature sets}, briefly CMC sets, although this analogy should be taken with care, cf.~\cref{Re:CMC}. 

\begin{definition}[Mean curvature barrier]\label{Def:Barriers} We say a function $\smash{H\in \Ell_\loc^1(\Sigma,\hh_0)}$ constitutes a \emph{mean curvature barrier} for $\Sigma$ if the following two inequalities hold in the sense of weak Radon functionals:
\begin{align*}
\BOX l_\Sigma \mres I^+(\Sigma) &\leq -(N-1)\int_Q \frac{\sfP_{\hat{k}_\alpha/(N-1),-H_0(\alpha)/(N-1)}'\circ l_\Sigma}{\sfP_{\hat{k}_\alpha/(N-1),-H_0(\alpha)/(N-1)}\circ l_\Sigma}\,\meas_\alpha\d\q(\alpha),\\
\BOX l_\Sigma \mres I^-(\Sigma) &\geq (N-1)\int_Q \frac{\sfP_{\hat{k}_\alpha/(N-1),H_0(\alpha)/(N-1)}'\circ l_\Sigma^\leftarrow}{\sfP_{\hat{k}_\alpha/(N-1),H_0(\alpha)/(N-1)}\circ l_\Sigma^\leftarrow}\,\meas_\alpha\d\q(\alpha).
\end{align*}
\end{definition}

The second inequality corresponds to a forward upper mean curvature bound of the signed Lorentz distance function $\smash{l_\Sigma^\leftarrow}$ relative to the causally reversed structure $\scrM^\leftarrow$ by $-H$. 

In the SEC case, the estimates from \cref{Def:Barriers} simplify to
\begin{align}\label{Eq:BAR}
\begin{split}
\BOX l_\Sigma \mres I^+(\Sigma) &\leq (N-1)\int_Q \frac{H_0(\alpha)}{N-1-H_0(\alpha)\,l_\Sigma}\,\meas_\alpha\d\q(\alpha),\\
\BOX l_\Sigma \mres I^-(\Sigma) &\geq (N-1)\int_Q \frac{H_0(\alpha)}{N-1+H_0(\alpha)\,l_\Sigma^\leftarrow}\,\meas_\alpha\d\q(\alpha).
\end{split}
\end{align}

\begin{remark}[Relation to CMC sets]\label{Re:CMC} In the smooth setting of \cref{Re:Compat smooth}, \cref{Th:Mean curv bd equiv} implies $\Sigma$ is a CMC surface (with constant mean curvature $H_0$) if and only if the number $H_0$ is a barrier for its mean curvature.

In the general nonsmooth framework, examples from metric measure geometry (e.g.~An\-tonelli--Pasqualetto--Pozzetta--Semola \cite{antonelli-pasqualetto-pozzetta-semola2022}*{Rem.~3.8}) suggest achronal area-regular TC sets may have an entire range of mean curvature barriers. The subtle issue lies in the failure of smoothness of the conditional densities, as best illustrated  when $\q$ is even a $\CD(0,N)$ disintegration. By  the one-sided continuity properties of $\smash{(\log h_\alpha)'^\pm}$ for $\q$-a.e.~$\alpha\in Q$ from \cref{Re:Prop CD},  \eqref{Eq:BAR} reduces to an upper bound on $\smash{H^+_\Sigma}$ and a lower bound on $\smash{-H^-_\Sigma}$ by $H_0$, respectively. On the other hand,  \cref{Re:Prop CD} again yields 
\begin{align*}
H_\Sigma = -H_\Sigma^-\quad \hh_0\mres S\textnormal{-a.e.} 
\end{align*}
Thus, the previous estimates give no information about the constancy of $H_\Sigma$ unless  the two one-sided  conditional logarithmic derivatives coincide.

Of course, one can define $\Sigma$ to be a CMC set if $\smash{H_\Sigma}$ is equal to a given constant $\hh_0$-a.e. However, unlike \cref{Def:Barriers} this notion is highly unstable.\hfill{\footnotesize{$\blacksquare$}}
\end{remark}

\subsection{Volume and area estimates by mean curvature}\label{Sub:HKIn} 

\subsubsection{Volume control of Heintze--Karcher-type} Let us now  recall \cref{Sub:FRAME} for the definition of  $l$-geodesic convexity of an arbitrary given subset of $\mms$. Under our standing hypotheses from \cref{Sub:MMS}, $\mms$ can be exhausted by (closed or open) causal emeralds, cf.~Braun--McCann \cite{braun-mccann2023}*{Rem.~2.30}, which will occasionally be used below without explicit notice. Also recall the correspondence between the $l$-diameter of $\mms$ and $k$ and $N$ from the sharp Bonnet--Myers theorems of Cavalletti--Mondino \cite{cavalletti-mondino2020}*{Prop.~5.10} and Braun--McCann \cite{braun-mccann2023}*{Cor.~7.7}.  Moreover, given $s,t\in \R$ with $s<t$ we define the sets
\begin{align*}
\Sigma_{[s,t]} := E_\Sigma \cap \{s\leq l_\Sigma \leq t\}.
\end{align*}
We abbreviate $\smash{\Sigma_{[t,t]}}$ by $\Sigma_t$. Lastly, given $z\in\Sigma$ we abbreviate
\begin{align*}
k_z := \hat{k}_{\Quot(z)}.
\end{align*}

We recall \cref{Re:Convent endpts} for the interpretation of all  terms below beyond endpoints.

\begin{theorem}[Heintze--Karcher-type inequality]\label{Th:HeintzeKarcher} Let $\scrM$ be a timelike essentially nonbranching $\smash{\TMCP^e(k,N)}$ metric measure spacetime. Assume $\Sigma$ is an achronal area-regular TC Borel subset of $\mms$. Let $\q$ constitute an $\MCP(k,N)$ disintegration $\q$ relative to $l_\Sigma$ and $E_\Sigma$. If $\Sigma$ has finite  curvature, then every $\smash{s,t\in(-\diam^l\mms,\diam^l\mms)}$  with $s<t$ such that $\smash{\vert t-s\vert \leq \diam^l\mms}$ and every pre\-compact, $l$-geodesically convex Borel subset $A$ of $\mms$ obey
\begin{align*}
\meas[\Sigma_{[s,t]} \cap A] \leq \int_{\sfp_0\circ\Quot(A)}\int_{[s,t]}\sfJ_{k_z,N,H_\Sigma(z)}\d \Leb^1\d\hh_0(z),
\end{align*}
and in particular
\begin{align*}
\meas[\Sigma_{[s,t]}] \leq \int_\Sigma\int_{[s,t]} \sfJ_{k_z,N,H_\Sigma(z)} \d\Leb^1\d\hh_0(z).
\end{align*}
\end{theorem}

More strongly,  in \cref{Th:HeintzeKarcher} the quantity $\smash{H_\Sigma}$ can be replaced by $\smash{H_\Sigma^+}$ if $s$ and $t$ are nonnegative and by $-H^-$ if $s$ and $t$ are nonpositive, respectively.

\begin{proof}[Proof of \cref{Th:HeintzeKarcher}] We only treat the case $s$ is negative and $t$ is positive, the others are dealt with analogously. 

Since $\Sigma$ has finite curvature, by the disintegration \cref{Th:Disintegration} and \cref{Re:Properties MCP} we get that for $\q$-a.e.~$\alpha\in Q$ the right derivative $\smash{(\log \hat{h}_\alpha)'^+(0)}$ is a real number. To all such $\alpha$, \cref{Le:CompII} applies and gives the following estimate on $I_\alpha$:
\begin{align*}
\hat{h}_\alpha \leq \hat{h}_\alpha(0)\,\sfJ_{\hat{k}_\alpha,N,(H_\Sigma)_0(\alpha)}.
\end{align*}
The disintegration \cref{Th:Disintegration} and \cref{Re:Convent endpts} imply
\begin{align*}
\meas[\Sigma_{[0,t]}\cap A] &= \int_{\Quot(A)}\int_{I_\alpha \cap[0,t]} \hat{h}_\alpha \d \Leb^1\\
&\leq \int_{\Quot(A)}\int_{I_\alpha\cap[0,t]} \sfJ_{\hat{k}_\alpha,N, (H_\Sigma^+)_0(\alpha)}\d\Leb^1\,\hat{h}_\alpha(0)\d\q(\alpha)\\
&\leq \int_{\sfp_0\circ\Quot(A)}\int_{[0,t]} \sfJ_{k_z,N,H_\Sigma^+(z)}\d\Leb^1\d\hh_0(z).
\end{align*}
Replacing $[0,t]$ with $[s,0]$ above gives 
\begin{align*}
\meas[\Sigma_{[s,0]}\cap A] \leq \int_{\sfp_0 \circ\Quot(A)}\int_{[s,0]} \sfJ_{k_z,N,-H_\Sigma^-(z)}\d\Leb^1\d\hh_0(z),
\end{align*}
and recalling $\meas[\Sigma] = 0$ by \cref{Sub:MMS} readily yields the claim by adding these estimates up.

The second claim follows from Levi's monotone convergence theorem.
\end{proof}

The above statement simplifies if both curvature bounds are constant. For the second part, we recall the future timelike Minkowski content of $\Sigma$ defined in \cref{Sub:Loci}.

\begin{corollary}[Minkowski content vs.~surface measure]\label{Cor:SMeasMink} In the framework of the previous \cref{Th:HeintzeKarcher}, assume $k$ and $H$ are bounded from below and above by numbers $K$ and $H_0$, respectively. Then for every $s,t\in (-\diam^l\mms,\diam^l\mms)$ with $s<t$,
\begin{align*}
\meas[\Sigma_{[s,t]}] &\leq \hh_0[\Sigma]\int_{[s,t]}\sfJ_{K,N,H_0}\d\Leb^1,\\
\meas^\pm[\Sigma] &\leq \hh_0[\Sigma].
\end{align*}

If in addition, $K$ is nonnegative and $H_0$ is nonpositive,
\begin{align*}
\meas[\Sigma_{[s,t]}] \leq \diam^l\mms\,\hh_0[\Sigma].
\end{align*}
\end{corollary}

\begin{remark}[Connection to \cite{cavalletti-mondino2024}] The latter estimate should be compared to the recent isometric-type inequality \cite{cavalletti-mondino2024}*{Thm.~4.11} proven by Cavalletti--Mondino. They obtain a \emph{lower} bound on the volume of the conical region between two appropriate achronal sets in terms of their mutual ``distance'' and the timelike Minkowski content of the set which lies in the chronological past of the other. As explained in \cite{cavalletti-mondino2024}*{§1}, unlike the Riemannian case the Lorentzian isoperimetric problem asks for the optimal area to enclose a given volume. Hence, our results from \cref{Th:HeintzeKarcher} et seq.~should not be interpreted as a Lorentzian isoperimetric-type inequality. Rather, they turn out to be very effective and general means to predict volume singularities, cf.~\cref{Sub:Vol sing} below.\hfill{\footnotesize{$\blacksquare$}}
\end{remark}

The same argument as for Ketterer \cite{ketterer2020-heintze-karcher}*{Cor.~1.4} yields the following. Given a number $K$, we recall the definition of $\pi_{K/(N-1)}$ as the first nonzero root of $\smash{\SIN_{K/(N-1)}}$ from \cref{Sub:Var dist coeff}.  It is the Lorentzian analog to a result from Heintze--Karcher's work \cite{heintze-karcher1978}*{Thm.~2.2}.

\begin{corollary}[Uniformly positive curvature]\label{Cor:Unifpos} In the framework of the previous \cref{Th:HeintzeKarcher}, assume $k$ is bounded from below by a positive number $K$. Then
\begin{align*}
\meas[E_\Sigma] &\leq \int_0^{\pi_{K/(N-1)}} \sin^{N-1}\Big[\sqrt{\frac{K}{N-1}}\Big]\d r\\
&\qquad\qquad\times\int_\Sigma\Big[\frac{K}{N-1} + \Big[\frac{H_\Sigma(z)}{N-1}\Big]^2\Big]^{(N-1)/2}\d\hh_0(z).
\end{align*}
\end{corollary}

\subsubsection{Area control} With the same techniques as above, we show the following.

\begin{theorem}[Area inequality]\label{Th:AREA} In the framework of \cref{Th:HeintzeKarcher},
\begin{align*}
\hh_t[A\cap \Sigma_t] \leq \int_{\sfp_0\circ\Quot(A)} \sfJ_{k_z,N,H_\Sigma(z)}\d\hh_0(z),
\end{align*}
and in particular
\begin{align*}
\hh_t[\Sigma_t] \leq \int_\Sigma\sfJ_{k_z,N,H_\Sigma(z)}\d\hh_0(z).
\end{align*}
\end{theorem}

\begin{proof} The first follows from the identity
\begin{align*}
\hh_t[A \cap \Sigma_t] = \int_{\Quot(A)} \hat{h}_\alpha(t)\d\q(\alpha),
\end{align*}
recalling our conventions from \cref{Re:Convent endpts}, and an argument analogous to the proof of \cref{Th:HeintzeKarcher} based on \cref{Le:CompII}.

The second estimate follows once again from Levi's monotone convergence theorem by the arbitrariness of $A$. This applies although $\hh_t$ may fail to be a Radon measure.
\end{proof}

\subsection{Volume singularity theorems}\label{Sub:Vol sing} Finally, we use our \cref{Th:HeintzeKarcher} to predict volume incompleteness of spacetimes. The terminology of volume incompleteness has been recently proposed by García-Heveling \cite{garcia-heveling2023-volume} (partly inspired by the volume comparison results of Treude--Grant \cite{treude-grant2013}  which were in turn motivated from Heintze--Karcher's work  \cite{heintze-karcher1978}). It complements the classical approach to ``singularities'' by geodesic incompleteness and covers more general circumstances even in the smooth case, cf.~\cref{Re:Finite vol,Re:Comments}. In general, the proposal from \Cref{Def:Vol sing} and the more traditional understanding of a ``singularity'' in terms of geodesic incompleteness are logically independent \cite{garcia-heveling2023-volume}*{Exs.~2.3, 2.4}. For more comments on their interrelation, see \cref{Re:Comments} below.

\begin{definition}[Future volume incompleteness \cite{garcia-heveling2023-volume}*{Def.~1.1}]\label{Def:Vol sing} We call $\scrM$ \emph{future volume incomplete} if for every $\varepsilon > 0$ there exists a point $x\in \mms$ such that
\begin{align*}
\meas[I^+(x)]\leq \varepsilon.
\end{align*}
\end{definition}

\begin{remark}[Finite volume characterization]\label{Re:Finite vol} Under our standing hypotheses from \cref{Sub:MMS}, $\scrM$ is future volume incomplete if and only if there exists a point $x\in\mms$ with $\meas[I^+(x)] < \infty$. The proof of \cite{garcia-heveling2023-volume}*{Thm.~2.1} (which combines the nonexistence of closed timelike curves with the full support of the canonical volume measure) carries over without any change.\hfill{\footnotesize{$\blacksquare$}}
\end{remark}

Now we prove several volume singularity theorems. 

First, we give a nonsmooth extension of the cosmological volume singularity theorem under uniform curvature bounds for certain Cauchy hypersurfaces in globally hyperbolic spacetimes by García-Heveling \cite{garcia-heveling2023-volume}*{Thm.~5.2}. It addresses the SEC version of \cite{garcia-heveling2023-volume}*{Conj. 4.3} affirmatively in the nonsmooth case.

\begin{theorem}[Constant volume singularity theorem]\label{Th:Const vol} Assume that $\scrM$ forms a timelike essentially nonbranching $\TMCP^e(K,N)$ metric measure spacetime for a constant $K$. Suppose $\mms$ contains an achronal area-regular FTC Borel subset $\Sigma$  with $\hh_0[\Sigma] < \infty$ whose mean curvature is bounded from above by a number $H_0$. Then 
\begin{align*}
\meas[I^+(\Sigma)] < \infty
\end{align*}
provided
\begin{enumerate}[label=\textnormal{\alph*\textcolor{black}{.}}]
\item $K>0$ and $H_0$ is arbitrary,
\item $K=0$ and $H_0 < 0$, or
\item\label{eq:Bkajsbdab} $K<0$ and $\smash{H_0 \leq -\sqrt{-K(N-1)}}$.
\end{enumerate}

In particular, $\scrM$ is future volume incomplete.
\end{theorem}

\begin{proof} The second statement follows from the first and \cref{Re:Finite vol}, while the first claim is shown by distinguishing the three named cases. 

The case $K>0$ immediately follows from \cref{Cor:Unifpos}.

Now we assume $K=0$ and $H_0 < 0$. In this case, the Jacobian  function takes the form
\begin{align*}
\sfJ_{0,N,H_0}(\theta) = \Big[1 + \frac{H_0}{N-1}\,\theta\Big]_+^{N-1}
\end{align*}
This function is positive on $[0,\theta_0)$ and vanishes identically on $[\theta_0,\infty)$, where
\begin{align*}
\theta_0 := \frac{N-1}{-H_0}.
\end{align*}
In particular, it is uniformly bounded from above on all of $\R$, say by $C$. Given any $t > \theta_0$, \cref{Cor:Unifpos} thus leads to
\begin{align*}
\meas[\Sigma_{[0,t]}] \leq  \int_\Sigma\int_{[0,\theta_0]} \sfJ_{K,N,H_0}\d\Leb^1\d\hh_0 \leq C\,\theta_0\,\hh_0[\Sigma].
\end{align*}
Levi's monotone convergence theorem provides the claim.

The  situation $K<0$ and $\smash{H_0 < -\sqrt{-K(N-1)}}$ is argued in a similar manner. Here the Jacobian function is given by
\begin{align}\label{Eq:JAV!}
\sfJ_{K,N,H_0}(\theta) = \Big[\!\cosh\!\Big[\sqrt{\frac{-K}{N-1}}\theta\Big] + \frac{H_0}{N-1}\,\sqrt{\frac{N-1}{-K}}\sinh\!\Big[\sqrt{\frac{-K}{N-1}}\theta\Big]\Big]_+^{N-1}.
\end{align}
It is positive on $[0,\theta_0)$ and vanishes identically on $[\theta_0,\infty)$, where
\begin{align*}
\theta_0 := \sqrt{\frac{N-1}{-K}}\coth^{-1}\!\Big[\frac{-H_0}{\sqrt{-K(N-1)}}\Big].
\end{align*}
It remains to argue as in the previous step.

In the borderline case $\smash{H_0 = -\sqrt{-K(N-1)}}$ of hypothesis \ref{eq:Bkajsbdab}, \eqref{Eq:JAV!} simplifies to
\begin{align*}
\sfJ_{K,N,H_0}(\theta) = \rme^{H_0\theta}.
\end{align*}
Levi's monotone convergence theorem and \cref{Cor:Unifpos} thus imply
\begin{align*}
\meas[I^+(\Sigma)] = \lim_{t\to \infty} \meas[\Sigma_{[0,t]}]  \leq \hh_0[\Sigma]\int_0^\infty \rme^{H_0\theta}\d\theta.
\end{align*}
Since $H_0$ is negative, the right-hand side is finite. 
\end{proof}

\begin{remark}[Comparison to \cite{cavalletti-mondino2020,garcia-heveling2023-volume}]\label{Re:Comments} The values of $\theta_0$ in the two intermediate steps of the above proof are exactly the upper estimates on the $l$-diameter of $I^+(\Sigma)$ predicted by Cavalletti--Mondino's synthetic Hawking-type singularity theorem \cite{cavalletti-mondino2020}*{Thm.~5.6} (under the same assumptions on $K$ and $H_0$). In these two cases, the proof of the smooth analog of \cref{Th:Const vol} \cite{garcia-heveling2023-volume}*{Thm.~5.2} relies on these uniform bounds. 

On the other hand, our more direct proof does not use  \cite{cavalletti-mondino2020}*{Thm.~5.6} (although our methods  of course compare to those of \cite{treude-grant2013,cavalletti-mondino2020,garcia-heveling2023-volume}). This stresses  the logical independence of volume and geodesic incompleteness indicated above.

Let us also comment on the borderline case $\smash{H_0 = -\sqrt{-K(N-1)}}$ in assumption \ref{eq:Bkajsbdab} above. It is not covered by \cite{cavalletti-mondino2020}*{Thm.~5.6} since, in fact, geodesic incompleteness is false here in general (e.g.~on model spaces, cf.~Treude--Grant \cite{treude-grant2013}*{Def.~4.4}). On the other hand, \cref{Th:Const vol} and its smooth predecessor by García-Heveling \cite{garcia-heveling2023-volume}*{Thm.~5.2} still entail volume incompleteness. In fact, this mixture of volume incompleteness and geodesic completeness under the stated relation of Ricci and mean curvature is rigid: it \emph{only} happens on model spaces by Andersson--Galloway \cite{andersson-galloway2002}*{Prop.~3.4} (see also Galloway--Woolgar \cite{galloway-woolgar2014}*{Thm.~1.3} for the Bakry--Émery case).\hfill{\footnotesize{$\blacksquare$}}
\end{remark}

Now we turn to volume singularity theorems variable curvature bounds. To the best of our knowledge, the three main  \cref{Th:Vol2,Th:Vol1,Th:Vol3} are entirely new. They confirm the hope to predict volume singularities under such hypotheses expressed by García-Heveling \cite{garcia-heveling2023-volume}*{§5}.

The following two facts should be compared to the recent  synthetic singularity theorems predicting geodesic incompleteness by Braun--McCann \cite{braun-mccann2023}*{Cor.~7.8, Thm.~A.7}.

\begin{theorem}[Variable volume singularity theorem I, see also \cref{Th:Vol1,Th:Vol3}]\label{Th:Vol2} Assume $\scrM$ is a timelike essentially nonbranching $\TMCP^e(k,N)$ metric measure spacetime. Suppose it contains an achronal area-regular FTC Borel subset $\Sigma$ with $\hh_0[\Sigma] < \infty$ whose forward mean curvature of $\Sigma$ is bounded from above by a number $H_0$. Moreover, assume there exist $\varepsilon > 0$ and $c > 0$ with 
\begin{align*}
\frac{2c}{\varepsilon} + \frac{H_0}{N-1} < 0
\end{align*}
such that everywhere on $I^+(\Sigma)$,
\begin{align*}
k_- \leq c\,\big[\varepsilon^{-2} \wedge l_\Sigma^{-2}\big].
\end{align*}
Then
\begin{align*}
\meas[I^+(\Sigma)] < \infty.
\end{align*}

In particular, $\scrM$ is future volume incomplete.
\end{theorem}

\begin{proof} Again we apply \cref{Th:HeintzeKarcher} and obtain for every $t >\varepsilon$ that
\begin{align}\label{Eq:Jac bounds}
\meas[\Sigma_{[0,t]}] \leq \int_\Sigma\int_{[0,t]}\sfJ_{-(k_-)_z,N,H_0}\d\Leb^1 \d\hh_0(z).
\end{align}
Given $z\in\Sigma$, we now take a closer look at the involved Jacobian function, viz.~
\begin{align*}
\sfJ_{-(k_-)_z,N,H_0}(\theta) &= \Big[1+\int_0^\theta (k_-)_z(r)\,v_z(r)\d r + \frac{H_0}{N-1}\,v_z(\theta)\Big]_+^{N-1}.
\end{align*}
The hypothesis on $k_-$ implies $\smash{k_- \leq c\varepsilon^{-2}}$ on the tube $\smash{\Sigma_{[0,\varepsilon]}}$. Sturm's comparison theorem as stated e.g.~in  \cite{ketterer2017}*{Thm.~3.1}) implies $r \leq v_z(r) \leq \sin_{-c\varepsilon^{-2}/(N-1)}$ on $[0,\varepsilon]$; moreover, the first inequality in fact holds on all of $[0,\infty)$.

In the case $\theta \leq \varepsilon$, using the  negativity of $H_0$ we thus get
\begin{align*}
\sfJ_{-(k_-)_z,N,H_0}(\theta) \leq \Big[1 + \frac{c}{\varepsilon}\SIN_{-c\varepsilon^{-2}/(N-1)}(\varepsilon)\Big]^{N-1}.
\end{align*}

On the other hand, if $\theta > \varepsilon$ we have
\begin{align*}
\sfJ_{-(k_-)_z,N,H_0}(\theta) &\leq \Big[1+\Big[\!\int_0^\varepsilon (k_-)(r)\d r + \int_\varepsilon^\theta (k_-)(r)\d r + \frac{H_0}{N-1}\Big]\, v_z(\theta)\Big]_+^{N-1}\\
&\leq \Big[1 + \Big[\frac{c}{\varepsilon} + c\int_\varepsilon^\theta r^{-2}\d r + \frac{H_0}{N-1}\Big]\,v_z(\theta)\Big]_+^{N-1}\\
&\leq \big[1 - \theta_0^{-1}\,\theta\big]_+^{N-1}.
\end{align*}
In the last step, we used the hypothesized negativity of
\begin{align*}
\theta_0 := \Big[\frac{2c}{\varepsilon} + \frac{H_0}{N-1}\Big]^{-1}.
\end{align*}

In summary, we have shown the Jacobian function is uniformly bounded and it vanishes identically on $[\theta_0,\infty)$. The conclusion follows from \eqref{Eq:Jac bounds} as for \cref{Th:Const vol}.
\end{proof}

\begin{theorem}[Variable volume singularity theorem II, see also \cref{Th:Vol2,Th:Vol3}]\label{Th:Vol1} Let $\scrM$ be a timelike essentially nonbranching $\TMCP^e(k,N)$ metric measure spacetime. Assume it contains an achronal area-regular FTC Borel subset $\Sigma$ with $\hh_0[\Sigma] < \infty$, $k$ is bounded from below by a negative constant $K$ on $I^+(\Sigma)$, $\q$-a.e.~ray has infinite $l$-length, and the forward mean curvature of $\Sigma$ is bounded from above by a number $H_0$ with
\begin{align*}
\limsup_{t\to \infty} c(t)^{-1}\,\hh_0[\Sigma]^{-1}\int_{\Sigma_{[0,t]}} k_-\d\meas < \frac{-H_0}{N-1},
\end{align*}
where for some $\delta >0$,
\begin{align*}
c(t) := \frac{\sinh(\sqrt{-K(N-1)}\,(t+\delta))^{N-1}}{\sinh(\sqrt{-K(N-1)}\,(\delta))^{N-1}}.
\end{align*}
Then there exists an $\hh_0$-measurable and $\hh_0$-nonnegligible subset $\Sigma'$ of $\Sigma$ such that
\begin{align*}
\meas[I^+(\Sigma')] < \infty.
\end{align*}

In particular, $\scrM$ is future volume incomplete.
\end{theorem}

The lower boundedness of $k$ on the chronological future of $\Sigma$ can be dropped at the cost of a more complicated expression for $c$, cf.~Braun--McCann \cite{braun-mccann2023}*{Prop.~7.13, Thm.~7.14}. 

\begin{proof}[Proof of \cref{Th:Vol1}] The theorem is a simple consequence of the synthetic Cheeger--Colding segment inequality established by Braun--McCann \cite{braun-mccann2023}*{Prop.~7.13}. Indeed, the hypothesis on the lengths of the transport rays makes it applicable to every $t>0$ and the given parameter $\delta$ gives
\begin{align}\label{Eq:Segm}
\hh_0\textnormal{-}\!\essinf_{z\in \Sigma} \int_0^t (k_-)_z(r)\d r \leq c(t)^{-1}\,\hh_0[\Sigma]^{-1}\int_{\Sigma_{[0,t]}}k_-\d\meas.
\end{align}

Let $\theta_0 > 0$ such that for all sufficiently large $t>0$,
\begin{align*}
c(t)^{-1}\,\hh_0[\Sigma]^{-1}\int_{\Sigma_{[0,t]}} k_-\d\meas \leq  \frac{-H_0}{N-1} - 2\theta_0^{-1}.
\end{align*}
By \eqref{Eq:Segm}, there exists an $\hh_0$-measurable and $\hh_0$-nonnegligible subset $\Sigma'$ of $\Sigma$ such that for all sufficiently large $t>0$, every $z\in \Sigma'$ satisfies
\begin{align*}
\int_0^t (k_-)_z(r)\d r \leq \frac{-H_0}{N-1} - \theta_0^{-1}.
\end{align*}

Applying \cref{Th:HeintzeKarcher} to $\Sigma'$ in place of $\Sigma$ entails
\begin{align*}
\meas[\Sigma_{[0,t]}'] \leq \int_{\Sigma'}\int_{[0,t]} \sfJ_{-(k_-)_z,N,H_0}\d\Leb^1\d\hh_0(z).
\end{align*}
As in the above proof, for every $z\in \Sigma'$ the inherent Jacobian function is estimated by
\begin{align*}
\sfJ_{-(k_-)_z,N,H_0}(\theta) &= \Big[1+\int_0^\theta(k_-)_z(r)\,v_z(r)\d r + \frac{H_0}{N-1}\,v_z(\theta)\Big]_+^{N-1}\\
&\leq \Big[1+\Big[\!\int_0^t (k_-)_z(r) + \frac{H_0}{N-1}\Big]\,v_z(\theta)\Big]_+^{N-1}\\
&\leq \big[1-\theta_0^{-1}\,\theta\big]_+^{N-1}.
\end{align*}
As in the proof of \cref{Th:Const vol}, this provides the claim.
\end{proof}

In light of \cref{Re:Comments} and the hypothesis of geodesic completeness in the previous \cref{Th:Vol1}, it is an interesting question how rigid the latter result is.

Lastly, we present the following simple result which assumes an integrability condition on the negative part of $k$ with respect to a Borel measure which is mutually equivalent to $\meas$ on the relevant chronological future.

\begin{theorem}[Variable volume singularity theorem III, see also \cref{Th:Vol1,Th:Vol2}]\label{Th:Vol3} Let $\scrM$ form a timelike essentially nonbranching $\TMCP^e(k,N)$ metric measure spacetime. Assume it contains an achronal area-regular FTC Borel subset $\Sigma$ with $\hh_0[\Sigma] < \infty$ whose forward mean curvature is bounded from above by a number $H_0$. Lastly, assume
\begin{align*}
\hh_0[\Sigma]^{-1}\int_{I^+(\Sigma)} k_-\d\mathfrak{n}  < \frac{-H_0}{N-1},
\end{align*}
where the Radon measure $\mathfrak{n}$ on $I^+(\Sigma)$ is defined by
\begin{align*}
\mathfrak{n} := \int_\Sigma\Leb_{\Quot(z)}^1\d\hh_0(z).
\end{align*}
Then there is an $\hh_0$-measurable and $\hh_0$-nonnegligible subset $\Sigma'$ of $\Sigma$ with
\begin{align*}
\meas[I^+(\Sigma')] <\infty.
\end{align*}

In particular, $\scrM$ is future volume incomplete.
\end{theorem}

\begin{proof} We fix $\theta_0 > 0$ such that
\begin{align*}
\hh_0[\Sigma]^{-1}\int_{I^+(\Sigma)} k_-\d\mathfrak{n} < \frac{-H_0}{N-1} - \theta_0^{-1}.
\end{align*}
In particular, the right-hand side is positive, and we abbreviate
\begin{align*}
\delta := \frac{-H_0}{N-1}-\theta_0^{-1}.
\end{align*}

We apply Markov's inequality to the probability measure $\hat{\hh}_0 := \hh_0[\Sigma]^{-1}\,\hh_0$ and get
\begin{align*}
\hh_0\Big[\!\int_0^\infty (k_-)_\cdot(r)\d r\geq \delta\Big] \leq \frac{1}{\delta}\int_{I^+(\Sigma)} k_-\d\mathfrak{n} < \hh_0[\Sigma].
\end{align*}
In other words, the $\hh_0$-measurable set
\begin{align*}
\Sigma' := \Big\lbrace\!\int_0^\infty (k_-)_\cdot(r)\d r < \delta\Big\rbrace
\end{align*}
has positive $\hh_0$-measure, hence is nonempty. 

Given $t>0$, applying \cref{Th:HeintzeKarcher} to $\Sigma'$ in place of $\Sigma$ gives
\begin{align*}
\meas[\Sigma_{[0,t]}'] \leq \int_{\Sigma'}\int_{[0,t]}\sfJ_{-(k_-)_z,N,H_0}\d\Leb^1 \d\hh_0(z).
\end{align*}
By definition of $\Sigma'$ and arguing as in the previous proofs, every $z\in\Sigma'$ satisfies
\begin{align*}
\sfJ_{-(k_-)_z,N,H_0}(\theta) &= \Big[1+\int_0^\theta (k_-)_z \,v_z(r)\d r +  \frac{H_0}{N-1}\,v_z(\theta)\Big]_+^{N-1}\\
&\leq \Big[1 + \Big[\delta + \frac{H_0}{N-1}\Big]\,v_z(\theta) \Big]_+^{N-1}\\
&\leq \big[1-\theta_0^{-1}\,\theta\big]_+^{N-1}.
\end{align*}
Arguing as for \cref{Th:Const vol}, this gives the desired inequality.
\end{proof}

\appendix

\section{Complementary material}

\subsection{Topological local anti-Lipschitz condition for Finsler spacetimes}\label{Sub:LorentzFinsler} The goal of this part is to verify the topological anti-Lipschitz condition of   \cref{Def:KS} for appropriate signed Lorentz distance functions on  Finsler spacetimes. In particular, our central \cref{Th:11} and \cref{Cor:22} (and, in the reversible case, \cref{Th:UnsigndalemI} and \cref{Th:UnsigndalemII}) hold in their strongest form. Even more, in \cref{Sub:Conclusion} we show that the integration by parts inequalities stipulated therein after \cref{Re:Two sided} actually become \emph{equalities}. This naturally leads to the condition of \emph{infinitesimally  strict concavity}, inspired by Gigli's ``infinitesimally strict convexity'' from metric measure geometry \cite{gigli2015}.

\subsubsection{Finsler spacetimes} We first recall some basic notions of Lorentz--Finsler geometry, referring to Minguzzi \cite{minguzzi2015-sprays,minguzzi2015-light,minguzzi2015-raychaudhuri,minguzzi2019-applications,minguzzi2019-causality} and Lu--Minguzzi--Ohta \cite{lu-minguzzi-ohta2022-range} for details.

Throughout the sequel, $\mms$ forms a connected smooth topological manifold. We fix an auxiliary Riemannian metric $r$ generating its topology; as all statements involving it will be local, they are independent of the choice of $r$. In order for the Legendre transform from \cref{Sub:Legendre} to have suitable smoothness properties, we suppose $\dim \mms$ is at least $3$. Lastly, we fix a \emph{smooth} measure $\meas$ on $\mms$, i.e.~it is mutually absolutely continuous with respect to the Lebesgue measure with smooth density on every chart. As in \cref{Ex:TCD}, this induces a canonical metric measure spacetime structure $\scrM$.

We call $\mms$ a  \emph{Lorentz--Finsler manifold} if it is endowed with a function $L$  on the tangent bundle $T\mms$ such that
\begin{itemize}
\item $L$ is smooth on the slit tangent bundle $T\mms\setminus\{0\}$,
\item for every $v\in T\mms$ and every $c > 0$, we have $L(cv)=c^2L(v)$, and
\item for every $\smash{v\in T\mms\setminus\{0\}}$, the symmetric $\dim\mms\times\dim\mms$-matrix with entries
\begin{align*}
\Rmet_{\alpha\beta}(v) := \frac{\partial^2L}{\partial v^\alpha\partial v^\beta}(v)
\end{align*}
is nondegenerate with signature $+,-,\dots,-$.
\end{itemize}
The function $L$ is assumed to be understood in the following. We will call $\mms$  \emph{reversible} if $L(v) = L(-v)$ for every $v\in T\mms$.

Every $v\in T\mms \setminus \{0\}$  induces a Lorentzian metric $\Rmet_v$ by the assignment
\begin{align*}
\Rmet_v\Big[a^\alpha\frac{\partial}{\partial x^\alpha}, b^\beta\frac{\partial}{\partial x^\beta}\Big] :=  g_{\alpha\beta}(V)\,a^\alpha\,b^\beta
\end{align*}
in local coordinates  (using Einstein summation convention). Then Euler's homogeneous function theorem yields $\smash{g_v(v,v) = 2L(v)}$.

We call a tangent vector $v\in T\mms$
\begin{itemize}
\item \emph{timelike}, symbolically $v\in \Omega'$, if $L(v) > 0$,
\item \emph{lightlike} if $L(v) = 0$ and $v$ is nonzero,
\item \emph{causal} if it is timelike or lightlike, and
\item \emph{spacelike} if it is not causal, i.e.~$L(v) < 0$ or $v$ vanishes.
\end{itemize}
For a causal vector $v\in T\mms$, we abbreviate
\begin{align*}
F(v) :=  \sqrt{g_v(v,v)} = \sqrt{2L(v)}.
\end{align*}

\begin{definition}[Finsler spacetime]\label{Def:Finsler spt} We call $\mms$ a \emph{Finsler spacetime} \textnormal{(}or \emph{time-oriented}\textnormal{)} if it is a Lorentz--Finsler manifold which admits a smooth timelike vector field $X$.
\end{definition}

A Lorentz--Finsler manifold can always be time-oriented.

Let $X$ be as in \cref{Def:Finsler spt}. We call $v\in T_x\mms$ \emph{future-directed} if it lies in the same connected component of $(\cl\,\Omega_x')\setminus \{0\}$ as  $X(x)$,  where $x\in\mms$, and \emph{past-oriented} if $-v$ is future-directed. Let $\Omega$ be the subset of future-directed elements of $\Omega'$. A curve $\gamma$ in $\mms$ is called timelike if it is continuously differentiable and its tangent vectors all belong to $\Omega$; we define causal curves analogously. Thus, unless explicitly stated otherwise every   causal curve is assumed to be future-directed.

\subsubsection{Legendre transform}\label{Sub:Legendre} On a genuine Lorentz spacetime, the identification between tangent and cotangent spaces is straightforward by employing the Lorentzian metric. The generalization thereof to Lorentz--Finsler geometry goes through the Legendre transform introduced and studied by Minguzzi \cite{minguzzi2015-light}. We also  refer to Lu--Minguzzi--Ohta \cite{lu-minguzzi-ohta2022-range}*{§4.4}, whose presentation we follow (modulo switched signature), for a comprehensive overview.

The \emph{polar cone} at $x\in\mms$ is the set $\Omega_x^*$ of all $\zeta\in T^*\mms$ such that $\zeta(v) > 0$ for every $\smash{v\in (\cl\,\Omega_x)\setminus \{0\}}$. It is an open convex cone in $T^*\mms$, cf.~Minguzzi \cite{minguzzi2015-light}*{Cor.~2}. For $\smash{\zeta\in\Omega^*_x}$, we define the quantity
\begin{align*}
L^*(\zeta) := \inf \zeta(v)^2,
\end{align*}
where the infimum is taken over all $v\in\Omega_x$ with $F(v) = 1$. This definition easily implies the reverse Cauchy--Schwarz inequality $\smash{4\,L^*(\zeta)\,L(v) \leq \zeta(v)^2}$ for every $\zeta\in \Omega_x^*$ and every $v\in \Omega_x$, cf.~Minguzzi \cite{minguzzi2015-light}*{Thm.~3}. The map $L^*$ inherits the smoothness properties from $L$. 

At every given $\zeta\in T^*M\setminus \{0\}$, the smooth map $\smash{L^*}$ induces a cometric of Lorentzian signature by the  formula 
\begin{align}\label{Eq:Cometricg}
\Rmet_\zeta^*\big[v_\alpha\d x^\alpha, w_\beta\d x^\beta\big] :=  {\Rmet^*}^{\alpha\beta}(\zeta)\,v_\alpha\,w_\beta,
\end{align}
in local coordinates, where
\begin{align*}
{\Rmet^*}^{\alpha\beta}(\zeta) := \frac{\partial^2L^*}{\partial\zeta^\alpha\partial\zeta^\beta}(\zeta).
\end{align*}

The \emph{Legendre transform} $\smash{\mathcal{L}^*\colon\Omega_x^* \to \Omega_x}$ is then variationally defined as follows. Given $\smash{\zeta\in \Omega_x^*}$, $\smash{\mathcal{L}^*(\zeta)}$ is the unique vector $v\in \Omega_x$ with $L^*(\zeta) = L(v) = \zeta(v)/2$\footnote{The uniqueness of $v$ follows from strict convexity of the superlevel sets of $F$ in $\Omega_x$.}. Analogously, one can define a map $\smash{\mathcal{L}\colon \Omega_x \to \Omega_x^*}$ which is the inverse of $\smash{\mathcal{L}^*}$. As $\dim\mms$ is at least $3$ as assumed above, $\smash{\mathcal{L}^*}$ and $\smash{\mathcal{L}}$ are in fact diffeomorphisms outside zero, cf.~Minguzzi \cite{minguzzi2015-light}*{Thm.~6}. Moreover,  in this case  $\smash{\Rmet^*_{\mathcal{L}(v)}}$ is the inverse matrix of the Lorentzian metric $\Rmet_v$ canonically induced by $L$ analogously to \eqref{Eq:Cometricg}  for every $v\in \Omega_x$, as stated e.g.~by Lu--Minguzzi--Ohta \cite{lu-minguzzi-ohta2022-range}*{Lem.~4.11}.

\begin{definition}[Gradient]\label{Def:Gradient} The  \emph{gradient} of a smooth function $\u$ on $\mms$ is defined by
\begin{align*}
\nabla \u := \mathcal{L}^*(\rmd \u).
\end{align*}
\end{definition}

For every smooth functions $\u$ and $f$ on $\mms$, as in Lu--Minguzzi--Ohta \cite{lu-minguzzi-ohta2022-range}*{p.~21} we get
\begin{align}\label{Eq:Cometr}
\Rmet_{\rmd f}^*(\rmd f,\rmd \u) = \rmd f(\nabla \u).
\end{align}

\begin{remark}[Nonlinearity] The differential $\rmd$ acting on smooth functions forms a linear operator since it only depends on the differentiable structure of $\mms$. On the other hand, the gradient operator is not linear unless $\mms$ is genuinely Lorentzian.\hfill{\footnotesize{$\blacksquare$}}
\end{remark}

\subsubsection{Regularity of Lorentz distance functions from  hypersurfaces} Let $\Sigma$ be an achronal spacelike\footnote{To keep the presentation short, we do not recapitulate the meaning of a  spacelike submanifold in the genuine Finsler setting. Instead, we refer to Minguzzi  \cite{minguzzi2015-light} for details.} TC submanifold of $\mms$ or just a point. We now briefly review some folklore regularity properties of $l_\Sigma$ on $\smash{I^+(\Sigma)}$; analogous discussions apply relative to $\smash{I^-(\Sigma)}$. Thorough details for Lorentz spacetimes are given by Treude \cite{treude2011}. For basic implicitly used properties of the exponential map, we refer to Minguzzi \cite{minguzzi2015-sprays}.

Recall from \cref{Sec:Signed} that by a timelike geodesic $\gamma$  every $y\in I^+(\Sigma)$ can be connected to a point $x\in\Sigma$ satisfying $l_\Sigma(y) = l(x,y)$. The set of all such $y$ for which $\gamma$ cannot be extended beyond $y$ is  the \emph{extended future timelike cut locus} $\ETCut^+(\Sigma)$, cf.~McCann  \cite{mccann2020}*{Def.~2.1}. By basic ODE theory, $\smash{\ETCut^+(\Sigma)}$ is closed  \cite{mccann2020}*{Thm.~3.6}.

The following \cref{Th:Properties smooth dist fct} is well-known. We refer to Treude  \cite{treude2011}*{Prop.~3.2.33} and Treude--Grant \cite{treude-grant2013}*{Prop.~2.7} for the case of causally complete submanifolds in Lorentz spacetimes. For $\Sigma$ being a singleton, see also McCann \cite[Thm.~3.6]{mccann2020} and Braun--Ohta \cite{braun-ohta2024}*{Thm.~3.11} in Lorentz and Finsler spacetimes, respectively.

\begin{proposition}[Regularity of $\smash{l_\Sigma}$]\label{Th:Properties smooth dist fct} The following hold.
\begin{enumerate}[label=\textnormal{(\roman*)}]
\item The function $l_\Sigma$ is smooth on $\smash{I^+(\Sigma)\setminus \ETCut^+(\Sigma)}$.
\item Everywhere on $I^+(\Sigma)\setminus \ETCut^+(\Sigma)$, $\nabla l_\Sigma$ is future-directed, timelike, and has unit magnitude with respect to $L$.
\item Given $\smash{y\in I^+(\Sigma)\setminus \ETCut^+(\Sigma)}$, let $\sigma$ the unique proper-time parametrized timelike geodesic from $\Sigma$ to $y$. Then
\begin{align*}
\nabla l_\Sigma(y) = \dot{\sigma}_{l_\Sigma(y)}.
\end{align*}
\end{enumerate}
\end{proposition}

\begin{remark}[Semiconcavity]\label{Re:Semiconcd} As in McCann \cite{mccann2020}*{Prop.~3.4, Thm.~3.5} and its Finslerian successor of Braun--Ohta \cite{braun-ohta2024}*{Prop.~3.9, Thm.~3.10}, $l_\Sigma$ can actually be shown to be locally semiconvex (hence locally Lipschitz) on all of $\smash{I^+(\Sigma)}$. In particular, by \cref{Th:Properties smooth dist fct} and Rademacher's theorem the extended future timelike cut locus is $\meas$-negligible; see also Treude \cite{treude2011}*{Prop.~3.2.32} and Treude--Grant \cite{treude-grant2013}*{Thm.~2.6}.\hfill{\footnotesize{$\blacksquare$}}
\end{remark}

\subsubsection{Conclusion}\label{Sub:Conclusion} We show $l_\Sigma$ is topologically locally anti-Lipschitz on $I^+(\Sigma) \cup I^-(\Sigma)$ according to \cref{Def:KS}. As the arguments are similar for both connected components, we concentrate on $I^+(\Sigma)$. Our arguments are strongly inspired by Sormani--Vega \cite{sormani-vega2016}. In fact, the property stated in \cref{Le:TF lv} below seems to be folklore (at least when $\Sigma$ is a Cauchy hypersurface, e.g.~Sormani--Vega \cite{sormani-vega2016} and Graf--Sormani \cite{graf-sormani2022}), yet we have decided to include a proof.

\begin{proposition}[Topological local anti-Lipschitzness]\label{Le:TF lv} Every point in $I^+(\Sigma)$ has an open neighborhood $U$  with the following property. There is a positive constant $c$ such that for every $x,y\in U$ satisfying $x\leq y$,
\begin{align*}
l_\Sigma(y) - l_\Sigma(x) \geq c\,r(x,y). 
\end{align*}

In particular, $l_\Sigma$ is topologically locally anti-Lipschitz on $\smash{I^+(\Sigma)}$.
\end{proposition}

\begin{proof} We claim for every open subset $U$ compactly contained in $\smash{I^+(\Sigma)}$ there is a constant $c'$ such that $\smash{\vert\nabla^r l_\Sigma\vert_r \leq c'}$ on $\smash{U \setminus \TCut^+(\Sigma)}$. Indeed, otherwise for some such set $U$ there is a sequence $(y_n)_{n\in\N}$ in $\smash{U \setminus \TCut^+(\Sigma)}$ converging to some $y\in \cl\,U$ with $\smash{\vert \nabla^r l_\Sigma\vert_r(y_n)\to\infty}$ as $n\to\infty$. Since $\Sigma$ is FTC and by nontrivial chronology, without loss of generality we may assume $(y_n)_{n\in\N}$ lies in a compact set of the form $\smash{J(C,y^+)}$ for some compact subset $C$ of $\Sigma$ and some $\smash{y^+\in I^+(y)}$. For $n\in\N$, let $\gamma_n$ be a proper-time parametrized past-directed timelike geodesic from $y_n$ to one of its footpoints in $C$ and $\smash{\eta_n}$ its time-reversed (hence future-directed) $r$-arclength parametrization. By the limit curve theorem for inextendible curves, cf.~Minguzzi \cite{minguzzi2008-limit-curve}*{Thm.~3.1}, the sequence $\smash{(\eta_n)_{n\in\N}}$ converges locally uniformly to a future inextendible causal curve $\smash{\eta}$. The latter  lies in $\smash{J(C,y^+)}$; this, however, contradicts nontotal imprisonment, cf.~Minguzzi  \cite{minguzzi2019-applications}*{Prop.~2.10}.

By \cref{Th:Properties smooth dist fct} we know $l_\Sigma$ is smooth on $\smash{I^+(\Sigma)\setminus \TCut^+(\Sigma)}$ with $\smash{L(\nabla l_\Sigma) = 1}$ there. With the previous paragraph, this implies the collection $\Pi$ of all vectors $\smash{\nabla l_\Sigma(x)}$, where $\smash{x\in I^+(\Sigma)\setminus \TCut^+(\Sigma)}$, is locally bounded away from the light cones on $\smash{I^+(\Sigma)}$ in the sense of Sormani--Vega  \cite{sormani-vega2016}*{Def.~4.13}. That is, for every $\smash{x\in I^+(\Sigma)}$ there exists a neighborhood $U$ of it and a  constant $c''$ such that $L(\xi)\geq c''\,(1 \vee \vert \xi\vert_r^2)$ for every $\smash{\xi\in \Pi \cap T^*U}$. 

The claim follows now as in Sormani--Vega \cite{sormani-vega2016}*{Thm.~4.18}.
\end{proof}

\begin{remark}[Smooth finite perturbations]\label{Re:Smth!pert} Let $U$ be as in \cref{Le:TF lv}. Every continuously differentiable function with compact support in $U$ is $c\,\met_r$-Lipschitz continuous. Thus, it is a symmetric finite perturbation of $l_\Sigma$ on $I^+(\Sigma)$. 

In turn, by partition of unity and since the set of symmetric finite perturbations of $l_\Sigma$ on $I^+(\Sigma)$ constitutes  a vector space, this implies every continuously differentiable function with compact support in $I^+(\Sigma)$ belongs to the set $\smash{\scrC_\Sigma^+}$ from \eqref{Eq:scrC}.\hfill{\footnotesize{$\blacksquare$}}
\end{remark}

\begin{proof}[Proof of \cref{Th:Main I Finsler,Th:Main II Finsler}] Recall  $\scrM$ is timelike nonbranching thanks to  the Cauchy--Lipschitz theorem. 

By Braun--Ohta \cite{braun-ohta2024}*{Thm.~5.9, Rem.~5.10}, we know the hypothesized curvature  bound implies $\smash{\TCD_\beta^e(0,N)}$ for \emph{every} $\beta\in (0,1)$ --- in fact, every negative $\beta$ as well according to Beran et al.~\cite{beran-braun-calisti-gigli-mccann-ohanyan-rott-samann+-}*{Rem.~1.2} ---, thus $\smash{\TMCP^e(0,N)}$. In turn, \cref{Le:TF lv} thus makes \cref{Th:11,Cor:22} applicable and yields the distributional ($\PP$-)\-d'Alembertians of the functions $\smash{l_\Sigma^\QQ}$ and $\smash{l_\Sigma}$ are generalized signed Radon measures on $\smash{I^+(\Sigma)\cup I^-(\Sigma)}$, where $\PP,\QQ\in (-\infty,1)\setminus \{0\}$ are mutually conjugate.

We are left to show the integration by parts identity. Owing to  \cref{Re:Smth!pert} we assume $\smash{f\in \Cont_\comp^\infty(I^+(\Sigma))}$, the case $\smash{f\in \Cont_\comp^\infty(I^-(\Sigma))}$ is analogous. We claim
\begin{align}\label{Eq:SCONV}
\rmd^+f(\nabla l_\Sigma) = \rmd^-f(\nabla l_\Sigma) = \rmd f(\nabla l_\Sigma)\quad\meas\mres I^+(\Sigma)\textnormal{-a.e.}
\end{align}
However, since $\smash{l_\Sigma}$ and $f$ are smooth outside a closed set of $\meas$-measure zero by \cref{Re:Semiconcd}, it suffices to show the statement in local coordinates around each point where both  are smooth. As $L^*$ is smooth outside zero, it is irrelevant from which side we differentiate it; this implies the first identity. The second follows from  the property \eqref{Eq:Cometr} of $\smash{g^*_{\rmd l_\Sigma}}$.

The discussion from \cref{Re:Two sided} forces equality to hold throughout \cref{Def:DAlem} in both cases, which terminates the proof.
\end{proof}

\subsection{Infinitesimally strict concavity}\label{Sub:Inf s} In \cite{gigli2015}, Gigli already observed an analog of \eqref{Eq:SCONV} in positive signature to yield the integration by parts \emph{identity} for (and the consequential uniqueness of) his distributional Laplacian. This motivates the subsequent  analog of his notion \cite{gigli2015}*{Def.~3.3} (which does not require curvature bounds, but only the material from  \cref{Sub:Sobo,Sub:Horiz}), which makes sense for general metric measure spacetimes after \cref{Sub:MMS}.

\begin{definition}[Infinitesimally strict concavity]\label{Re:Inf str conv} For a nonzero exponent $\PP$ less than $1$, let $\QQ$ denote its conjugate exponent. The metric measure spacetime $\scrM$ will be called \emph{$\QQ$-infinitesimally strictly concave} if for every $l$-causal function $\u$ as in \cref{Sub:Horiz} and every symmetric finite perturbation $f$ of $\u$,
\begin{align}\label{Eq:Sub:asd}
\rmd^+f(\nabla\u)\,\vert \rmd\u\vert^{\PP-2} = \rmd^-f(\nabla\u)\,\vert\rmd\u\vert^{\PP-2}\quad\meas\textnormal{-a.e.}
\end{align}
\end{definition}

The terminology is evident from the observations around the uniqueness properties of the Legendre transform from \cref{Sub:Legendre}.

In the context of \cref{Re:Inf str conv}, we use the $\meas$-a.e.~well-defined abbreviation
\begin{align*}
\rmd f(\nabla\u)\,\vert\rmd\u\vert^{\PP-2} := \rmd^+f(\nabla\u)\,\vert\rmd\u\vert^{\PP-2} = \rmd^-f(\nabla\u)\,\vert\rmd\u\vert^{\PP-2}.
\end{align*}

Of course, for the last arguments of \cref{Sub:LorentzFinsler} it would suffice to require \eqref{Eq:Sub:asd} only for a specific choice of $\u$.

\begin{example}[Basic examples] In the Finsler spacetime setting of \cref{Sub:LorentzFinsler}, the induced metric measure spacetime is $\QQ$-infinitesimally strictly concave for every $\QQ\in (-\infty,1)\setminus \{0\}$. This follows as for \eqref{Eq:SCONV} by observing that real-valued $l$-causal functions and their symmetric finite perturbations are functions of locally bounded variation, cf.~Beran et al.~\cite{beran-braun-calisti-gigli-mccann-ohanyan-rott-samann+-}*{Thm.~A.2}, so that the differential objects appearing in \eqref{Eq:SCONV} make sense $\meas$-a.e.

Moreover, the proof of \cite{beran-braun-calisti-gigli-mccann-ohanyan-rott-samann+-}*{Thm.~4.16} yields $\scrM$ is $\QQ$-infinitesimally strictly concave for every $\QQ$ as above if it is infinitesimally Minkowskian in the sense of  \cite{beran-braun-calisti-gigli-mccann-ohanyan-rott-samann+-}*{Def.~1.4}.\hfill{\footnotesize{$\blacksquare$}}
\end{example}

The following should be compared to Gigli \cite{gigli2015}*{Cor.~3.4} in positive signature.

\begin{proposition}[Linearity]\label{Pr:El prop} Assume that $\scrM$ is $\QQ$-infinitesimally strictly concave, where $\QQ$ is the conjugate exponent to a nonzero number $\PP$ less than $1$. Then for every $l$-causal function $\u$ as in \eqref{Sub:Horiz},  the quantity $\rmd\cdot(\nabla\u)\,\vert\rmd \u\vert^{\PP-2}$ is $\meas$-a.e.~linear on the real vector space $\FPert(\u)$. That is, for every $f,g\in \FPert(\u)$ and every $\lambda \in\R$,
\begin{align*}
\rmd(\lambda f + g)(\nabla\u)\,\vert\rmd\u\vert^{\PP-2} = \lambda\,\rmd f(\nabla\u)\,\vert\rmd\u\vert^{\PP-2} + \rmd g(\nabla\u)\,\vert\rmd\u\vert^{\PP-2}\quad\meas\textnormal{-a.e.}
\end{align*}
\end{proposition}

\begin{proof} This is a  consequence of  the $\meas$-a.e.~concave dependence of $\smash{\rmd^+\cdot (\nabla\u)\,\vert\rmd\u\vert^{\PP-2}}$ on $\FPert(\u)$ and \cref{Le:Diff quot}.
\end{proof}

\subsection{Proofs pertaining to finite perturbations}\label{App:A} 

For the first fact, we recall the function $\varphi_\QQ$ from \eqref{Eq:varphip} used to define $\u_\QQ := \varphi_\QQ\circ l_\Sigma$, where $\QQ$ is a positive exponent less than $1$.

\begin{lemma}[Repetition of \cref{Le:Invar p}] In the framework  of \cref{Gen Sigm}, the two classes $\Pert_\bc(l_\Sigma,I^+(\Sigma)\cup I^-(\Sigma))$ and $\Pert_\bc(\u_\QQ,I^+(\Sigma)\cup I^-(\Sigma))$ coincide.
\end{lemma}

\begin{proof} This is a consequence of \cref{Pr:Relations}, though not entirely evident, thus we include some details for the  convenience of the reader. 

Thanks to the $l$-causal of $l_\Sigma$ and the disjointness of $I^+(\Sigma)$ and $I^-(\Sigma)$, every element of $\smash{\Pert_\bc(l_\Sigma,I^+(\Sigma)\cup I^-(\Sigma))}$ is the sum of a function in $\smash{\Pert_\bc(l_\Sigma,I^+(\Sigma))}$ and a function in $\smash{\Pert_\bc(l_\Sigma,I^-(\Sigma))}$. The same holds by replacing $l_\Sigma$ by $\smash{\u_\QQ}$. It thus suffices to prove
\begin{align*}
\Pert_\bc(l_\Sigma,I^+(\Sigma)) &= \Pert_\bc(\u_\QQ,I^+(\Sigma)),\\
\Pert_\bc(l_\Sigma,I^-(\Sigma)) &= \Pert_\bc(\u_\QQ,I^-(\Sigma)).
\end{align*}
We concentrate on the first equality, the second follows analogously.

Let $\smash{f\in \Pert_\bc(l_\Sigma,I^+(\Sigma))}$, which is extended trivially to $\mms$ according to \cref{Re:ExtensionII}. Let $U$ be an open set which is compactly contained in $\smash{I^+(\Sigma)}$ and contains the support of $f$. By continuity of $l_\Sigma$ on $\smash{I^+(\Sigma)}$, there is  $\varepsilon\in (0,1/2)$ such that $U$ is contained in $\smash{\{l_\Sigma \geq 2\varepsilon\}}$. Interpolating $\smash{\varphi_\QQ(-\varepsilon)}$ and $\smash{\varphi_\QQ(\varepsilon)}$ linearly and smoothing out the corners, we can construct a strictly increasing smooth function $\varphi$ on $\R$ with Lipschitz continuous inverse with  $\smash{\varphi = \varphi_\QQ}$ on $\R \setminus [-\varepsilon,\varepsilon]$. Then \cref{Pr:Relations} implies $\smash{f\in \Pert_\bc(\v, I^+(\Sigma))}$, where
\begin{align*}
\v := \varphi\circ l_\Sigma.
\end{align*}

We replace $\varphi$ by $\smash{\varphi_\QQ}$ as follows. Assume $\v + \tau f$ is $l$-causal on $\smash{I^+(\Sigma)}$ for some given $\tau >0$. Let $\smash{x,y\in I^+(\Sigma)}$ satisfy $x\leq y$. We distinguish four cases. 
\begin{itemize}
\item If $l_\Sigma(x) \geq 2\varepsilon$ and $l_\Sigma(y) \geq 2\varepsilon$, clearly
\begin{align*}
\u_\QQ(y) - \u_\QQ(x) + \tau\big[f(y) - f(x)\big]  = \v(y) - \v(x) + \tau\big[f(y) - f(x)\big],
\end{align*}
which is nonnegative. 
\item If $l_\Sigma(x) < 2\varepsilon$ and $l_\Sigma(y) < 2\varepsilon$, since the support of $f$ has empty intersection with the set $\{l_\Sigma < 2\varepsilon\}$ by construction, we obtain
\begin{align*}
\u_\QQ(y) - \u_\QQ(x) + \tau\big[f(y) - f(x)\big] = \u_\QQ(y) - \u_\QQ(x),
\end{align*}
which is nonnegative by the $l$-causal of $\smash{\u_\QQ}$.
\item If $l_\Sigma(x) < 2 \varepsilon$ yet $l_\Sigma(y) \geq 2\varepsilon$, let $\gamma$ be a continuous $l$-causal curve from $x$ to $y$. Then there exists $t\in (0,1)$ such that $l_\Sigma(\gamma_t)\in (l_\Sigma(x)\vee\varepsilon, 2\varepsilon)$. In conjunction with the $l$-causal of $\smash{\u_\QQ}$, this entails nonnegativity of the rightmost side of
\begin{align*}
&\u_\QQ(y) - \u_\QQ(x) + \tau\big[f(y) - f(x)\big]\\
&\qquad\qquad = \u_\QQ(y) - \u_\QQ(x) + \tau\,f(y)\\
&\qquad\qquad = \v(y) - \v(\gamma_t) + \tau\big[f(y) - f(\gamma_t)\big] + \u_\QQ(\gamma_t) - \u_\QQ(x).
\end{align*}
\item The situation $l_\Sigma(x) \geq 2\varepsilon$ yet $l_\Sigma(y) < 2\varepsilon$ cannot occur by \cref{Cor:Steepness}.
\end{itemize}

An analogous argument replacing $\smash{\varphi_\QQ}$ with its inverse gives the converse.
\end{proof}

\begin{lemma}[Repetition of \cref{Le:VsTop}] In the framework of \cref{Gen Sigm}, assume that $\smash{l_\Sigma}$ is topologically locally anti-Lipschitz on $\smash{I^\pm(\Sigma)}$ according to \cref{Def:KS}. Then  $\smash{\scrC_\Sigma^\pm}$ is uniformly dense in $\smash{\Cont_\comp(I^\pm(\Sigma))}$.
\end{lemma}

\begin{proof} We only prove uniform density of $\smash{\scrC_\Sigma^+}$ in $\smash{\Cont_\comp(I^+(\Sigma))}$, the argument relative to the chronological past of $\Sigma$ is analogous.

Using the symmetry and boundedness of all elements of $\smash{\scrC_\Sigma^+}$, $\smash{\scrC_\Sigma^+}$ is easily verified to be an algebra. By using distance functions from appropriate points (which are continuous, hence the additional topological component in \cref{Def:KS}) from the definition of topological local anti-Lipschitzness, a cutoff-argument, and the second part of \cref{Re:ExtensionII}, the class $\smash{\scrC_\Sigma^+}$ separates points and vanishes nowhere on $I^+(\Sigma)$. The statement then follows from the Stone--Weierstraß theorem.
\end{proof}

\bibliographystyle{amsrefs}


\end{document}